\documentclass[aos]{imsart}

\RequirePackage{amsthm,amsmath,amsfonts,amssymb}
\RequirePackage[numbers]{natbib}
\RequirePackage[colorlinks,citecolor=blue,urlcolor=blue]{hyperref}
\RequirePackage{graphicx}
\usepackage{booktabs}

\startlocaldefs
\theoremstyle{plain}
\newtheorem{theorem}{Theorem}
\newtheorem{cor}{Corollary}
\theoremstyle{definition}
\newtheorem{remark}{Remark} 
\newcommand{\E}{\mathbb{E}}
\newcommand{\bbP}{\mathbb{P}}
\newcommand{\Var}{\mathrm{Var}}
\newcommand{\Cov}{\mathrm{Cov}}
\newcommand{\D}{\mathcal{D}}

\newcommand{\biota}{\boldsymbol{\iota}}
\newcommand{\bx}{\boldsymbol{x}}
\newcommand{\by}{\boldsymbol{y}}
\newcommand{\bz}{\boldsymbol{z}}
\newcommand{\bX}{\boldsymbol{X}}
\newcommand{\boldm}{\boldsymbol{m}}
\newcommand{\bO}{\boldsymbol{O}}
\newcommand{\bo}{\boldsymbol{o}}

\newcommand{\cP}{\mathcal{P}}
\endlocaldefs

\begin{document}

\begin{frontmatter}
\title{Nuisance Function Tuning and Sample Splitting for Optimally Estimating a Doubly Robust Functional}
\runtitle{On Nuisance Function Tuning and Sample Splitting}

\begin{aug}
\author[A]{\fnms{Sean}~\snm{McGrath}\ead[label=e1]{mcgrathsean@wustl.edu}}
\author[B]{\fnms{Rajarshi}~\snm{Mukherjee}\ead[label=e2]{ram521@mail.harvard.edu}}
\address[A]{Department of Statistics and Data Science, Washington University in St. Louis\printead[presep={,\ }]{e1}}

\address[B]{Department of Biostatistics, Harvard University\printead[presep={,\ }]{e2}}
\end{aug}

\begin{abstract}
Estimators of doubly robust functionals typically rely on estimating two complex nuisance functions, such as the propensity score and conditional outcome mean for the average treatment effect functional. We consider the problem of how to estimate nuisance functions to obtain optimal rates of convergence for a doubly robust nonparametric functional that has witnessed applications across the causal inference and conditional independence testing literature. For several plug-in estimators and a first-order bias-corrected estimator, we illustrate the interplay between different tuning parameter choices for the nuisance function estimators and sample splitting strategies on the optimal rate of estimating the functional of interest. For each of these estimators and each sample splitting strategy, we show the necessity to either undersmooth or oversmooth the nuisance function estimators under low regularity conditions to obtain optimal rates of convergence for the functional of interest. Unlike the existing literature, we show that plug-in and first-order bias-corrected estimators can achieve minimax rates of convergence across all Hölder smoothness classes of the nuisance functions by careful combinations of sample splitting and nuisance function tuning strategies. We complement these results with numerical simulations illustrating the impact of different nuisance function tuning and sample splitting strategies.
\end{abstract}

\begin{keyword}[class=MSC]
\kwdgroup[type=primary]{\kwd{62C20}
\kwd{62G20}}
\kwdgroup[type=secondary]{\kwd{62D20}
\kwd{62G08}}
\end{keyword}

\begin{keyword}
\kwd{nuisance function tuning}
\kwd{sample splitting}
\kwd{doubly robust functional}
\end{keyword}

\end{frontmatter}

\section{Introduction}

Inferring functionals of observed data distributions where one needs to estimate complex nuisance parameters and functions is a staple in the modern statistical paradigm. In this regard, a class of problems has gathered recent momentum through a combination of ideas from the machine learning literature and modern semiparametric theory. This class of problems is characterized by the presence of two complex nuisance functions which are often estimated through state-of-the-art machine learning techniques before being subjected to the lens of semiparametric de-biasing ideas to yield asymptotically normal inference of the functional of interest. These methods -- often referred to as Double Machine Learning (DML) methods -- furnish a rich tapestry of examples with instances spread across causal inference, missing data problems, high dimensional regression, and conditional independence testing among many others \cite{bang2005doubly,chernozhukov2018double,shah2020hardness}. As a specific example, when estimating the average treatment effect of a binary exposure on an outcome of interest, one often estimates the propensity score and outcome regression before plugging them into the (now) classical doubly robust estimator of the average treatment effect (see Bang and Robins \cite{bang2005doubly} and Chernozhukov et al. \cite{chernozhukov2018double}). The appeal of this approach derives from the fact that one only needs to control the product of the errors for estimating the two nuisance functions instead of individual control of the errors, thereby justifying the doubly robust terminology. Recent literature has explored a broad set of sufficient conditions under which functionals of data-generating distributions can benefit from such DML methods. This class of functionals has been referred to by Robins et al. \cite{robins2008higher} and Rotnitzky et al. \cite{rotnitzky2021characterization} as doubly robust functionals. 

The approach described above for estimating doubly robust functionals requires estimating two nuisance functions (e.g. propensity score and outcome regression functions for causal inference type problems) through machine learning algorithms which require appropriately setting tuning parameters. Examples of such tuning parameters include regularizing parameters in penalized regression methods (e.g., Lasso) and bandwidths/resolutions in nonparametric regression methods (e.g., kernel regression). At this point, the philosophical scope of the investigation can be divided into two broad themes. In one regime, the researcher uses an optimal tuning for the estimation of nuisance functions and relies on semiparametric theory of bias correction to yield the best possible estimation of the functional of interest that can be obtained using these nuisance function estimators \cite{robins2008higher,robins2017higher,liu2021adaptive}. In a parallel tenet, researchers choose a specific estimator of the functional of interest and aim to obtain the optimal tuning/estimation of the nuisance functions that produces the best estimation of the functional under study. Several works have identified benefits of undersmoothing nuisance function estimators and performing sample splitting (i.e., estimating the nuisance function(s) from a part of the sample and using them for the functional estimation in a different part of the sample) for this second perspective \cite{hall1992effect,newey1998undersmoothing,gine2008simple,paninski2008undersmoothed,newey2018cross,van2019efficient,van2019causal,fisher2023three,kennedy2024minimax,balakrishnan2026fundamental, bruns2023augmented,mcgrath2025optimal}. In this regard, this line of research has mostly focused on demonstrating a sufficiency of the aforementioned undersmoothing principle. However, it remains unclear whether undersmoothing is also necessary. Indeed, the development of such a theory will depend on the functional of interest, type of estimator of the functional (which involves the nuisance function estimators in its construction), function space for the nuisance functions, and sample splitting strategy.  In this paper, to take the first steps towards such a theory, we consider this second perspective and explore in detail the estimation of a specific doubly robust functional concerning its interplay with the estimator chosen by the investigator, sample splitting strategies to estimate the nuisance functions, smoothness of the nuisance functions, the tuning of the nuisance function estimators, and the best possible mean squared error one might obtain as a result for estimating the functional. Since we aim to explore a complicated canvas of interplay, we keep our discussions focused through a specific functional, nuisance function estimation strategy, and type of function space for the nuisance functions that we describe below.

Specifically, in this paper, we explore the interplay between different types of estimators (e.g., plug-in type estimators, first-order bias-corrected estimator), tuning parameter choices, and sample splitting strategies through the lens of a doubly robust functional which has gained recent popularity in both the causal inference and conditional independence testing literature. We consider observing $n$ i.i.d.\ copies of $\bO = (\bX, A, Y) \sim \bbP$ where $\bX \in \mathbb{R}^d$, $A \in \mathbb{R}$, and $Y \in \mathbb{R}$
and focus on the functional  
\begin{equation*}
    \psi(\bbP) = \E_\bbP[\Cov_\bbP(A, Y|\bX)].
\end{equation*}
As we elaborate on in Section \ref{sec: motivation}, $\psi(\bbP)$ has connections to an average treatment effect functional when we consider $\bX$ to be a vector of baseline covariates, $A$ a binary treatment, and $Y$ a binary outcome of interest. Indeed, estimating $\psi(\bbP)$ is challenging due to its dependence on the nuisance functions $p(\bx):=\E_\bbP[A | \bX = \bx]$ and $b(\bx):=\E_\bbP[Y | \bX = \bx]$ when one does not assume parametric function classes, such as Hölder classes in our case. The main focus of this paper is to consider estimators of the form $\hat{\psi}(\{\bO_i\}_{i\in \mathcal{D}_3};\hat{p}_{\mathcal{D}_{1}}, \hat{b}_{\mathcal{D}_2})$ where one uses estimators of $p,b$ based on sample indices $\mathcal{D}_1,\mathcal{D}_2\subseteq \{1,\ldots,n\}$ respectively and uses them in estimating $\psi(\bbP)$ with samples indexed by $\mathcal{D}_3$  (see Section \ref{sec: estimation strategies} for details). By changing the overlap between $\mathcal{D}_j$'s for $j=1,2,3$ we present a comprehensive study of the interplay between the optimal tuning of estimators of $p$ and $b$, sample splitting, and optimal rates of estimation of $\psi(\bbP)$. Specifically, the main contributions of this paper are listed below: 
\begin{enumerate}
    \item We derive matching necessary and sufficient conditions for tuning the nuisance function estimators to achieve optimal estimation rates for $\psi(\bbP)$ under various regularity regimes and sample splitting strategies. Specifically, we describe when one can tune the nuisance function estimators with prediction-optimal resolutions to obtain optimal rates of estimation for $\psi(\bbP)$, as well as when using prediction-\emph{suboptimal} resolutions (e.g., undersmoothing and oversmoothing) is necessary to obtain optimal rates. We ultimately show that all of the estimators we consider require performing some degree of undersmoothing or oversmoothing in low regularity regimes. We summarize these results in Figures \ref{fig:tuning double}, \ref{fig:tuning single}, and \ref{fig:tuning no} once we introduce the necessary notation and setup. 

    \item We describe the impact of different sample splitting techniques on whether the plug-in and first-order type estimators can be minimax optimal, which we illustrate in Figure \ref{fig:minimax} once we introduce the notation and setup. When not performing sample splitting, none of the estimators can achieve minimax optimal rates of convergence for any resolution choice in low regularity regimes. However, when using single or double sample splitting, some of the estimators (but not all) can be minimax optimal when appropriately tuning the nuisance function estimators. Further, we show that these different sample splitting techniques impose different requirements on how to optimally tune the nuisance function estimators. We also give a detailed description of how the best rates of estimation for these estimators change when comparing the different sample splitting strategies.     

    \item We perform numerical simulations assessing the impact of nuisance function tuning and sample splitting strategies for the plug-in estimators and the first-order estimator across different regularity regimes ($\sqrt{n}$ and non-$\sqrt{n}$ regimes). In line with our  theoretical results, we find that optimal nuisance function tuning requires undersmoothing and/or oversmoothing in low regularity regimes, and that doing so leads to substantial improvements in estimating $\psi(\bbP)$ compared to prediction-optimal tuning.   
\end{enumerate}

The rest of the paper is organized as follows. In Section \ref{sec: setup}, we further motivate the functional $\psi(\bbP)$, introduce the mathematical framework under which we will operate, and discuss estimation strategies. We present the main results of the paper in Section \ref{sec: results}. In Section~\ref{sec: simulations}, we present the results of our numerical simulations. We discuss our results and directions for future work in Section \ref{sec: discussion}. In Appendix A of the Supplementary Material \cite{mcgrath2026supplement}, we review some results on the theory of function spaces and orthonormal wavelet bases which we use throughout the paper. Appendix B contains some extensions of our results. The proofs of our results are given in Appendices C--J. Appendix K contains additional results from our numerical simulations.

\subsection{Notation}

Our results focus on the rates of convergence of estimators, which involves the following asymptotic notation. For nonnegative sequences $(a_n)_{n = 1}^{\infty}$ and $(b_n)_{n= 1}^{\infty}$, let $a_n \lesssim b_n$ and $a_n = O(b_n)$ indicate that there exists a positive constant $C$ such that $a_n \leq Cb_n$ for all $n \geq n_0$. Similarly, let $a_n \gtrsim b_n$ indicate that $a_n \geq C b_n$ for all $n \geq n_0$ (i.e., $b_n \lesssim a_n$). Let $a_n \asymp b_n$ indicate that $a_n \lesssim b_n$ and $b_n \lesssim a_n$. Moreover, let $a_n \ll b_n$ indicate that $\lim_{n \to \infty} a_n / b_n = 0$ and $a_n \gg b_n$ indicate that $\lim_{n \to \infty} a_n / b_n = \infty$. 

For $1 \leq p < \infty$, let $L^p([0,1]^d)$ denote the space of functions $h:[0,1]^d \to \mathbb{R}$ such that $\| h \|_p := (\int_{[0,1]^d} |h(\bx)|^{p} d\bx)^{1/p} < \infty$. Similarly, let $L^{\infty}([0,1]^d)$ denote the space of functions $h:[0,1]^d \to \mathbb{R}$ such that $\| h \|_{\infty} := \sup_{\bx \in [0,1]^d} |h(\bx)| < \infty$. We consider Hölder type regularity conditions on the nuisance functions. Let $H(\alpha, M)$ denote the Hölder ball with exponent $\alpha$ and radius $M$. See Appendix A for details.

For $a, b \in \mathbb{R}$, we let $a \vee b := \max(a, b)$ and $a \wedge b := \min(a, b)$. Throughout the paper, all integrals are over $[0, 1]^d$ unless otherwise stated.

The estimators we study are based on splitting the sample into subsamples, which involves the following notation. Suppose we split the sample into $m$ disjoint subsamples denoted by $\D_1, \dots, \D_m$. For index set $\mathcal{J} \subseteq \{1, \dots, m\}$, let $\E_{\bbP, \mathcal{J}}$ and $\Var_{\bbP, \mathcal{J}}$ denote the expectation and variance under distribution $\bbP$ over sets $\cup_{\ell \in \mathcal{J}} \D_\ell$ conditional on $\cup_{\ell \not\in \mathcal{J}} \D_\ell$. 

\section{Setup} \label{sec: setup}

\subsection{Motivation and background} \label{sec: motivation}

The functional $\E_\bbP[\Cov_\bbP(A, Y|\bX)]$ arose  from the causal inference literature \cite{robins1992estimating, crump2009dealing}. Specifically, consider an epidemiological study where $\bX$ denotes a vector of baseline covariates, $A$ denotes a binary treatment, and $Y$ denotes a binary outcome of interest. Under the standard assumptions of no unmeasured confounding, positivity, and consistency \cite{hernan2020causal}, the average treatment effect for the population with $\bX = \bx$ is given by $ c(\bx) = \mathbb{E}_\bbP[Y | A = 1, \bX = \bx] - \mathbb{E}_\bbP[Y | A = 0, \bX = \bx].$
Thereafter, the variance-weighted average treatment effect is given by
$ \tau := \E_\bbP \left(\frac{ \Var_\bbP(A | \bX) c(\bX) }{\E_\bbP \left[ \Var_\bbP(A | \bX) \right]}\right)$ and a straightforward calculation shows that $\tau$ can be expressed in terms of $\E_\bbP \left[ \Cov_\bbP(A, Y | \bX) \right]$ as $ \tau = \frac{\E_\bbP \left[ \Cov_\bbP(A, Y | \bX) \right]}{\E_\bbP \left[ \Var_\bbP(A | \bX) \right]}.$ It is worth noting that exploring inference on $\E_\bbP \left[ \Cov_\bbP(A, Y | \bX) \right]$ sheds light on strategies for $\E_\bbP \left[ \Var_\bbP(A | \bX) \right]$. Indeed, as noted by \cite{robins2008higher}, point estimators and confidence intervals for $\E_\bbP \left[ \Cov_\bbP(A, Y | \bX) \right]$ can be used to construct point estimators and confidence intervals for $\tau$. Apart from the variance-weighted average treatment effect, $\E_\bbP \left[ \Cov_\bbP(A, Y | \bX) \right]$ has recently appeared in other causal functionals such as in measures of causal influence \cite{diaz2023non} and in marginal interventional effects \cite{zhou2022marginal}. Beyond motivations in causal studies, $\E_\bbP \left[ \Cov_\bbP(A, Y | \bX) \right]$ has found applications in conditional independence testing \cite{shah2020hardness} and also appeared as one of the classical examples of doubly robust functionals \cite{rotnitzky2021characterization} in semiparametric theory.

To describe the literature for inferring $\psi(\bbP)$, we focus on H\"{o}lder type regularity conditions. Let $p(\bx) := \E_\bbP[A | \bX = \bx]$ and $b(\bx) := \E_\bbP[Y | \bX = \bx]$. In the parlance of the causal inference literature, $p$ is referred to as the propensity score, and $b$ as the outcome regression function. Let $\cP_{(\alpha,\beta)}$ denote the set of distributions $\bbP$ such that $A, Y \in [C_1, C_2]$  $\text{a.s.} \, \bbP$ and $p \in H(\alpha, M), \, b \in H(\beta, M)$ where $C_1, C_2 \in \mathbb{R}$ and $\alpha, \beta, M \in \mathbb{R}^+$ are known constants.\footnote{Instead of assuming that $A, Y$ are bounded, our results can be seen to hold when assuming that $\mathbb{E}_{\mathbb{P}}[A^4|\bX = \bx]$ and $\mathbb{E}_{\mathbb{P}}[Y^4|\bX = \bx]$ are uniformly bounded on $[0, 1]^d$.} For ease of illustration, we also let $\bX \sim \text{Uniform}([0,1]^d)$ although it is straightforward to see that our results hold under Hölder type regularity conditions on the density of $\bX$.\footnote{Specifically, let $\bX$ have a known density $f$ with respect to the Lebesgue measure on $\mathbb{R}^d$ that has a compact support, which we take to be $[0, 1]^d$. Our results can easily be seen to hold for any $f$ such that $ f \in H(\gamma, M) $ where $\gamma \geq \alpha \vee \beta$ and $f(\bx) \in [M_1, M_2]$ for all $\bx \in [0, 1]^d$ where $M_1, M_2 \in \mathbb{R}^{+}$ are known constants. In fact, the proofs of all of the upper bounds are written out under these conditions, and the proofs of the lower bounds can be seen to hold when using a wavelet basis orthonormalized with respect to $f$ rather than the Lebesgue measure (see Section \ref{sec: nuisance function estimators}). We elaborate on our motivation to consider $f$ known and extend our results for unknown $f$ in Appendix B.}

Robins et al.\ \cite{robins2009semiparametric,robins2008higher} established that the minimax risk for estimating $\psi(\bbP)$ associated with the model $\cP_{(\alpha,\beta)}$ is
\begin{equation}  \label{eq: minimax rate}
\inf_{\hat{\psi}}\sup_{\bbP\in \cP_{(\alpha,\beta)}}\mathbb{E}_\mathbb{\bbP}\left[(\hat{\psi}-\psi(\bbP))^2\right] \asymp \begin{cases}
n^{-1}, & \frac{\alpha + \beta}{2} \geq \frac{d}{4}\\
n^{-\frac{4\alpha + 4\beta}{2\alpha + 2\beta +d}}, & \frac{\alpha + \beta}{2} < \frac{d}{4}
\end{cases}.
\end{equation}
General rate-optimal estimators of $\psi(\bbP)$ can be constructed based on higher-order bias corrections of plug-in estimators and sample splitting \cite{robins2008higher}, which we describe next along with other possible strategies. Finally, it is worth mentioning that for $\frac{\alpha+\beta}{2}>\frac{d}{4}$, Newey and Robins \cite{newey2018cross} established that a first-order bias-corrected estimator (see Section \ref{sec: estimation strategies} for more details) achieves this sharp optimality in this regime. In contrast, our results offer a complementary set of evidence in the non-$\sqrt{n}$-regime.   

\subsection{Estimation strategies} \label{sec: estimation strategies}

We now gather possible strategies and intuitions for estimating $\psi(\bbP)$ from the literature which we will eventually compare in subsequent sections. Indeed, $\psi(\bbP)$ depends on the nuisance functions $p$ and $b$ and we can describe possible estimation procedures which depend on suitable estimators of these unknown quantities. In increasing order of sophistication, these can be described as plug-in estimation, one-step bias-corrected estimation, and higher-order bias-corrected estimation. It is now standard in the literature \cite{bickel1982adaptive, bickel1988estimating,  powell1989semiparametric,laurent1996efficient, zheng2010asymptotic,chernozhukov2018double, newey2018cross, kennedy2023towards, kennedy2024minimax} to employ additional ideas such as sample splitting to estimate the nuisance functions and thereby avoid overfitting and simplify the statistical analysis. In order to provide a streamlined analysis of these principles, below we focus on a specific class of nuisance function estimation arising from wavelet projections and set up the corresponding notation and estimation strategies.

\subsubsection{Sample Splitting} \label{sec: sample splitting}

We consider splitting the sample into disjoint subsamples, $\{\D_1, \dots, \D_m\}$. Without meaningful loss of generality (in terms of optimal rates of convergence of the estimators of $\psi(\bbP)$), we take the subsamples to be of equal size. The number of subsamples, $m$, will depend on the choice of the estimator of $\psi(\bbP)$. For notational simplicity, we consider the total size of the sample to be $m n$ so that each of the subsamples is of size $n$. 

We consider two different sample splitting techniques which differ based on whether the nuisance functions are estimated in the same subsample versus separate subsamples. In the \emph{single sample splitting} case, the nuisance functions are estimated in one subsample ($\D_j$) and $\psi(\bbP)$ is estimated from the remaining data ($\cup_{\ell \neq j} \D_\ell$). In the \emph{double sample splitting} case, the nuisance functions are estimated in separate subsamples themselves. We also consider not performing sample splitting, in which case the nuisance functions and $\psi(\bbP)$ are all estimated from $\D_1$.

In practice, one may be concerned about a loss of efficiency when performing sample splitting because it involves estimating the nuisance functions and $\psi(\bbP)$ based on a subsample of the full data set. A commonly used approach for improving efficiency when performing sample splitting is the practice of so-called \emph{cross-fitting}. Cross-fitting involves estimating $\psi(\bbP)$ by exchanging the roles of $\{\D_1, \dots, \D_m\}$ and then averaging the $m!$ estimates of $\psi(\bbP)$ to obtain a final estimate. In Sections \ref{sec: double ss} and \ref{sec: single ss}, we discuss the rates of estimation of $\psi(\bbP)$ under cross-fitting strategies as well.

\subsubsection{Estimators of nuisance functions} \label{sec: nuisance function estimators}
Our strategy for estimating the nuisance functions can be described as follows. The approximate wavelet projection estimator of $p$ with resolution $k_1$ that is fit in subsample $\D_j$ is given by
\begin{equation*}
    \hat{p}_{k_1}^{(j)}(\bx) = \frac{1}{n} \sum_{i \in \mathcal{D}_j} A_i K_{V_{k_1}}(\bX_i,\bx), \qquad \bx \in [0, 1]^d
\end{equation*}
where $K_{V_{k_1}}$ denotes a kernel of an orthogonal projection onto the linear subspace $V_{k_1}$ of $L^2([0,1]^d)$. For an orthonormal basis $\{\Phi_{k_1 \boldm}\}_{\boldm \in \mathcal{Z}_{k_1}}$ of $V_{k_1}$, the kernel $K_{V_{k_1}}$ is given by
\begin{equation*}
    K_{V_{k_1}}(\bx, \by) =  \sum_{\boldm \in \mathcal{Z}_{k_1}} \Phi_{k_1 \boldm}(\bx) \Phi_{k_1 \boldm}(\by), \qquad \bx, \by \in [0, 1]^d.
\end{equation*}
In Appendix A, we define the orthonormal basis we adopt. Intuitively, $\hat{p}_{k_1}^{(j)}$ estimates the projection of $p$ onto the span of the wavelet basis functions at the $k_1$th resolution level of a multiresolution analysis of $L^2([0,1]^d)$. We estimate $b$ analogously by
\begin{equation*}
    \hat{b}_{k_2}^{(j^{\prime})}(\bx) = \frac{1}{n} \sum_{i \in \mathcal{D}_{j^{\prime}}} Y_i K_{V_{k_2}}(\bX_i,\bx), \qquad \bx \in [0, 1]^d.
\end{equation*}

The choice of $k_1$ and $k_2$ controls the bias-variance trade-off of the estimators of $p$ and $b$. When choosing resolutions for $\hat{p}_{k_1}^{(j)}$ and $\hat{b}_{k_2}^{(j^{\prime})}$ to minimize their respective mean (integrated) squared errors, i.e. for minimizing the rates of convergence of 
\begin{align*}
    & \sup_{\bbP \in \cP_{(\alpha, \beta)}} \mathbb{E}_\bbP \int (\hat{p}_{k_1}^{(j)}(\bx) - p(\bx))^2 d\bx \quad \text{and} \quad \sup_{\bbP \in \cP_{(\alpha, \beta)}} \mathbb{E}_\bbP \int (\hat{b}_{k_2}^{(j^{\prime})}(\bx) - b(\bx))^2 d\bx,
\end{align*}
the optimal choices are $k_1^{\text{pred}} = c_1 n^\frac{d}{2\alpha + d}$ and $k_2^{\text{pred}} = c_2 n^\frac{d}{2\beta + d}$, respectively, for fixed $c_1, c_2 > 0$ (Härdle et al. \cite{hardle2012wavelets}, Chapter 10). We refer to such choices of $k_1$ and $k_2$ as \emph{prediction-optimal} resolutions. Choosing $k_1$ that grows faster than $k_1^{\text{pred}}$ is referred to as \emph{undersmoothing}, as the squared bias of $\hat{p}_{k_1}^{(j)}$ shrinks faster than its variance; Choosing $k_1$ that grows slower than $k_1^{\text{pred}}$ is referred to as \emph{oversmoothing}. Undersmoothing and oversmoothing are defined analogously for $k_2$. As we will see in Section \ref{sec: results}, the choice of $k_1$ and $k_2$ will be critical in constructing rate optimal estimators of $\psi(\bbP)$. 

Finally, it is worth mentioning that $\hat{p}_{k_1}^{(j)},\hat{b}_{k_2}^{(j')}$ described above use wavelet basis orthonormal with respect to the Lebesgue measure since $\bX \sim \text{Uniform}([0,1]^d)$ in our set up. As described earlier, one can easily extend the results to the known density of $\bX$ by simple modifications of our arguments and by using an orthonormalization of the basis functions under the known density. The case of unknown covariate density presents several subtleties in terms of information-theoretic limits of the problem, the nature of sample splitting, and the choice of estimators of nuisance functions. We defer these discussions to Appendix B.

\subsubsection{Estimators of \texorpdfstring{$\psi(\bbP)$}{psi}} \label{sec: estimators of psi}

We next describe strategies for constructing estimators of $\psi(\bbP)$ based on these nuisance function estimators.

\paragraph{Plug-in type estimators}

We consider the following three plug-in type estimators. 

Considering the representation $\psi(\bbP) = \E_{\bbP}[AY] - \E_{\bbP}[p(\bX)b(\bX)]$, one can estimate $\E_{\bbP}[AY]$ at a $\sqrt{n}$ rate by taking the sample mean of the $A_iY_i$. The main challenge lies in estimating $\E_{\bbP}[p(\bX)b(\bX)]$ due to its dependence on the nuisance functions. The first plug-in type estimators we consider are based on different approaches for estimating $\E_{\bbP}[p(\bX)b(\bX)]$. The \emph{integral-based plug-in estimator} of $\psi(\bbP)$ is given by
\begin{equation*}
    \hat{\psi}_{k_1, k_2}^{\mathrm{INT}} = \frac{1}{n}\sum_{i \in \D_{\ell_3}} A_i Y_i -  \int \hat{p}_{k_1}^{(\ell_1)}(\bx)\hat{b}_{k_2}^{(\ell_2)}(\bx)d\bx .
\end{equation*}
In practice, evaluating the integral in $\hat{\psi}_{k_1, k_2}^{\mathrm{INT}}$ may be challenging when $d$ is moderately large. Naturally, one may consider a  plug-in estimator of $\psi(\bbP)$  which replaces the integral appearing in $\hat{\psi}_{k_1, k_2}^{\mathrm{INT}}$ with Monte Carlo integration over a subsample of the data. The \emph{Monte Carlo-based plug-in estimator} of $\psi(\bbP)$ is given by
\begin{equation*}
    \hat{\psi}^{\mathrm{MC}}_{k_1, k_2} = \frac{1}{n}\sum_{i \in \D_{\ell_3}} A_i Y_i -  \frac{1}{n}\sum_{i \in \D_{\ell_4}} \hat{p}^{(\ell_1)}_{k_1}(\bX_i)\hat{b}^{(\ell_2)}_{k_2}(\bX_i).
\end{equation*}

We also consider plug-in estimators that only depend on one nuisance function. Considering the representation $\psi(\bbP) = \E_{\bbP}[A(Y-b(\bX))]$, the \emph{Newey and Robins plug-in estimator} of $\psi(\bbP)$ is given by
\begin{equation*}
   \hat{\psi}_{k}^{\mathrm{NR}} = \frac{1}{n} \sum_{i \in \D_{\ell_2}} A_i(Y_i - \hat{b}_k^{(\ell_1)}(\bX_i)).
\end{equation*}
We refer to this estimator as the Newey and Robins plug-in estimator because a similar estimator was considered by Newey and Robins \cite{newey2018cross}. However, it is important to note that Newey and Robins \cite{newey2018cross} used regression splines to estimate $b$ rather than approximate wavelet projections as we consider. Some of the implications of this distinction will be discussed in Section \ref{sec: single ss}. One could also consider a similar estimator by exchanging the roles of $A$ and $Y$, i.e. $\frac{1}{n} \sum_{i \in \D_{\ell_2}} Y_i(A_i - \hat{p}_k^{(\ell_1)}(\bX_i))$. As we will note in Section \ref{sec: single ss one}, analogous results can be seen to hold for this estimator. 

\paragraph{First-order bias-corrected estimators}

Consider an initial estimator $\hat{\bbP}$ of $\bbP$ based on some nonparametric estimators $\hat{p}$ and $\hat{b}$. Typically, one chooses rate-optimal estimators of the nuisance functions and the resulting plug-in estimator, $\psi(\hat{\bbP})$, has a sub-optimal rate of estimation in low regularity regimes. Consequently, one may consider a first-order bias-corrected estimator. Such estimators are based on a von Mises expansion of $\psi$
\begin{equation*}
   \psi(\bbP) =  \psi(\hat{\bbP}) + \int \mathrm{IF}_{\hat{\bbP}}(\bo) d\bbP(\bo) + R_2(\hat{\bbP}, \bbP)
\end{equation*}
where $\mathrm{IF}_{\bbP}(\bo) = (y-b(\bx))(a - p(\bx)) - \psi(\bbP)$ is the first-order influence function and $R_2(\hat{\bbP}, \bbP)$ is a remainder term \cite{kennedy2024semiparametric, tsiatis2006semiparametric, robins2008higher}. Adding an estimate of the derivative term to $\psi(\hat{\bbP})$ gives rise to the first-order bias-corrected estimator $\hat{\psi}^{\mathrm{IF,1}}$
\begin{equation*}
    \hat{\psi}^{\mathrm{IF,1}} := \psi(\hat{\bbP})+\frac{1}{n}\sum\limits_{i=1}^n\mathrm{IF}_{\hat{\bbP}}(\bO_i)  = \frac{1}{n} \sum_{i = 1}^n (A_i - \hat{p}(\bX_i))(Y_i - \hat{b}(\bX_i)).
\end{equation*}
In particular, we consider the following first-order estimator
\begin{equation*}
    \hat{\psi}^{\mathrm{IF}}_{k_1, k_2} = \frac{1}{n} \sum_{i \in \D_{\ell_3}} (A_i - \hat{p}^{(\ell_1)}_{k_1}(\bX_i))(Y_i - \hat{b}^{(\ell_2)}_{k_2}(\bX_i)).
\end{equation*}

First-order bias-corrected estimators of this nature have witnessed immense popularity owing to their attractive doubly robust properties \cite{bang2005doubly,chernozhukov2018double} and the implied flexibility thereof to use off-the-shelf machine learning tools to estimate the nuisance functions $p$ and $b$. This flexibility, however, comes at a subtle price when viewed through the lens of minimax optimality in the non-$\sqrt{n}$ rate of convergence regime. Indeed, off-the-shelf machine learning methods are typically aimed at estimating $p$ and $b$ optimally, and, as we show in our main results, such a strategy can imply sub-optimal rates of convergence for estimating $\psi(\mathbb{P})$.

\paragraph{Higher-order estimators}

Although we do not analyze the following estimator (which has been now thoroughly analyzed elsewhere \cite{robins2008higher}), we provide the last rung of the ladder for the sake of completeness and also to benchmark our somewhat simpler estimators compared to existing ones. In particular, to further reduce the bias of the first-order estimator and achieve rate optimality in low regularity regimes, higher-order bias corrections can be performed \cite{robins2008higher, robins2017minimax, robins2017higher}. 
In this regard, Robins et al. \cite{robins2008higher} showed that one version of a  second-order sample split estimator of $\psi(\bbP)$ is given by
\begin{align*}
    \hat{\psi}^{\mathrm{IF,2}}_k & = \frac{1}{n} \sum_{i \in \D_2} (A_i - \hat{p}^{(1)}(\bX_i))(Y_i - \hat{b}^{(1)}(\bX_i)) \\ & \quad - \frac{1}{n(n-1)} \sum_{\substack{i_1, i_2 \in \D_2 \\ i_1 \neq i_2}} ( A_{i_1} - \hat{p}^{(1)}(\bX_{i_1}) ) K_{V_k}(\bX_{i_1}, \bX_{i_2}) ( Y_{i_2} - \hat{b}^{(1)}(\bX_{i_2}) ).
\end{align*}
It has been shown that $\hat{\psi}^{\mathrm{IF,2}}_k$ is rate optimal for a suitably chosen resolution $k$ provided that one tunes $\hat{p}^{(1)}$ and $\hat{b}^{(1)}$ so that they are rate optimal. In fact, one can be agnostic about the tuning of $\hat{p}^{(1)}$ and $\hat{b}^{(1)}$ by performing further bias corrections on $\hat{\psi}^{\mathrm{IF,2}}_k$ (see Robins et al. \cite{robins2017higher} and Remark 7 in Liu et al. \cite{liu2017semiparametric}). However, our work considers a different estimation strategy based on optimally tuning estimators of $p$ and $b$ for downstream inference on $\psi(\bbP)$ when using plug-in and first-order estimators.

\section{Main results} \label{sec: results}

Before presenting the theoretical results in full detail, we first give a high-level summary of the impact of sample splitting and optimal nuisance function tuning. 
    
The impact of sample splitting is largely seen in the reduction of the bias of the estimator of $\psi(\bbP)$. For any $\bbP \in \cP_{(\alpha,\beta)}$ and each estimator $\hat{\psi}_{k_1, k_2}$, we can decompose its bias into three parts: 
    \begin{equation*}
        \E_\bbP\left(\hat{\psi}_{k_1, k_2} - \psi(\bbP)\right) = g_{1, \bbP}(k_1, k_2, n) + g_{2, \bbP}(k_1, k_2, n) + g_{3, \bbP}(k_1, k_2, n).
    \end{equation*}
    The first part, $g_{1, \bbP}(k_1, k_2, n)$, arises due to using the same sample to estimate the nuisance functions $p,b$ and the functional of interest $\psi(\bbP)$ (i.e., estimation of $\psi(\bbP)$ involves $\hat{p}(\bX_i)$ where $\bX_i$ was used to obtain $\hat{p}$), called \emph{own-observation bias} by Newey and Robins \cite{newey2018cross} in a similar context. The second part, $g_{2, \bbP}(k_1, k_2, n)$ arises due to estimating the nuisance functions $p, b$ in the same subsample, called \emph{nonlinearity bias} by Newey and Robins \cite{newey2018cross} because $\E_{\bbP}[\hat{p}(\bx) \hat{b}(\bx)] \neq \E_{\bbP}[\hat{p}(\bx)]  \E_{\bbP}[\hat{b}(\bx)]$. 
    The last part, $g_{3, \bbP}(k_1, k_2, n)$, which we refer to as the \emph{approximation bias}, is the remaining bias of the estimator. While estimators without sample splitting have all three sources of bias, single sample splitting removes own-observation bias and double sample splitting additionally removes nonlinearity bias. 

    Our results provide sharp bounds on these three sources of bias. We show that the own-observation bias can be characterized as $\sup_{\bbP \in \cP_{(\alpha,\beta)}} \left| g_{1, \bbP}(k_1,k_2,n) \right| \asymp  \frac{k_1 \vee k_2}{n}$ and the nonlinearity bias can be characterized by $\sup_{\bbP \in \cP_{(\alpha,\beta)}} \left| g_{2, \bbP}(k_1,k_2,n) \right| \asymp  \frac{k_1 \wedge k_2}{n}$. The form of $\sup_{\bbP \in \cP_{(\alpha,\beta)}} \left| g_{3, \bbP}(k_1, k_2, n) \right|$ depends on the choice of estimator, where
    \begin{equation*}
        \sup_{\bbP \in \cP_{(\alpha,\beta)}} \left| g_{3, \bbP}(k_1,k_2,n) \right| \asymp \begin{cases}  k^{-(\alpha+\beta)/d} \quad & \text{for $\hat{\psi}_{k}^{\mathrm{NR}}$} \\ (k_1 \wedge k_2)^{-(\alpha+\beta)/d}  & \text{for $\hat{\psi}_{k_1, k_2}^{\mathrm{INT}}, \hat{\psi}_{k_1, k_2}^{\mathrm{MC}}$} \\(k_1 \vee k_2)^{-(\alpha+\beta)/d} & \text{for $\hat{\psi}_{k_1, k_2}^{\mathrm{IF}}$}\\ \end{cases}.
    \end{equation*}
    
    To optimally tune the nuisance function estimators, we choose $k_1, k_2$ to minimize the rate of convergence of the mean squared error of the estimator. In regimes with strong regularity conditions, the optimal $k_1, k_2$ coincide with prediction-optimal choices. However, without such strong regularity conditions, the optimal $k_1, k_2$ involve performing undersmoothing and/or oversmoothing to balance the sizes of the different sources of bias and variance. With suitable nuisance function tuning and appropriate sample splitting, we show that several of these estimators can achieve minimax rates of convergence across all Hölder smoothness classes.

    In the following subsection, we make these notions more rigorous. We analyze the estimators with double sample splitting in Section \ref{sec: double ss}, single sample splitting in Section \ref{sec: single ss}, and no sample splitting in Section \ref{sec: no ss}.
    
\subsection{Double sample splitting} \label{sec: double ss}

In this subsection, we analyze the estimators in the double sample splitting case. We analyze $\hat{\psi}_{k_1, k_2}^{\mathrm{INT}}$, $\hat{\psi}_{k_1, k_2}^{\mathrm{MC}}$, and $\hat{\psi}_{k_1, k_2}^{\mathrm{IF}}$ as defined in Section \ref{sec: estimators of psi}, where we take each of the $\ell_j$ to be distinct.

The following theorem establishes bounds on the bias and variance of the estimators in terms of the resolutions of the nuisance function estimators and the sample size. We include lower bounds on the bias of all the estimators in order to establish when using prediction-optimal resolutions is sub-optimal. We include additional lower bounds for $\hat{\psi}_{k_1, k_2}^{\mathrm{MC}}$ and $\hat{\psi}_{k_1, k_2}^{\mathrm{IF}}$ for the following reasons. Since the upper bounds corresponding to $\hat{\psi}^{\mathrm{MC}}_{k_1, k_2}$ are too large to establish minimax rate optimality in low regularity regimes, we provide corresponding lower bounds in order to eventually show that $\hat{\psi}^{\mathrm{MC}}_{k_1, k_2}$ is sub-optimal for any choice of $k_1$ and $k_2$ in low regularity regimes. We include lower bounds for $\hat{\psi}^{\mathrm{IF}}_{k_1, k_2}$ because we will show that rate optimality is impossible to achieve for this estimator when letting the resolutions of the two nuisance function estimators have the same rate (see Remark \ref{rem: same resolution known f}).

\begin{theorem} \label{theorem: plugin known f} 
Suppose that double sample splitting is performed. Under the assumptions given in Section \ref{sec: motivation}, the following statements hold:
\begin{enumerate} 
    \item[1)] The estimator $\hat{\psi}_{k_1, k_2}^{\mathrm{INT}}$ satisfies
\begin{align*}
    \sup_{\bbP \in \cP_{(\alpha,\beta)}} \left|\E_\bbP\left(\hat{\psi}_{k_1, k_2}^{\mathrm{INT}} - \psi(\bbP)\right)\right| & \asymp (k_1 \wedge k_2)^{-(\alpha + \beta)/d} \\
    \sup_{\bbP \in \cP_{(\alpha,\beta)}} \Var_\bbP(\hat{\psi}_{k_1, k_2}^{\mathrm{INT}}) & \lesssim \frac{1}{n}  +  \frac{k_1 \wedge k_2}{n^2}. 
\end{align*}

    \item[2a)] The estimator $\hat{\psi}_{k_1, k_2}^{\mathrm{MC}}$ satisfies
\begin{align*}
    \sup_{\bbP \in \cP_{(\alpha,\beta)}} \left|\E_\bbP\left(\hat{\psi}_{k_1, k_2}^{\mathrm{MC}} - \psi(\bbP)\right)\right| & \asymp (k_1 \wedge k_2)^{-(\alpha + \beta)/d} \\
    \sup_{\bbP \in \cP_{(\alpha,\beta)}} \Var_\bbP(\hat{\psi}_{k_1, k_2}^{\mathrm{MC}}) & \lesssim \frac{1}{n}  +  \frac{k_1 \vee k_2}{n^2} + \frac{k_1k_2}{n^3}.
\end{align*}
    \item[2b)]
    There exists a $\bbP \in \cP_{(\alpha,\beta)}$ such that for $k_1, k_2 \gg n$
\begin{align*}
    \left|\E_\bbP\left(\hat{\psi}^{\mathrm{MC}}_{k_1, k_2} - \psi(\bbP)\right)\right| & \gtrsim (k_1 \wedge k_2)^{-(\alpha + \beta)/d} \\
    \Var_\bbP(\hat{\psi}^{\mathrm{MC}}_{k_1, k_2}) & \gtrsim \frac{k_1k_2}{n^3}.
\end{align*}
    
    \item[3a)] The estimator $\hat{\psi}^{\mathrm{IF}}_{k_1, k_2}$ satisfies
\begin{align*}
    \sup_{\bbP \in \cP_{(\alpha,\beta)}} \left|\E_\bbP\left(\hat{\psi}^{\mathrm{IF}}_{k_1, k_2} - \psi(\bbP)\right)\right| & \asymp (k_1 \vee k_2)^{-(\alpha + \beta)/d}   \\
    \sup_{\bbP \in \cP_{(\alpha,\beta)}} \Var_\bbP(\hat{\psi}^{\mathrm{IF}}_{k_1, k_2}) & \lesssim  \frac{1}{n}  +  \frac{k_1 \vee k_2}{n^2} + \frac{k_1k_2}{n^3}.
\end{align*}
\item[3b)] There exists a $\bbP \in \cP_{(\alpha,\beta)}$ such that for $k_1,k_2 \gg n$
\begin{align*}
    \left|\E_\bbP\left(\hat{\psi}^{\mathrm{IF}}_{k_1,k_2} - \psi(\bbP)\right)\right| & \gtrsim (k_1 \vee k_2)^{-(\alpha + \beta)/d} \\
    \Var_\bbP(\hat{\psi}^{\mathrm{IF}}_{k_1,k_2}) & \gtrsim \frac{k_1k_2}{n^3}.
\end{align*}
\end{enumerate}
\end{theorem}

\noindent We defer the proof of this result to Appendix C and discuss some subtleties of the proof instead. The main architecture of the proof relies on expressing the bias for each of the estimators in terms of projections of $p$ and $b$ onto the suitable subspaces of dimensions $k_1$ and $k_2$ respectively and subsequently using the geometry of these subspaces to inform our arguments. Specifically, we use known results on function approximations in Hölder spaces based on wavelet projections (see Appendix A) to obtain tight bounds (both upper and lower) on the bias. The understanding of the variances, in turn, relies on the compact support of the wavelet bases used in the construction of the estimators of $p$ and $b$.

The following corollary establishes the optimal resolution choices for each of the estimators based on the bounds given in Theorem \ref{theorem: plugin known f} and whether such resolution choices result in achieving minimax rate optimality for $\psi(\bbP)$ (i.e., the rate described by (\ref{eq: minimax rate})).

\begin{cor} \label{cor: opt known f double}
If $\frac{\alpha + \beta}{2} \geq \frac{d}{4}$, then
\begin{enumerate}
    \item $\hat{\psi}^{\mathrm{INT}}_{k_1, k_2}$ is minimax rate optimal when choosing $ n^{\frac{d}{2\alpha + 2\beta}} \lesssim k_1 \wedge k_2 \lesssim n$.
    \item $\hat{\psi}^{\mathrm{MC}}_{k_1, k_2}$ is minimax rate optimal when choosing $n^{\frac{d}{2\alpha + 2\beta}} \lesssim k_\ell \lesssim n$ for $\ell = 1, 2$.
    \item $\hat{\psi}^{\mathrm{IF}}_{k_1, k_2}$ is minimax rate optimal when choosing $ n^{\frac{d}{2\alpha + 2\beta}} \lesssim k_1 \vee k_2 \lesssim n$. 
\end{enumerate}
If $\frac{\alpha + \beta}{2} < \frac{d}{4}$, then
\begin{enumerate}
    \item $\hat{\psi}^{\mathrm{INT}}_{k_1, k_2}$ is minimax rate optimal when choosing $k_1 \wedge k_2 \asymp n^{\frac{2d}{2\alpha + 2\beta + d}}$.
    \item $\hat{\psi}^{\mathrm{MC}}_{k_1, k_2}$ cannot be minimax rate optimal for any choice of $k_1, k_2$. The best rate of estimation of $\psi(\bbP)$ for $\hat{\psi}^{\mathrm{MC}}_{k_1, k_2}$ is $n^{-\frac{3\alpha + 3\beta}{\alpha + \beta + d}}$ which occurs when choosing $k_\ell \asymp  n^{\frac{3d}{2\alpha + 2\beta + 2d}} $ for $\ell = 1, 2$.
    \item $\hat{\psi}^{\mathrm{IF}}_{k_1, k_2}$ is minimax rate optimal when choosing $k_1 \vee k_2 \asymp n^{\frac{2d}{2\alpha + 2\beta + d}}$ and $k_1 \wedge k_2 \lesssim n $.
\end{enumerate}
\end{cor}

The top panel of Figure \ref{fig:minimax} summarizes some of the results in Corollary \ref{cor: opt known f double}, illustrating the regularity conditions where there exist resolution choices that yield minimax rate optimality. Unlike the other estimators, $\hat{\psi}^{\mathrm{MC}}_{k_1, k_2}$ cannot be minimax rate optimal across all Hölder smoothness classes because its bias and variance properties impose conflicting requirements on $k_1, k_2$. Specifically, since $\hat{\psi}^{\mathrm{MC}}_{k_1, k_2}$ is not doubly robust, its bias only shrinks sufficiently fast when both $k_1$ and $k_2$ are large. However, 
having both $k_1$ and $k_2$ large causes the variance to grow too fast due to 
the term of order $\frac{k_1 k_2}{n^3}$, preventing $\hat{\psi}^{\mathrm{MC}}_{k_1, k_2}$ from 
achieving the minimax rate when $\frac{\alpha + \beta}{2} < \frac{d}{4}$.

\begin{figure}[!ht]
  \centering
    {
      \includegraphics[width=0.8\textwidth]{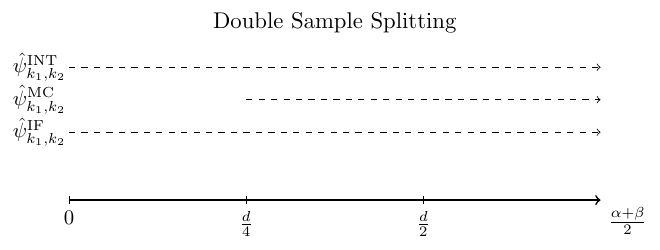}
    }
    \\
    {
      \includegraphics[width=0.8\textwidth]{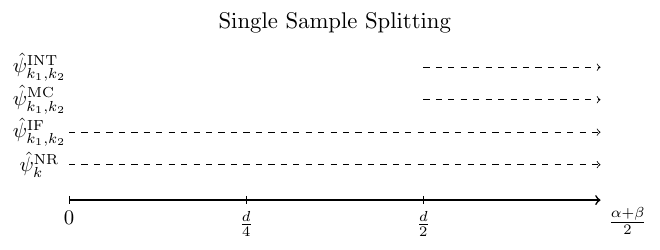}
    } \\
    {
      \includegraphics[width=0.8\textwidth]{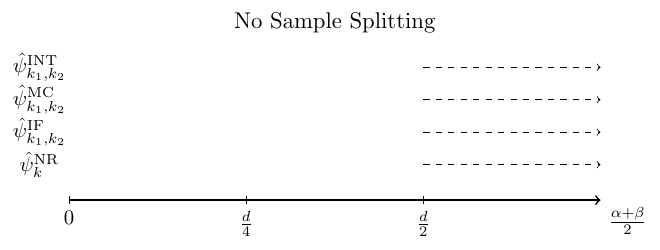}
    }
    \caption{Rate-optimality of the estimators when using double sample splitting (top panel), single sample splitting (middle panel), and no sample splitting (bottom panel). The dashed lines illustrate the region of the parameter space where the respective estimators are rate optimal when using optimal resolution choices.}
    \label{fig:minimax}
  \end{figure}

\begin{remark}
    It is straightforward to see that the bounds for the optimal resolutions given in Corollary \ref{cor: opt known f double} are not only sufficient but also necessary. In fact, the lower bounds on the bias in Theorem \ref{theorem: plugin known f} imply that the lower bounds for the optimal resolutions in Corollary \ref{cor: opt known f double} are indeed necessary (which, as we will see, imply the necessity to perform undersmoothing in low regularity regimes). Some additional lower bounds on the variance are needed to establish the necessity of some of the upper bounds on the optimal resolutions. Nevertheless, the necessity of the upper bounds on the optimal resolution choices is inconsequential to our discussion of the necessity to undersmooth versus use prediction-optimal resolutions.
\end{remark}

\begin{remark} \label{rem: same resolution known f}
In some settings (e.g., when $A = Y$), it is natural to consider tuning both of the nuisance functions in the same manner. While the optimal resolution choices for $\hat{\psi}^{\mathrm{INT}}_{k_1, k_2}$ and $\hat{\psi}^{\mathrm{MC}}_{k_1, k_2}$ allow setting $k_1 = k_2$, this is not the case for $\hat{\psi}^{\mathrm{IF}}_{k_1, k_2}$ in low regularity regimes. That is, if we enforce $k_1 = k_2 = k$,  then $\hat{\psi}^{\mathrm{IF}}_{k_1, k_2}$ cannot be minimax rate optimal for any choice of $k$ when $\frac{\alpha + \beta}{2} < \frac{d}{4}$. The best rate of estimation of $\psi(\bbP)$ for $\hat{\psi}^{\mathrm{IF}}_{k_1, k_2}$ in this case is $n^{-\frac{3\alpha + 3\beta}{\alpha + \beta + d}}$ which occurs if and only if choosing $k \asymp  n^{\frac{3d}{2\alpha + 2\beta + 2d}} $. The requirement that $k_1$ and $k_2$ be of different orders arises in order to simultaneously control the size of the bias and variance of $\hat{\psi}^{\mathrm{IF}}_{k_1, k_2}$. Due to the double robustness property of $\hat{\psi}^{\mathrm{IF}}_{k_1, k_2}$, only one of the nuisance function estimators needs to converge quickly (i.e., only one of $k_1, k_2$ needs to be large) so that the bias of $\hat{\psi}^{\mathrm{IF}}_{k_1, k_2}$ is rate optimal. However, in order to keep the variance of the estimator of $\psi(\mathbb{P})$ from growing too large, the other tuning parameter must be kept small in low regularity regimes. Although $\hat{\psi}^{\mathrm{IF}}_{k_1, k_2}$ does not allow $k_1 = k_2$ in low regularity regimes, note that $\hat{\psi}^{\mathrm{IF}}_{k_1, k_2}$ does not have more stringent 
requirements on $k_1, k_2$ than $\hat{\psi}^{\mathrm{INT}}_{k_1, k_2}$. Rather, they have different requirements that are not strictly stronger or weaker, which is due to the different rates of convergence of the bias and variance of these estimators. For example, $\hat{\psi}^{\mathrm{IF}}_{k_1, k_2}$ allows for $k_1 \wedge k_2 \lesssim n$ in this regime, whereas $\hat{\psi}^{\mathrm{INT}}_{k_1, k_2}$ requires that $k_1 \wedge k_2 \asymp n^{\frac{2d}{2\alpha + 2\beta + d}}$.

\end{remark}

\begin{remark}
Note that the estimators $\hat{\psi}_{k_1, k_2}^{\mathrm{INT}}$, $\hat{\psi}_{k_1, k_2}^{\mathrm{MC}}$, and $\hat{\psi}_{k_1, k_2}^{\mathrm{IF}}$ are minimax rate optimal in the $\frac{\alpha + \beta}{2} \geq \frac{d}{4}$ regime when setting $k_1 = k_2 = n$. Therefore, these estimators are adaptive over the Hölder smoothness classes of $p$ and $b$ in the $\frac{\alpha + \beta}{2} \geq \frac{d}{4}$ regime in the sense that they do not require explicit knowledge of $\alpha$ and $\beta$ to achieve minimax rate optimality. This is not the case in the $\frac{\alpha + \beta}{2} < \frac{d}{4}$ regime, as our choices of optimal resolutions depend on $\alpha$ and $\beta$. As we will see in Sections \ref{sec: single ss} and \ref{sec: no ss}, the same phenomenon holds in the single sample splitting case and the no sample splitting case (when $\frac{\alpha + \beta}{2} \geq \frac{d}{2}$). Although tangential to our main research questions, one may potentially produce minimax optimal estimators that are adaptive over the Hölder smoothness classes of the nuisance functions in this regime by applying Lepski-type methods \cite{lepskii1991problem, lepskii1992asymptotically}.
\end{remark}

Next, we analyze when using prediction-optimal resolutions is allowed for achieving rate optimality. We then compare prediction-optimal resolutions and resolutions that yield the best possible rate for estimating $\psi(\bbP)$.

\begin{cor} \label{cor: pred known f double}
Suppose we choose prediction-optimal resolutions $k_1 = k_1^{\text{pred}}$ and $k_2 = k_2^{\text{pred}}$. Then, 
\begin{enumerate}
    \item $\hat{\psi}_{k_1^{\text{pred}}, k_2^{\text{pred}}}^{\mathrm{INT}}$ achieves its best rate for estimating $\psi(\bbP)$ (i.e., the minimax rate in (\ref{eq: minimax rate})) if and only if $\alpha \wedge \beta \geq \frac{d}{2}$.
    \item $\hat{\psi}_{k_1^{\text{pred}}, k_2^{\text{pred}}}^{\mathrm{MC}}$ achieves its best rate for estimating $\psi(\bbP)$ (i.e., $n^{-1}$ when $\frac{\alpha + \beta}{2} \geq \frac{d}{4}$ and $n^{-\frac{3\alpha + 3\beta}{\alpha + \beta + d}}$ if $\frac{\alpha + \beta}{2} < \frac{d}{4}$) if and only if $\alpha \wedge \beta \geq \frac{d}{2}$.
    \item $\hat{\psi}_{k_1^{\text{pred}}, k_2^{\text{pred}}}^{\mathrm{IF}}$ achieves its best rate for estimating $\psi(\bbP)$ (i.e., the minimax rate in (\ref{eq: minimax rate})) if and only if $\alpha \vee \beta \geq \frac{d}{2}$.
\end{enumerate}
\end{cor}

\begin{cor} \label{cor: undermoothing known f double} The following statements hold:
    \begin{enumerate}
    \item $\hat{\psi}_{k_1, k_2}^{\mathrm{INT}}$ requires undersmoothing $k_1$ when $\beta < \frac{d}{2}$ and undersmoothing $k_2$ when $\alpha < \frac{d}{2}$ to achieve its best rate for estimating $\psi(\bbP)$ (i.e., the minimax rate in (\ref{eq: minimax rate})). 
    \item $\hat{\psi}_{k_1, k_2}^{\mathrm{MC}}$ requires undersmoothing $k_1$ when $\beta < \frac{d}{2}$ and undersmoothing $k_2$ when $\alpha < \frac{d}{2}$ to achieve its best rate for estimating $\psi(\bbP)$ (i.e., $n^{-1}$ when $\frac{\alpha + \beta}{2} \geq \frac{d}{4}$ and $n^{-\frac{3\alpha + 3\beta}{\alpha + \beta + d}}$ if $\frac{\alpha + \beta}{2} < \frac{d}{4}$).
    \item $\hat{\psi}_{k_1, k_2}^{\mathrm{IF}}$ requires undersmoothing $k_1$ or $k_2$ (and it does not matter which) when $\alpha \vee \beta < \frac{d}{2}$ to achieve its best rate for estimating $\psi(\bbP)$ (i.e., the minimax rate in (\ref{eq: minimax rate})).
\end{enumerate}
\end{cor}

These corollaries are summarized in Figure \ref{fig:tuning double}. Specifically, Figure \ref{fig:tuning double} describes when prediction-optimal resolutions are satisfactory and when prediction-suboptimal resolutions are required to obtain the best possible rate of estimation of $\psi(\bbP)$ (even if minimax optimal rates are not possible, such as for $\hat{\psi}^{\mathrm{MC}}_{k_1, k_2}$). In sufficiently low regularity regimes (e.g., when $\alpha \vee \beta < \frac{d}{2}$), all of the estimators require performing some degree of undersmoothing, which is needed to make the bias of the estimators converge at a fast enough rate. While both of the plug-in estimators require undersmoothing both nuisance functions, the first-order estimator only requires undersmoothing one of the nuisance functions and it does not matter which one is undersmoothed. Interestingly, the other nuisance function need not be undersmoothed; The only requirement on the other nuisance function resolution when $\frac{\alpha + \beta}{2} < \frac{d}{4}$ is that it is $O(n)$,  which ensures that the variance of $\hat{\psi}^{\mathrm{IF}}_{k_1, k_2}$ does not grow too large. The key reason why nuisance function tuning requirements for $\hat{\psi}^{\mathrm{IF}}_{k_1, k_2}$ differ from those of $\hat{\psi}^{\mathrm{INT}}_{k_1, k_2}$ and $\hat{\psi}^{\mathrm{MC}}_{k_1, k_2}$ is that the bias of $\hat{\psi}^{\mathrm{IF}}_{k_1, k_2}$ is controlled entirely by the less smooth nuisance function estimator whereas the bias of $\hat{\psi}^{\mathrm{INT}}_{k_1, k_2}$ and $\hat{\psi}^{\mathrm{MC}}_{k_1, k_2}$ is controlled by the smoother nuisance function estimator. 

\begin{remark}
    It is an interesting question to compare rates of estimation under double sample splitting to those under a corresponding cross-fitting strategy (see Section \ref{sec: sample splitting}). When performing cross-fitting, the bias of the estimators are identical and it is not hard to see that the same variance upper bounds hold. Therefore, Corollaries \ref{cor: opt known f double}, \ref{cor: pred known f double}, and \ref{cor: undermoothing known f double} also hold for cross-fit versions of $\hat{\psi}_{k_1, k_2}^{\mathrm{INT}}$ and $\hat{\psi}_{k_1, k_2}^{\mathrm{IF}}$. These corollaries also hold for the cross-fit version of $\hat{\psi}_{k_1, k_2}^{\mathrm{MC}}$, except that the optimality of the rate of $n^{-\frac{3\alpha + 3\beta}{\alpha + \beta + d}}$ in the $\frac{\alpha + \beta}{2} < \frac{d}{4}$ regime is not immediately implied. An analogous variance lower bound for the cross-fit estimator would be needed to establish that $n^{-\frac{3\alpha + 3\beta}{\alpha + \beta + d}}$ is also the optimal rate in the $\frac{\alpha + \beta}{2} < \frac{d}{4}$ regime. 
\end{remark}

\begin{remark}
    When using prediction-optimal resolutions, the first-order estimator is minimax rate optimal for a larger set of $(\alpha, \beta)$ compared to the plug-in estimators (Corollary \ref{cor: pred known f double}). In fact, it is straightforward to see from Theorem \ref{theorem: plugin known f} that the rate of estimation of the first-order estimator is at least as fast as the plug-in estimators for any $(\alpha, \beta)$ when using prediction-optimal resolutions. However, the advantage of the first-order estimator disappears when selecting optimal resolutions for estimating $\psi(\bbP)$, as the plug-in estimator $\hat{\psi}_{k_1, k_2}^{\mathrm{INT}}$ with suitable $k_1, k_2$ achieves minimax rate optimality across all Hölder smoothness classes. 
\end{remark}

\begin{figure}[!h]
\centering
\begin{tabular}{ll}
  \includegraphics[width=0.45\textwidth]{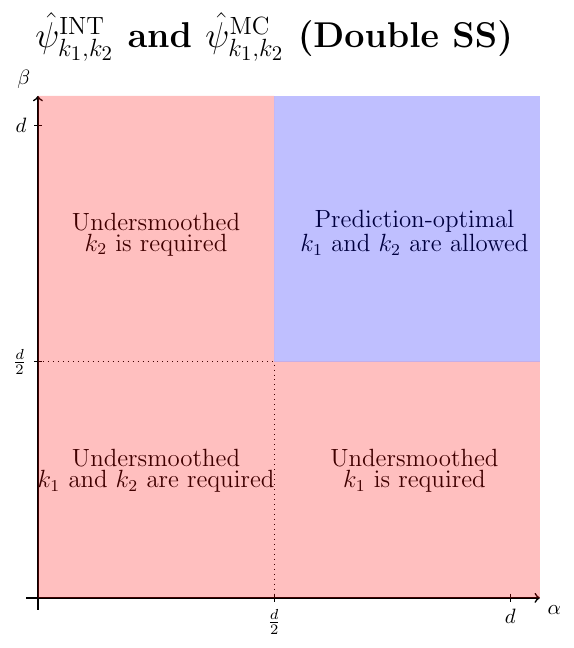} & \includegraphics[width=0.45\textwidth]{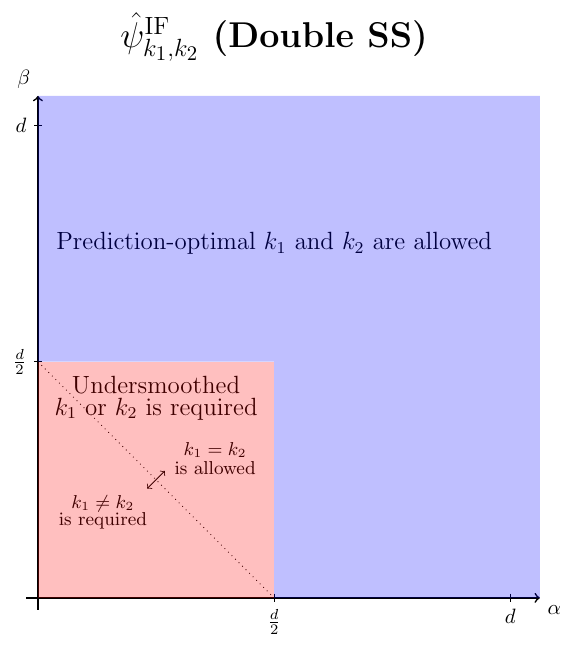}
\end{tabular}
\caption{Optimally tuning the resolutions for $\hat{\psi}_{k_1, k_2}^{\mathrm{INT}}$, $\hat{\psi}_{k_1, k_2}^{\mathrm{MC}}$, and $\hat{\psi}_{k_1, k_2}^{\mathrm{IF}}$ with double sample splitting (SS).}
    \label{fig:tuning double}
  \end{figure}

\subsection{Single sample splitting} \label{sec: single ss}

Here, we analyze the estimators in the single sample splitting case. In Section \ref{sec: single ss two}, we analyze single sample split versions of the estimators that depend on two nuisance functions. That is, we analyze $\hat{\psi}_{k_1, k_2}^{\mathrm{INT}}$, $\hat{\psi}_{k_1, k_2}^{\mathrm{MC}}$, and $\hat{\psi}_{k_1, k_2}^{\mathrm{IF}}$ when taking $\ell_1 = \ell_2$ and the other $\ell_j$ distinct. In Section \ref{sec: single ss one}, we analyze the single sample split Newey and Robins plug-in estimator (i.e., $\hat{\psi}_{k}^{\mathrm{NR}}$ when taking $\ell_1 \neq \ell_2$), which only relies on one nuisance function. 

\subsubsection{Estimators that depend on two nuisance functions} \label{sec: single ss two}

The following theorem establishes bounds on the bias and variance of the estimators that depend on two nuisance functions in the single sample splitting case. Since the bounds depend on a number of additional terms (compared to the double sample splitting case) which are not symmetric in $k_1$ and $k_2$, we present the bounds in this subsection in terms of $k_{\mathrm{min}} := k_1 \wedge k_2$ and $k_{\mathrm{max}} := k_1 \vee k_2$. We do not provide lower bounds on the variance because the lower bounds on the bias are sufficient to establish the optimal resolution choices and when prediction-optimal resolutions are satisfactory for obtaining the best possible rate of estimation of $\psi(\bbP)$.

\begin{theorem} \label{theorem: plugin known f single} 
Suppose that single sample splitting is performed. Under the assumptions given in Section \ref{sec: motivation}, the following statements hold:
\begin{enumerate} 
    \item The estimator $\hat{\psi}_{k_1, k_2}^{\mathrm{INT}}$ satisfies
\begin{align*}
    \sup_{\bbP \in \cP_{(\alpha,\beta)}} \left|\E_\bbP\left(\hat{\psi}_{k_1, k_2}^{\mathrm{INT}} - \psi(\bbP)\right)\right| & \asymp k_{\mathrm{min}}^{-(\alpha + \beta) / d}   \vee \frac{k_{\mathrm{min}}}{n} \\
    \sup_{\bbP \in \cP_{(\alpha,\beta)}} \Var_\bbP(\hat{\psi}_{k_1, k_2}^{\mathrm{INT}}) & \lesssim \frac{1}{n} + \frac{k_{\mathrm{min}}}{n^2} + \frac{k_{\mathrm{min}}^2}{n^3}.
\end{align*}

    \item The estimator $\hat{\psi}_{k_1, k_2}^{\mathrm{MC}}$ satisfies 
\begin{align*}
    \sup_{\bbP \in \cP_{(\alpha,\beta)}} \left|\E_\bbP\left(\hat{\psi}_{k_1, k_2}^{\mathrm{MC}} - \psi(\bbP)\right)\right| & \asymp k_{\mathrm{min}}^{-(\alpha + \beta) / d}   \vee \frac{k_{\mathrm{min}}}{n} \\
    \sup_{\bbP \in \cP_{(\alpha,\beta)}} \Var_\bbP(\hat{\psi}_{k_1, k_2}^{\mathrm{MC}}) & \lesssim \frac{1}{n} + \frac{k_{\mathrm{max}}}{n^2} + \frac{k_{\mathrm{min}}k_{\mathrm{max}}}{n^3} + \frac{k_{\mathrm{min}}^2 k_{\mathrm{max}}}{n^4}.
\end{align*}

\item The estimator $\hat{\psi}^{\mathrm{IF}}_{k_1, k_2}$ satisfies
\begin{align*}
    \sup_{\bbP \in \cP_{(\alpha,\beta)}} \left|\E_\bbP\left(\hat{\psi}^{\mathrm{IF}}_{k_1, k_2} - \psi(\bbP)\right)\right| & \asymp k_{\mathrm{max}}^{-\frac{\alpha + \beta}{d}}   \vee \frac{k_{\mathrm{min}}}{n} \\
    \sup_{\bbP \in \cP_{(\alpha,\beta)}} \Var_\bbP(\hat{\psi}^{\mathrm{IF}}_{k_1, k_2}) & \lesssim  \frac{1}{n} + \frac{k_{\mathrm{max}}}{n^2} + \frac{k_{\mathrm{min}}k_{\mathrm{max}}}{n^3} + \frac{k_{\mathrm{min}}^2 k_{\mathrm{max}}}{n^4}.
\end{align*}
\end{enumerate}
\end{theorem} 
\noindent The proof is given in Appendix D. We express the bias of these estimators in terms of the bias in the double sample splitting case plus a remainder term which we refer to as nonlinearity bias. We show that the nonlinearity bias is of order $\frac{k_{\mathrm{min}}}{n}$ by analyzing expectations of products of kernel functions. The variance bounds are derived in a similar manner as those in the double sample splitting case. The key difference in the variance bounds here is the existence of expectations of higher-order products of kernels, which contribute the additional $O(\frac{k_{\mathrm{min}}^2}{n^3})$ term for $\hat{\psi}^{\mathrm{INT}}_{k_1, k_2}$ and $O( \frac{k_{\mathrm{min}}^2 k_{\mathrm{max}}}{n^4})$ term for $\hat{\psi}^{\mathrm{MC}}_{k_1, k_2}$ and $\hat{\psi}^{\mathrm{IF}}_{k_1, k_2}$.

Next, we describe the optimal resolution choices for each of the estimators and whether such resolution choices yield minimax rate optimality for estimating $\psi(\bbP)$.
\begin{cor} \label{cor: opt known f single}
If $\frac{\alpha + \beta}{2} \geq \frac{d}{2}$, then 
\begin{enumerate}
    \item $\hat{\psi}^{\mathrm{INT}}_{k_1, k_2}$ is minimax rate optimal when choosing $n^{\frac{d}{2\alpha + 2\beta}} \lesssim k_{\mathrm{min}}  \lesssim \sqrt{n}$.
    \item $\hat{\psi}^{\mathrm{MC}}_{k_1, k_2}$ is minimax rate optimal when choosing $n^{\frac{d}{2\alpha + 2\beta}} \lesssim k_{\mathrm{min}} \lesssim \sqrt{n}$ and $k_{\mathrm{max}} \lesssim n$.
    \item $\hat{\psi}^{\mathrm{IF}}_{k_1, k_2}$ is minimax rate optimal when choosing $k_{\mathrm{min}} \lesssim \sqrt{n}$ and $n^{\frac{d}{2\alpha + 2\beta}} \lesssim k_{\mathrm{max}} \lesssim n$. 
\end{enumerate}
If $\frac{\alpha + \beta}{2} < \frac{d}{2}$, then 
\begin{enumerate}
    \item $\hat{\psi}^{\mathrm{INT}}_{k_1, k_2}$ cannot be minimax rate optimal for any choice of $k_1, k_2$. The best rate of estimation of $\psi(\bbP)$ for $\hat{\psi}^{\mathrm{INT}}_{k_1, k_2}$ is $n^{-\frac{2\alpha + 2\beta}{\alpha + \beta + d}}$ which occurs when choosing $k_{\mathrm{min}} \asymp  n^{\frac{d}{\alpha + \beta + d}} $.
    \item $\hat{\psi}^{\mathrm{MC}}_{k_1, k_2}$ cannot be minimax rate optimal for any choice of $k_1, k_2$. As with $\hat{\psi}^{\mathrm{INT}}_{k_1, k_2}$, the best rate of estimation of $\psi(\bbP)$ for $\hat{\psi}^{\mathrm{MC}}_{k_1, k_2}$ is $n^{-\frac{2\alpha + 2\beta}{\alpha + \beta + d}}$ which occurs when choosing $k_{\mathrm{min}} \asymp  n^{\frac{d}{\alpha + \beta + d}} $ and $k_{\mathrm{max}} \lesssim  n^{\frac{2d}{\alpha + \beta + d}} $.
    \item $\hat{\psi}^{\mathrm{IF}}_{k_1, k_2}$ is minimax rate optimal when choosing $k_{\mathrm{min}} \lesssim \sqrt{n}$ and $n^{\frac{d}{2\alpha + 2\beta}} \lesssim k_{\mathrm{max}} \lesssim n$ when $\frac{d}{4} \leq \frac{\alpha + \beta}{2} < \frac{d}{2}$ and $k_{\mathrm{min}} \lesssim n^{\frac{d}{2\alpha + 2\beta + d}}$ and $k_{\mathrm{max}} \asymp n^{\frac{2d}{2\alpha + 2\beta + d}}$ when $\frac{\alpha + \beta}{2} < \frac{d}{4}$.
\end{enumerate}
\end{cor}

The middle panel of Figure \ref{fig:minimax} describes the region of the parameter space where minimax rate optimality is possible for each of the estimators. Observe that both of the plug-in estimators have slower rates of convergence in the single sample splitting case compared to the double sample splitting case whenever $\frac{\alpha + \beta}{2} < \frac{d}{2}$. This is due to the nonlinearity bias arising in the single sample splitting case, which prevents these estimators from having a bias that converges at a sufficiently fast rate in this regime. However, the first-order estimator can be minimax rate optimal across the entire parameter space regardless of whether single or double sample splitting is used. This is because the nonlinearity bias increases with respect to $k_{\mathrm{min}}$ (which can be set arbitrarily small) and the other terms in the bias of the first-order estimator decrease with respect to $k_{\mathrm{max}}$ (which can be set sufficiently large).

We next make a few remarks on the optimal resolution choices.

\begin{remark}
As in the double sample splitting case, we note that it is straightforward to see that the optimal resolution bounds in Corollary \ref{cor: opt known f single} are indeed necessary. The lower bounds on the bias in Theorem \ref{theorem: plugin known f single} imply the necessity of the lower bounds on the optimal resolutions in Corollary \ref{cor: opt known f single}. The necessity of some of the upper bounds on the optimal resolutions requires lower bounds on the variance of $\hat{\psi}^{\mathrm{MC}}_{k_1, k_2}$ and $\hat{\psi}^{\mathrm{IF}}_{k_1, k_2}$ which follow from the same proof techniques used in the double sample splitting case. An analogous comment applies to the bounds in Sections \ref{sec: single ss one} and \ref{sec: no ss}.
\end{remark}

\begin{remark}
Recall that it may be natural to set $k_1 = k_2$ in some settings. As in the double sample splitting case, the optimal resolution choices for $\hat{\psi}^{\mathrm{INT}}_{k_1, k_2}$ and $\hat{\psi}^{\mathrm{MC}}_{k_1, k_2}$ allow for setting $k_1 = k_2$, although this is not the case for $\hat{\psi}^{\mathrm{IF}}_{k_1, k_2}$ in low regularity regimes. If we enforce $k_1 = k_2 = k$,  then $\hat{\psi}^{\mathrm{IF}}_{k_1, k_2}$ cannot be minimax rate optimal for any choice of $k$ when $\frac{\alpha + \beta}{2} < \frac{d}{2}$. The best rate of estimation of $\psi(\bbP)$ for $\hat{\psi}^{\mathrm{IF}}_{k_1, k_2}$ is given by $n^{-\frac{2\alpha + 2\beta}{\alpha + \beta + d}}$ which occurs when choosing $k \asymp n^{\frac{d}{\alpha + \beta + d}}$.
\end{remark}

In the following two corollaries, we describe when using prediction-optimal resolutions yields minimax rate optimality and compare them to the optimal resolutions.
\begin{cor} \label{cor: pred known f single}
Suppose we choose prediction-optimal resolutions $k_1 = k_1^{\text{pred}}$ and $k_2 = k_2^{\text{pred}}$. Then, 
\begin{enumerate}
    \item $\hat{\psi}_{k_1^{\text{pred}}, k_2^{\text{pred}}}^{\mathrm{INT}}$ achieves its best rate for estimating $\psi(\bbP)$ (i.e., $n^{-1}$ when $\frac{\alpha + \beta}{2} \geq \frac{d}{2}$ and $n^{-\frac{2\alpha + 2\beta}{\alpha + \beta + d}}$ if $\frac{\alpha + \beta}{2} < \frac{d}{2}$) if and only if $\alpha \wedge \beta \geq \frac{d}{2}$ or $\alpha = \beta$.
    \item $\hat{\psi}_{k_1^{\text{pred}}, k_2^{\text{pred}}}^{\mathrm{MC}}$ achieves its best rate for estimating $\psi(\bbP)$ (i.e., $n^{-1}$ when $\frac{\alpha + \beta}{2} \geq \frac{d}{2}$ and $n^{-\frac{2\alpha + 2\beta}{\alpha + \beta + d}}$ if $\frac{\alpha + \beta}{2} < \frac{d}{2}$) if and only if $\alpha \wedge \beta \geq \frac{d}{2}$ or $\alpha = \beta$.
    \item $\hat{\psi}_{k_1^{\text{pred}}, k_2^{\text{pred}}}^{\mathrm{IF}}$ achieves its best rate for estimating $\psi(\bbP)$ (i.e., the minimax rate in (\ref{eq: minimax rate})) if and only if $\alpha \vee \beta \geq \frac{d}{2}$.
\end{enumerate}
\end{cor}

\begin{cor} \label{cor: undermoothing known f single}
The following statements hold:
\begin{enumerate}
    \item $\hat{\psi}_{k_1, k_2}^{\mathrm{INT}}$ requires undersmoothing $k_1$ when $\beta < \alpha \wedge \frac{d}{2}$ and requires undersmoothing $k_2$ when $\alpha < \beta \wedge \frac{d}{2}$ to achieve its best rate for estimating $\psi(\bbP)$ (i.e., $n^{-1}$ when $\frac{\alpha + \beta}{2} \geq \frac{d}{2}$ and $n^{-\frac{2\alpha + 2\beta}{\alpha + \beta + d}}$ if $\frac{\alpha + \beta}{2} < \frac{d}{2}$).
    \item $\hat{\psi}_{k_1, k_2}^{\mathrm{MC}}$ requires undersmoothing $k_1$ when $\beta < \alpha \wedge \frac{d}{2}$ and requires undersmoothing $k_2$ when $\alpha < \beta \wedge \frac{d}{2}$ to achieve its best rate for estimating $\psi(\bbP)$ (i.e., $n^{-1}$ when $\frac{\alpha + \beta}{2} \geq \frac{d}{2}$ and $n^{-\frac{2\alpha + 2\beta}{\alpha + \beta + d}}$ if $\frac{\alpha + \beta}{2} < \frac{d}{2}$).
    \item $\hat{\psi}_{k_1, k_2}^{\mathrm{IF}}$ requires undersmoothing $k_{\ell_1}$ and oversmoothing $k_{\ell_2}$ ($\ell_1, \ell_2 \in \{1, 2\}$, $\ell_1 \neq \ell_2$) when $\alpha \vee \beta < \frac{d}{2}$ to achieve its best rate for estimating $\psi(\bbP)$ (i.e., the minimax rate in (\ref{eq: minimax rate})).
\end{enumerate}
\end{cor}

The top panels of Figure \ref{fig:tuning single} summarize these corollaries. For each of the three estimators, the region in the parameter space where some degree of undersmoothing is needed to achieve the best rate of estimation of $\psi(\bbP)$ is the same regardless of whether single or double sample splitting is used (except when $\alpha = \beta$). However, there are some differences between single and double sample splitting in terms of how to optimally tune the nuisance functions. First, for the plug-in estimators, the degree to which one must undersmooth and the best rate of estimation differs in the single sample splitting case. Single sample splitting requires a lesser degree of undersmoothing and results in slower optimal rates of convergence compared to the double sample splitting case when $\frac{\alpha + \beta}{2} < \frac{d}{2}$. Second, for the first-order estimator, there are stricter requirements on how to optimally tune the nuisance functions in the single sample splitting case. Recall that $\hat{\psi}^{\mathrm{IF}}_{k_1, k_2}$ required setting one of the resolutions to be of order $c n^{\frac{2d}{2\alpha + 2\beta + d}}$ and the other resolution to be $O(n)$ (i.e., either undersmoothed or oversmoothed) in the double sample splitting case when $\frac{\alpha + \beta}{2} < \frac{d}{4}$. In the single sample splitting case, the smaller resolution has to be $O(n^{\frac{d}{2\alpha + 2\beta + d}})$, which implies oversmoothing, in order for the nonlinearity bias to be sufficiently small.

\begin{remark}
    It is straightforward to see that all the bias and variance bounds in Theorem \ref{theorem: plugin known f single} hold for the corresponding cross-fit estimators. Therefore, all of the conclusions in this subsection immediately hold for the corresponding cross-fit estimators.
\end{remark}

\subsubsection{Estimators that depend on one nuisance function} \label{sec: single ss one}

We next analyze the Newey and Robins plug-in estimator in the single sample splitting case. We begin with describing bounds on the bias and variance of this estimator in terms of the sample size and resolution $k$. 

\begin{theorem} \label{theorem: plugin known f single nr}
Suppose that single sample splitting is performed. Under the assumptions given in Section \ref{sec: motivation}, the estimator $\hat{\psi}_{k}^{\mathrm{NR}}$ satisfies
\begin{align*}
    \sup_{\bbP \in \cP_{(\alpha,\beta)}} \left|\E_\bbP\left(\hat{\psi}_{k}^{\mathrm{NR}} - \psi(\bbP)\right)\right| & \asymp k^{-(\alpha + \beta)/d} \\
    \sup_{\bbP \in \cP_{(\alpha,\beta)}} \Var_\bbP(\hat{\psi}_{k}^{\mathrm{NR}}) & \lesssim \frac{1}{n}  +  \frac{k}{n^2}. 
\end{align*}
\end{theorem}
\noindent The proof is given in Appendix E. The proof strategy is similar to that used for the double sample split estimators (as $\hat{\psi}_{k}^{\mathrm{NR}}$ trivially has no nonlinearity bias arising from estimating two nuisance functions in the same subsample). Due to the geometry of the wavelet projection estimator, the bias of $\hat{\psi}_{k}^{\mathrm{NR}}$ can be expressed in terms of orthogonal projections of $p$ and $b$ onto the span of the wavelet basis functions at the $k$th resolution level of $L^2([0,1]^d)$, which results in the bias scaling with the smoothness of $p$ and $b$ despite only estimating $b$.

\begin{remark}
    Observe that the rates of convergence of the bias and variance of $\hat{\psi}^{\mathrm{NR}}_{k}$ are identical to those of $\hat{\psi}^{\mathrm{IF}}_{k_1, k_2}$ when setting either $k_1$ or $k_2$ to $O(1)$, which recall is optimal for estimating $\psi(\bbP)$. Consequently, one way of interpreting the optimal resolution choices for $\hat{\psi}^{\mathrm{IF}}_{k_1, k_2}$ in low regularity regimes is to make $\hat{\psi}^{\mathrm{IF}}_{k_1, k_2}$ equivalent to a singly robust estimator. A similar phenomenon holds in the no sample splitting case (see Section \ref{sec: no ss one}).
\end{remark}

The following corollary describes how to optimally choose the resolution for $\hat{\psi}^{\mathrm{NR}}_{k}$.

\begin{cor} \label{cor: opt known f nr}
The estimator $\hat{\psi}^{\mathrm{NR}}_{k}$ is minimax rate optimal when choosing $n^{\frac{d}{2\alpha + 2\beta}} \lesssim k \lesssim n$ if $\frac{\alpha + \beta}{2} \geq \frac{d}{4}$ and $k \asymp n^{\frac{2d}{2\alpha + 2\beta + d}}$ if $\frac{\alpha + \beta}{2} < \frac{d}{4}$.
\end{cor}

The following two corollaries describe when  using the prediction-optimal resolution makes $\hat{\psi}^{\mathrm{NR}}_{k}$ minimax rate optimal and when undersmoothing is needed.

\begin{cor} \label{cor: pred known f nr}
Suppose we choose the prediction-optimal resolution $k = k_2^{\text{pred}}$. Then, $\hat{\psi}_{k_2^{\text{pred}}}^{\mathrm{NR}}$ is minimax rate optimal if and only if $\alpha \geq \frac{d}{2}$.
\end{cor}

\begin{cor} \label{cor: undersmoothing known f nr}
The estimator $\hat{\psi}_{k}^{\mathrm{NR}}$ requires undersmoothing $k$ when $\alpha < \frac{d}{2}$ to achieve its best rate for estimating $\psi(\bbP)$ (i.e., the minimax rate in (\ref{eq: minimax rate})).
\end{cor}

These results are included in Figures \ref{fig:minimax} and \ref{fig:tuning single}. Once again, undersmoothing is necessary to achieve optimal rates of convergence in low regularity regimes.

\begin{remark}
    If one were to consider an analogous estimator by exchanging the roles of $A$ and $Y$ (i.e., $\frac{1}{n} \sum_{i \in \D_2} Y_i(A_i - \hat{p}_k^{(1)}(\bX_i))$), the same results would hold when exchanging the roles of $\alpha$ and $\beta$. That is, Theorem \ref{theorem: plugin known f single nr} and Corollary \ref{cor: opt known f nr} would hold, but Corollaries \ref{cor: pred known f nr} and \ref{cor: undersmoothing known f nr} would imply that the prediction-optimal resolution ($k = k_1^{\text{pred}}$) results in minimax rate optimality if and only if $\beta \geq \frac{d}{2}$ and undersmoothing is needed if and only if $\beta < \frac{d}{2}$. An analogous comment applies in the no sample splitting case as well.
\end{remark}

\begin{remark}
    Similar to the analysis of the other single sample split estimators, the bias and variance bounds in Theorem \ref{theorem: plugin known f single nr} hold for the cross-fit version of $\hat{\psi}_{k}^{\mathrm{NR}}$. Thus, all of the conclusions in this subsection also apply to the cross-fit estimator.
\end{remark}

\begin{figure}[!h]
\centering
\begin{tabular}{ll}
  \includegraphics[width=0.45\textwidth]{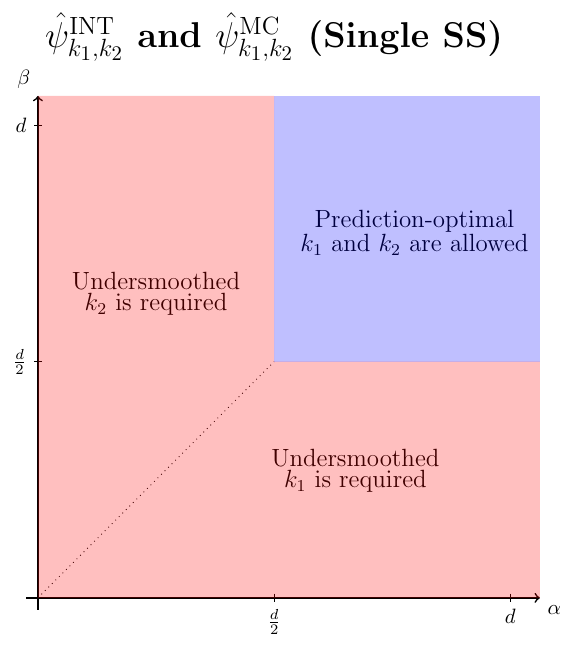} & 
  \includegraphics[width=0.45\textwidth]{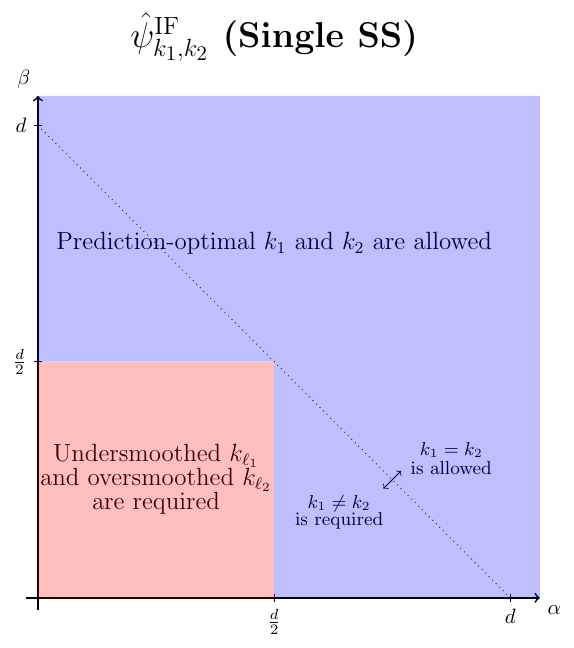} \\ 
  \multicolumn{2}{c}{\includegraphics[width=0.45\textwidth]{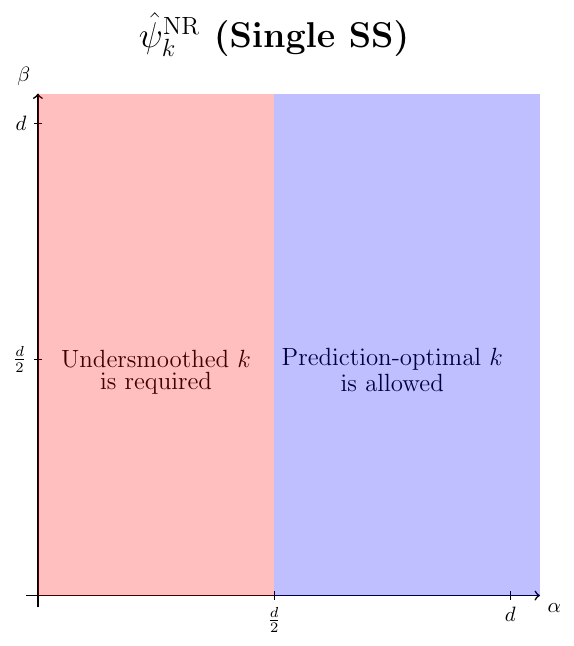}}
\end{tabular}
    \caption{Optimally tuning the resolutions for $\hat{\psi}_{k_1, k_2}^{\mathrm{INT}}$, $\hat{\psi}_{k_1, k_2}^{\mathrm{MC}}$, $\hat{\psi}_{k_1, k_2}^{\mathrm{IF}}$, and  $\hat{\psi}_{k}^{\mathrm{NR}}$ with single sample splitting (SS).}
    \label{fig:tuning single}
\end{figure}

\subsection{No sample splitting} \label{sec: no ss}

Last, we analyze all four estimators when no sample splitting is used (i.e., when all $\ell_j$'s are equal). 

\subsubsection{Estimators that depend on two nuisance functions} \label{sec: no ss two}

Theorem \ref{theorem: plugin known f no} establishes bounds on the bias and variance of the three estimators that depend on two nuisance functions. Similar to the single sample splitting case, we do not provide lower bounds on the variance of the estimators because the lower bounds on the bias are sufficient to establish optimal resolution choices and when prediction-optimal resolutions are satisfactory.

\begin{theorem} \label{theorem: plugin known f no} 
Suppose that no sample splitting is performed. Under the assumptions given in Section \ref{sec: motivation}, the following statements hold:
\begin{enumerate} 
    \item The estimator $\hat{\psi}_{k_1, k_2}^{\mathrm{INT}}$ satisfies
\begin{align*}
    \sup_{\bbP \in \cP_{(\alpha,\beta)}} \left|\E_\bbP\left(\hat{\psi}_{k_1, k_2}^{\mathrm{INT}} - \psi(\bbP)\right)\right| & \asymp k_{\mathrm{min}}^{-(\alpha + \beta) / d}   \vee \frac{k_{\mathrm{min}}}{n} \\
    \sup_{\bbP \in \cP_{(\alpha,\beta)}} \Var_\bbP(\hat{\psi}_{k_1, k_2}^{\mathrm{INT}}) & \lesssim \frac{1}{n} + \frac{k_{\mathrm{min}}}{n^2} + \frac{k_{\mathrm{min}}^2}{n^3}.
\end{align*}

    \item The estimator $\hat{\psi}_{k_1, k_2}^{\mathrm{MC}}$ satisfies for $k_1, k_2 \ll n$
\begin{equation*}
    \sup_{\bbP \in \cP_{(\alpha,\beta)}} \left|\E_\bbP\left(\hat{\psi}_{k_1, k_2}^{\mathrm{MC}} - \psi(\bbP)\right)\right| \asymp k_{\mathrm{min}}^{-(\alpha + \beta) / d}   \vee \frac{k_{\mathrm{max}}}{n}
\end{equation*}
and for any $k_1, k_2$
\begin{equation*}
        \sup_{\bbP \in \cP_{(\alpha,\beta)}} \Var_\bbP(\hat{\psi}_{k_1, k_2}^{\mathrm{MC}}) \lesssim \frac{1}{n} + \frac{k_{\mathrm{max}}}{n^2} + \frac{k_{\mathrm{max}}^2}{n^3} + \frac{k_{\mathrm{min}} k_{\mathrm{max}}^2}{n^4} + \frac{k_{\mathrm{min}}^2 k_{\mathrm{max}}^2}{n^5}.
\end{equation*}

\item The estimator $\hat{\psi}^{\mathrm{IF}}_{k_1, k_2}$ satisfies for $k_1,k_2 \ll n$, 
\begin{equation*}
    \sup_{\bbP \in \cP_{(\alpha,\beta)}} \left|\E_\bbP\left(\hat{\psi}^{\mathrm{IF}}_{k_1, k_2} - \psi(\bbP)\right)\right| \asymp k_{\mathrm{max}}^{-\frac{\alpha + \beta}{d}}   \vee \frac{k_{\mathrm{max}}}{n}
\end{equation*}
and for any $k_1, k_2$
\begin{equation*}
    \sup_{\bbP \in \cP_{(\alpha,\beta)}} \Var_\bbP(\hat{\psi}^{\mathrm{IF}}_{k_1, k_2}) \lesssim  \frac{1}{n} + \frac{k_{\mathrm{max}}}{n^2} + \frac{k_{\mathrm{max}}^2}{n^3} + \frac{k_{\mathrm{min}} k_{\mathrm{max}}^2}{n^4} + \frac{k_{\mathrm{min}}^2 k_{\mathrm{max}}^2}{n^5}.
\end{equation*}
\end{enumerate}
\end{theorem} 

\noindent The proof is given in Appendix F. Compared to the single sample splitting case, the bias of $\hat{\psi}_{k_1, k_2}^{\mathrm{MC}}$ and $\hat{\psi}_{k_1, k_2}^{\mathrm{IF}}$ incur an additional term of order $\frac{k_{\mathrm{max}}}{n}$ due to evaluating the nuisance function estimators at the same observations used to fit them, which we refer to as own-observation bias. The variance bounds of these estimators also include additional terms in the no sample splitting case, although these terms do not dominate the variance when using optimal and prediction-optimal resolutions. The bias and variance bounds of $\hat{\psi}_{k_1, k_2}^{\mathrm{INT}}$ are trivially the same as in the single sample splitting case since the nuisance function estimators are not evaluated at sample points for this estimator. 

\begin{remark}
    We took $k_1, k_2 \ll n$ to simplify the presentation of the bias bounds of $\hat{\psi}^{\mathrm{MC}}_{k_1, k_2}$ and $\hat{\psi}^{\mathrm{IF}}_{k_1, k_2}$, as otherwise the bias of these estimators does not converge to 0. Indeed, our proofs show that $k_{\mathrm{min}}^{-\frac{\alpha + \beta}{d}}   \vee \frac{k_{\mathrm{max}}}{n} \vee \frac{k_1 k_2}{n^2}$ and $k_{\mathrm{max}}^{-\frac{\alpha + \beta}{d}}   \vee \frac{k_{\mathrm{max}}}{n} \vee \frac{k_1 k_2}{n^2}$ are upper bounds for the bias of $\hat{\psi}^{\mathrm{MC}}_{k_1, k_2}$ and $\hat{\psi}^{\mathrm{IF}}_{k_1, k_2}$ without this restriction of $k_1,k_2$. Straightforward extensions of the arguments in the proof can show that these are lower bounds as well. However, these considerations are inconsequential for our investigation of optimal versus prediction-optimal nuisance function tuning strategies, as the optimal and prediction-optimal rates for $k_1, k_2$ grow slower than $n$.
\end{remark}

We next describe the optimal resolution choices and optimal rates of convergence for these estimators in Corollary \ref{cor: opt known f no}.
\begin{cor} \label{cor: opt known f no}
If $\frac{\alpha + \beta}{2} \geq \frac{d}{2}$, then 
\begin{enumerate}
    \item $\hat{\psi}^{\mathrm{INT}}_{k_1, k_2}$ is minimax rate optimal when choosing $n^{\frac{d}{2\alpha + 2\beta}} \lesssim k_{\mathrm{min}}  \lesssim \sqrt{n}$.
    \item $\hat{\psi}^{\mathrm{MC}}_{k_1, k_2}$ is minimax rate optimal when choosing $n^{\frac{d}{2\alpha + 2\beta}} \lesssim k_{\mathrm{min}}$ and $k_{\mathrm{max}} \lesssim \sqrt{n}$.
    \item $\hat{\psi}^{\mathrm{IF}}_{k_1, k_2}$ is minimax rate optimal when choosing $n^{\frac{d}{2\alpha + 2\beta}} \lesssim k_{\mathrm{max}} \lesssim \sqrt{n}$. 
\end{enumerate}
If $\frac{\alpha + \beta}{2} < \frac{d}{2}$, then $\hat{\psi}^{\mathrm{INT}}_{k_1, k_2}$, $\hat{\psi}^{\mathrm{MC}}_{k_1, k_2}$, and $\hat{\psi}^{\mathrm{IF}}_{k_1, k_2}$ cannot be minimax rate optimal for any choice of $k_1, k_2$. The best rate of estimation of $\psi(\bbP)$ for each of these estimators is $n^{-\frac{2\alpha + 2\beta}{\alpha + \beta + d}}$. Further, 
\begin{enumerate}
    \item The best rate of estimation for $\hat{\psi}^{\mathrm{INT}}_{k_1, k_2}$ occurs when choosing $k_{\mathrm{min}} \asymp  n^{\frac{d}{\alpha + \beta + d}} $.
    \item The best rate of estimation for $\hat{\psi}^{\mathrm{MC}}_{k_1, k_2}$ occurs when choosing $k_{\mathrm{min}} \asymp k_{\mathrm{max}} \asymp  n^{\frac{d}{\alpha + \beta + d}} $.
    \item The best rate of estimation for $\hat{\psi}^{\mathrm{IF}}_{k_1, k_2}$ occurs when choosing $k_{\mathrm{max}} \asymp  n^{\frac{d}{\alpha + \beta + d}} $.
\end{enumerate}
\end{cor}
The bottom panel of Figure \ref{fig:minimax} illustrates where the estimators can be minimax rate optimal in the no sample splitting case. These results establish the necessity to perform sample splitting for $\hat{\psi}^{\mathrm{IF}}_{k_1, k_2}$ in low regularity regimes: the own-observation bias arising from not performing sample splitting prevents $\hat{\psi}^{\mathrm{IF}}_{k_1, k_2}$ from achieving minimax rate optimality when $\frac{\alpha + \beta}{2} < \frac{d}{2}$. 

\begin{remark}
    The regime where the estimators without sample splitting fail to achieve minimax rate optimality can be interpreted in terms of Donsker conditions on the nuisance function classes. Donsker conditions have been routinely used in the semiparametric literature to control the size of empirical process remainder terms in order to establish that estimators without sample splitting attain certain optimality properties \cite{van2000asymptotic, tsiatis2006semiparametric}. In our setting, the function class for $p$ is Donsker if and only if $\alpha > \frac{d}{2}$ and similarly the function class for $b$ is Donsker if and only if $\beta > \frac{d}{2}$ \cite{van2023weak, gine2021mathematical}.
    Therefore, in the regime where each of the estimators without sample splitting cannot achieve minimax rate optimality (i.e., $\frac{\alpha + \beta}{2} < \frac{d}{2}$), the model is so large that at least one of the nuisance function classes fails to be Donsker. Unlike existing literature that establishes sufficiency of Donsker conditions in the absence of sample splitting, these results illustrate examples of their necessity.
\end{remark}

\begin{remark}
Unlike the double and single sample splitting cases, setting $k_1 = k_2$ is optimal for $\hat{\psi}^{\mathrm{IF}}_{k_1, k_2}$ (in the sense that $\hat{\psi}^{\mathrm{IF}}_{k_1, k_2}$ achieves its best rate of estimation) in all Hölder smoothness classes of the nuisance functions. This is because $k_{\mathrm{max}}$ cannot be undersmoothed as much as in the double and single sample splitting cases due to the presence of the own-observation bias of order $\frac{k_{\mathrm{max}}}{n}$. However, as in the double and single sample splitting cases, setting $k_1 = k_2$ results in the rate of estimation for $\hat{\psi}^{\mathrm{IF}}_{k_1, k_2}$ being slower than the minimax rate. Similarly, setting $k_1 = k_2$ is optimal for $\hat{\psi}^{\mathrm{INT}}_{k_1, k_2}$ and $\hat{\psi}^{\mathrm{MC}}_{k_1, k_2}$ across Hölder smoothness classes.
\end{remark}

Next, we compare the optimal resolution choices to the prediction-optimal resolutions in Corollaries \ref{cor: pred known f no} and \ref{cor: undermoothing known f no}.

\begin{cor} \label{cor: pred known f no}
Suppose we choose prediction-optimal resolutions $k_1 = k_1^{\text{pred}}$ and $k_2 = k_2^{\text{pred}}$. Then, $\hat{\psi}_{k_1^{\text{pred}}, k_2^{\text{pred}}}^{\mathrm{INT}}$, $\hat{\psi}_{k_1^{\text{pred}}, k_2^{\text{pred}}}^{\mathrm{MC}}$, $\hat{\psi}_{k_1^{\text{pred}}, k_2^{\text{pred}}}^{\mathrm{IF}}$ achieve their best rate for estimating $\psi(\bbP)$ (i.e., $n^{-1}$ when $\frac{\alpha + \beta}{2} \geq \frac{d}{2}$ and $n^{-\frac{2\alpha + 2\beta}{\alpha + \beta + d}}$ when $\frac{\alpha + \beta}{2} < \frac{d}{2}$) if and only if $\alpha \wedge \beta \geq \frac{d}{2}$ or $\alpha = \beta$.
\end{cor}

\begin{cor} \label{cor: undermoothing known f no}
The following statements hold:
\begin{enumerate}
    \item $\hat{\psi}_{k_1, k_2}^{\mathrm{INT}}$ requires undersmoothing $k_1$ when $\beta < \alpha \wedge \frac{d}{2}$ and requires undersmoothing $k_2$ when $\alpha < \beta \wedge \frac{d}{2}$ to achieve its best rate for estimating $\psi(\bbP)$ (i.e., $n^{-1}$ when $\frac{\alpha + \beta}{2} \geq \frac{d}{2}$ and $n^{-\frac{2\alpha + 2\beta}{\alpha + \beta + d}}$ if $\frac{\alpha + \beta}{2} < \frac{d}{2}$).
    \item $\hat{\psi}_{k_1, k_2}^{\mathrm{MC}}$ requires undersmoothing $k_1$ and oversmoothing $k_2$ when $\beta < \alpha \wedge \frac{d}{2}$ and requires undersmoothing $k_2$ and oversmoothing $k_1$ when $\alpha < \beta \wedge \frac{d}{2}$ to achieve its best rate for estimating $\psi(\bbP)$ (i.e., $n^{-1}$ when $\frac{\alpha + \beta}{2} \geq \frac{d}{2}$ and $n^{-\frac{2\alpha + 2\beta}{\alpha + \beta + d}}$ if $\frac{\alpha + \beta}{2} < \frac{d}{2}$).
    \item $\hat{\psi}_{k_1, k_2}^{\mathrm{IF}}$ requires oversmoothing $k_1$ when $\alpha < \beta \wedge \frac{d}{2}$ and requires oversmoothing $k_2$ when $\beta < \alpha \wedge \frac{d}{2}$ to achieve its best rate for estimating $\psi(\bbP)$ (i.e., $n^{-1}$ when $\frac{\alpha + \beta}{2} \geq \frac{d}{2}$ and $n^{-\frac{2\alpha + 2\beta}{\alpha + \beta + d}}$ if $\frac{\alpha + \beta}{2} < \frac{d}{2}$).
\end{enumerate}
\end{cor}

These results are illustrated in Figure \ref{fig:tuning no}. Unlike the double and single sample splitting cases, prediction-optimal resolutions are only satisfactory for $\hat{\psi}_{k_1, k_2}^{\mathrm{IF}}$ when $\alpha, \beta \geq \frac{d}{2}$ (except if $\alpha = \beta$). Otherwise, $\hat{\psi}_{k_1, k_2}^{\mathrm{IF}}$ requires oversmoothing one of the nuisance function estimators to prevent the own-observation bias from growing too large. Additionally, when $\hat{\psi}_{k_1, k_2}^{\mathrm{IF}}$ is able to achieve rate optimality (i.e., when $\frac{\alpha + \beta}{2} \geq \frac{d}{2}$), there are stricter nuisance function tuning requirements compared to the sample splitting cases. In particular, one cannot undersmooth the nuisance functions as much as in the cases with sample splitting in order to prevent the own-observation bias from growing too large. 

Similar conclusions hold for $\hat{\psi}_{k_1, k_2}^{\mathrm{MC}}$. This estimator now requires oversmoothing in low regularity regimes so that the own-observation bias does not grow too large, unlike the sample splitting cases which did not require any oversmoothing. Also, there are stricter nuisance function tuning requirements when this estimator is able to achieve rate optimality (i.e., when $\frac{\alpha + \beta}{2} \geq \frac{d}{2}$) due to the presence of own-observation bias compared to the sample splitting cases.

\begin{remark}
These results emphasize additional subtleties arising in estimating the expected conditional covariance compared to the expected conditional variance \cite{liu2020nearly, liu2021adaptive, shen2020optimal}. Since $\alpha = \beta$ in the expected conditional variance problem, these results imply that using prediction-optimal resolutions is satisfactory over all Hölder smoothness classes of the nuisance function when estimating the expected conditional variance with $\hat{\psi}_{k_1, k_2}^{\mathrm{INT}}$, $\hat{\psi}_{k_1, k_2}^{\mathrm{MC}}$, and $\hat{\psi}_{k_1, k_2}^{\mathrm{IF}}$ in the no sample splitting case. 
\end{remark}

\subsubsection{Estimators that depend on one nuisance function} \label{sec: no ss one}

Here, we analyze the Newey and Robins plug-in estimator in the no sample splitting case. We first bound the bias and variance of this estimator. As in Section \ref{sec: no ss two}, we only provide a lower bound on the bias because it is sufficient to establish optimal resolutions and analyze when prediction-optimal resolutions are satisfactory (due to the presence of own-observation bias).

\begin{theorem} \label{theorem: plugin known f no nr}
Suppose that no sample splitting is performed. Under the assumptions given in Section \ref{sec: motivation}, the estimator $\hat{\psi}_{k}^{\mathrm{NR}}$ satisfies
\begin{align*}
    \sup_{\bbP \in \cP_{(\alpha,\beta)}} \left|\E_\bbP\left(\hat{\psi}_{k}^{\mathrm{NR}} - \psi(\bbP)\right)\right| & \asymp k^{-(\alpha + \beta)/d} \vee \frac{k}{n} \\
    \sup_{\bbP \in \cP_{(\alpha,\beta)}} \Var_\bbP(\hat{\psi}_{k}^{\mathrm{NR}}) & \lesssim \frac{1}{n}  +  \frac{k}{n^2} + \frac{k^2}{n^3}. 
\end{align*}
\end{theorem}

\noindent The proof is given in Appendix G. The proof strategy is similar to that used for the other estimators in the no sample splitting case, as $\hat{\psi}_{k}^{\mathrm{NR}}$ is also subject to own-observation bias. Similar to the other estimators, the variance bound on $\hat{\psi}_{k}^{\mathrm{NR}}$ includes an additional term compared to the single sample splitting case, although this additional term does not dominate the variance when using optimal or prediction-optimal resolutions.

The following corollary establishes the optimal resolution choices and optimal rate of estimation of $\psi(\bbP)$ for $\hat{\psi}^{\mathrm{NR}}_{k}$.

\begin{cor} \label{cor: opt known f no nr}
If $\frac{\alpha + \beta}{2} \geq \frac{d}{2}$, $\hat{\psi}^{\mathrm{NR}}_{k}$ is minimax rate optimal when choosing $n^{\frac{d}{2\alpha + 2\beta}} \lesssim k  \lesssim \sqrt{n}$. If $\frac{\alpha + \beta}{2} < \frac{d}{2}$, $\hat{\psi}^{\mathrm{NR}}_{k}$ cannot be minimax rate optimal for any choice of $k$. The best rate of estimation of $\psi(\bbP)$ for $\hat{\psi}^{\mathrm{NR}}_{k}$ is $n^{-\frac{2\alpha + 2\beta}{\alpha + \beta + d}}$ which occurs when choosing $k \asymp  n^{\frac{d}{\alpha + \beta + d}}$.
\end{cor}

The results are illustrated in the bottom panel in Figure \ref{fig:minimax}. Unlike the sample splitting case, $\hat{\psi}^{\mathrm{NR}}_{k}$ cannot be rate optimal across all Hölder smoothness classes of the nuisance functions in the no sample splitting case. Due to the presence of own-observation bias, the bias of $\hat{\psi}^{\mathrm{NR}}_{k}$ cannot converge at a fast enough rate when $\frac{\alpha + \beta}{2} < \frac{d}{2}$, similar to the estimators with two nuisance functions.

We last compare the optimal and prediction-optimal resolutions in Corollaries \ref{cor: pred known f nr no} and \ref{cor: undersmoothing known f nr no}.

\begin{cor} \label{cor: pred known f nr no}
Suppose we choose the prediction-optimal resolution $k = k_2^{\text{pred}}$. Then, $\hat{\psi}_{k_2^{\text{pred}}}^{\mathrm{NR}}$ achieves its best rate for estimating $\psi(\bbP)$ (i.e., $n^{-1}$ when $\frac{\alpha + \beta}{2} \geq \frac{d}{2}$ and $n^{-\frac{2\alpha + 2\beta}{\alpha + \beta + d}}$ if $\frac{\alpha + \beta}{2} < \frac{d}{2}$) if and only if $\alpha \wedge \beta \geq \frac{d}{2}$ or $\alpha = \beta$.
\end{cor}

\begin{cor} \label{cor: undersmoothing known f nr no}
The estimator $\hat{\psi}_{k}^{\mathrm{NR}}$ requires undersmoothing $k$ when $\alpha < \beta \wedge \frac{d}{2}$ and oversmoothing $k$ when $\beta < \alpha \wedge \frac{d}{2}$ to achieve its best rate for estimating $\psi(\bbP)$ (i.e., $n^{-1}$ when $\frac{\alpha + \beta}{2} \geq \frac{d}{2}$ and $n^{-\frac{2\alpha + 2\beta}{\alpha + \beta + d}}$ if $\frac{\alpha + \beta}{2} < \frac{d}{2}$).
\end{cor}

The results are summarized in the bottom right panel of Figure \ref{fig:tuning no}. Similar to the first-order estimator, the own-observation bias affects when the Newey and Robins plug-in estimator can use prediction-optimal resolutions to achieve the best rate for estimating $\psi(\bbP)$. Moreover, when minimax rate optimality is possible, there are stricter nuisance function tuning requirements compared to the single sample splitting case so that the own-observation bias does not grow too large.

\begin{figure}[!h]
\centering
\begin{tabular}{ll}
  \includegraphics[width=0.45\textwidth]{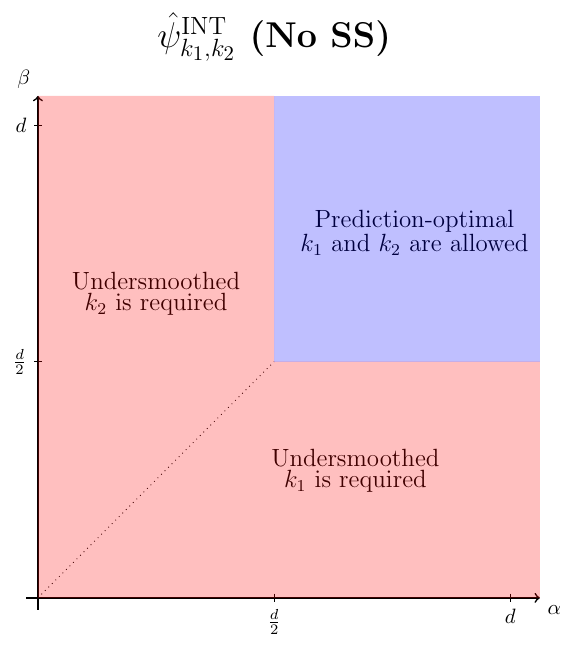} & \includegraphics[width=0.45\textwidth]{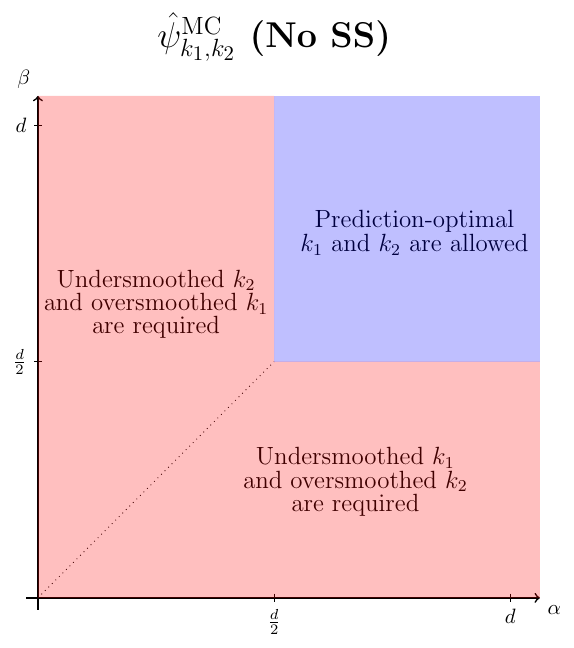} \\ 
  \includegraphics[width=0.45\textwidth]{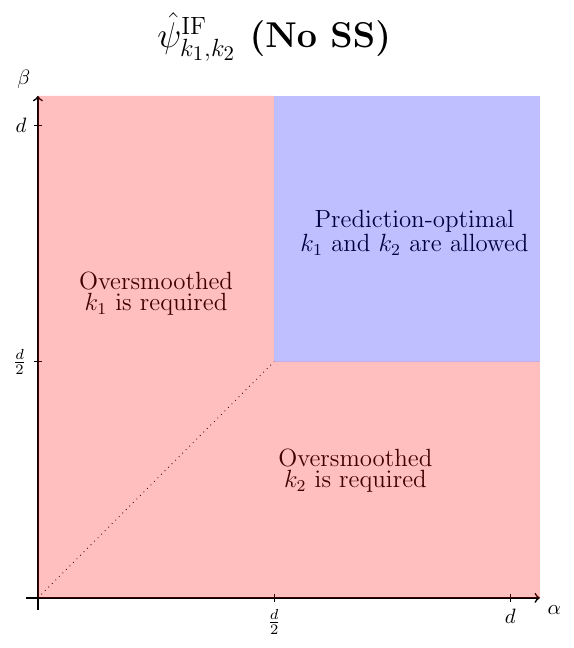} & 
  \includegraphics[width=0.45\textwidth]{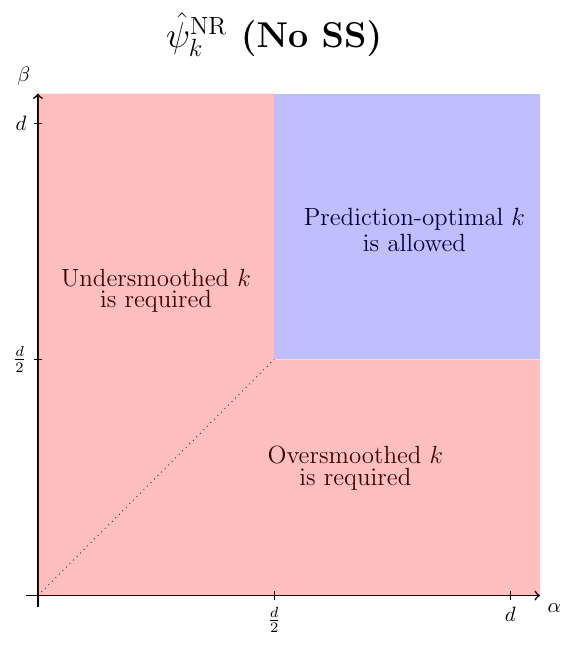}
\end{tabular}
\caption{Optimally tuning the resolutions for $\hat{\psi}_{k_1, k_2}^{\mathrm{INT}}$, $\hat{\psi}_{k_1, k_2}^{\mathrm{MC}}$, $\hat{\psi}_{k_1, k_2}^{\mathrm{IF}}$ and $\hat{\psi}_{k}^{\mathrm{NR}}$ with no sample splitting (SS).}
    \label{fig:tuning no}
  \end{figure}

\subsection{Comparison with existing results}

A few works have explored the possibility of producing minimax rate optimal estimators of $\psi(\bbP)$  by using sample splitting and suitable nuisance function tuning techniques. In this subsection, we summarize these works and compare their results to ours.

\subsubsection{Newey and Robins \cite{newey2018cross}}Newey and Robins \cite{newey2018cross} analyzed a sample split plug-in estimator and first-order estimator of $\psi(\bbP)$ under Hölder smoothness assumptions on the nuisance functions. They used regression splines to estimate the nuisance functions, in which case the tuning parameter is the number of splines. They found that their double sample split first-order estimator can be rate optimal in the $\frac{\alpha + \beta}{2} > \frac{d}{4}$ regime when using a suitable number of regression splines. Moreover, they showed that their sample split Newey and Robins plug-in estimator can be rate optimal in this regime when the density of $\bX$ is known (i.e., the second-moment matrix of the regression splines is known), and they extended this result in the unknown density case when $\alpha, \beta \leq 1$ or $\beta > \frac{d}{2}$. They also found that the bias of their first-order estimator incurs an additional $O(\frac{k}{n})$ term in the single sample splitting and no sample splitting cases when setting $k_1 = k_2 = k$, which no longer yields rate optimality when $\alpha$ and $\beta$ are sufficiently small (see also Fisher and Fisher \cite{fisher2023three}, which established a corresponding lower bound in this setting). Their use of nuisance function estimators based on least squares type estimation does not allow for the regime $\frac{\alpha + \beta}{2} < \frac{d}{4}$ due to the Gram matrix being non-invertible when using greater than $n$ regression splines.

In contrast, we adopt wavelet-based estimators of the nuisance functions which allows us to study the $\frac{\alpha + \beta}{2} < \frac{d}{4}$  regime. Our results can be seen as natural extensions to those of \cite{newey2018cross} in this regime: we find that (i) the double sample split first-order estimator can be rate optimal in this regime with suitable undersmoothing, (ii) the sample split Newey and Robins plug-in estimator is rate optimal in this regime when performing suitable undersmoothing, and (iii) the single sample split first-order estimator cannot be rate optimal in this regime when choosing $k_1 = k_2$. Moreover, since our analyses do not require $k_1 = k_2$, we clarify that the single sample split first-order estimator can be minimax rate optimal across all Hölder smoothness classes when undersmoothing only one of the nuisance function estimators. More broadly, the key contribution of our work is clarifying when prediction-optimal tuning is satisfactory versus when undersmoothing or oversmoothing is necessary across different regularity regimes, sample splitting strategies, and types of estimators.

\subsubsection{McClean et al.\ \cite{mcclean2024double}} After a preprint of this work appeared, McClean et al.\ \cite{mcclean2024double} analyzed a double sample split first-order estimator of $\psi(\bbP)$ under Hölder smoothness assumptions as well as under weaker assumptions on the nuisance functions. In their most closely related result (Theorem 2) to our work, they adopted Hölder smoothness assumptions and showed that the first-order estimator is minimax rate optimal when using undersmoothed kernel regression estimators of the nuisance functions. More specifically, they selected bandwidths of the kernel regression estimators so that one of the nuisance functions is undersmoothed while the other nuisance function estimator is consistent. These bandwidth conditions are analogous to our resolution conditions for the first-order estimator with wavelet-based nuisance function estimators in Corollary \ref{cor: opt known f double}. In addition to rates of estimation, McClean et al.\ \cite{mcclean2024double} showed that the double sample split first-order estimator with suitable undersmoothing is asymptotically normal in $\sqrt{n}$ and non-$\sqrt{n}$ regimes. Under less restrictive assumptions -- including when the smoothness parameters $\alpha, \beta$ and the density of $\bX$ are unknown -- they showed that the first-order estimator is rate optimal in the $\sqrt{n}$ regime when using local polynomial regression and k-nearest neighbors estimators of the nuisance functions. Finally, under no assumptions on the nuisance functions, they provide an analysis of the first-order estimator which illustrates how undersmoothing can help achieve rate optimality.

In contrast to McClean et al.\ \cite{mcclean2024double} that studied the benefits of undersmoothing nuisance function estimators under different types of models and nuisance function estimators, our work studies the benefits of undersmoothing (and oversmoothing) under different types of sample splitting strategies (i.e., double, single, and no sample splitting) and estimators of the functional of interest (i.e., three plug-in estimators and the first-order estimator). By doing so, we illustrate a delicate interplay between the optimal way to tune nuisance function estimators and the type of sample splitting strategy and estimator of the functional of interest. Furthermore, a distinguishing feature of our work are the lower bounds on the bias of the estimators, which establish the \emph{necessity} to undersmooth or oversmooth the nuisance function estimators for obtaining optimal rates of estimating $\psi(\bbP)$.

\section{Simulations} \label{sec: simulations} 

In this section, we perform numerical simulations to illustrate the finite sample behavior of the four types of estimators of $\psi(\bbP)$ (i.e., $\hat{\psi}_{k_1, k_2}^{\mathrm{INT}}$, $\hat{\psi}_{k_1, k_2}^{\mathrm{MC}}$, $\hat{\psi}_{k}^{\mathrm{NR}}$, $\hat{\psi}_{k_1, k_2}^{\mathrm{IF}}$). Specifically, we compare the prediction-optimal resolutions to the optimal resolutions for estimating $\psi(\bbP)$ for each of the four estimators in each possible sample splitting strategy and in various regularity regimes. We then illustrate the impact of using prediction-optimal versus optimal resolutions on the mean squared error, bias, and variance of the estimators of $\psi(\bbP)$. We also perform these analyses with corresponding cross-fit estimators in Appendix K.

\subsection{Design} \label{sec: sim design}

We set $n \in \{300, 30000\}$ and generated $n$ i.i.d.\ copies of $(X, A, Y)$ as follows. We first drew $X \sim \text{Uniform}(0, 1)$. We adopted the setting where the conditional means of $A$ and $Y$ are the same (i.e., $\mu := p = b$). We independently drew $A$ and $Y$ from the $\text{Normal}(\mu(X), 1)$ distribution.  We set $\mu$ to lie near the boundary of a Hölder space with smoothness class $s$ (see Appendix K for details). We set $s \in \{0.05, 0.25, 0.75\}$, which we refer to as the low, medium, and high regularity settings, to illustrate $\sqrt{n}$ and non-$\sqrt{n}$ regimes. Figure \ref{fig: distributions} illustrates $\mu$ in these settings.

We applied each estimator in each sample splitting strategy, where the sample splits were of equal size. To estimate the nuisance functions, we use the approximate wavelet projection estimator with the Haar basis. We numerically computed the prediction-optimal resolutions (i.e., minimizing the mean integrated squared error for $\mu$) and the optimal resolutions for estimating $\psi(\bbP)$ (i.e., minimizing the mean squared error of the estimator of $\psi(\bbP)$). In brief, we considered a discrete grid of values for $(k_1, k_2)$ and performed Monte Carlo integration for each $(k_1, k_2)$.  See Appendix K for additional details.

\begin{figure}[!ht] 
  \centering
   \includegraphics[width=\textwidth]{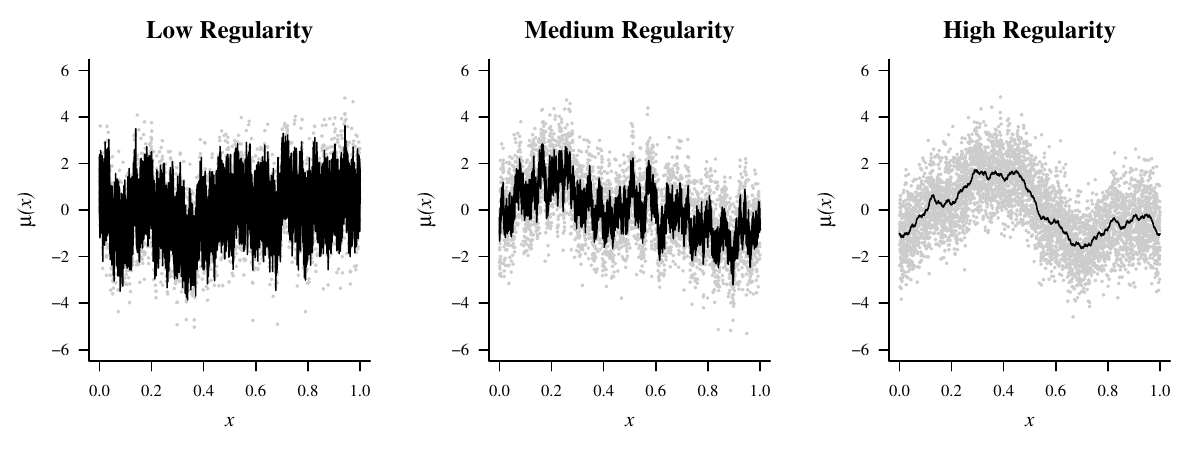}
   \caption{Conditional mean function $\mu$ in the low, medium, and high regularity regimes in the simulations. Grey points show 5,000 samples from the distribution of $A$ and $Y$.} \label{fig: distributions}
\end{figure}

\subsection{Results}
Here, we present the simulation results for $n = 300$ and present the remaining simulation results (i.e., cross-fit estimators and the $n=30,000$ case) in Appendix K.

 The results for the double sample splitting estimators are presented in Table \ref{tab:summary_double}. As suggested by our theoretical results, in the low and medium regularity regimes, the optimal resolution choices for $\hat{\psi}_{k_1, k_2}^{\mathrm{INT}}$ and  $\hat{\psi}_{k_1, k_2}^{\mathrm{MC}}$ were undersmoothed $k_1, k_2$ values and the optimal resolution choices for $\hat{\psi}_{k_1, k_2}^{\mathrm{IF}}$ involved substantially undersmoothing only one of $k_1, k_2$. By doing so, these estimators reduced their MSE for $\psi(\bbP)$ by reducing their squared bias at the expense of slightly increasing their variance. In the high regularity regime, the prediction-optimal and optimal resolution choices for each estimator were either exactly or nearly equal. Therefore, the estimators had nearly the same MSE regardless of whether prediction-optimal and optimal resolutions were used. This finding also aligns with our theoretical results, as prediction-optimal and optimal resolutions achieve the same rates for estimating $\psi(\bbP)$ in this regime.

\begin{table}[!h]
\centering
\caption{Simulation results for double sample split estimators in settings with $n = 300$.} 
\label{tab:summary_double}
\begin{tabular}{llllllllllll}
  \hline
& & \multicolumn{5}{c}{Prediction-Optimal Resolutions} & \multicolumn{5}{c}{Optimal Resolutions} \\
\cmidrule(lr){3-7} \cmidrule(lr){8-12}
Regularity & Estimator & MSE & Bias\textsuperscript{2} & Var & $k_1$  & $k_2$ & MSE & Bias\textsuperscript{2} & Var  & $k_1$  & $k_2$ \\ 
  \hline
Low & Integral & 0.789 & 0.734 & 0.055 & 4 & 4 & 0.390 & 0.263 & 0.127 & 128 & 128 \\ 
   & Monte Carlo & 0.789 & 0.734 & 0.055 & 4 & 4 & 0.452 & 0.328 & 0.124 & 64 & 64 \\ 
   & First-Order & 0.779 & 0.734 & 0.044 & 4 & 4 & 0.345 & 0.201 & 0.144 & 256 & 8 \\ 
  Medium & Integral & 0.278 & 0.206 & 0.072 & 8 & 8 & 0.156 & 0.039 & 0.117 & 64 & 64 \\ 
   & Monte Carlo & 0.277 & 0.206 & 0.071 & 8 & 8 & 0.183 & 0.077 & 0.106 & 32 & 32 \\ 
   & First-Order & 0.237 & 0.206 & 0.031 & 8 & 8 & 0.089 & 0.039 & 0.050 & 8 & 64 \\ 
  High & Integral & 0.085 & 0.009 & 0.076 & 8 & 8 & 0.083 & 0.003 & 0.080 & 16 & 16 \\ 
   & Monte Carlo & 0.082 & 0.009 & 0.073 & 8 & 8 & 0.082 & 0.009 & 0.073 & 8 & 8 \\ 
   & First-Order & 0.028 & 0.009 & 0.019 & 8 & 8 & 0.024 & 0.003 & 0.021 & 16 & 8 \\ 
   \hline
\end{tabular}
\end{table}

Table \ref{tab:summary_single} summarizes the single sample splitting results. The optimal resolutions for $\hat{\psi}_{k_1, k_2}^{\mathrm{INT}}$ and  $\hat{\psi}_{k_1, k_2}^{\mathrm{MC}}$ often involved some degree of undersmoothing (e.g., in the low and medium regularity regimes), despite our theoretical results implying that prediction-optimal resolutions produce rate optimal estimation of $\psi(\bbP)$ for these estimators in these settings. As expected from our theoretical results, the optimal resolutions for $\hat{\psi}_{k_1, k_2}^{\mathrm{IF}}$ involved undersmoothing one of the nuisance functions and oversmoothing the other nuisance function in the low and medium regularity regimes. With optimal resolutions, the single sample split estimators generally reduced their MSEs.

\begin{table}[!h]
\centering
\caption{Simulation results for single sample split estimators in settings with $n = 300$.} 
\label{tab:summary_single}
\begin{tabular}{llllllllllll}
  \hline
& & \multicolumn{5}{c}{Prediction-Optimal Resolutions} & \multicolumn{5}{c}{Optimal Resolutions} \\
\cmidrule(lr){3-7} \cmidrule(lr){8-12}
Regularity & Estimator & MSE & Bias\textsuperscript{2} & Var & $k_1$  & $k_2$ & MSE & Bias\textsuperscript{2} & Var  & $k_1$  & $k_2$ \\ 
  \hline
Low & Integral & 0.598 & 0.553 & 0.044 & 8 & 8 & 0.111 & 0.023 & 0.087 & 64 & 64 \\ 
   & Monte Carlo & 0.598 & 0.554 & 0.044 & 8 & 8 & 0.137 & 0.024 & 0.113 & 64 & 64 \\ 
   & First-Order & 0.747 & 0.717 & 0.030 & 8 & 8 & 0.293 & 0.176 & 0.117 & 4 & 512 \\ 
   & Newey-Robins & 0.664 & 0.633 & 0.032 &  & 8 & 0.283 & 0.204 & 0.080 &  & 256 \\ 
  Medium & Integral & 0.221 & 0.164 & 0.057 & 8 & 8 & 0.091 & 0.005 & 0.087 & 32 & 32 \\ 
   & Monte Carlo & 0.220 & 0.164 & 0.056 & 8 & 8 & 0.101 & 0.005 & 0.097 & 32 & 32 \\ 
   & First-Order & 0.273 & 0.254 & 0.019 & 8 & 8 & 0.062 & 0.023 & 0.039 & 128 & 4 \\ 
   & Newey-Robins & 0.236 & 0.206 & 0.029 &  & 8 & 0.075 & 0.016 & 0.059 &  & 128 \\ 
  High & Integral & 0.064 & 0.002 & 0.062 & 8 & 8 & 0.064 & 0.002 & 0.062 & 8 & 8 \\ 
   & Monte Carlo & 0.062 & 0.002 & 0.060 & 8 & 8 & 0.062 & 0.002 & 0.060 & 8 & 8 \\ 
   & First-Order & 0.032 & 0.020 & 0.012 & 8 & 8 & 0.017 & 0.001 & 0.015 & 4 & 32 \\ 
   & Newey-Robins & 0.036 & 0.009 & 0.027 &  & 8 & 0.031 & 0.003 & 0.028 &  & 16 \\ 
   \hline
\end{tabular}
\end{table}

Table \ref{tab:summary_none} summarizes the no sample splitting results. Our theoretical results imply that prediction-optimal resolutions are satisfactory for obtaining the best rates of estimation possible for $\psi(\bbP)$ in all these settings. Here, we find that the optimal resolution choices often involved some degree of undersmoothing, especially in the low and medium regularity regimes. Choosing optimal resolutions resulted in substantial reductions in MSE in the low and medium regularity regimes (e.g., a reduction by a factor of 35 for $\hat{\psi}_{k_1, k_2}^{\mathrm{IF}}$ in the low regularity regime).

\begin{table}[!h]
\centering
\caption{Simulation results for the estimators without sample splitting in settings with $n = 300$.} 
\label{tab:summary_none}
\begin{tabular}{llllllllllll}
  \hline
& & \multicolumn{5}{c}{Prediction-Optimal Resolutions} & \multicolumn{5}{c}{Optimal Resolutions} \\
\cmidrule(lr){3-7} \cmidrule(lr){8-12}
Regularity & Estimator & MSE & Bias\textsuperscript{2} & Var & $k_1$  & $k_2$ & MSE & Bias\textsuperscript{2} & Var  & $k_1$  & $k_2$ \\ 
  \hline
Low & Integral & 0.442 & 0.430 & 0.012 & 16 & 16 & 0.022 & 0.008 & 0.015 & 128 & 128 \\ 
   & Monte Carlo & 0.406 & 0.389 & 0.017 & 16 & 16 & 0.059 & 0.000 & 0.059 & 64 & 128 \\ 
   & First-Order & 0.483 & 0.473 & 0.011 & 16 & 16 & 0.014 & 0.000 & 0.014 & 256 & 32 \\ 
   & Newey-Robins & 0.442 & 0.430 & 0.012 &  & 16 & 0.022 & 0.008 & 0.015 &  & 128 \\ 
  Medium & Integral & 0.119 & 0.110 & 0.009 & 16 & 16 & 0.013 & 0.000 & 0.013 & 64 & 64 \\ 
   & Monte Carlo & 0.091 & 0.072 & 0.019 & 16 & 16 & 0.036 & 0.000 & 0.036 & 16 & 128 \\ 
   & First-Order & 0.164 & 0.157 & 0.007 & 16 & 16 & 0.009 & 0.001 & 0.008 & 16 & 128 \\ 
   & Newey-Robins & 0.119 & 0.110 & 0.009 &  & 16 & 0.013 & 0.000 & 0.013 &  & 64 \\ 
  High & Integral & 0.012 & 0.005 & 0.007 & 8 & 8 & 0.007 & 0.000 & 0.007 & 16 & 16 \\ 
   & Monte Carlo & 0.019 & 0.001 & 0.018 & 8 & 8 & 0.019 & 0.001 & 0.018 & 8 & 8 \\ 
   & First-Order & 0.017 & 0.013 & 0.004 & 8 & 8 & 0.004 & 0.000 & 0.004 & 4 & 32 \\ 
   & Newey-Robins & 0.012 & 0.005 & 0.007 &  & 8 & 0.007 & 0.000 & 0.007 &  & 16 \\ 
   \hline
\end{tabular}
\end{table}

\section{Discussion} \label{sec: discussion}

In this paper, we provide a detailed description of how to optimally tune nuisance function estimators for estimating a doubly robust functional which has witnessed interest across the causal inference and conditional independence testing literature. Our results describe when and how one has to tailor nuisance function tuning to obtain optimal rates of convergence for estimating the functional of interest. In particular, our results illustrate the necessity to perform some degree of undersmoothing or oversmoothing in low regularity regimes. Our results additionally demonstrate the impact of different sample splitting schemes and the necessity to perform sample splitting for the sake of eventually obtaining minimax rate-optimal functional estimation. We demonstrate this through the lens of a popular class of estimators consisting of simple plug-in principles and influence function based first-order bias-corrected methods. 

While our analyses consider the nuisance functions to belong to Hölder spaces, the bounds we derived extend to other Besov spaces (see Appendix A). One may also potentially produce minimax optimal estimators that are adaptive over the Hölder smoothness classes of the nuisance functions by applying Lepski-type methods \cite{lepskii1991problem, lepskii1992asymptotically}. For the sake of concreteness, our analyses pertain to estimators of nuisance functions based on wavelet projections. Recently, McClean et al. \cite{mcclean2024double} considered a similar setup as this work and provided some parallel upper bounds when using other types of nuisance function estimators (e.g., kernel regression, local polynomial regression, k-nearest neighbors) for the double sample split first-order estimator. We expect similar lower bounds and optimal nuisance function tuning strategies to hold when using these types of nuisance function estimators.

Although the choice of the specific functional is somewhat driven by the fact that one can perform detailed calculations to establish both sharp upper and lower bounds for many classes of estimators for the problem, our results provide intuition for the more general bracket of doubly robust functionals in the literature. That is, it is natural to expect similar phenomena for estimating other doubly robust functionals in the literature, such as the average treatment effect. In particular, as we discussed in Section \ref{sec: setup}, our object of study has deep connections with the average treatment effect functional. However, estimation of the average treatment effect functional necessitates estimating the inverse propensity score (i.e. $1/p$ in our context). Although recent strategies of directly obtaining a projection estimator of $1/p$ \cite{chernozhukov2018double,newey2018cross,chernozhukov2022automatic,chernozhukov2018double2} can be analyzed, these strategies can only be operationalized in the $\sqrt{n}$-regime. It is unclear at this moment how to modify these estimators for the slower rate of convergence regimes that we explore in this paper. We keep these explorations for future research endeavors.

%
%

\begin{acks}[Acknowledgments]
We would like to thank James Robins for insightful discussions. We would also like to thank the Associate Editor and referees for their comments which helped improve the paper. SM is jointly affiliated with the Bursky School of Public Health at Washington University in St. Louis. This work was primarily conducted when SM was at Harvard University, and part of the work was conducted when SM was at Yale University. The simulations were run on the FASRC Cannon cluster supported by the FAS Division of Science Research Computing Group at Harvard University.
\end{acks}
\begin{funding}
SM was supported by the National Science Foundation Graduate Research Fellowship Program under Grant No.\ DGE2140743, National Library of Medicine of the National Institutes of Health under Award Number T32LM012411, and Fonds de recherche du Québec-Nature et technologies B1X research scholarship. RM was supported by the National Science Foundation (DMS CAREER 2338760).  Any opinions, findings, and conclusions or recommendations expressed in this material are those of the authors and do not necessarily reflect the views of the funding agencies.
\end{funding}

\begin{supplement}
\stitle{Supplementary Material: Additional Results and Proofs}
\sdescription{This supplement includes a review of wavelet bases and Hölder spaces, extensions of our theoretical results, proofs of all our results, and additional simulation results.}
\end{supplement}

\begin{supplement}
\stitle{Supplementary Material: Code}
\sdescription{This supplement includes all code used in the simulations. The simulation code and output files are also available at \url{https://github.com/stmcg/nuisance-tuning}.}
\end{supplement}


\bibliographystyle{imsart-number} 
\bibliography{bibliography}       

@techreport{chernozhukov2018double2,
  title={Double/de-biased machine learning using regularized Riesz representers},
  author={Chernozhukov, Victor and Newey, Whitney K and Robins, James},
  year={2018},
  institution={Centre for Microdata Methods and Practice}, 
  type = {Working Paper}, 
  number = {15/18},
url = {http://hdl.handle.net/10419/189711},
}

@article{chernozhukov2022automatic,
  title={Automatic debiased machine learning of causal and structural effects},
  author={Chernozhukov, Victor and Newey, Whitney K and Singh, Rahul},
  journal={Econometrica},
  volume={90},
  number={3},
  pages={967--1027},
  year={2022},
  publisher={Wiley Online Library}
}

@article{robins2009semiparametric,
  title={Semiparametric minimax rates},
  author={Robins, James and Tchetgen, Eric Tchetgen and Li, Lingling and van der Vaart, Aad},
  journal={Electronic journal of statistics},
  volume={3},
  pages={1305},
  year={2009},
  publisher={NIH Public Access}
}

@article{hall1992effect,
  title={Effect of bias estimation on coverage accuracy of bootstrap confidence intervals for a probability density},
  author={Hall, Peter},
  journal={The Annals of Statistics},
  pages={675--694},
  year={1992},
  publisher={JSTOR}
}

@article{paninski2008undersmoothed,
  title={Undersmoothed kernel entropy estimators},
  author={Paninski, Liam and Yajima, Masanao},
  journal={IEEE Transactions on Information Theory},
  volume={54},
  number={9},
  pages={4384--4388},
  year={2008},
  publisher={IEEE}
}

@article{van2019efficient,
url = {https://doi.org/10.1515/ijb-2019-0092},
title = {Efficient estimation of pathwise differentiable target parameters with the undersmoothed highly adaptive lasso},
author = {Mark J. van der Laan and David Benkeser and Weixin Cai},
journal = {The International Journal of Biostatistics},
doi = {10.1515/ijb-2019-0092},
year = {2022},
lastchecked = {2023-03-02}
}

@techreport{newey1998undersmoothing,
title = {Undersmoothing and Bias Corrected Functional Estimation},
author = {Newey, Whitney and Hsieh, Fushing and Robins, James},
year = {1998},
institution = {Massachusetts Institute of Technology (MIT), Department of Economics},
type = {Working Paper},
number = {98-17},
url = {https://EconPapers.repec.org/RePEc:mit:worpap:98-17}
}

@article{bang2005doubly,
  title={Doubly robust estimation in missing data and causal inference models},
  author={Bang, Heejung and Robins, James M},
  journal={Biometrics},
  volume={61},
  number={4},
  pages={962--973},
  year={2005},
  publisher={Wiley Online Library}
}

@article{shah2020hardness,
  title={The hardness of conditional independence testing and the generalised covariance measure},
  author={Shah, Rajen D and Peters, Jonas},
  journal={The Annals of Statistics},
  volume={48},
  number={3},
  pages={1514--1538},
  year={2020},
  publisher={Institute of Mathematical Statistics}
}

@article{liu2021adaptive,
  title={Adaptive estimation of nonparametric functionals},
  author={Liu, Lin and Mukherjee, Rajarshi and Robins, James M and Tchetgen, Eric Tchetgen},
  journal={Journal of Machine Learning Research},
  volume={22},
  number={99},
  pages={1--66},
  year={2021}
}

@article{rotnitzky2021characterization,
  title={Characterization of parameters with a mixed bias property},
  author={Rotnitzky, Andrea and Smucler, Ezequiel and Robins, James M},
  journal={Biometrika},
  volume={108},
  number={1},
  pages={231--238},
  year={2021},
  publisher={Oxford University Press}
}

@book{gine2021mathematical,
  title={Mathematical foundations of infinite-dimensional statistical models},
  author={Gin{\'e}, Evarist and Nickl, Richard},
  year={2021},
  publisher={Cambridge University Press}
}

@article{kennedy2023towards,
  title={Towards optimal doubly robust estimation of heterogeneous causal effects},
  author={Kennedy, Edward H},
  journal={Electronic Journal of Statistics},
  volume={17},
  number={2},
  pages={3008--3049},
  year={2023},
  publisher={The Institute of Mathematical Statistics and the Bernoulli Society}
}

@article{kennedy2024minimax,
  title={Minimax rates for heterogeneous causal effect estimation},
  author={Kennedy, Edward H and Balakrishnan, Sivaraman and Robins, James M and Wasserman, Larry},
  journal={The Annals of Statistics},
  volume={52},
  number={2},
  pages={793--816},
  year={2024},
  publisher={Institute of Mathematical Statistics}
}

@article{chernozhukov2018double,
    author = {Chernozhukov, Victor and Chetverikov, Denis and Demirer, Mert and Duflo, Esther and Hansen, Christian and Newey, Whitney and Robins, James},
    title = "{Double/debiased machine learning for treatment and structural parameters}",
    journal = {The Econometrics Journal},
    volume = {21},
    number = {1},
    pages = {C1-C68},
    year = {2018},
    month = {01},
    issn = {1368-4221},
    doi = {10.1111/ectj.12097},
    url = {https://doi.org/10.1111/ectj.12097},    eprint = {https://academic.oup.com/ectj/article-pdf/21/1/C1/27684918/ectj00c1.pdf},
}

@article{bickel1988estimating,
  title={Estimating integrated squared density derivatives: sharp best order of convergence estimates},
  author={Bickel, Peter J and Ritov, Yaacov},
  journal={Sankhy{\=a}: The Indian Journal of Statistics, Series A},
  pages={381--393},
  year={1988},
  publisher={JSTOR}
}

@techreport{zheng2010asymptotic,
  title={Asymptotic theory for cross-validated targeted maximum likelihood estimation},
  author={Zheng, Wenjing and Van Der Laan, Mark J},
  institution={U.C. Berkeley Division of Biostatistics}, 
  year={2010},
  type    = {Working Paper},
number  = {273},
url={https://biostats.bepress.com/ucbbiostat/paper273}
}

@article{gine2008simple,
  title={A simple adaptive estimator of the integrated square of a density},
  author={Gin{\'e}, Evarist and Nickl, Richard},
  journal={Bernoulli},
  volume={14},
  number={1},
  pages={47--61},
  year={2008},
  publisher={Bernoulli Society for Mathematical Statistics and Probability}
}

@incollection{robins2008higher,
  title={Higher order influence functions and minimax estimation of nonlinear functionals},
  author={Robins, James and Li, Lingling and Tchetgen, Eric and van der Vaart, Aad},
  booktitle={Probability and statistics: essays in honor of David A. Freedman},
  pages={335--421},
  year={2008},
  publisher={Institute of Mathematical Statistics}
}

@article{newey2018cross,
  title={Cross-fitting and fast remainder rates for semiparametric estimation},
  author={Newey, Whitney K and Robins, James R},
  journal={arXiv preprint arXiv:1801.09138},
  year={2018}
}

@article{robins2017minimax,
  title={Minimax estimation of a functional on a structured high-dimensional model},
  author={Robins, James and Li, Lingling and Mukherjee, Rajarshi and Tchetgen, Eric Tchetgen and van der Vaart, Aad and others},
  journal={The Annals of Statistics},
  volume={45},
  number={5},
  pages={1951--1987},
  year={2017},
  publisher={Institute of Mathematical Statistics}
}

@article{robins2017higher,
  title={Higher order estimating equations for high-dimensional models},
  author={Robins, James and Li, Lingling and Mukherjee, Rajarshi and Tchetgen, Eric Tchetgen and van der Vaart, Aad},
  journal={Annals of statistics},
  volume={45},
  number={5},
  pages={1951},
  year={2017},
  publisher={NIH Public Access}
}

@book{hernan2020causal,
  title={Causal Inference: What If},
  author={Hern{\'a}n, Miguel A and Robins, James M},
  year={2020},
  publisher={CRC Boca Raton, FL}
}

@article{crump2009dealing,
  title={Dealing with limited overlap in estimation of average treatment effects},
  author={Crump, Richard K and Hotz, V Joseph and Imbens, Guido W and Mitnik, Oscar A},
  journal={Biometrika},
  volume={96},
  number={1},
  pages={187--199},
  year={2009},
  publisher={Oxford University Press}
}

@article{robins1992estimating,
  title={Estimating exposure effects by modelling the expectation of exposure conditional on confounders},
  author={Robins, James M and Mark, Steven D and Newey, Whitney K},
  journal={Biometrics},
  pages={479--495},
  year={1992},
  publisher={JSTOR}
}

@article{bickel1982adaptive,
  title={On adaptive estimation},
  author={Bickel, Peter J},
  journal={The Annals of Statistics},
  pages={647--671},
  year={1982},
  publisher={JSTOR}
}

@article{powell1989semiparametric,
  title={Semiparametric estimation of index coefficients},
  author={Powell, James L and Stock, James H and Stoker, Thomas M},
  journal={Econometrica: Journal of the Econometric Society},
  pages={1403--1430},
  year={1989},
  publisher={JSTOR}
}

@article{laurent1996efficient,
  title={Efficient estimation of integral functionals of a density},
  author={Laurent, B{\'e}atrice},
  journal={The Annals of Statistics},
  volume={24},
  number={2},
  pages={659--681},
  year={1996},
  publisher={Institute of Mathematical Statistics}
}

@book{van2023weak,
  title={Weak Convergence and Empirical Processes: With Applications to Statistics},
  author={van der Vaart, A.W. and Wellner, J.A.},
  isbn={9783031290404},
  series={Springer Series in Statistics},
  url={https://books.google.com/books?id=vfzKEAAAQBAJ},
  year={2023},
  publisher={Springer International Publishing}
}

@book{van2000asymptotic,
  title={Asymptotic statistics},
  author={Van der Vaart, Aad W},
  volume={3},
  year={2000},
  publisher={Cambridge university press}
}

@article{lepskii1991problem,
  title={On a problem of adaptive estimation in Gaussian white noise},
  author={Lepskii, OV},
  journal={Theory of Probability \& Its Applications},
  volume={35},
  number={3},
  pages={454--466},
  year={1991},
  publisher={SIAM}
}

@article{lepskii1992asymptotically,
  title={Asymptotically minimax adaptive estimation. i: Upper bounds. optimally adaptive estimates},
  author={Lepskii, OV},
  journal={Theory of Probability \& Its Applications},
  volume={36},
  number={4},
  pages={682--697},
  year={1992},
  publisher={SIAM}
}

@article{diaz2023non,
  title={Non-agency interventions for causal mediation in the presence of intermediate confounding},
  author={D{\'\i}az, Iv{\'a}n},
  journal={Journal of the Royal Statistical Society Series B: Statistical Methodology},
  pages={qkad130},
  year={2023},
  publisher={Oxford University Press US}
}

@article{zhou2022marginal,
  title={Marginal Interventional Effects},
  author={Zhou, Xiang and Opacic, Aleksei},
  journal={arXiv preprint arXiv:2206.10717},
  year={2022}
}

@article{liu2017semiparametric,
  title={Semiparametric efficient empirical higher order influence function estimators},
  author={Liu, Lin and Mukherjee, Rajarshi and Newey, Whitney K and Robins, James M},
  journal={arXiv preprint arXiv:1705.07577},
  year={2017}
}

@book{hardle2012wavelets,
  title={Wavelets, approximation, and statistical applications},
  author={H{\"a}rdle, Wolfgang and Kerkyacharian, Gerard and Picard, Dominique and Tsybakov, Alexander},
  volume={129},
  year={2012},
  publisher={Springer Science \& Business Media}
}

@article{mcclean2024double,
  title={Double cross-fit doubly robust estimators: Beyond series regression},
  author={McClean, Alec and Balakrishnan, Sivaraman and Kennedy, Edward H and Wasserman, Larry},
  journal={Journal of the Royal Statistical Society Series B: Statistical Methodology},
  pages={qkag057},
  year={2026},
  publisher={Oxford University Press UK}
}

@article{balakrishnan2026fundamental,
  title={The fundamental limits of structure-agnostic functional estimation},
  author={Balakrishnan, Sivaraman and Kennedy, Edward and Wasserman, Larry},
  journal={Statistical Science},
  volume={41},
  number={3},
  pages={659--670},
  year={2026},
  publisher={Institute of Mathematical Statistics}
}

@article{fisher2023three,
  title={Three-way Cross-Fitting and Pseudo-Outcome Regression for Estimation of Conditional Effects and other Linear Functionals},
  author={Fisher, Aaron and Fisher, Virginia},
  journal={arXiv preprint arXiv:2306.07230},
  year={2023}
}

@article{liu2020nearly,
  title={On nearly assumption-free tests of nominal confidence interval coverage for causal parameters estimated by machine learning},
  author={Liu, Lin and Mukherjee, Rajarshi and Robins, James M},
  journal={Statistical Science},
  volume={35},
  number={3},
  pages={518--539},
  year={2020},
  publisher={JSTOR}
}

@article{shen2020optimal,
author = {Yandi Shen and Chao Gao and Daniela Witten and Fang Han},
title = {{Optimal estimation of variance in nonparametric regression with random design}},
volume = {48},
journal = {The Annals of Statistics},
number = {6},
publisher = {Institute of Mathematical Statistics},
pages = {3589 -- 3618},
year = {2020},
doi = {10.1214/20-AOS1944},
URL = {https://doi.org/10.1214/20-AOS1944}
}

@article{bruns2023augmented,
  title={Augmented balancing weights as linear regression},
  author={Bruns-Smith, David and Dukes, Oliver and Feller, Avi and Ogburn, Elizabeth L},
  journal={arXiv preprint arXiv:2304.14545},
  year={2023}
}

@inproceedings{van2019causal,
  title={Causal inference based on undersmoothing the highly adaptive lasso},
  author={van der Laan, Mark J and Benkeser, David and Cai, Weixin},
  booktitle={Association for the Advancement of Artificial Intelligence Spring Symposia, Palo Alto, California},
  year={2019}
}

@article{kennedy2024semiparametric,
  title={Semiparametric doubly robust targeted double machine learning: a review},
  author={Kennedy, Edward H},
  journal={Handbook of statistical methods for precision medicine},
  pages={207--236},
  year={2024},
  publisher={Chapman and Hall/CRC}
}

@book{tsiatis2006semiparametric,
  title={Semiparametric theory and missing data},
  author={Tsiatis, Anastasios A},
  year={2006},
  publisher={Springer}
}

@article{mcgrath2025optimal,
  title={Optimal Nuisance Function Tuning for Estimating a Doubly Robust Functional under Proportional Asymptotics},
  author={McGrath, Sean and Mukherjee, Debarghya and Mukherjee, Rajarshi and Wang, Zixiao Jolene},
  journal={Advances in Neural Information Processing System},
  volume={38},
  pages={90233--90305},
  year={2025}
}

@article{mcgrath2026supplement,
  title={Supplement to ``Nuisance Function Tuning and Sample Splitting for Optimally Estimating a Doubly Robust Functional"},
  author={McGrath, Sean and Mukherjee, Rajarshi},
  year={2026}, 
  doi={10.1214/[provided by typesetter]}
}


\end{document}


\begin{frontmatter}
\title{Supplement to ``Nuisance Function Tuning and Sample Splitting for Optimally Estimating a Doubly Robust Functional''}
\runtitle{On Nuisance Function Tuning and Sample Splitting}

\begin{aug}
\author[A]{\fnms{Sean}~\snm{McGrath}\ead[label=e1]{mcgrathsean@wustl.edu}}
\author[B]{\fnms{Rajarshi}~\snm{Mukherjee}\ead[label=e2]{ram521@mail.harvard.edu}}
\address[A]{Department of Statistics and Data Science, Washington University in St. Louis\printead[presep={,\ }]{e1}}

\address[B]{Department of Biostatistics, Harvard University\printead[presep={,\ }]{e2}}
\end{aug}

\end{frontmatter}
\tableofcontents

\setcounter{theorem}{5}
\setcounter{equation}{1}
\setcounter{cor}{15}
\setcounter{remark}{15}

\begin{appendix}
\section{Wavelet bases and Hölder spaces} \label{sec: wavelets}

We review some results on orthonormal wavelet bases and characterizations of Hölder spaces. We consider representations of $h \in L^2([0,1]^d)$ in terms of an orthonormal wavelet basis. Specifically, we consider a compactly supported Cohen-Daubechies-Vial type orthonormal wavelet basis of regularity $S \geq 1$ \cite{gine2021mathematical} given by
\begin{equation*}
    \{ \Phi_{J_0 \boldm}, \Psi^{\biota}_{l \boldm^{\prime}}: \boldm \in \mathcal{Z}_{J_0}, \boldm^{\prime} \in \mathcal{Z}_{l}, \biota \in \mathcal{I}, l \in \mathbb{Z}, l \geq J_0 \},
\end{equation*} 
where $J_0 := J_0(S) \geq 1$ is such that $2^{J_0} \geq S$ and the index sets satisfy $\mathcal{Z}_{J_0}, \mathcal{Z}_{l} \subset \mathbb{Z}^d$ and $\mathcal{I}:= \{0, 1\}^d \setminus \{\boldsymbol{0}\}$. The basis functions $\Phi_{J_0 \boldm}, \Psi^{\biota}_{l \boldm}$ are obtained by translating and scaling $\Phi$ and $\Psi$. For $\bx = (x_1, \dots, x_d)^{\top} \in [0, 1]^d$, 
\begin{align*}
    \Phi_{l\boldm}(\bx) & := 2^{ld/2} \Phi(2^{l}\bx - \boldm), \qquad \Phi(\bx) := \phi(x_1) \cdots \phi(x_d) \\
    \Psi^{\biota}_{l\boldm^{\prime}}(\bx) & := 2^{ld/2} \Psi^{\biota}(2^l \bx - \boldm^{\prime}), \qquad  \Psi^{\biota}(\bx) := \psi^{\iota_1}(x_1) \cdots \psi^{\iota_d}(x_d)
\end{align*}
where $\psi^0 = \phi$ is the scaling function and $\psi^1 = \psi$ is the wavelet function.\footnote{To be precise, the construction of the wavelet basis of $L^2([0,1]^d)$ requires a bit more care than just translating and scaling $\Phi$ and $\Psi$ (e.g., including edge functions and subsequently performing Gram-Schmidt orthonormalization). See Chapter 4 of Giné and Nickl \cite{gine2021mathematical} for a detailed construction of the wavelet basis of $L^2([0,1]^d)$. We presented the wavelet basis in this manner to simplify the presentation, noting that our proofs only use bounds on the size and support of the wavelet basis functions (i.e., Property P2).} Due to the compact support of $\Phi$ and $\Psi$, we have that $|\mathcal{Z}_{J_0}| \asymp 1$ and $|\mathcal{Z}_{l}| \asymp 2^{ld}$. Letting $\langle g, h \rangle := \int g(\bx) h(\bx) d\bx$ denote the usual $L^2([0,1]^d)$ inner product, the wavelet series expansion of $h \in L^2([0,1]^d)$ is given by
\begin{equation*}
    h = \sum_{\boldm \in \mathcal{Z}_{J_0}} \langle \Phi_{J_0 \boldm}, h \rangle \Phi_{J_0 \boldm} + \sum_{l = J_0}^{\infty} \sum_{\boldm \in \mathcal{Z}_l} \sum_{\biota \in \mathcal{I}} \langle \Psi_{l \boldm}^{\biota}, h \rangle \Psi_{l \boldm}^{\biota}.
\end{equation*}

The nuisance function estimators we consider are based on approximate projections of $h \in L^2([0,1]^d)$ onto subspaces $V_j$ defined by
\begin{equation*}
    V_j := \text{span} \left\{ \Phi_{j\boldm}, \boldm \in \mathcal{Z}_j \right\}.
\end{equation*}
By standard properties of a multiresolution analysis, the projection kernel $K_{V_j}$ onto $V_j$ can be expressed as
\begin{equation*}
    K_{V_j}(\bx, \by) =  \sum_{\boldm \in \mathcal{Z}_j} \Phi_{j \boldm}(\bx) \Phi_{j \boldm}(\by)
\end{equation*}
and the projection of $h$ onto $V_j$ can be expressed as
\begin{align*}
    \Pi(h | V_j) & := \sum_{\boldm \in \mathcal{Z}_{j}} \langle \Phi_{j \boldm}, h \rangle \Phi_{j \boldm} \\
    & = \sum_{\boldm \in \mathcal{Z}_{J_0}} \langle \Phi_{J_0 \boldm}, h \rangle \Phi_{J_0 \boldm} + \sum_{l = J_0}^{j-1} \sum_{\boldm \in \mathcal{Z}_l} \sum_{\biota \in \mathcal{I}} \langle \Psi_{l \boldm}^{\biota}, h \rangle \Psi_{l \boldm}^{\biota}
\end{align*}
where $j \geq J_0$. With a slight abuse of notation, we will denote $V_j$ as $V_k$ where $k = 2^{jd}$. 

Many of the functions that we work with belong to Hölder spaces which can be characterized as follows. Let $C_u([0,1]^d)$ denote the space of uniformly continuous functions defined on $[0,1]^d$. Then, we define the Hölder space $H(\alpha, M)$ by
\begin{equation*} 
    H(\alpha, M) := \left\{ h \in C_u([0,1]^d) : \| h \|_{B_{\infty \infty}^{\alpha}} \leq M  \right\}
\end{equation*}
where for wavelet basis with regularity $S > \alpha$,
\begin{equation*}
    \| h \|_{B_{\infty \infty}^{\alpha}} := 2^{J_0(\alpha+\frac{d}{2})} \| \langle h, \Phi_{J_0 \cdot}  \rangle \|_{\infty} + \sup_{l \geq J_0, \boldm \in \mathcal{Z}_l, \biota \in \mathcal{I}} 2^{l(\alpha+\frac{d}{2})}  | \langle h, \Psi_{l\boldm}^{\biota} \rangle|.
\end{equation*}
Note that this definition of $H(\alpha, M)$ coincides with the Besov space $B^{\alpha}_{\infty \infty}$ defined in Giné and Nickl \cite{gine2021mathematical}. We can extend our results to hold for $B^{\alpha}_{2 \infty}$ using straightforward arguments.

Throughout the proofs, we will repeatedly use the following standard properties (e.g., Chapter 4 of Giné and Nickl \cite{gine2021mathematical}) of functions belonging to $H(\alpha, M)$ and of our choice of wavelet basis functions: \begin{itemize}
    \item[] (P1) We have that
\begin{equation*}
    \sup_{h \in H(\alpha, M)}\| h - \Pi(h | V_k) \|_{\infty} \leq c k^{-\frac{\alpha}{d}}
\end{equation*}
where $c$ only depends on $\alpha$, $M$, and $d$.
    \item[] (P2) For any $k \geq 2^{J_0d}$ and $\boldm \in \mathcal{Z}_{k}$, $|\Phi_{k \boldm}|$ is uniformly bounded above by $O(\sqrt{k})$ and has a support contained in a set of Lebesgue measure $O(\frac{1}{k})$.

    \item[] (P3) For any $k \geq 2^{J_0d}$, $V_k$ contains all polynomials on $[0,1]^d$ up to degree $S-1$.
\end{itemize}
Property P1 is used in deriving the upper bounds on the bias of all of the estimators we consider. Property P2 is used in deriving upper bounds on the variance of the estimators, specifically in our technical lemmas that bound expectations of products of kernel functions. Property P3 is used in deriving lower bounds on the bias of some of the estimators in the no sample splitting case.

\section{Unknown 
covariate density} \label{sec: unknown f}

In this section, we discuss issues, challenges, and details involved in extending our results when the density of the covariates $\bX$, denoted by $f$, is unknown. 

Before presenting the precise mathematical formalization, we begin with a discussion of related literature in this domain. The literature in this regard can be broadly divided into two regimes of asymptotic inference. Specifically, when $\sqrt{n}$-consistent estimation of $\psi(\mathbb{P})$ is possible in an asymptotic minimax sense, several papers have considered an adaptation to unknown $f$ under minimal requirements. Specifically, Robins et al. \cite{robins2017minimax,robins2008higher} demonstrate a framework to obtain rate optimal minimax estimation of $\psi(\bbP)$ by performing sophisticated higher-order bias corrections starting from a plug-in principle that required estimation of $f$. Subsequently, Liu et al. \cite{liu2017semiparametric} constructed an estimator based on higher order influence functions that does not require estimation of $f$ and yet produces semiparametric efficient $\sqrt{n}$-consistent estimation of $\psi(\bbP)$ whenever information-theoretically possible.  In terms of simply using the first-order influence function-based estimation using $\hat{\psi}^{\mathrm{IF},1}$, Newey and Robins \cite{newey2018cross} demonstrated how to obtain $\sqrt{n}$-consistent semiparametric estimation through sample splitting and suitable undersmoothing of the estimators of $p$ and $b$. In this regard, they used least squares series estimators of $p$ and $b$ that do not require estimating $f$ but instead require inverting the Gram matrix of the series-features and hence can only operate when the number of features ($k_1$ and $k_2$ in our setting) is less than the sample size. A similar qualification holds for the higher order influence function-based estimator considered in Liu et al. \cite{liu2017semiparametric}. As a result, this estimator is not well-defined when one needs to operate under more features than the sample size i.e. $k_1,k_2>n$ in our notation. Since the focus of this paper is to understand the necessary and sufficient conditions of tuning estimators of $p$ and $b$ to obtain rate optimal estimation of $\psi(\bbP)$ in any information-theoretic regime of estimation, we do not consider these estimators in this paper. 

To take a step towards addressing the unknown $f$ case, we devote this section to explore the case when one estimates $f$ and uses such estimates in estimating $p$, $b$, and $\psi(\bbP)$. Indeed, this raises further questions regarding the interplay between the type of sample splitting and the nuisance function tuning strategies. To keep our discussions focused and consistent with the exposition in the main paper, we consider the case when $f$ is estimated optimally from a separate subsample. We further consider that all the nuisance functions are estimated in separate subsamples. We show that the same conclusions hold as in the known $f$ case in the non-$\sqrt{n}$ regime, provided that $f$ is sufficiently regular. That is, we illustrate that (i) all of the estimators require performing some degree of undersmoothing to obtain optimal rates for estimating $\psi(\bbP)$ and (ii) all of the estimators except the Monte Carlo-based plug-in estimator can be minimax rate optimal when suitably tuning the nuisance function estimators. The complete understanding of results parallel to the main results in our paper for the unknown $f$ case remains beyond the scope of the current paper and we plan to explore it further in future work.

\subsection{Setup}

\subsubsection{Model and minimax risk} 

First, we define the model in the unknown $f$ case, which is a straightforward extension of the model in the known $f$ case  (see Section 2.1 in the main text). Let $\cP_{(\alpha,\beta,\gamma,g)}$ denote the set of distributions $\bbP$ such that the following conditions hold
\begin{align*}
    (i) & \quad  A, Y \in [C_1, C_2] \quad \text{a.s.} \, \bbP \\
    (ii) & \quad   p \in H(\alpha, M), \, b \in H(\beta, M), \, f \in H(\gamma, M) \\
    (iii) & \quad  f(\bx) \in [M_1, M_2] \quad \forall \bx \in [0, 1]^d 
\end{align*}
where $C_1, C_2 \in \mathbb{R}$ and $M, M_1, M_2 \in \mathbb{R}^+$ are known constants. We consider that $\theta := (\alpha, \beta, \gamma) \in \Theta_g$ where
\begin{equation*}
    \Theta_g = \{ (\alpha, \beta, \gamma): 0 < \alpha, \beta, \gamma < \infty, \, \gamma > g(\alpha, \beta)\}.
\end{equation*}
We will require a specific $g(\alpha, \beta)$  and we have not aimed to optimize this requirement since the necessary and sufficient conditions on $g$ for the full spectrum of minimax estimation of $\psi(\bbP)$, to the best of our knowledge, are yet to be established in full generality. Letting 
\begin{equation*}
    g^*(\alpha, \beta) = \frac{2\delta(\Delta + 1)(1 - 4 \delta / d)}{(\Delta + 2)(1 + 4\delta / d) - 4(\delta / d)(1 - 4\delta / d)(\Delta + 1)}
\end{equation*}
where $\delta := \delta(\alpha, \beta) := \frac{\alpha + \beta}{2}$ and $\Delta := \Delta(\alpha, \beta) :=  |\frac{\alpha}{\beta} - 1|$, Robins et al. \cite{robins2009semiparametric,robins2008higher} showed that the minimax risk for estimating $\psi(\bbP)$ associated with the model $\cP_{(\alpha,\beta,\gamma,g^*)}$ is the same as in the known $f$ case, i.e.,
\begin{equation} \label{eq: minimax rate unknown f}
\inf_{\hat{\psi}}\sup_{\bbP\in \cP_{(\alpha,\beta,\gamma,g^*)}}\mathbb{E}_\mathbb{\bbP}\left[(\hat{\psi}-\psi(\bbP))^2\right] \asymp \begin{cases}
n^{-1}, & \frac{\alpha + \beta}{2} \geq \frac{d}{4}\\
n^{-\frac{4\alpha + 4\beta}{2\alpha + 2\beta +d}}, & \frac{\alpha + \beta}{2} < \frac{d}{4}
\end{cases}.
\end{equation}
The choices of $g$ in our results will therefore automatically satisfy $g\geq g^*$. To lessen the burden of notation, we will drop the subscript of $g$ in the model $\cP_{(\alpha,\beta,\gamma,g)}$.

\subsubsection{Estimators}

The nuisance function estimators and integral-based plug-in estimators described in the main paper will require an estimate of $f$ in the unknown $f$ case. We consider estimating $f$ with a suitably bounded, rate-optimal estimator $\hat{f}$. Specifically, we require the following properties on $\hat{f}$: 
\begin{align*}
    & (i) \quad  0 < c_1 \leq \hat{f}(\bx) \leq c_{2} < \infty \quad \forall \bx \in [0, 1]^d \\
    & (ii) \quad \sup_{\bbP \in \cP_{(\alpha,\beta,\gamma)}} \E_{\bbP}[\| \hat{f} - f \|_{\infty}^2 ] \lesssim \left( \frac{n}{\log n}\right)^{-\frac{2\gamma}{2\gamma + d}}.
\end{align*}
One example of such an estimator is the following bounded, approximate wavelet projection estimator. Consider the approximate wavelet projection estimator of $f$ fit in subsample $\D_j$
\begin{equation*}
    \tilde{f}^{(j)}_k(\bx) = \frac{1}{n} \sum_{i \in \D_j}  K_{V_{k}}(\bX_i,\bx),\qquad \bx \in [0, 1]^d
\end{equation*}
where $k \asymp ( \frac{n}{\log n})^{\frac{d}{2\gamma + d}} $. It follows from standard results of compactly supported wavelet bases with regularity $\gamma$ that $\sup_{\bbP \in \cP_{(\alpha,\beta,\gamma)}} \E_{\bbP}[\| \tilde{f}^{(j)}_k - f \|_{\infty}^2 ] \lesssim \left( \frac{n}{\log n}\right)^{-\frac{2\gamma}{2\gamma + d}}$ (Chapter 5 of Giné and Nickl \cite{gine2021mathematical}). Then, define $\hat{f}^{(j)}_k(\bx)$ to be $\tilde{f}^{(j)}_k(\bx)$ bounded in the interval $[M_1, M_2]$. Note that $\hat{f}^{(j)}_k$ satisfies (ii) since $|\hat{f}^{(j)}_k(\bx) - f(\bx)| \leq |\tilde{f}^{(j)}_k(\bx) - f(\bx)|$ for all $\bx \in [0, 1]^d$.

We consider using separate subsamples to estimate all nuisance functions, analogous to double sample splitting when $f$ is known. Our approximate wavelet projection estimator of $p$ that is constructed in subsample $\D_{j_1} \cup \D_{j_1^{\prime}}$ ($j_1 \neq j_1^{\prime}$) is given by
\begin{equation*}
    \hat{p}^{(j_1, j_1^{\prime})}_{k_1}(\bx) = \frac{1}{n} \sum_{i \in \D_{j_1}} \frac{A_i K_{V_{k_1}}(\bX_i,\bx)}{\hat{f}^{(j_1^{\prime})}(\bX_i)}, \qquad \bx \in [0, 1]^d.
\end{equation*}
Similarly, we estimate $b$ in subsample $\D_{j_2} \cup \D_{j_2^{\prime}}$ ($j_2 \neq j_2^{\prime}$) by 
\begin{equation*}
    \hat{b}^{(j_2, j_2^{\prime})}_{k_2}(\bx) = \frac{1}{n} \sum_{i \in \D_{j_2}} \frac{Y_i K_{V_{k_2}}(\bX_i,\bx) }{\hat{f}^{(j_2^{\prime})}(\bX_i)}, \qquad \bx \in [0, 1]^d.  
\end{equation*}

The plug-in type estimators and first-order estimator in the unknown $f$ case are then given by
\begin{align*}
    \hat{\psi}_{k_1, k_2}^{\mathrm{INT}} & = \frac{1}{n}\sum_{i \in \D_5} A_i Y_i -  \int \hat{p}^{(1, 2)}_{k_1}(\bx)\hat{b}^{(3, 4)}_{k_2}(\bx) \hat{f}^{(6)}(\bx) d\bx \\
    \hat{\psi}^{\mathrm{MC}}_{k_1, k_2} & = \frac{1}{n}\sum_{i \in \D_5} A_i Y_i -  \frac{1}{n}\sum_{i \in \D_6} \hat{p}^{(1, 2)}_{k_1}(\bX_i)\hat{b}^{(3, 4)}_{k_2}(\bX_i) \\
    \hat{\psi}_{k}^{\mathrm{NR}} & = \frac{1}{n} \sum_{i \in \D_3} A_i(Y_i - \hat{b}_k^{(1,2)}(\bX_i)) \\
    \hat{\psi}^{\mathrm{IF}}_{k_1, k_2} & = \frac{1}{n} \sum_{i \in \D_5} (A_i - \hat{p}^{(1,2)}_{k_1}(\bX_i))(Y_i - \hat{b}^{(3,4)}_{k_2}(\bX_i)).
\end{align*}

\subsection{Results}

We first establish bounds on the bias and variance of the plug-in type estimators and first-order estimator. As in the results in the main text, we provide lower bounds on the bias of all the estimators to establish the necessity to undersmooth the nuisance functions in low regularity regimes. Moreover, as in the double sample splitting case, we also include lower bounds on the variance of $\hat{\psi}^{\mathrm{MC}}_{k_1, k_2}$ to establish that the rate of convergence of the mean squared error of $\hat{\psi}^{\mathrm{MC}}_{k_1, k_2}$ is sub-optimal for any $k_1$ and $k_2$ and similarly for $\hat{\psi}^{\mathrm{IF}}_{k_1, k_2}$ when setting $k_1 = k_2$.

\begin{theorem} \label{theorem: plugin unknown f} 
Under the assumptions given in this section, the following statements hold:
\begin{enumerate} 
    \item[1a)] The estimator $\hat{\psi}_{k_1, k_2}^{\mathrm{INT}}$ satisfies
\begin{align*}
    \sup_{\bbP \in \cP_{(\alpha,\beta,\gamma)}} \left|\E_\bbP\left(\hat{\psi}_{k_1, k_2}^{\mathrm{INT}} - \psi(\bbP)\right)\right| & \lesssim (k_1 \wedge k_2)^{-(\alpha + \beta)/d} + \left( \frac{n}{\log n}\right)^{-\frac{\gamma}{2\gamma + d}}\\
    \sup_{\bbP \in \cP_{(\alpha,\beta,\gamma)}} \Var_\bbP(\hat{\psi}_{k_1, k_2}^{\mathrm{INT}}) & \lesssim \frac{1}{n}  +  \frac{k_1 \wedge k_2}{n^2} + \left( \frac{n}{\log n}\right)^{-\frac{2\gamma}{2\gamma + d}}. 
\end{align*}
    \item[1b)] If $\gamma$ is sufficiently large, there exists a $\bbP \in \cP_{(\alpha,\beta,\gamma)}$ such that for $k_1 \wedge k_2 \ll n^{\frac{d}{2\alpha + 2\beta}}$
\begin{equation*}
    \left|\E_\bbP\left(\hat{\psi}^{\mathrm{INT}}_{k_1, k_2} - \psi(\bbP)\right)\right| \gtrsim (k_1 \wedge k_2)^{-(\alpha + \beta)/d}. 
\end{equation*}

    \item[2a)] The estimator $\hat{\psi}_{k_1, k_2}^{\mathrm{MC}}$ satisfies
\begin{align*}
    \sup_{\bbP \in \cP_{(\alpha,\beta,\gamma)}} \left|\E_\bbP\left(\hat{\psi}_{k_1, k_2}^{\mathrm{MC}} - \psi(\bbP)\right)\right| & \lesssim (k_1 \wedge k_2)^{-(\alpha + \beta)/d} + \left( \frac{n}{\log n}\right)^{-\frac{\gamma}{2\gamma + d}}\\
    \sup_{\bbP \in \cP_{(\alpha,\beta,\gamma)}} \Var_\bbP(\hat{\psi}_{k_1, k_2}^{\mathrm{MC}}) & \lesssim \frac{1}{n}  +  \frac{k_1 \vee k_2}{n^2} + \frac{k_1k_2}{n^3} + \left( \frac{n}{\log n}\right)^{-\frac{2\gamma}{2\gamma + d}}.
\end{align*}
    \item[2b)] If $\gamma$ is sufficiently large, there exists a $\bbP \in \cP_{(\alpha,\beta,\gamma)}$ such that for $k_1 \wedge k_2 \ll n^{\frac{d}{2\alpha + 2\beta}}$
\begin{equation*}
    \left|\E_\bbP\left(\hat{\psi}^{\mathrm{MC}}_{k_1, k_2} - \psi(\bbP)\right)\right| \gtrsim (k_1 \wedge k_2)^{-(\alpha + \beta) / d}
\end{equation*}
and for $k_1, k_2 \gg n$,
\begin{equation*}
    \Var_\bbP(\hat{\psi}^{\mathrm{MC}}_{k_1, k_2}) \gtrsim \frac{k_1k_2}{n^3}.
\end{equation*}
    
    \item[3a)] The estimator $\hat{\psi}_{k}^{\mathrm{NR}}$ satisfies
\begin{align*}
    \sup_{\bbP \in \cP_{(\alpha,\beta,\gamma)}} \left|\E_\bbP\left(\hat{\psi}_{k}^{\mathrm{NR}} - \psi(\bbP)\right)\right| & \lesssim k^{-(\alpha + \beta)/d} + \left( \frac{n}{\log n}\right)^{-\frac{\gamma}{2\gamma + d}} \\
    \sup_{\bbP \in \cP_{(\alpha,\beta,\gamma)}} \Var_\bbP(\hat{\psi}_{k}^{\mathrm{NR}}) & \lesssim \frac{1}{n}  +  \frac{k}{n^2} + \left( \frac{n}{\log n}\right)^{-\frac{2\gamma}{2\gamma + d}}. 
\end{align*}
\item[3b)] If $\gamma$ is sufficiently large, there exists a $\bbP \in \cP_{(\alpha,\beta,\gamma)}$ such that for $k \ll n^{\frac{d}{2\alpha + 2\beta}}$
\begin{equation*}
    \left|\E_\bbP\left(\hat{\psi}^{\mathrm{NR}}_{k} - \psi(\bbP)\right)\right| \gtrsim k^{-(\alpha + \beta) / d}.
\end{equation*}

\item[4a)] The estimator $\hat{\psi}^{\mathrm{IF}}_{k_1, k_2}$ satisfies
\begin{align*}
    \sup_{\bbP \in \cP_{(\alpha,\beta,\gamma)}} \left|\E_\bbP\left(\hat{\psi}^{\mathrm{IF}}_{k_1, k_2} - \psi(\bbP)\right)\right| & \lesssim (k_1 \vee k_2)^{-(\alpha+\beta)/d} + \left( \frac{n}{\log n}\right)^{-\frac{\gamma}{2\gamma + d}}  \\
    \sup_{\bbP \in \cP_{(\alpha,\beta,\gamma)}} \Var_\bbP(\hat{\psi}^{\mathrm{IF}}_{k_1, k_2}) & \lesssim  \frac{1}{n}  +  \frac{k_1 \vee k_2}{n^2} + \frac{k_1k_2}{n^3} + \left( \frac{n}{\log n}\right)^{-\frac{2\gamma}{2\gamma + d}}.
\end{align*}

\item[4b)] If $\gamma$ is sufficiently large, there exists a $\bbP \in \cP_{(\alpha,\beta,\gamma)}$ such that for $k_1 \vee k_2 \ll n^{\frac{d}{2\alpha + 2\beta}}$
\begin{equation*}
    \left|\E_\bbP\left(\hat{\psi}^{\mathrm{IF}}_{k_1, k_2} - \psi(\bbP)\right)\right| \gtrsim (k_1 \vee k_2)^{-(\alpha+\beta)/d} 
\end{equation*}
and for $k_1,k_2 \gg n$
\begin{equation*}
    \Var_\bbP(\hat{\psi}^{\mathrm{IF}}_{k_1, k_2}) \gtrsim \frac{k_1 k_2}{n^3}.
\end{equation*}
\end{enumerate}
\end{theorem}

The proof is given in Appendix \ref{sec: proof unknown f}. We express the bias of these estimators in terms of their bias in the known $f$ case (with double sample splitting for $\hat{\psi}_{k_1, k_2}^{\mathrm{INT}}$, $\hat{\psi}_{k_1, k_2}^{\mathrm{MC}}$, $\hat{\psi}^{\mathrm{IF}}_{k_1, k_2}$, single sample splitting for $\hat{\psi}^{\mathrm{NR}}_{k}$) plus a remainder term which is $O(( \frac{n}{\log n})^{-\frac{\gamma}{2\gamma + d}})$ due to property (ii) of $\hat{f}$. The derivations of the variance upper bounds are similar to those in the known $f$ case, although they are lengthier due to additional variance decompositions needed from using more subsamples. Compared to the known $f$ case, the upper bounds on the variance of each of the estimators have an additional term of $O(( \frac{n}{\log n})^{-\frac{2\gamma}{2\gamma + d}})$, which also arises due to property (ii) of $\hat{f}$. The lower bounds on the variance of $\hat{\psi}^{\mathrm{MC}}_{k_1, k_2}$ and $\hat{\psi}^{\mathrm{IF}}_{k_1, k_2}$ follow from the same steps in the known $f$ case since $\hat{f}$ is bounded.

We next describe the optimal resolutions of the estimators in the $\frac{\alpha + \beta}{2} < \frac{d}{4}$ regime. Since the additional $O(( \frac{n}{\log n})^{-\frac{\gamma}{2\gamma + d}})$ terms in the bias and the $O(( \frac{n}{\log n})^{-\frac{2\gamma}{2\gamma + d}})$ terms in the variance are negligible in this regime when $\gamma$ is sufficiently large, we arrive at the following corollary which is the same as the known $f$ case (i.e., Corollaries 1 and 7 in the main text) in the $\frac{\alpha + \beta}{2} < \frac{d}{4}$ regime.

\begin{cor} \label{cor: unknown f}
If $\frac{\alpha + \beta}{2} < \frac{d}{4}$ and $\gamma$ is sufficiently large, then
\begin{enumerate}
    \item $\hat{\psi}^{\mathrm{INT}}_{k_1, k_2}$ is minimax rate optimal when choosing $k_1 \wedge k_2 \asymp n^{\frac{2d}{2\alpha + 2\beta + d}}$.
    \item $\hat{\psi}^{\mathrm{MC}}_{k_1, k_2}$ cannot be minimax rate optimal for any choice of $k_1, k_2$. The best rate of estimation of $\psi(\bbP)$ for $\hat{\psi}^{\mathrm{MC}}_{k_1, k_2}$ is $n^{-\frac{3\alpha + 3\beta}{\alpha + \beta + d}}$
    which occurs when choosing $k_\ell \asymp  n^{\frac{3d}{2\alpha + 2\beta + 2d}} $ for $\ell = 1, 2$.
    \item $\hat{\psi}^{\mathrm{NR}}_{k}$ is minimax rate optimal when choosing $k \asymp n^{\frac{2d}{2\alpha + 2\beta + d}}$.
    \item $\hat{\psi}^{\mathrm{IF}}_{k_1, k_2}$ is minimax rate optimal when choosing $k_1 \vee k_2 \asymp n^{\frac{2d}{2\alpha + 2\beta + d}}$ and $k_1 \wedge k_2 \lesssim n $.
\end{enumerate}
\end{cor}

We next make a few remarks on the optimal resolution choices:

\begin{remark}
As in the known $f$ case, the lower bounds for the optimal resolutions given in Corollary \ref{cor: unknown f} are indeed necessary. This directly follows from the lower bounds on the bias in Theorem \ref{theorem: plugin unknown f}.
\end{remark}

\begin{remark}
As discussed in Remark 2 in the main text, one may consider letting $k_1 = k_2$ in some settings. If we enforce $k_1 = k_2 = k$,  then $\hat{\psi}^{\mathrm{IF}}_{k_1, k_2}$ cannot be rate optimal for any choice of $k$ when $\frac{\alpha + \beta}{2} < \frac{d}{4}$ and $\gamma$ is sufficiently large. The best rate of estimation of $\psi(\bbP)$ for $\hat{\psi}^{\mathrm{IF}}_{k_1, k_2}$ is $n^{-\frac{3\alpha + 3\beta}{\alpha + \beta + d}}$ which occurs if and only if choosing $k \asymp  n^{\frac{3d}{2\alpha + 2\beta + 2d}} $. 
\end{remark}

Next, we consider using prediction-optimal resolutions. The following theorem establishes that the prediction-optimal resolution choices for $\hat{p}^{(j_1, j_1^{\prime})}_{k_1}$ and $\hat{b}^{(j_2, j_2^{\prime})}_{k_2}$ are the same as in the known $f$ case (i.e., $k_1^{\text{pred}} = c_1 n^\frac{d}{2\alpha + d}$ and $k_2^{\text{pred}} = c_2 n^\frac{d}{2\beta + d}$ respectively, for fixed $c_1, c_2 > 0$).

\begin{theorem} \label{theorem: pred opt unknown f}
    If $\gamma$ is sufficiently large and $k_1 \ll n^{\frac{d}{2\alpha}}$, then
    \begin{align*}
        \sup_{\bbP \in \cP_{(\alpha, \beta, \gamma)}} \int \left( \E_\bbP\left(\hat{p}^{(j_1, j_1^{\prime})}_{k_1}(\bx) - p(\bx)\right)\right)^2 d\bx & \asymp k_1^{-2\alpha/d}. 
    \end{align*}
    Moreover, for $k_1 \gtrsim n^\epsilon$ for some $\epsilon>0$ and $\gamma$ sufficiently large,
    \begin{equation*}
        \sup_{\bbP \in \cP_{(\alpha, \beta, \gamma)}} \int \Var_{\bbP}(\hat{p}^{(j_1, j_1^{\prime})}_{k_1}(\bx)) d\bx \asymp \frac{k_1}{n}.
    \end{equation*}
    Therefore, the rate of convergence of
    \begin{equation*}
        \sup_{\bbP \in \cP_{(\alpha, \beta, \gamma)}} \E_{\bbP} \int (\hat{p}^{(j_1, j_1^{\prime})}_{k_1}(\bx) - p(\bx))^2 d\bx
    \end{equation*}
    is minimized if and only if $k_1 = c n^{\frac{d}{2\alpha + d}} $ for fixed $c > 0$.
\end{theorem}

The proof is given in Appendix \ref{sec: proof pred opt unknown f}. The upper bounds on the bias and variance in Theorem \ref{theorem: pred opt unknown f} are rather common for nonparametric regression methods, and, in particular, they are identical to the bounds in the known $f$ case (see Chapter 10 of Härdle et al. \cite{hardle2012wavelets}). The main subtlety is that the bias of $\hat{p}^{(j_1, j_1^{\prime})}_{k_1}$ has an additional term of $O(( \frac{n}{\log n})^{-\frac{\gamma}{2\gamma + d}})$ due to estimating $f$, which we show becomes asymptotically negligible when $\gamma$ is sufficiently large and $k_1$ is sufficiently small. 

Now, we can compare using the optimal resolutions versus prediction-optimal resolutions. The following corollaries establish the necessity to perform undersmoothing in the $\frac{\alpha + \beta}{2} < \frac{d}{4}$ regime when $\gamma$ is sufficiently large. As one may expect, these results are the same as in the known $f$ case (with double sample splitting for $\hat{\psi}^{\mathrm{INT}}_{k_1, k_2}$, $\hat{\psi}^{\mathrm{MC}}_{k_1, k_2}$, and $\hat{\psi}^{\mathrm{IF}}_{k_1, k_2}$) when $\frac{\alpha + \beta}{2} < \frac{d}{4}$. 

\begin{cor} \label{cor: pred unknown f}
Suppose we choose prediction-optimal resolutions $k_1 = k_1^{\text{pred}}$ and $k_2 = k_2^{\text{pred}}$. If $\frac{\alpha + \beta}{2} < \frac{d}{4}$ and $\gamma$ is sufficiently large, then $\hat{\psi}_{k_1^{\text{pred}}, k_2^{\text{pred}}}^{\mathrm{INT}}$, $\hat{\psi}_{k_1^{\text{pred}}, k_2^{\text{pred}}}^{\mathrm{MC}}$, $\hat{\psi}_{k_2^{\text{pred}}}^{\mathrm{NR}}$, and $\hat{\psi}_{k_1^{\text{pred}}, k_2^{\text{pred}}}^{\mathrm{IF}}$ are not minimax rate optimal.
\end{cor}

\begin{cor} \label{cor: undermoothing unknown f} If $\frac{\alpha + \beta}{2} < \frac{d}{4}$ and $\gamma$ is sufficiently large, then
    \begin{enumerate}
    \item $\hat{\psi}_{k_1, k_2}^{\mathrm{INT}}$ requires undersmoothing $k_1$ and $k_2$ to achieve its best rate for estimating $\psi(\bbP)$ (i.e., the minimax rate in (\ref{eq: minimax rate unknown f})). 
    \item $\hat{\psi}_{k_1, k_2}^{\mathrm{MC}}$ requires undersmoothing $k_1$ and $k_2$ to achieve its best rate for estimating $\psi(\bbP)$ (i.e., $n^{-\frac{3\alpha + 3\beta}{\alpha + \beta + d}}$).
    \item $\hat{\psi}_{k}^{\mathrm{NR}}$ requires undersmoothing $k$ to achieve its best rate for estimating $\psi(\bbP)$ (i.e., the minimax rate in (\ref{eq: minimax rate unknown f})).
    \item $\hat{\psi}_{k_1, k_2}^{\mathrm{IF}}$ requires undersmoothing $k_1$ or $k_2$ (and it does not matter which) to achieve its best rate for estimating $\psi(\bbP)$ (i.e., the minimax rate in (\ref{eq: minimax rate unknown f})).
\end{enumerate}
\end{cor}

\section{Proof of Theorem 1} \label{sec: proof known f double}

Recall from Section 2.1 in the main text that we set the covariate density $f = 1$ for ease of illustration. We will write out the proofs of the upper bounds when $f$ is such that 
\begin{align*}
    (i) & \quad  f(\bx) \in [M_1, M_2] \quad \forall \bx \in [0, 1]^d\\
    (ii) & \quad f \in H(\gamma, M)
\end{align*}
where $\gamma \geq \alpha \vee \beta$ and $M, M_1, M_2 \in \mathbb{R}^{+}$ are known constants. In this case, the approximate wavelet projection estimators of the nuisance functions described in Section 2.2 in the main text now require division by $f(\bX_i)$, i.e. 
\begin{align}
    \hat{p}_{k_1}^{(j)}(\bx) & = \frac{1}{n} \sum_{i \in \mathcal{D}_j} \frac{A_i K_{V_{k_1}}(\bX_i,\bx)}{f(\bX_i)}, \qquad \bx \in [0, 1]^d \label{eq: wavelet estimator of p}\\
    \hat{b}_{k_2}^{(j^{\prime})}(\bx) & = \frac{1}{n} \sum_{i \in \mathcal{D}_{j^{\prime}}} \frac{Y_i K_{V_{k_2}}(\bX_i,\bx)}{f(\bX_i)}, \qquad \bx \in [0, 1]^d \label{eq: wavelet estimator of b}.
\end{align}
Moreover, the estimators of $\psi(\bbP)$ that we analyze are given by\footnote{Although we described the plug-in and first-order estimators in Section 2.2.3 in the main text, we define these estimators here for two reasons. First, $\hat{\psi}_{k_1, k_2}^{\mathrm{INT}}$ involves including $f(\bx)$ in the integral now that we do not enforce $f = 1$. Second, we introduce the convention used for the subsample labels which is used throughout the proofs.}
\begin{align*}
    \hat{\psi}_{k_1, k_2}^{\mathrm{INT}} & = \frac{1}{n}\sum_{i \in \D_3} A_i Y_i -  \int \hat{p}_{k_1}^{(1)}(\bx)\hat{b}_{k_2}^{(2)}(\bx)f(\bx) d\bx \\
    \hat{\psi}^{\mathrm{MC}}_{k_1, k_2} & = \frac{1}{n}\sum_{i \in \D_3} A_i Y_i -  \frac{1}{n}\sum_{i \in \D_4} \hat{p}^{(1)}_{k_1}(\bX_i)\hat{b}^{(2)}_{k_2}(\bX_i) \\
    \hat{\psi}^{\mathrm{IF}}_{k_1, k_2} & = \frac{1}{n} \sum_{i \in \D_3} (A_i - \hat{p}^{(1)}_{k_1}(\bX_i))(Y_i - \hat{b}^{(2)}_{k_2}(\bX_i)).
\end{align*}

\begin{remark} \label{rem:mc double modified}
    One can consider an alternative double sample split Monte Carlo-based plug-in estimator that uses three folds instead of four:
    \begin{equation*}
        \tilde{\psi}^{\mathrm{MC}}_{k_1, k_2}  = \frac{1}{n}\sum_{i \in \D_3} A_i Y_i -  \frac{1}{n}\sum_{i \in \D_3} \hat{p}^{(1)}_{k_1}(\bX_i)\hat{b}^{(2)}_{k_2}(\bX_i).
    \end{equation*}
    The same bias and variance bounds in Theorem 1 hold for this estimator as well. We show this in Section \ref{sec: mc three split}. 
\end{remark}

\subsection{Integral-based plug-in estimator}

\subsubsection{Proof of the upper bound} Let $\bbP \in \cP_{(\alpha,\beta)}$ be arbitrary. \\

\noindent \textbf{Bounding the bias:}

We can express the expectation of $\hat{\psi}_{k_1, k_2}^{\mathrm{INT}}$ by
\begin{align*}
    \E_\bbP(\hat{\psi}_{k_1, k_2}^{\mathrm{INT}}) & = \E_\bbP \left( A Y \right) - \int \Pi(p|V_{k_1})(\bx) \Pi(b|V_{k_2})(\bx) f(\bx) d\bx \\
    & = \E_\bbP\left( A Y \right) - \int ( p(\bx) - \Pi(p|V_{k_1}^{\perp})(\bx) ) ( b(\bx) - \Pi(b|V_{k_2}^{\perp})(\bx)) f(\bx) d\bx \\
    & = \E_\bbP\left( A Y \right) - \int p(\bx)b(\bx) f(\bx) d\bx -  \int \Pi(p|V_{k_1}^{\perp})(\bx) \Pi(b|V_{k_2}^{\perp})(\bx) f(\bx) d\bx   \\ & \quad + \int  b(\bx)f(\bx) \Pi(p|V_{k_1}^{\perp})(\bx)  d\bx +  \int p(\bx)f(\bx) \Pi(b|V_{k_2}^{\perp})(\bx) d\bx. 
\end{align*}
By the triangle inequality, the bias is bounded by
\begin{align*}
    \left|\E_\bbP\left(\hat{\psi}_{k_1, k_2}^{\mathrm{INT}} - \psi(\bbP)\right)\right|  & \leq  \left| \int \Pi(p|V_{k_1}^{\perp})(\bx) \Pi(b|V_{k_2}^{\perp})(\bx) f(\bx) d\bx \right |  \\
    & \quad + \left| \int  b(\bx) f(\bx) \Pi(p|V_{k_1}^{\perp})(\bx)   d\bx \right|  +  \left| \int p(\bx)f(\bx) \Pi(b|V_{k_2}^{\perp})(\bx)  d\bx  \right|.
\end{align*}

We first bound $\left| \int \Pi(p|V_{k_1}^{\perp})(\bx) \Pi(b|V_{k_2}^{\perp})(\bx) f(\bx) d\bx \right |$. By boundedness of $f$, 
\begin{equation*}
    \left| \int \Pi(p|V_{k_1}^{\perp})(\bx) \Pi(b|V_{k_2}^{\perp})(\bx) f(\bx) d\bx \right | \lesssim \| \Pi(p|V_{k_1}^{\perp}) \|_{\infty} \| \Pi(b|V_{k_2}^{\perp}) \|_{\infty}. 
\end{equation*}
Since $p$ is $\alpha$-Hölder and $b$ is $\beta$-Hölder, we have by Property P1 (Appendix \ref{sec: wavelets}) 
\begin{align*}
    \| \Pi(p|V_{k_1}^{\perp}) \|_{\infty} \| \Pi(b|V_{k_2}^{\perp}) \|_{\infty}  \lesssim k_1^{-\alpha/d}k_2^{-\beta/d} .
\end{align*}

To bound $\left| \int  b(\bx) f(\bx) \Pi(p|V_{k_1}^{\perp})(\bx)   d\bx \right|$, notice that by orthogonality of $\Pi(bf | V_{k_1})$ and $\Pi(p|V_{k_1}^{\perp})$ with respect to the Lebesgue measure 
\begin{align*}
    \int  b(\bx) f(\bx) \Pi(p|V_{k_1}^{\perp})(\bx)   d\bx  & =  \int  \Pi(bf | V_{k_1})(\bx) \Pi(p|V_{k_1}^{\perp})(\bx)  d\bx   \\
    & \quad +  \int  \Pi(bf | V_{k_1}^{\perp})(\bx)\Pi(p|V_{k_1}^{\perp})(\bx)  d\bx  \\
    & = \int  \Pi(bf | V_{k_1}^{\perp})(\bx)\Pi(p|V_{k_1}^{\perp})(\bx)  d\bx.  
\end{align*}
Since $p$ is $\alpha$-Hölder and $bf$ is $\beta$-Hölder, we have by Property P1 (Appendix \ref{sec: wavelets})
\begin{align*}
    \left| \int  \Pi(bf | V_{k_1}^{\perp})(\bx)\Pi(p|V_{k_1}^{\perp})(\bx)  d\bx  \right| & \leq  \| \Pi(bf | V_{k_1}^{\perp}) \|_{\infty} \| \Pi(p|V_{k_1}^{\perp}) \|_{\infty}  \\
    & \lesssim k_1^{-(\alpha + \beta)/d}.
\end{align*}

By the same arguments, we can bound $\left| \int p(\bx)f(\bx) \Pi(b|V_{k_2}^{\perp})(\bx)  d\bx  \right|$ by
\begin{equation*}
    \left| \int p(\bx)f(\bx) \Pi(b|V_{k_2}^{\perp})(\bx)  d\bx  \right| \lesssim k_{2}^{-(\alpha + \beta)/d}.
\end{equation*}

Therefore, we can bound the bias by
\begin{equation*}
    \left|\E_\bbP\left(\hat{\psi}_{k_1, k_2}^{\mathrm{INT}} - \psi(\bbP)\right)\right| \lesssim (k_1 \wedge k_2)^{-(\alpha + \beta)/d}.
\end{equation*}

\noindent \textbf{Bounding the variance:}

We can express the variance of $\hat{\psi}_{k_1, k_2}^{\mathrm{INT}}$ by
\begin{align*}
    \Var_\bbP(\hat{\psi}_{k_1, k_2}^{\mathrm{INT}}) & = \Var_\bbP\left( \frac{1}{n}\sum_{i \in \D_3} A_i Y_i\right) + \Var_\bbP\left( \int \hat{p}^{(1)}_{k_1}(\bx)\hat{b}^{(2)}_{k_2}(\bx) f(\bx) d\bx  \right) \\
    & = \frac{1}{n}\Var_\bbP(AY) +  \Var_{\bbP,1} \left[ \E_{\bbP,2}\left(\int \hat{p}^{(1)}_{k_1}(\bx)\hat{b}^{(2)}_{k_2}(\bx) f(\bx) d\bx\right)  \right] \\ & \quad +  \E_{\bbP,1} \left[ \Var_{\bbP,2}\left( \int \hat{p}^{(1)}_{k_1}(\bx)\hat{b}^{(2)}_{k_2}(\bx) f(\bx) d\bx  \right)  \right].
\end{align*}

The second term in the above expression can be bounded as follows:
\begin{align} \label{eq: var exp int known f}
    & \Var_{\bbP,1} \left[ \E_{\bbP,2}\left(\int \hat{p}^{(1)}_{k_1}(\bx)\hat{b}^{(2)}_{k_2}(\bx) f(\bx) d\bx\right)  \right] \nonumber \\
    & \quad = \Var_{\bbP,1} \left[ \int \hat{p}_{k_1}^{(1)}(\bx) \Pi(b|V_{k_2})(\bx) f(\bx) d\bx \right] \nonumber \\
    & \quad = \E_{\bbP,1} \left[ \left( \int \hat{p}_{k_1}^{(1)}(\bx) \Pi(b|V_{k_2})(\bx) f(\bx) d\bx \right)^2 \right] \nonumber \\
    & \quad \quad - \left[ \E_{\bbP,1}  \left( \int \hat{p}_{k_1}^{(1)}(\bx) \Pi(b|V_{k_2})(\bx) f(\bx) d\bx \right) \right]^2 \nonumber \\
    & \quad  = \E_{\bbP,1} \left[  \iint \hat{p}_{k_1}^{(1)}(\bx)\hat{p}_{k_1}^{(1)}(\by) \Pi(b|V_{k_2})(\bx)\Pi(b|V_{k_2})(\by) f(\bx)f(\by) d\bx d\by  \right] \nonumber \\ & \quad \quad  - \left[  \int \E_{\bbP,1}  \left( \hat{p}_{k_1}^{(1)}(\bx)\right) \Pi(b|V_{k_2})(\bx) f(\bx) d\bx \right]^2 \nonumber \\
    & \quad  =  \iint  \E_{\bbP,1} \left[\hat{p}_{k_1}^{(1)}(\bx)\hat{p}_{k_1}^{(1)}(\by) \right] \Pi(b|V_{k_2})(\bx)\Pi(b|V_{k_2})(\by) f(\bx)f(\by) d\bx d\by  \nonumber \\
    & \quad \quad  - \left[  \int \Pi(p|V_{k_1})(\bx) \Pi(b|V_{k_2})(\bx) f(\bx) d\bx \right]^2 \nonumber \\
    & \quad  =  \iint  \E_{\bbP,1} \left[\hat{p}_{k_1}^{(1)}(\bx)\hat{p}_{k_1}^{(1)}(\by) \right] \Pi(b|V_{k_2})(\bx)\Pi(b|V_{k_2})(\by) f(\bx)f(\by) d\bx d\by  \nonumber \\
    & \quad \quad  -  \iint \Pi(p|V_{k_1})(\bx)\Pi(p|V_{k_1})(\by) \Pi(b|V_{k_2})(\bx)\Pi(b|V_{k_2})(\by) f(\bx)f(\by) d\bx d\by  \nonumber \\
    & \quad \lesssim  \iint \left| \E_{\bbP,1} \left[\hat{p}_{k_1}^{(1)}(\bx)\hat{p}_{k_1}^{(1)}(\by) \right] - \Pi(p|V_{k_1})(\bx)\Pi(p|V_{k_1})(\by)  \right| d\bx d\by \nonumber \\
    & \quad  \lesssim \frac{1}{n} + \iint \frac{1}{n} \E_{\bbP} \left[ \left| K_{V_{k_1}}(\bX, \bx)K_{V_{k_1}}(\bX, \by) \right| \right] d\bx d\by \quad (\text{by Lemma \ref{lem: exp pxpy}}) \nonumber \\
    & \quad  \lesssim \frac{1}{n} \quad (\text{by Lemma \ref{lem: exp kernel}}).
\end{align}

The third term can be bounded as follows:
\begin{align} \label{eq: exp var int known f}
    & \E_{\bbP,1} \left[ \Var_{\bbP,2}\left( \int \hat{p}^{(1)}_{k_1}(\bx)\hat{b}^{(2)}_{k_2}(\bx) f(\bx) d\bx  \right)  \right] \nonumber \\
    & \quad = \E_{\bbP,1} \left[ \E_{\bbP,2}\left[ \left(\int \hat{p}^{(1)}_{k_1}(\bx)\hat{b}^{(2)}_{k_2}(\bx) f(\bx) d\bx \right)^2 \right] \right] \nonumber \\
    & \quad \quad - \E_{\bbP,1} \left[ \left( \E_{\bbP,2}\left[ \int \hat{p}^{(1)}_{k_1}(\bx)\hat{b}^{(2)}_{k_2}(\bx) f(\bx) d\bx  \right] \right)^2 \right] \nonumber \\
    & \quad = \E_{\bbP, 1, 2} \left[  \iint \hat{p}_{k_1}^{(1)}(\bx)\hat{p}_{k_1}^{(1)}(\by)\hat{b}_{k_2}^{(2)}(\bx)\hat{b}_{k_2}^{(2)}(\by) f(\bx)f(\by)d\bx d\by   \right] \nonumber \\ & \quad \quad - \E_{\bbP,1} \left[ \left( \int \hat{p}_{k_1}^{(1)}(\bx)\Pi(b|V_{k_2})(\bx) f(\bx) d\bx \right)^2  \right] \nonumber \\
    & \quad =  \iint \E_{\bbP,1} \left(\hat{p}_{k_1}^{(1)}(\bx)\hat{p}_{k_1}^{(1)}(\by)\right)\E_{\bbP,2}\left(\hat{b}_{k_2}^{(2)}(\bx)\hat{b}_{k_2}^{(2)}(\by)\right) f(\bx)f(\by) d\bx d\by  \nonumber \\
    & \quad \quad  -  \iint \E_{\bbP,1} \left(\hat{p}_{k_1}^{(1)}(\bx)\hat{p}_{k_1}^{(1)}(\by)\right)\Pi(b|V_{k_2})(\bx)\Pi(b|V_{k_2})(\by) f(\bx)f(\by)d\bx d\by  \nonumber \\
    & \quad  \lesssim \iint \left| \begin{array}{c} \E_{\bbP,1} \left(\hat{p}_{k_1}^{(1)}(\bx)\hat{p}_{k_1}^{(1)}(\by)\right) \\ \times \left\{ \E_{\bbP,2}\left(\hat{b}_{k_2}^{(2)}(\bx)\hat{b}_{k_2}^{(2)}(\by)\right) - \Pi(b|V_{k_2})(\bx)\Pi(b|V_{k_2})(\by) \right\} \end{array}\right| d\bx d\by \nonumber \\
    & \quad  \lesssim \iint \begin{array}{c}\left( 1 + \frac{1}{n}\E_{\bbP} \left[ \left| K_{V_{k_1}}(\bX, \bx)K_{V_{k_1}}(\bX, \by) \right| \right] \right) \\ \times \left( \frac{1}{n} + \frac{1}{n}\E_{\bbP} \left[ \left| K_{V_{k_2}}(\bX, \bx)K_{V_{k_2}}(\bX, \by) \right| \right] \right) \end{array} d\bx d\by \quad (\text{by Lemma \ref{lem: exp pxpy}}) \nonumber \\
    & \quad  \lesssim \frac{1}{n} + \frac{k_1 \wedge k_2}{n^2} \quad (\text{by Lemma \ref{lem: exp kernel}}).
\end{align}

\subsubsection{Proof of the lower bound}

Let $\bbP \in \cP_{(\alpha,\beta)}$ and $\bX \sim \text{Uniform}([0,1]^d)$. Recall from the proof of the upper bound that we can express the bias of $\hat{\psi}_{k_1, k_2}^{\mathrm{INT}}$ by 
\begin{align*}
    \E_\bbP\left(\hat{\psi}_{k_1, k_2}^{\mathrm{INT}} - \psi(\bbP)\right)   & = - \int  \Pi(p|V_{k_1}^{\perp})(\bx) \Pi(b|V_{k_2}^{\perp})(\bx) d\bx   \\
    & \quad +  \int  b(\bx) \Pi(p|V_{k_1}^{\perp})(\bx)   d\bx  +  \int p(\bx) \Pi(b|V_{k_2}^{\perp})(\bx)  d\bx .
\end{align*}

Suppose that $k_1 \leq k_2$ for $n$ sufficiently large. It follows from orthogonality of $\Pi(b | V_{k_1})$ and $\Pi(p|V_{k_1}^{\perp})$ with respect to the Lebesgue measure that 
\begin{align*}
    \int  b(\bx) \Pi(p | V_{k_1}^{\perp})(\bx)    d\bx &  = \int  \Pi(p | V_{k_1}^{\perp})(\bx)\Pi(b | V_{k_1})(\bx)  d\bx + \int  \Pi(p | V_{k_1}^{\perp})(\bx)\Pi(b | V_{k_1}^{\perp})(\bx)  d\bx \\
    &  = \int  \Pi(p | V_{k_1}^{\perp})(\bx)\Pi(b | V_{k_1}^{\perp})(\bx)  d\bx.
\end{align*}
Since $V_{k_1} \subseteq V_{k_2}$, we have by orthogonality of $\Pi(p | V_{k_1})$ and $\Pi(b|V_{k_2}^{\perp})$ with respect to the Lebesgue measure that
\begin{align*}
    \int p(\bx) \Pi(b|V_{k_2}^{\perp})(\bx)  d\bx  & =   \int  \Pi(p | V_{k_1})(\bx) \Pi(b|V_{k_2}^{\perp})(\bx)  d\bx  + \int  \Pi(p | V_{k_1}^{\perp})(\bx)\Pi(b|V_{k_2}^{\perp})(\bx)  d\bx  \\
    & = \int  \Pi(p | V_{k_1}^{\perp})(\bx)\Pi(b|V_{k_2}^{\perp})(\bx)  d\bx.  
\end{align*}
Thus,
\begin{equation*}
    \left| \E_\bbP(\hat{\psi}^{\mathrm{INT}}_{k_1, k_2}) - \psi(\bbP) \right|   = \left|   \int  \Pi(p | V_{k_1}^{\perp})(\bx)\Pi(b | V_{k_1}^{\perp})(\bx)  d\bx\right|.
\end{equation*}

We can choose suitable $p$ and $b$ as follows. Adopting the notation introduced in Appendix \ref{sec: wavelets}, for $\bx \in [0, 1]^d$ let 
\begin{align} 
    p(\bx) & := \epsilon \sum_{l = J_0}^{\infty} \sum_{\boldm \in \mathcal{Z}_l} \sum_{\biota \in \mathcal{I}} 2^{-l(\alpha + \frac{d}{2})} \Psi_{l\boldm}^{\biota}(\bx) \label{eq: p lb} \\
    b(\bx) & := \epsilon \sum_{l = J_0}^{\infty}  \sum_{\boldm \in \mathcal{Z}_l} \sum_{\biota \in \mathcal{I}} 2^{-l(\beta + \frac{d}{2})} \Psi_{l\boldm}^{\biota}(\bx) \label{eq: b lb}
\end{align}
where $\epsilon > 0$. Then, $\| p \|_{B_{\infty \infty}^{\alpha}} = \epsilon$ and $\| b \|_{B_{\infty \infty}^{\beta}} = \epsilon$, which implies that $p \in H(\alpha, M)$ and $b \in H(\beta, M)$ for $\epsilon$ sufficiently small. 

By Parseval's identity,
\begin{align*}
    & \left|   \int  \Pi(p | V_{k_1}^{\perp})(\bx)\Pi(b | V_{k_1}^{\perp})(\bx)  d\bx\right|\\
    & \quad = \epsilon^2 \left| \int \left( \sum_{l \geq j_1} \sum_{\boldm \in \mathcal{Z}_l} \sum_{\biota \in \mathcal{I}}  2^{-l(\alpha + \frac{d}{2})} \Psi_{l\boldm}^{\biota}(\bx) \right) \left( \sum_{l \geq j_1} \sum_{\boldm \in \mathcal{Z}_l} \sum_{\biota \in \mathcal{I}} 2^{-l(\beta + \frac{d}{2})} \Psi_{l\boldm}^{\biota}(\bx) \right) d\bx \right| \\
    & \quad =  \epsilon^2 \sum_{l \geq j_1} \sum_{\boldm \in \mathcal{Z}_l} \sum_{\biota \in \mathcal{I}} 2^{-l(\alpha + \beta + d)}  \\
    & \quad \gtrsim  \sum_{l \geq j_1} 2^{-l(\alpha + \beta )} \\
    & \quad \gtrsim 2^{-j_1(\alpha + \beta )} = k_1^{-(\alpha + \beta) / d}. 
\end{align*}

The case of $k_1 \geq k_2$ can be seen to follow in the same manner. By orthogonality of $\Pi(p | V_{k_2})$ and $\Pi(b | V_{k_2}^{\perp})$ with respect to the Lebesgue measure, 
\begin{equation*}
    \int  p(\bx) \Pi(b | V_{k_2}^{\perp})(\bx)  d\bx = \int  \Pi(p | V_{k_2}^{\perp})(\bx) \Pi(b | V_{k_2}^{\perp})(\bx)  d\bx.
\end{equation*}
Since $V_{k_2} \subseteq V_{k_1}$, we have by orthogonality of $\Pi(b | V_{k_2}^{\perp})$ and $\Pi(p | V_{k_1}^{\perp})$ with respect to the Lebesgue measure that
\begin{equation*}
    \int  b(\bx) \Pi(p | V_{k_1}^{\perp})(\bx)    d\bx = \int  \Pi(p | V_{k_1}^{\perp})(\bx)\Pi(b | V_{k_2}^{\perp})(\bx)  d\bx.  
\end{equation*}
Therefore,
\begin{equation*}
    \left| \E_\bbP(\hat{\psi}^{\mathrm{INT}}_{k_1, k_2}) - \psi(\bbP) \right|   = \left|   \int \Pi(p | V_{k_2}^{\perp})(\bx) \Pi(b | V_{k_2}^{\perp})(\bx)   d\bx\right|  \gtrsim k_2^{-(\alpha + \beta ) / d}
\end{equation*}
for our choice of $p$ and $b$.

\subsection{Monte Carlo-based plug-in estimator}

\subsubsection{Proof of the upper bound}
Let $\bbP \in \cP_{(\alpha,\beta)}$ be arbitrary. First, note that the expectation of $\hat{\psi}_{k_1, k_2}^{\mathrm{MC}}$ is the same as that of $\hat{\psi}_{k_1, k_2}^{\mathrm{INT}}$. Therefore, 
\begin{equation*}
    \left|\E_\bbP\left(\hat{\psi}_{k_1, k_2}^{\mathrm{MC}} - \psi(\bbP)\right)\right| \lesssim k_1^{-(\alpha + \beta)/d} \vee k_2^{-(\alpha + \beta)/d}.
\end{equation*}

To bound the variance, we consider the decomposition
\begin{align*}
    \Var_\bbP(\hat{\psi}_{k_1, k_2}^{\mathrm{MC}}) & = \Var_\bbP\left( \frac{1}{n}\sum_{i \in \D_3} A_i Y_i\right) + \Var_\bbP\left( \frac{1}{n}\sum_{i \in \D_4} \hat{p}^{(1)}_{k_1}(\bX_i)\hat{b}^{(2)}_{k_2}(\bX_i)  \right) \\
    & = \frac{1}{n}\Var_\bbP(AY) +  \Var_{\bbP,1} \left[ \E_{\bbP,2,4}(\hat{p}^{(1)}_{k_1}(\bX)\hat{b}^{(2)}_{k_2}(\bX))  \right]  \\
    & \quad +  \E_{\bbP,1} \left[ \Var_{\bbP,2,4}\left( \frac{1}{n}\sum_{i \in \D_4} \hat{p}^{(1)}_{k_1}(\bX_i)\hat{b}^{(2)}_{k_2}(\bX_i)  \right)  \right].
\end{align*}

The second term can be bounded by
\begin{equation*}
    \Var_{\bbP,1} \left[ \E_{\bbP,2,4}(\hat{p}^{(1)}_{k_1}(\bX)\hat{b}^{(2)}_{k_2}(\bX))  \right] = \Var_{\bbP,1} \left[ \E_{\bbP,2} \left( \int \hat{p}^{(1)}_{k_1}(\bx) \hat{b}^{(2)}_{k_2}(\bx)f(\bx) d\bx\right) \right]  \lesssim \frac{1}{n}
\end{equation*}
where the inequality follows from (\ref{eq: var exp int known f}) (in the derivation of the variance of $\hat{\psi}^{\mathrm{INT}}_{k_1, k_2}$). 

To bound the third term, we further decompose the variance as follows:
\begin{align*}
    & \E_{\bbP,1} \left[ \Var_{\bbP,2,4}\left( \frac{1}{n}\sum_{i \in \D_4} \hat{p}^{(1)}_{k_1}(\bX_i)\hat{b}^{(2)}_{k_2}(\bX_i)  \right)  \right] \\
    & \quad = \E_{\bbP,1} \left[ \E_{\bbP,2}  \left[ \Var_{\bbP,4}\left( \frac{1}{n}\sum_{i \in \D_4} \hat{p}^{(1)}_{k_1}(\bX_i)\hat{b}^{(2)}_{k_2}(\bX_i)  \right)  \right] \right] \\ & \quad \quad  + \E_{\bbP,1} \left[ \Var_{\bbP,2} \left[ \E_{\bbP,4}\left( \frac{1}{n}\sum_{i \in \D_4} \hat{p}^{(1)}_{k_1}(\bX_i)\hat{b}^{(2)}_{k_2}(\bX_i)  \right)  \right] \right] \\
    & \quad = \frac{1}{n} \E_{\bbP,1,2}  \left[ \Var_{\bbP,4}\left( \hat{p}^{(1)}_{k_1}(\bX)\hat{b}^{(2)}_{k_2}(\bX)  \right)  \right]  \\ & \quad \quad  + \E_{\bbP,1} \left[ \Var_{\bbP,2} \left[ \E_{\bbP,4}\left( \hat{p}^{(1)}_{k_1}(\bX)\hat{b}^{(2)}_{k_2}(\bX)  \right)  \right] \right].
\end{align*}

We have that
\begin{align*}
    \E_{\bbP,1} \left[ \Var_{\bbP,2} \left[ \E_{\bbP,4}\left( \hat{p}^{(1)}_{k_1}(\bX)\hat{b}^{(2)}_{k_2}(\bX)  \right)  \right] \right] & = \E_{\bbP,1} \left[ \Var_{\bbP,2} \left[ \int \hat{p}^{(1)}_{k_1}(\bx)\hat{b}^{(2)}_{k_2}(\bx) f(\bx) d\bx   \right] \right] \\
    & \lesssim \frac{1}{n} + \frac{k_1 \wedge k_2}{n^2}
\end{align*}
where the inequality follows from (\ref{eq: exp var int known f}) (in the derivation of the variance of $\hat{\psi}^{\mathrm{INT}}_{k_1, k_2}$). Moreover, 
\begin{align*}
    \frac{1}{n} \E_{\bbP,1,2}  \left[ \Var_{\bbP,4}\left( \hat{p}^{(1)}_{k_1}(\bX)\hat{b}^{(2)}_{k_2}(\bX)  \right)  \right] & \leq \frac{1}{n} \E_{\bbP,1,2}  \left[ \E_{\bbP,4}\left[ \left( \hat{p}^{(1)}_{k_1}(\bX)\hat{b}^{(2)}_{k_2}(\bX)  \right)^2 \right]  \right] \\
    & \lesssim \frac{1}{n} + \frac{k_1}{n^2} + \frac{k_2}{n^2} +  \frac{k_1k_2}{n^3} \quad (\text{by Lemma \ref{lem: nuisance function bounds}}).
\end{align*}
Thus,
\begin{equation*}
    \Var_\bbP(\hat{\psi}_{k_1, k_2}^{\mathrm{MC}}) \lesssim \frac{1}{n}  +  \frac{k_1 \vee k_2}{n^2} + \frac{k_1k_2}{n^3}.
\end{equation*}

\subsubsection{Proof of the lower bound}

Since the expectation of $\hat{\psi}_{k_1, k_2}^{\mathrm{MC}}$ is the same as that of $\hat{\psi}_{k_1, k_2}^{\mathrm{INT}}$, we conclude that the lower bound of the bias holds when $\bX \sim \text{Uniform}([0,1]^d)$ and $\bbP \in \cP_{(\alpha,\beta)}$ where $p$ and $b$ are defined as in (\ref{eq: p lb}) and (\ref{eq: b lb}). For this choice of $\bbP$, we derive the lower bound on the variance when $k_1, k_2 \gg n$. We have that 
\begin{align*}
    \Var_\bbP(\hat{\psi}^{\mathrm{MC}}_{k_1, k_2}) & \geq \Var_{\bbP, 1, 2, 4}\left( \frac{1}{n}\sum_{i \in \D_4} \hat{p}^{(1)}_{k_1}(\bX_i)\hat{b}^{(2)}_{k_2}(\bX_i)  \right) \\
    & \geq \E_{\bbP,1, 2} \left[ \Var_{\bbP,4}\left( \frac{1}{n}\sum_{i \in \D_4} \hat{p}^{(1)}_{k_1}(\bX_i)\hat{b}^{(2)}_{k_2}(\bX_i)  \right)  \right] \\
    & = \frac{1}{n}\E_{\bbP,1, 2} \left[ \Var_{\bbP,4}\left(  \hat{p}^{(1)}_{k_1}(\bX)\hat{b}^{(2)}_{k_2}(\bX)  \right)  \right]. 
\end{align*}

It remains to show
\begin{equation} \label{eq: if exp var lb}
    \E_{\bbP, 1, 2} \left[\Var_{\bbP, 4} \left( \hat{p}^{(1)}_{k_1}(\bX) \hat{b}^{(2)}_{k_2}(\bX)  \right) \right] \gtrsim \frac{k_1k_2}{n^2}.
\end{equation}
We have that
\begin{align*}
    \E_{\bbP, 1, 2} \left[\Var_{\bbP, 4} \left( \hat{p}^{(1)}_{k_1}(\bX) \hat{b}^{(2)}_{k_2}(\bX)  \right) \right] & =   \E_{\bbP, 1, 2} \left[\E_{\bbP, 4} \left[ \left( \hat{p}^{(1)}_{k_1}(\bX) \hat{b}^{(2)}_{k_2}(\bX) \right)^2  \right] \right] \\
    & \quad -   \E_{\bbP, 1, 2} \left[\E_{\bbP, 4} \left[ \hat{p}^{(1)}_{k_1}(\bX) \hat{b}^{(2)}_{k_2}(\bX)   \right]^2 \right].
\end{align*}
By Lemma \ref{lem: nuisance function bounds}, we have for this choice of $\bbP$
\begin{equation*}
    \E_{\bbP, 1, 2} \left[\E_{\bbP, 4} \left[ \left( \hat{p}^{(1)}_{k_1}(\bX) \hat{b}^{(2)}_{k_2}(\bX) \right)^2  \right] \right] \gtrsim \frac{k_1k_2}{n^2}.
\end{equation*}
Moreover,
\begin{align*}
    \E_{\bbP, 1, 2} \left[\E_{\bbP, 4} \left[ \hat{p}^{(1)}_{k_1}(\bX) \hat{b}^{(2)}_{k_2}(\bX)   \right]^2 \right] & = \E_{\bbP, 1, 2} \left[ \iint \hat{p}^{(1)}_{k_1}(\bx)\hat{p}^{(1)}_{k_1}(\by) \hat{b}^{(2)}_{k_2}(\bx)\hat{b}^{(2)}_{k_2}(\by) d\bx d\by \right] \\
    & =  \iint \E_{\bbP, 1} \left[\hat{p}^{(1)}_{k_1}(\bx)\hat{p}^{(1)}_{k_1}(\by) \right] \E_{\bbP, 2} \left[ \hat{b}^{(2)}_{k_2}(\bx)\hat{b}^{(2)}_{k_2}(\by) \right] d\bx d\by.
\end{align*}
It follows from Lemma \ref{lem: exp pxpy} that for any $\bx, \by \in [0, 1]^d$ 
\begin{align*}
    \E_{\bbP, 1} \left[\hat{p}^{(1)}_{k_1}(\bx)\hat{p}^{(1)}_{k_1}(\by) \right] & =  \Pi(p|V_{k_1})(\bx) \Pi(p|V_{k_1})(\by) + S_{n,k_1}(\bx, \by) \\
    \E_{\bbP, 2} \left[ \hat{b}^{(2)}_{k_2}(\bx)\hat{b}^{(2)}_{k_2}(\by) \right] & =  \Pi(b|V_{k_2})(\bx) \Pi(b|V_{k_2})(\by) + S^{\prime}_{n,k_2}(\bx, \by) 
\end{align*}
where there exists some $C_1, C_2 \in \mathbb{R}^+$ such that for all $\bx, \by \in [0, 1]^d$
\begin{align*}
    | S_{n,k_1}(\bx, \by) | & \leq C_1 \left( \frac{1}{n} + \frac{1}{n} \E_{\bbP} \left[ \left| K_{V_{k_1}}(\bX, \bx)K_{V_{k_1}}(\bX, \by) \right| \right] \right) \\
    | S^{\prime}_{n,k_2}(\bx, \by) | & \leq C_2 \left( \frac{1}{n} + \frac{1}{n} \E_{\bbP} \left[ \left| K_{V_{k_2}}(\bX, \bx)K_{V_{k_2}}(\bX, \by) \right| \right] \right).
\end{align*}
Therefore,
\begin{align*}
    & \E_{\bbP, 1, 2} \left[\E_{\bbP, 4} \left[ \hat{p}^{(1)}_{k_1}(\bX) \hat{b}^{(2)}_{k_2}(\bX)   \right]^2 \right] \\
    & \quad \lesssim 1 + \iint | S_{n,k_1}(\bx, \by)  S^{\prime}_{n,k_2}(\bx, \by) | d\bx d\by \\
    & \quad \lesssim 1 + \iint \begin{array}{c}\left( \frac{1}{n} + \frac{1}{n} \E_{\bbP} \left[ \left| K_{V_{k_1}}(\bX, \bx)K_{V_{k_1}}(\bX, \by) \right| \right] \right) \\ \times \left( \frac{1}{n} + \frac{1}{n} \E_{\bbP} \left[ \left| K_{V_{k_2}}(\bX, \bx)K_{V_{k_2}}(\bX, \by) \right| \right] \right) \end{array} d\bx d\by \\
    & \quad \lesssim 1 + \frac{k_1 \wedge k_2}{n^2}
\end{align*}
which establishes (\ref{eq: if exp var lb}).

\subsubsection{Bounds for the three split estimator} \label{sec: mc three split}

Here, we show that the same bias and variance bounds hold for the Monte Carlo-based plug-in estimator with three splits:
\begin{equation*}
        \tilde{\psi}^{\mathrm{MC}}_{k_1, k_2}  = \frac{1}{n}\sum_{i \in \D_3} A_i Y_i -  \frac{1}{n}\sum_{i \in \D_3} \hat{p}^{(1)}_{k_1}(\bX_i)\hat{b}^{(2)}_{k_2}(\bX_i).
    \end{equation*}
Let $\mathbb{P} \in \mathcal{P}_{(\alpha, \beta)}$ be arbitrary. The bias of the three fold estimator $\tilde{\psi}^{\mathrm{MC}}_{k_1, k_2}$ is identical to that of the four fold estimator $\hat{\psi}^{\mathrm{MC}}_{k_1, k_2}$ by linearity of expectation. Moreover, the same upper bound on the variance immediately holds, as 
\begin{align*}
    \Var_{\mathbb{P}}(\tilde{\psi}^{\mathrm{MC}}_{k_1, k_2}) & \leq 2 \Var_{\mathbb{P}}\left(\frac{1}{n}\sum_{i \in \D_3} A_i Y_i\right) +   2\Var_{\mathbb{P}}\left(\frac{1}{n}\sum_{i \in \D_3} \hat{p}^{(1)}_{k_1}(\bX_i)\hat{b}^{(2)}_{k_2}(\bX_i) \right) \\
    & = 2\Var_{\mathbb{P}}(\hat{\psi}^{\mathrm{MC}}_{k_1, k_2}) \\
    & \lesssim \frac{1}{n} + \frac{k_1 \vee k_2}{n^2} + \frac{k_1k_2}{n^3}.
\end{align*}
Last, we establish the lower bound on the variance of $\tilde{\psi}^{\mathrm{MC}}_{k_1, k_2}$ when $\bX \sim \text{Uniform}([0,1]^d)$ and $\bbP \in \cP_{(\alpha,\beta)}$ where $p$ and $b$ are defined as in (\ref{eq: p lb}) and (\ref{eq: b lb}). We have that 
\begin{align*}
    \Var_{\mathbb{P}}(\tilde{\psi}^{\mathrm{MC}}_{k_1, k_2}) & = \Var_{\mathbb{P}}\left(\frac{1}{n}\sum_{i \in \D_3} A_i Y_i\right) + \Var_{\mathbb{P}}\left(\frac{1}{n}\sum_{i \in \D_3} \hat{p}^{(1)}_{k_1}(\bX_i)\hat{b}^{(2)}_{k_2}(\bX_i) \right) \\
    & \quad - 2 \Cov_{\mathbb{P}}\left(\frac{1}{n}\sum_{i \in \D_3} A_i Y_i, \frac{1}{n}\sum_{i \in \D_3} \hat{p}^{(1)}_{k_1}(\bX_i)\hat{b}^{(2)}_{k_2}(\bX_i)\right).
\end{align*}
We will show that the second term dominates when $k_1, k_2 \gg n$. Recall that the first term is $O(1/n)$. To bound the third term, we have that
\begin{align*}
    & \left| \Cov_{\mathbb{P}}\left(\frac{1}{n}\sum_{i \in \D_3} A_i Y_i, \frac{1}{n}\sum_{i \in \D_3} \hat{p}^{(1)}_{k_1}(\bX_i)\hat{b}^{(2)}_{k_2}(\bX_i)\right) \right| \\
    & \quad \leq \sqrt{\Var_{\mathbb{P}}\left( \frac{1}{n}\sum_{i \in \D_3} A_i Y_i \right) \Var_{\mathbb{P}} \left( \frac{1}{n}\sum_{i \in \D_3} \hat{p}^{(1)}_{k_1}(\bX_i)\hat{b}^{(2)}_{k_2}(\bX_i) \right)}
\end{align*}
where recall from our analysis of the four fold estimator that
\begin{align*}
    \Var_{\mathbb{P}}\left( \frac{1}{n}\sum_{i \in \D_3} A_i Y_i \right) & \lesssim \frac{1}{n} \\
    \Var_{\mathbb{P}} \left( \frac{1}{n}\sum_{i \in \D_3} \hat{p}^{(1)}_{k_1}(\bX_i)\hat{b}^{(2)}_{k_2}(\bX_i) \right) & \lesssim \frac{k_1k_2}{n^3}.
\end{align*}
Therefore, the third term is $O(\sqrt{\frac{k_1k_2}{n^4}})$. For the second term, recall that we showed in our analysis of the four fold estimator that for our choice of $\mathbb{P}$,
\begin{equation*}
    \Var_{\bbP}\left(\frac{1}{n}\sum_{i \in \D_3} \hat{p}^{(1)}_{k_1}(\bX_i)\hat{b}^{(2)}_{k_2}(\bX_i) \right) \gtrsim \frac{k_1k_2}{n^3}
\end{equation*}
and $\left|\E_\bbP\left(\tilde{\psi}^{\mathrm{MC}}_{k_1, k_2} - \psi(\bbP)\right)\right| \gtrsim (k_1 \wedge k_2)^{-(\alpha + \beta)/d}$. Since $k_1, k_2 \gg n$, it follows that $\frac{k_1 k_2}{n^3} \gg \sqrt{\frac{k_1k_2}{n^4}}$, i.e., the second term dominates the variance. Thus, the same lower bound holds,
\begin{equation*}
    \Var_\bbP(\tilde{\psi}^{\mathrm{MC}}_{k_1, k_2}) \gtrsim \frac{k_1k_2}{n^3}.
\end{equation*}

\subsection{First-order bias-corrected estimator}

\subsubsection{Proof of the upper bound} Let $\bbP \in \cP_{(\alpha,\beta)}$ be arbitrary.\\ 

\noindent \textbf{Bounding the bias:}

We can express the expectation of $\hat{\psi}_{k_1, k_2}^{\mathrm{IF}}$ by
\begin{align*}
    \E_\bbP(\hat{\psi}^{\mathrm{IF}}_{k_1, k_2}) & = \E_{\bbP, 1, 2}\left[  \E_{\bbP,3} ((A - \hat{p}^{(1)}_{k_1}(\bX))(Y - \hat{b}^{(2)}_{k_2}(\bX)))  \right] \\
    & = \E_{\bbP, 1, 2}\left[  \E_{\bbP,3} \left(  \begin{array}{c}(A - p(\bX) + p(\bX) - \hat{p}^{(1)}_{k_1}(\bX)) \\ \times (Y - b(\bX) + b(\bX)  - \hat{b}^{(2)}_{k_2}(\bX)) \end{array} \right)  \right] \\
    & = \E_{\bbP, 1, 2}\left[  \E_{\bbP, 3} \left(  (A-p(\bX))(Y - b(\bX))  \right) \right] \\
    & \quad + \E_{\bbP, 1, 2}\left[  \E_{\bbP, 3} \left(  (p(\bX) - \hat{p}^{(1)}_{k_1}(\bX))(b(\bX) - \hat{b}^{(2)}_{k_2}(\bX))  \right) \right] \\ & \quad + \E_{\bbP, 1, 2}\left[  \E_{\bbP, 3} \left(  (p(\bX) - \hat{p}^{(1)}_{k_1}(\bX))(Y - b(\bX))  \right) \right] \\
    & \quad + \E_{\bbP, 1, 2}\left[  \E_{\bbP, 3} \left(  (A - p(\bX))(b(\bX) - \hat{b}^{(2)}_{k_2}(\bX))  \right) \right] \\
    & = \psi(\bbP) + \E_{\bbP, 1, 2} \left[  \E_{\bbP, 3} \left(  (p(\bX) - \hat{p}^{(1)}_{k_1}(\bX))(b(\bX) - \hat{b}^{(2)}_{k_2}(\bX)) \right) \right] \\
    & = \psi(\bbP) + \E_{\bbP, 1, 2} \left[ \int (p(\bx) - \hat{p}^{(1)}_{k_1}(\bx))(b(\bx) - \hat{b}^{(2)}_{k_2}(\bx)) f(\bx) d\bx \right] \\
    & = \psi(\bbP) + \int \left( p(\bx) - \Pi(p | V_{k_1})(\bx)\right)\left( b(\bx) - \Pi(b | V_{k_2})(\bx)\right) f(\bx) d\bx \\
    & = \psi(\bbP) + \int \Pi(p | V_{k_1}^{\perp})(\bx) \Pi(b | V_{k_2}^{\perp})(\bx) f(\bx) d\bx.
\end{align*}

Suppose without loss of generality that $k_1 \leq k_2$ for $n$ sufficiently large. Then
\begin{align*}
    \E_\bbP(\hat{\psi}^{\mathrm{IF}}_{k_1, k_2}) - \psi(\bbP) & =  \int p(\bx)f(\bx) \Pi(b | V_{k_2}^{\perp})(\bx)  d\bx  - \int \Pi(p | V_{k_1})(\bx) f(\bx) \Pi(b | V_{k_2}^{\perp})(\bx)  d\bx. 
\end{align*}
Recall from the derivation of the bias of $\hat{\psi}_{k_1, k_2}^{\mathrm{INT}}$ that
\begin{equation*}
    \left| \int p(\bx)f(\bx) \Pi(b | V_{k_2}^{\perp})(\bx)  d\bx \right| \lesssim k_2^{-(\alpha + \beta)/ d}.
\end{equation*}
Moreover, by orthogonality of $\Pi\left(\Pi(p | V_{k_1}) f | V_{k_2}^{\perp}\right)$ and $\Pi(b | V_{k_2}^{\perp})$ with respect to the Lebesgue measure
\begin{align*}
    \int \Pi(p | V_{k_1})(\bx) f(\bx) \Pi(b | V_{k_2}^{\perp})(\bx)  d\bx &  = \int \Pi\left(\Pi(p | V_{k_1}) f | V_{k_2}^{\perp}\right)(\bx) \Pi(b | V_{k_2}^{\perp})(\bx)  d\bx \\
    & \quad + \int \Pi\left(\Pi(p | V_{k_1}) f | V_{k_2}\right)(\bx) \Pi(b | V_{k_2}^{\perp})(\bx)  d\bx \\
    & = \int \Pi\left(\Pi(p | V_{k_1}) f | V_{k_2}^{\perp}\right)(\bx) \Pi(b | V_{k_2}^{\perp})(\bx)  d\bx.
\end{align*}
Since $p$ is $\alpha$-Hölder, then so is $\Pi(p | V_{k_1})$. Recalling that $b$ is $\beta$-Hölder, it follows from Property P1 (Appendix \ref{sec: wavelets}) that
\begin{align*}
    \left| \int \Pi\left(\Pi(p | V_{k_1}) f | V_{k_2}^{\perp}\right)(\bx) \Pi(b | V_{k_2}^{\perp})(\bx)  d\bx \right| & \leq \| \Pi\left(\Pi(p | V_{k_1}) f | V_{k_2}^{\perp}\right) \|_{\infty} \| \Pi(b | V_{k_2}^{\perp}) \|_{\infty} \\
    & \lesssim k_2^{-(\alpha + \beta)/d}.
\end{align*}
Therefore, we can bound the bias by
\begin{equation*}
    \left|\E_\bbP\left(\hat{\psi}_{k_1, k_2}^{\mathrm{IF}} - \psi(\bbP)\right)\right| \lesssim k_2^{-\frac{\alpha + \beta}{d}}. 
\end{equation*}

\noindent \textbf{Bounding the variance:}
    
We will show that
\begin{align}
    \Var_{\bbP,1,2} \left[ \E_{\bbP,3}\left( \hat{\psi}^{\mathrm{IF}}_{k_1, k_2}\right)  \right] & \lesssim \frac{1}{n} + \frac{k_1 \wedge k_2}{n^2} \label{var pt2 cov if} \\
    \E_{\bbP,1,2} \left[ \Var_{\bbP,3} \left( \hat{\psi}^{\mathrm{IF}}_{k_1, k_2}\right)  \right] & \lesssim \frac{1}{n} + \frac{k_1 \vee k_2}{n^2} + \frac{k_1k_2}{n^3} \label{eq: var pt1 cov if}.
\end{align}

We first show (\ref{var pt2 cov if}). We have that
\begin{align*}
    \Var_{\bbP,1,2} \left[ \E_{\bbP,3}\left( \hat{\psi}^{\mathrm{IF}}_{k_1, k_2}\right)  \right] & = \Var_{\bbP,1,2} \left[ \E_{\bbP,3}\left( \hat{\psi}^{\mathrm{IF}}_{k_1, k_2}\right) - \psi(\bbP)  \right]  \\
    & = \Var_{\bbP,1,2} \left[  \E_{\bbP,3} \left((A - \hat{p}^{(1)}_{k_1}(\bX))(Y - \hat{b}^{(2)}_{k_2}(\bX))\right) - \psi(\bbP)  \right]  \\
    & = \Var_{\bbP,1,2} \left[  \E_{\bbP,3} \left((p(\bX) - \hat{p}^{(1)}_{k_1}(\bX))(b(\bX) - \hat{b}^{(2)}_{k_2}(\bX))\right)  \right] \\
    & = \E_{\bbP,1,2} \left[ \left( \E_{\bbP,3} \left((p(\bX) - \hat{p}^{(1)}_{k_1}(\bX))(b(\bX) - \hat{b}^{(2)}_{k_2}(\bX))\right)\right)^2  \right] \\
    & \quad - \left( \E_{\bbP,1,2} \left[  \E_{\bbP,3} \left((p(\bX) - \hat{p}^{(1)}_{k_1}(\bX))(b(\bX) - \hat{b}^{(2)}_{k_2}(\bX))\right)  \right] \right)^2
\end{align*}
where the third line follows from the same arguments as in the bias derivation. Observe that
\begin{align*}
    & \left( \E_{\bbP,1,2} \left[  \E_{\bbP,3} \left((p(\bX) - \hat{p}^{(1)}_{k_1}(\bX))(b(\bX) - \hat{b}^{(2)}_{k_2}(\bX))\right)  \right] \right)^2 \\
    & \quad = \iint \Pi(p|V_{k_1}^{\perp})(\bx) \Pi(p|V_{k_1}^{\perp})(\by)\Pi(b|V_{k_2}^{\perp})(\bx) \Pi(b|V_{k_2}^{\perp})(\by)  f(\bx) f(\by)  d\bx d\by.
\end{align*}
Further,
\begin{align*}
    & \E_{\bbP,1,2} \left[ \left( \E_{\bbP,3} \left((p(\bX) - \hat{p}^{(1)}_{k_1}(\bX))(b(\bX) - \hat{b}^{(2)}_{k_2}(\bX))\right)\right)^2  \right] \\
    & \quad = \E_{\bbP,1,2} \left[ \left( \int (p(\bx) - \hat{p}^{(1)}_{k_1}(\bx))(b(\bx) - \hat{b}^{(2)}_{k_2}(\bx)) f(\bx) d\bx \right)^2  \right] \\
    & \quad = \E_{\bbP,1,2} \left[ \begin{array}{c} \iint (p(\bx) - \hat{p}^{(1)}_{k_1}(\bx))(p(\by) - \hat{p}^{(1)}_{k_1}(\by))(b(\bx) \\- \hat{b}^{(2)}_{k_2}(\bx))(b(\by) - \hat{b}^{(2)}_{k_2}(\by)) f(\bx)f(\by) d\bx d\by   \end{array}\right] \\
    & \quad =   \iint \begin{array}{c}\E_{\bbP,1} \left[(p(\bx) - \hat{p}^{(1)}_{k_1}(\bx)) (p(\by) - \hat{p}^{(1)}_{k_1}(\by)) \right] \\ \times \E_{\bbP,2} \left[(b(\bx) - \hat{b}^{(2)}_{k_2}(\bx))(b(\by) - \hat{b}^{(2)}_{k_2}(\by))\right] \end{array} f(\bx)f(\by) d\bx d\by.   
\end{align*}
It follows from Lemma \ref{lem: exp pxpy} that for any $\bx, \by \in [0, 1]^d$
\begin{align*}
    \E_{\bbP,1} \left[(p(\bx) - \hat{p}^{(1)}_{k_1}(\bx)) (p(\by) - \hat{p}^{(1)}_{k_1}(\by)) \right] & =  \Pi(p|V_{k_1}^{\perp})(\bx) \Pi(p|V_{k_1}^{\perp})(\by) + S_{n,k_1}(\bx, \by) \\
    \E_{\bbP,2} \left[(b(\bx) - \hat{b}^{(2)}_{k_2}(\bx)) (b(\by) - \hat{b}^{(2)}_{k_2}(\by)) \right] & =  \Pi(b|V_{k_2}^{\perp})(\bx) \Pi(b|V_{k_2}^{\perp})(\by) + S^{\prime}_{n,k_2}(\bx, \by) 
\end{align*}
where there exists $C_1, C_2 \in \mathbb{R}^+$ such that for all $\bx, \by \in [0, 1]^d$
\begin{align*} 
    | S_{n,k_1}(\bx, \by) | & \leq C_1 \left( \frac{1}{n} + \frac{1}{n} \E_{\bbP} \left[ \left| K_{V_{k_1}}(\bX, \bx)K_{V_{k_1}}(\bX, \by) \right| \right] \right) \\
    | S^{\prime}_{n,k_2}(\bx, \by) | & \leq C_2 \left( \frac{1}{n} + \frac{1}{n} \E_{\bbP} \left[ \left| K_{V_{k_2}}(\bX, \bx)K_{V_{k_2}}(\bX, \by) \right| \right] \right).
\end{align*}
Then,
\begin{align*}
    \Var_{\bbP,1,2} \left[ \E_{\bbP,3}\left( \hat{\psi}^{\mathrm{IF}}_{k_1, k_2}\right)  \right] & \leq  \left| \iint \Pi(p|V_{k_1}^{\perp})(\bx) \Pi(p|V_{k_1}^{\perp})(\by) S^{\prime}_{n,k_2}(\bx, \by) f(\bx)f(\by) d\bx d\by \right| \\ 
    & \quad + \left| \iint \Pi(b|V_{k_2}^{\perp})(\bx) \Pi(b|V_{k_2}^{\perp})(\by) S_{n,k_1}(\bx, \by) f(\bx)f(\by) d\bx d\by \right| \\
    & \quad + \left| \iint S_{n,k_1}(\bx, \by) S^{\prime}_{n,k_2}(\bx, \by) f(\bx)f(\by) d\bx d\by \right| \\ 
    & \leq   \iint \| \Pi(p|V_{k_1}^{\perp}) \|_{\infty}^2 |S^{\prime}_{n,k_2}(\bx, \by)| f(\bx)f(\by) d\bx d\by  \\ 
    & \quad +  \iint \| \Pi(b|V_{k_2}^{\perp}) \|_{\infty}^2 | S_{n,k_1}(\bx, \by)| f(\bx)f(\by) d\bx d\by  \\
    & \quad + \iint |S_{n,k_1}(\bx, \by) S^{\prime}_{n,k_2}(\bx, \by)| f(\bx)f(\by) d\bx d\by  \\ 
    & \lesssim  k_1^{- \frac{2\alpha}{d}}  \iint  | S^{\prime}_{n,k_2}(\bx, \by)|  d\bx d\by  \\ 
    & \quad + k_2^{- \frac{2\beta}{d}}  \iint |S_{n,k_1}(\bx, \by)| d\bx d\by \\
    & \quad +  \iint |S_{n,k_1}(\bx, \by) S^{\prime}_{n,k_2}(\bx, \by)| d\bx d\by .
\end{align*}
Moreover, it follows from Lemma \ref{lem: exp kernel} that
\begin{align*}
    k_1^{- \frac{2\alpha}{d}}  \iint  | S^{\prime}_{n,k_2}(\bx, \by)|  d\bx d\by & \lesssim \frac{1}{n} \\
    k_2^{- \frac{2\beta}{d}}  \iint |S_{n,k_1}(\bx, \by)| d\bx d\by & \lesssim  \frac{1}{n} \\
    \iint |S_{n,k_1}(\bx, \by) S^{\prime}_{n,k_2}(\bx, \by)| d\bx d\by & \lesssim \frac{1}{n^2} + \frac{k_1 \wedge k_2}{n^2}
\end{align*}
which establishes (\ref{var pt2 cov if}).

Next, we show (\ref{eq: var pt1 cov if}). We have that
\begin{align*}
    & \E_{\bbP,1,2} \left[ \Var_{\bbP,3} \left( \hat{\psi}^{\mathrm{IF}}_{k_1, k_2}\right)  \right] \\
    & \quad = \frac{1}{n} \E_{\bbP,1,2} \left[ \Var_{\bbP,3} \left( (A - \hat{p}^{(1)}_{k_1}(\bX))(Y - \hat{b}^{(2)}_{k_2}(\bX)) \right) \right] \\
    & \quad \leq \frac{1}{n} \E_{\bbP,1,2} \left[ \E_{\bbP,3} \left( (A - \hat{p}^{(1)}_{k_1}(\bX))^2(Y - \hat{b}^{(2)}_{k_2}(\bX))^2 \right) \right] \\
    & \quad = \frac{1}{n} \E_{\bbP,1,2} \left[ \E_{\bbP,3} \left( (A^2 - 2A \hat{p}^{(1)}_{k_1}(\bX) + \hat{p}^{(1)}_{k_1}(\bX)^2)(Y^2 - 2Y\hat{b}^{(2)}_{k_2}(\bX) + \hat{b}^{(2)}_{k_2}(\bX)^2) \right) \right] \\
    & \quad = \frac{1}{n} \E_{\bbP,3} \left( \begin{array}{c} \left(A^2 - 2A \Pi(p | V_{k_1})(\bX) + \E_{\bbP,1} \left[\hat{p}^{(1)}_{k_1}(\bX)^2 \right]\right) \\ \times  \left(Y^2 - 2Y\Pi(b | V_{k_2})(\bX) + \E_{\bbP,2}\left[\hat{b}^{(2)}_{k_2}(\bX)^2\right]\right) \end{array}  \right).
\end{align*}
Since $A$, $Y$, $\bX$ are bounded random variables and $\Pi(p | V_{k_1})$ and $\Pi(b | V_{k_2})$ are bounded, 
\begin{align*}
    \E_{\bbP,1,2} \left[ \Var_{\bbP,3} \left( \hat{\psi}^{\mathrm{IF}}_{k_1, k_2}\right)  \right] & \lesssim \frac{1}{n} + \frac{1}{n} \E_{\bbP,3} \left( \E_{\bbP,1}\left[\hat{p}^{(1)}_{k_1}(\bX)^2\right] \right) +  \frac{1}{n} \E_{\bbP,3} \left( \E_{\bbP,2}\left[\hat{b}^{(2)}_{k_2}(\bX)^2\right] \right) \\ & \quad + \frac{1}{n} \E_{\bbP,3} \left(  \E_{\bbP,1}\left[\hat{p}^{(1)}_{k_1}(\bX)^2\right] \E_{\bbP,2}\left[\hat{b}^{(2)}_{k_2}(\bX)^2\right] \right).
\end{align*}
It then follows from Lemma \ref{lem: nuisance function bounds} that
\begin{equation*}
    \E_{\bbP,1,2} \left[ \Var_{\bbP,3} \left( \hat{\psi}^{\mathrm{IF}}_{k_1, k_2}\right)  \right] \lesssim \frac{1}{n} + \frac{k_1 \vee k_2}{n^2} + \frac{k_1k_2}{n^3}.
\end{equation*}

\subsubsection{Proof of the lower bound} 

Let $\bbP \in \cP_{(\alpha,\beta)}$ and $\bX \sim \text{Uniform}([0,1]^d)$. \\

\noindent \textbf{Bounding the bias:}

Recall from the proof of the upper bound that
\begin{equation*}
    \E_\bbP(\hat{\psi}^{\mathrm{IF}}_{k_1, k_2}) - \psi(\bbP)   =  \int \Pi(p | V_{k_1}^{\perp})(\bx)\Pi(b | V_{k_2}^{\perp})(\bx) d\bx.
\end{equation*}
Suppose without loss of generality that $k_1 \leq k_2$ for $n$ sufficiently large. We choose $p$ and $b$ as given in (\ref{eq: p lb}) and (\ref{eq: b lb}), and we follow the same approach used in the proof of the lower bound of the bias of $\hat{\psi}_{k_1, k_2}^{\mathrm{INT}}$. That is,  by Parseval's identity,
\begin{align*}
    & \left|   \int  \Pi(p | V_{k_1}^{\perp})(\bx)\Pi(b | V_{k_2}^{\perp})(\bx)  d\bx\right|\\
    & \quad = \epsilon^2 \left| \int \begin{array}{c} \left( \sum_{l \geq j_1} \sum_{\boldm \in \mathcal{Z}_l} \sum_{\biota \in \mathcal{I}}  2^{-l(\alpha + \frac{d}{2})} \Psi_{l\boldm}^{\biota}(\bx) \right) \\ \times \left( \sum_{l \geq j_2} \sum_{\boldm \in \mathcal{Z}_l} \sum_{\biota \in \mathcal{I}} 2^{-l(\beta + \frac{d}{2})} \Psi_{l\boldm}^{\biota}(\bx) \right) \end{array} d\bx \right| \\
    & \quad = \epsilon^2 \left| \int \begin{array}{c} \left( \sum_{l \geq j_1} \sum_{\boldm \in \mathcal{Z}_l} \sum_{\biota \in \mathcal{I}}  2^{-l(\alpha + \frac{d}{2})} \Psi_{l\boldm}^{\biota}(\bx) \right) \\ \times \left( \sum_{l \geq j_1} \sum_{\boldm \in \mathcal{Z}_l} \sum_{\biota \in \mathcal{I}} 2^{-l(\beta + \frac{d}{2})} I(l \geq j_2) \Psi_{l\boldm}^{\biota}(\bx) \right) \end{array} d\bx \right| \\
    & \quad =  \epsilon^2 \sum_{l \geq j_1} \sum_{\boldm \in \mathcal{Z}_l} \sum_{\biota \in \mathcal{I}} 2^{-l(\alpha + \beta + d)} I(l \geq j_2)  \\
    & \quad =  \epsilon^2 \sum_{l \geq j_2} \sum_{\boldm \in \mathcal{Z}_l} \sum_{\biota \in \mathcal{I}} 2^{-l(\alpha + \beta + d)}  \\
    & \quad \gtrsim  \sum_{l \geq j_2} 2^{-l(\alpha + \beta )} \\
    & \quad \gtrsim 2^{-j_2(\alpha + \beta )} = k_2^{-(\alpha + \beta) / d}. 
\end{align*}

\noindent \textbf{Bounding the variance:}

Next, we show $\Var_\bbP(\hat{\psi}^{\mathrm{IF}}_{k_1,k_2}) \gtrsim \frac{k_1k_2}{n^3}$ by showing $\E_{\bbP, 1, 2} \left[ \Var_{\bbP, 3}(\hat{\psi}^{\mathrm{IF}}_{k_1,k_2})  \right] \gtrsim  \frac{k_1k_2}{n^3}$. We have that
\begin{align*}
    & \E_{\bbP, 1, 2} \left[ \Var_{\bbP, 3}(\hat{\psi}^{\mathrm{IF}}_{k_1,k_2})  \right] \\
    & \quad = \frac{1}{n}  \E_{\bbP, 1, 2} \left[ \Var_{\bbP, 3} \left( (A - \hat{p}^{(1)}_{k_1}(\bX))(Y - \hat{b}^{(2)}_{k_2}(\bX))  \right) \right] \\
    & \quad = \frac{1}{n}  \E_{\bbP, 1, 2} \left[\Var_{\bbP, 3} \left( AY - Y\hat{p}^{(1)}_{k_1}(\bX) - A\hat{b}^{(2)}_{k_2}(\bX) + \hat{p}^{(1)}_{k_1}(\bX) \hat{b}^{(2)}_{k_2}(\bX)  \right) \right] \\
    & \quad \geq \frac{1}{n}  \E_{\bbP, 1, 2} \left[\Var_{\bbP, 3} \left( AY - Y\hat{p}^{(1)}_{k_1}(\bX) - A\hat{b}^{(2)}_{k_2}(\bX)   \right) \right] \\
    & \quad \quad + \frac{1}{n}  \E_{\bbP, 1, 2} \left[\Var_{\bbP, 3} \left( \hat{p}^{(1)}_{k_1}(\bX) \hat{b}^{(2)}_{k_2}(\bX)  \right) \right] \\
    & \quad \quad - \frac{2}{n}  \E_{\bbP, 1, 2} \left[\left( \Var_{\bbP, 3} \left( AY - Y\hat{p}^{(1)}_{k_1}(\bX) - A\hat{b}^{(2)}_{k_2}(\bX)   \right) \Var_{\bbP, 3} \left( \hat{p}^{(1)}_{k_1}(\bX) \hat{b}^{(2)}_{k_2}(\bX)  \right)\right)^{1/2} \right] \\
    & \quad \geq \frac{1}{n}  \E_{\bbP, 1, 2} \left[\Var_{\bbP, 3} \left( \hat{p}^{(1)}_{k_1}(\bX) \hat{b}^{(2)}_{k_2}(\bX)  \right) \right] \\
    & \quad \quad - \frac{2}{n}  \E_{\bbP, 1, 2} \left[\left( \Var_{\bbP, 3} \left( AY - Y\hat{p}^{(1)}_{k_1}(\bX) - A\hat{b}^{(2)}_{k_2}(\bX)   \right) \Var_{\bbP, 3} \left( \hat{p}^{(1)}_{k_1}(\bX) \hat{b}^{(2)}_{k_2}(\bX)  \right)\right)^{1/2} \right] \\
    & \quad \geq \frac{1}{n}  \E_{\bbP, 1, 2} \left[\Var_{\bbP, 3} \left( \hat{p}^{(1)}_{k_1}(\bX) \hat{b}^{(2)}_{k_2}(\bX)  \right) \right] \\
    & \quad \quad -   \frac{2}{n} \left(\begin{array}{c} \E_{\bbP, 1, 2} \left[ \Var_{\bbP, 3} \left( AY - Y\hat{p}^{(1)}_{k_1}(\bX) - A\hat{b}^{(2)}_{k_2}(\bX)    \right) \right] \\ \times  \E_{\bbP, 1, 2} \left[ \Var_{\bbP, 3} \left( \hat{p}^{(1)}_{k_1}(\bX) \hat{b}^{(2)}_{k_2}(\bX)  \right) \right] \end{array}\right)^{1/2}.
\end{align*}
It then suffices to show that
\begin{align}
    \E_{\bbP, 1, 2} \left[\Var_{\bbP, 3} \left( \hat{p}^{(1)}_{k_1}(\bX) \hat{b}^{(2)}_{k_2}(\bX)  \right) \right] & \asymp \frac{k_1k_2}{n^2} \label{eq: if known f var lb helper1} \\
    \E_{\bbP, 1, 2} \left[ \Var_{\bbP, 3} \left( AY - Y\hat{p}^{(1)}_{k_1}(\bX) - A\hat{b}^{(2)}_{k_2}(\bX)    \right) \right] & \lesssim \frac{k_1 \vee k_2}{n} \label{eq: if known f var lb helper2}.
\end{align}
The bound in (\ref{eq: if known f var lb helper1}) can be seen to follow from the same steps used in proving (\ref{eq: if exp var lb}) (appearing in the derivation of the lower bound of the variance of $\hat{\psi}_{k_1, k_2}^{\mathrm{MC}}$). For the sake of completeness, we write out the proof of (\ref{eq: if known f var lb helper1}).  We have that
\begin{align*}
    \E_{\bbP, 1, 2} \left[\Var_{\bbP, 3} \left( \hat{p}^{(1)}_{k_1}(\bX) \hat{b}^{(2)}_{k_2}(\bX)  \right) \right] & =   \E_{\bbP, 1, 2} \left[\E_{\bbP, 3} \left[ \left( \hat{p}^{(1)}_{k_1}(\bX) \hat{b}^{(2)}_{k_2}(\bX) \right)^2  \right] \right] \\
    & \quad -   \E_{\bbP, 1, 2} \left[\E_{\bbP, 3} \left[ \hat{p}^{(1)}_{k_1}(\bX) \hat{b}^{(2)}_{k_2}(\bX)   \right]^2 \right].
\end{align*}
By Lemma \ref{lem: nuisance function bounds},
\begin{equation*}
    \E_{\bbP, 1, 2} \left[\E_{\bbP, 3} \left[ \left( \hat{p}^{(1)}_{k_1}(\bX) \hat{b}^{(2)}_{k_2}(\bX) \right)^2  \right] \right] \asymp \frac{k_1k_2}{n^2}.
\end{equation*}
Moreover,
\begin{align*}
     & \E_{\bbP, 1, 2} \left[\E_{\bbP, 3} \left[ \hat{p}^{(1)}_{k_1}(\bX) \hat{b}^{(2)}_{k_2}(\bX)   \right]^2 \right] \\
     & \quad = \E_{\bbP, 1, 2} \left[ \iint \hat{p}^{(1)}_{k_1}(\bx)\hat{p}^{(1)}_{k_1}(\by) \hat{b}^{(2)}_{k_2}(\bx)\hat{b}^{(2)}_{k_2}(\by) f(\bx) f(\by) d\bx d\by \right] \\
    & \quad \lesssim  \iint \left| \E_{\bbP, 1} \left[\hat{p}^{(1)}_{k_1}(\bx)\hat{p}^{(1)}_{k_1}(\by) \right] \E_{\bbP, 2} \left[ \hat{b}^{(2)}_{k_2}(\bx)\hat{b}^{(2)}_{k_2}(\by) \right] \right| d\bx d\by \\
    & \quad  \lesssim \iint \begin{array}{c} \left( 1 + \frac{1}{n} \E_{\bbP} \left[ \left| K_{V_{k_1}}(\bX, \bx)K_{V_{k_1}}(\bX, \by) \right| \right] \right) \\ \times \left( 1 + \frac{1}{n} \E_{\bbP} \left[ \left| K_{V_{k_2}}(\bX, \bx)K_{V_{k_2}}(\bX, \by) \right| \right] \right) \end{array} d\bx d\by \quad (\text{by Lemma \ref{lem: exp pxpy}})\\
    & \quad  \lesssim 1 + \frac{k_1 \wedge k_2}{n^2} \quad (\text{by Lemma \ref{lem: exp kernel}}).
\end{align*}

Next, we have that
\begin{align*} 
    & \E_{\bbP, 1, 2} \left[ \Var_{\bbP, 3} \left( AY - Y\hat{p}^{(1)}_{k_1}(\bX) - A\hat{b}^{(2)}_{k_2}(\bX)    \right) \right]  \\
    \quad & \leq \E_{\bbP, 1, 2} \left[ \E_{\bbP, 3} \left[ \left( AY - Y\hat{p}^{(1)}_{k_1}(\bX) - A\hat{b}^{(2)}_{k_2}(\bX)    \right)^2 \right] \right] \\
    & \quad \lesssim \frac{k_1 \vee k_2}{n}.
\end{align*}
The last line can be seen to hold from the following bounds, which follows from Lemma \ref{lem: nuisance function bounds} and the observation that $A, Y, \bX$ are bounded random variables and $\Pi(p | V_{k_1}), \Pi(b | V_{k_2})$ are bounded and
\begin{align*}
    \E_{\bbP, 1} [\E_{\bX} ( \hat{p}^{(1)}_{k_1}(\bX)^2 )] & \lesssim \frac{k_1}{n} \\
    \E_{\bbP, 2} [\E_{\bX} ( \hat{b}^{(2)}_{k_2}(\bX)^2 )]  & \lesssim \frac{k_2}{n}.
\end{align*}

\section{Proof of Theorem 2} \label{sec: proof known f single}

As in the proof of Theorem 1, we write out the proofs of the upper bounds when $f$ is such that 
\begin{align*}
    (i) & \quad  f(\bx) \in [M_1, M_2] \quad \forall \bx \in [0, 1]^d\\
    (ii) & \quad f \in H(\gamma, M)
\end{align*}
where $\gamma \geq \alpha \vee \beta$ and $M, M_1, M_2 \in \mathbb{R}^{+}$ are known constants. Recall that we use the approximate wavelet projection estimators of the nuisance functions given by (\ref{eq: wavelet estimator of p}) and (\ref{eq: wavelet estimator of b}) in this case.

Moreover, the estimators of $\psi(\bbP)$ that we analyze are given by
\begin{align*}
    \hat{\psi}_{k_1, k_2}^{\mathrm{INT}} & = \frac{1}{n}\sum_{i \in \D_2} A_i Y_i -  \int \hat{p}_{k_1}^{(1)}(\bx)\hat{b}_{k_2}^{(1)}(\bx)f(\bx) d\bx \\
    \hat{\psi}^{\mathrm{MC}}_{k_1, k_2} & = \frac{1}{n}\sum_{i \in \D_2} A_i Y_i -  \frac{1}{n}\sum_{i \in \D_3} \hat{p}^{(1)}_{k_1}(\bX_i)\hat{b}^{(1)}_{k_2}(\bX_i) \\
    \hat{\psi}^{\mathrm{IF}}_{k_1, k_2} & = \frac{1}{n} \sum_{i \in \D_2} (A_i - \hat{p}^{(1)}_{k_1}(\bX_i))(Y_i - \hat{b}^{(1)}_{k_2}(\bX_i)).
\end{align*}

\begin{remark}
    Similar to Remark \ref{rem:mc double modified} in the double sample splitting case, one can consider an alternative single sample split Monte Carlo-based plug-in estimator that uses only two folds:
    \begin{equation*}
        \tilde{\psi}^{\mathrm{MC}}_{k_1, k_2}  = \frac{1}{n}\sum_{i \in \D_2} A_i Y_i -  \frac{1}{n}\sum_{i \in \D_2} \hat{p}^{(1)}_{k_1}(\bX_i)\hat{b}^{(1)}_{k_2}(\bX_i).
    \end{equation*}
    The same bias and variance bounds in Theorem 2 hold for this estimator as well. In particular, the bias of the two-fold estimator $\tilde{\psi}^{\mathrm{MC}}_{k_1, k_2}$ is identical to that of the three fold version $\hat{\psi}^{\mathrm{MC}}_{k_1, k_2}$ by linearity of expectation. Moreover, we can bound the variance of $\tilde{\psi}^{\mathrm{MC}}_{k_1, k_2}$ by $\Var_{\bbP}(\tilde{\psi}^{\mathrm{MC}}_{k_1, k_2}) \leq 2\Var_{\bbP}(\hat{\psi}^{\mathrm{MC}}_{k_1, k_2})$ using the Cauchy-Schwarz inequality.
\end{remark}

\subsection{Integral-based plug-in estimator}

\subsubsection{Proof of the upper bound} 

Let $\bbP \in \cP_{(\alpha,\beta)}$ be arbitrary. \\

\noindent \textbf{Bounding the bias:}

We have that
\begin{equation*}
    \E_\bbP(\hat{\psi}_{k_1, k_2}^{\mathrm{INT}}) = \E_\bbP \left( A Y \right) - \int \E_{\bbP}[\hat{p}^{(1)}_{k_1}(\bx)\hat{b}^{(1)}_{k_2}(\bx)] f(\bx) d\bx.
\end{equation*}
Our strategy is to express the bias of the single sample split estimator as the bias of the double sample split estimator plus a remainder. It follows from the proof of Lemma \ref{lem: exp pxpy} that we may write
\begin{equation*}
    \E_{\bbP}[\hat{p}^{(1)}_{k_1}(\bx)\hat{b}^{(1)}_{k_2}(\bx)] = \Pi(p|V_{k_1})(\bx) \Pi(b|V_{k_2})(\bx) + O(n^{-1}) + r_{n, k_1, k_2}(\bx)
\end{equation*}
where
\begin{equation*}
    r_{n, k_1, k_2}(\bx) = \frac{1}{n}  \E_{\bbP,1}\left[ \frac{ AY  K_{V_{k_1}}(\bX, \bx)K_{V_{k_2}}(\bX, \bx)}{\left( f(\bX) \right)^2} \right].
\end{equation*}
Then
\begin{align*}
    \E_\bbP(\hat{\psi}_{k_1, k_2}^{\mathrm{INT}}) & = \E_\bbP \left( A Y \right) - \int \Pi(p|V_{k_1})(\bx) \Pi(b|V_{k_2})(\bx) f(\bx) d\bx \\
    & \quad - \int r_{n, k_1, k_2}(\bx) f(\bx) d\bx + O(n^{-1}).
\end{align*}
It then follows from the work in the double sample splitting case that
\begin{align*}
    \E_\bbP\left(\hat{\psi}_{k_1, k_2}^{\mathrm{INT}} - \psi(\bbP)\right) & = -\int \Pi(p|V_{k_1}^{\perp})(\bx) \Pi(b|V_{k_2}^{\perp})(\bx) f(\bx) d\bx + \int  b(\bx) f(\bx) \Pi(p|V_{k_1}^{\perp})(\bx)   d\bx \\
    & \qquad + \int p(\bx)f(\bx) \Pi(b|V_{k_2}^{\perp})(\bx)  d\bx  - \int r_{n, k_1, k_2}(\bx) f(\bx) d\bx + O(n^{-1}).
\end{align*}
Observe that the bias of the single sample split estimator is the bias of the double sample split estimator plus $-\int r_{n, k_1, k_2}(\bx) f(\bx) d\bx + O(n^{-1})$. 

It then suffices to show that $\left| \int r_{n, k_1, k_2}(\bx) f(\bx) d\bx \right| \lesssim \frac{k_1 \wedge k_2}{n}$. Observe that
\begin{align*}
    \left| \int r_{n, k_1, k_2}(\bx) f(\bx) d\bx \right| & \lesssim \int \left| r_{n, k_1, k_2}(\bx) \right| d\bx \\
    & \leq \frac{1}{n} \int \E_{\bbP,1}\left[\left| \frac{ AY  K_{V_{k_1}}(\bX, \bx)K_{V_{k_2}}(\bX, \bx)}{\left( f(\bX) \right)^2} \right|\right] d\bx \\
    & \lesssim \frac{1}{n} \int \E_{\bbP,1}\left[\left| K_{V_{k_1}}(\bX, \bx)K_{V_{k_2}}(\bX, \bx) \right|\right] d\bx \\
    & \lesssim \frac{k_1 \wedge k_2}{n} \quad (\text{by Lemma \ref{lem: exp kernel single}}).
\end{align*}
Therefore, we can bound the bias by
\begin{equation*}
    \left|\E_\bbP\left(\hat{\psi}_{k_1, k_2}^{\mathrm{INT}} - \psi(\bbP)\right)\right| \lesssim (k_1 \wedge k_2)^{-(\alpha + \beta)/d} + \frac{k_1 \wedge k_2}{n}.
\end{equation*}

\noindent \textbf{Bounding the variance:}

We have that
\begin{align*}
    \Var_\bbP(\hat{\psi}_{k_1, k_2}^{\mathrm{INT}}) & = \Var_\bbP\left( \frac{1}{n}\sum_{i \in \D_2} A_i Y_i\right) + \Var_\bbP\left( \int \hat{p}^{(1)}_{k_1}(\bx)\hat{b}^{(1)}_{k_2}(\bx) f(\bx) d\bx  \right) \\
    & = O\left(\frac{1}{n}\right) + \Var_\bbP\left( \int \hat{p}^{(1)}_{k_1}(\bx)\hat{b}^{(1)}_{k_2}(\bx) f(\bx) d\bx  \right).
\end{align*}
We follow a similar approach used in the double sample splitting case to bound the second term. We have that
\begin{align}
    & \Var_\bbP\left( \int \hat{p}^{(1)}_{k_1}(\bx)\hat{b}^{(1)}_{k_2}(\bx) f(\bx) d\bx  \right) \nonumber \\
    & \quad = \E_\bbP \left[ \left( \int \hat{p}^{(1)}_{k_1}(\bx)\hat{b}^{(1)}_{k_2}(\bx) f(\bx) d\bx  \right)^2\right] - \left( \E_\bbP \left[  \int \hat{p}^{(1)}_{k_1}(\bx)\hat{b}^{(1)}_{k_2}(\bx) f(\bx) d\bx  \right] \right)^2 \nonumber \\
    & \quad  = \E_\bbP \left[  \iint \hat{p}^{(1)}_{k_1}(\bx)\hat{p}^{(1)}_{k_1}(\by)\hat{b}^{(1)}_{k_2}(\bx)\hat{b}^{(1)}_{k_2}(\by) f(\bx)f(\by) d\bx  d\by \right] \nonumber \\
    & \quad  \quad - \left(   \int \E_\bbP \left[\hat{p}^{(1)}_{k_1}(\bx)\hat{b}^{(1)}_{k_2}(\bx) \right] f(\bx) d\bx   \right)^2 \nonumber\\
    & \quad  =   \iint \E_\bbP \left[\hat{p}^{(1)}_{k_1}(\bx)\hat{p}^{(1)}_{k_1}(\by)\hat{b}^{(1)}_{k_2}(\bx)\hat{b}^{(1)}_{k_2}(\by) \right] f(\bx)f(\by) d\bx  d\by  \nonumber \\
    & \quad  \quad -   \iint \E_\bbP \left[\hat{p}^{(1)}_{k_1}(\bx)\hat{b}^{(1)}_{k_2}(\bx) \right] \E_\bbP \left[\hat{p}^{(1)}_{k_1}(\by)\hat{b}^{(1)}_{k_2}(\by) \right] f(\bx)f(\by) d\bx d\by  \nonumber \\
    & \quad  \lesssim   \iint \left| \begin{array}{l} \E_\bbP \left[\hat{p}^{(1)}_{k_1}(\bx)\hat{p}^{(1)}_{k_1}(\by)\hat{b}^{(1)}_{k_2}(\bx)\hat{b}^{(1)}_{k_2}(\by) \right] \\ \quad  - \E_\bbP \left[\hat{p}^{(1)}_{k_1}(\bx)\hat{b}^{(1)}_{k_2}(\bx) \right] \E_\bbP \left[\hat{p}^{(1)}_{k_1}(\by)\hat{b}^{(1)}_{k_2}(\by) \right] \end{array} \right| d\bx  d\by \label{eq: var bound ss int}.
\end{align}
To simplify the integrand, we first focus on $\E_\bbP \left[\hat{p}^{(1)}_{k_1}(\bx)\hat{b}^{(1)}_{k_2}(\bx) \right] \E_\bbP \left[\hat{p}^{(1)}_{k_1}(\by)\hat{b}^{(1)}_{k_2}(\by) \right]$. It follows from the proof of Lemma \ref{lem: exp pxpy} that
\begin{align*}
    \E_\bbP \left[\hat{p}^{(1)}_{k_1}(\bx)\hat{b}^{(1)}_{k_2}(\bx) \right] & = \left(1 - \frac{1}{n}\right) \Pi(p|V_{k_1})(\bx) \Pi(b|V_{k_2})(\bx) \\
    & \quad + \frac{1}{n} \E_{\bbP,1}\left[ \frac{AY  K_{V_{k_1}}(\bX, \bx)K_{V_{k_2}}(\bX, \bx)}{\left( f(\bX) \right)^2} \right]  
\end{align*}
and so
\begin{align*}
    & \E_\bbP \left[\hat{p}^{(1)}_{k_1}(\bx)\hat{b}^{(1)}_{k_2}(\bx) \right] \E_\bbP \left[\hat{p}^{(1)}_{k_1}(\by)\hat{b}^{(1)}_{k_2}(\by) \right] \\
    & \quad = \Pi(p|V_{k_1})(\bx)\Pi(p|V_{k_1})(\by) \Pi(b|V_{k_2})(\bx)\Pi(b|V_{k_2})(\by) + O\left( \frac{1}{n} \right) \\
    & \quad \quad + \frac{1}{n}\left(1 - \frac{1}{n}\right) \Pi(p|V_{k_1})(\bx) \Pi(b|V_{k_2})(\bx) \E_{\bbP,1}\left[ \frac{AY  K_{V_{k_1}}(\bX, \by)K_{V_{k_2}}(\bX, \by)}{\left( f(\bX) \right)^2} \right]\\
    & \quad \quad + \frac{1}{n}\left(1 - \frac{1}{n}\right) \Pi(p|V_{k_1})(\by) \Pi(b|V_{k_2})(\by) \E_{\bbP,1}\left[ \frac{AY  K_{V_{k_1}}(\bX, \bx)K_{V_{k_2}}(\bX, \bx)}{\left( f(\bX) \right)^2} \right] \\
    & \quad \quad + \frac{1}{n^2} \E_{\bbP,1}\left[ \frac{AY  K_{V_{k_1}}(\bX, \bx)K_{V_{k_2}}(\bX, \bx)}{\left( f(\bX) \right)^2} \right]\E_{\bbP,1}\left[ \frac{AY  K_{V_{k_1}}(\bX, \by)K_{V_{k_2}}(\bX, \by)}{\left( f(\bX) \right)^2} \right].
\end{align*}
Next, we focus on $\E_\bbP \left[\hat{p}^{(1)}_{k_1}(\bx)\hat{p}^{(1)}_{k_1}(\by)\hat{b}^{(1)}_{k_2}(\bx)\hat{b}^{(1)}_{k_2}(\by) \right]$. Following a similar approach as in Lemma \ref{lem: exp pxpy},
\begin{align*}
    & \E_\bbP \left[\hat{p}^{(1)}_{k_1}(\bx)\hat{p}^{(1)}_{k_1}(\by)\hat{b}^{(1)}_{k_2}(\bx)\hat{b}^{(1)}_{k_2}(\by) \right] \\
    & \quad = \underbrace{\frac{1}{n^4} \sum_{i_1, i_2, j_1, j_2 \in \D_1} \E_\bbP\left[ \frac{A_{i_1}A_{i_2}Y_{j_1}Y_{j_2}\left(\begin{array}{c} K_{V_{k_1}}(\bX_{i_1},\bx)K_{V_{k_1}}(\bX_{i_2},\by) \\ \times K_{V_{k_2}}(\bX_{j_1},\bx)K_{V_{k_2}}(\bX_{j_2},\by) \end{array} \right)}{f(\bX_{i_1})f(\bX_{i_2})f(\bX_{j_1})f(\bX_{j_2})}  \right]}_{(*)}.
\end{align*}
We evaluate $(*)$ by breaking it down into the following cases.
\begin{itemize}
    \item \emph{All distinct} ($i_1, i_2, j_1, j_2$ all distinct): There are $n(n-1)(n-2)(n-3)$ such cases in $(*)$, each of which evaluates to   
    \begin{equation*}
        \Pi(p|V_{k_1})(\bx)\Pi(p|V_{k_1})(\by) \Pi(b|V_{k_2})(\bx)\Pi(b|V_{k_2})(\by).
    \end{equation*}
    \item \emph{All equal} ($i_1 = i_2 = j_1 = j_2)$: There are $n$ such cases in $(*)$, each of which is bounded above by
    \begin{equation*}
        c \E_{\bbP}\left[ |K_{V_{k_1}}(\bX,\bx)K_{V_{k_2}}(\bX,\bx)K_{V_{k_1}}(\bX,\by)K_{V_{k_2}}(\bX,\by)| \right].
    \end{equation*}

    \item \emph{Two distinct pairs:} 
    \begin{itemize}
        \item $i_1 = j_1$, $i_2 = j_2$ ($i_1 \neq i_2$): There are $n(n-1)$ such cases in $(*)$, each of which evaluates to
    \begin{equation*}
        \E_{\bbP}\left[\frac{AY K_{V_{k_1}}(\bX,\bx)K_{V_{k_2}}(\bX,\bx)}{f(\bX)^2}\right] \E_{\bbP}\left[\frac{AY K_{V_{k_1}}(\bX,\by)K_{V_{k_2}}(\bX,\by)}{f(\bX)^2} \right].
    \end{equation*}
    \item $i_1 = i_2$, $j_1 = j_2$ ($i_1 \neq j_1$): There are $n(n-1)$ such cases in $(*)$, each of which is bounded above by
    \begin{equation*}
        c \E_{\bbP}\left[ |K_{V_{k_1}}(\bX,\bx)K_{V_{k_1}}(\bX,\by)|\right] \E_{\bbP}\left[ |K_{V_{k_2}}(\bX,\bx)K_{V_{k_2}}(\bX,\by)|\right]. 
    \end{equation*}
    \item $i_1 = j_2$, $i_2 = j_1$ ($i_1 \neq i_2$): There are $n(n-1)$ such cases in $(*)$, each of which is bounded above by
    \begin{equation*}
        c \E_{\bbP}\left[ |K_{V_{k_1}}(\bX,\bx)K_{V_{k_2}}(\bX,\by)|\right]^2.
    \end{equation*}
    \end{itemize}

    \item \emph{One pair:}
    \begin{itemize}
        \item $i_1 = j_1$, $i_2 \neq j_2$ ($i_1 \neq i_2$): There are $n(n-1)(n-2)$ such cases in $(*)$, each of which evaluates to
    \begin{equation*}
        \E_{\bbP}\left[\frac{AY K_{V_{k_1}}(\bX,\bx)K_{V_{k_2}}(\bX,\bx)}{f(\bX)^2}\right] \Pi(p|V_{k_1})(\by)\Pi(b|V_{k_2})(\by).
    \end{equation*}

    \item $i_1 \neq j_1$, $i_2 = j_2$ ($i_1 \neq i_2$): There are $n(n-1)(n-2)$ such cases in $(*)$, each of which evaluates to
    \begin{equation*}
        \E_{\bbP}\left[\frac{AY K_{V_{k_1}}(\bX,\by)K_{V_{k_2}}(\bX,\by)}{f(\bX)^2}\right] \Pi(p|V_{k_1})(\bx)\Pi(b|V_{k_2})(\bx).
    \end{equation*}

    \item By symmetry, there are $n(n-1)(n-2)$ cases in $(*)$ bounded by each of the following
    \begin{align*}
        & c \E_{\bbP}\left[ |K_{V_{k_1}}(\bX,\bx)K_{V_{k_1}}(\bX,\by)|\right] \\
        & c \E_{\bbP}\left[ |K_{V_{k_1}}(\bX,\bx)K_{V_{k_2}}(\bX,\by)|\right] \\
        & c \E_{\bbP}\left[ |K_{V_{k_1}}(\bX,\by)K_{V_{k_2}}(\bX,\bx)|\right] \\
        & c \E_{\bbP}\left[ |K_{V_{k_2}}(\bX,\bx)K_{V_{k_2}}(\bX,\by)|\right].
    \end{align*}
    \end{itemize}

    \item \emph{One triple:}
    \begin{itemize}
        \item $i_1 = i_2 = j_1$ ($j_2 \neq i_1)$; There are $n(n-1)$ such cases in $(*)$, each of which is bounded above by
        \begin{equation*}
            c \E_{\bbP}\left[ |K_{V_{k_1}}(\bX,\bx)K_{V_{k_2}}(\bX,\bx)K_{V_{k_1}}(\bX,\by)|\right].
        \end{equation*}
        \item By symmetry, there are $n(n-1)$ cases in $(*)$ bounded by each of the following
        \begin{align*}
            & c \E_{\bbP}\left[ |K_{V_{k_1}}(\bX,\bx)K_{V_{k_2}}(\bX,\bx)K_{V_{k_2}}(\bX,\by)|\right] \\
            & c \E_{\bbP}\left[ |K_{V_{k_1}}(\bX,\bx)K_{V_{k_1}}(\bX,\by)K_{V_{k_2}}(\bX,\by)|\right] \\
            & c \E_{\bbP}\left[ |K_{V_{k_2}}(\bX,\bx)K_{V_{k_1}}(\bX,\by)K_{V_{k_2}}(\bX,\by)|\right]. 
        \end{align*}
    \end{itemize}
\end{itemize}

Upon plugging in the expressions for $\E_\bbP \left[\hat{p}^{(1)}_{k_1}(\bx)\hat{b}^{(1)}_{k_2}(\bx) \right] \E_\bbP \left[\hat{p}^{(1)}_{k_1}(\by)\hat{b}^{(1)}_{k_2}(\by) \right]$ and \newline \noindent $\E_\bbP \left[\hat{p}^{(1)}_{k_1}(\bx)\hat{p}^{(1)}_{k_1}(\by)\hat{b}^{(1)}_{k_2}(\bx)\hat{b}^{(1)}_{k_2}(\by) \right]$ in (\ref{eq: var bound ss int}), cancelling like terms, and applying the triangle inequality, we have that
\begin{align*}
    & \Var_\bbP\left( \int \hat{p}^{(1)}_{k_1}(\bx)\hat{b}^{(1)}_{k_2}(\bx) f(\bx) d\bx  \right) \\
    & \quad \lesssim \frac{1}{n}   + \frac{1}{n} \iint \E_{\bbP}\left[ | K_{V_{k_1}}(\bX,\bx)K_{V_{k_1}}(\bX,\by) | \right] d\bx d\by\\ 
    & \quad \quad +  \frac{1}{n} \iint \E_{\bbP}\left[ | K_{V_{k_2}}(\bX,\bx)K_{V_{k_2}}(\bX,\by) | \right] d\bx d\by\\ 
    & \quad \quad +  \frac{1}{n} \iint \E_{\bbP}\left[ | K_{V_{k_1}}(\bX,\bx)K_{V_{k_2}}(\bX,\by) | \right] d\bx d\by\\ 
    & \quad \quad + \frac{1}{n^2} \iint \E_{\bbP}\left[ |K_{V_{k_1}}(\bX,\bx)K_{V_{k_1}}(\bX,\by)|\right] \E_{\bbP}\left[ |K_{V_{k_2}}(\bX,\bx)K_{V_{k_2}}(\bX,\by)|\right]  d\bx d\by\\
    & \quad \quad + \frac{1}{n^2} \iint \E_{\bbP}\left[ | K_{V_{k_1}}(\bX,\bx)K_{V_{k_2}}(\bX,\by)|\right]^2  d\bx d\by \\ 
    & \quad \quad + \frac{1}{n^2} \iint \E_{\bbP}\left[ |K_{V_{k_1}}(\bX,\bx)K_{V_{k_2}}(\bX,\bx)K_{V_{k_1}}(\bX,\by) |\right] d\bx d\by\\ 
    & \quad \quad + \frac{1}{n^2} \iint \E_{\bbP}\left[ |K_{V_{k_1}}(\bX,\bx)K_{V_{k_2}}(\bX,\bx)K_{V_{k_2}}(\bX,\by)|\right] d\bx d\by \\
    & \quad \quad + \frac{1}{n^3} \iint \E_{\bbP}\left[ |K_{V_{k_1}}(\bX,\bx)K_{V_{k_2}}(\bX,\bx)K_{V_{k_1}}(\bX,\by)K_{V_{k_2}}(\bX,\by)| \right] d\bx d\by \\
    & \quad \quad + \frac{1}{n^3} \iint \E_{\bbP}\left[ |K_{V_{k_1}}(\bX,\bx)K_{V_{k_2}}(\bX,\bx)|\right] \E_{\bbP}\left[|K_{V_{k_1}}(\bX,\by)K_{V_{k_2}}(\bX,\by) | \right] d\bx d\by.
\end{align*}
By Lemma \ref{lem: exp kernel},
\begin{align*}
    \frac{1}{n} \iint \E_{\bbP}\left[ | K_{V_{k_1}}(\bX,\bx)K_{V_{k_1}}(\bX,\by) | \right] d\bx d\by & \lesssim \frac{1}{n} \\
    \frac{1}{n} \iint \E_{\bbP}\left[ | K_{V_{k_2}}(\bX,\bx)K_{V_{k_2}}(\bX,\by) | \right] d\bx d\by & \lesssim \frac{1}{n} \\
    \frac{1}{n} \iint \E_{\bbP}\left[ | K_{V_{k_1}}(\bX,\bx)K_{V_{k_2}}(\bX,\by) | \right] d\bx d\by & \lesssim \frac{1}{n} \\
    \frac{1}{n^2} \iint \E_{\bbP} \left[ \left| K_{V_{k_1}}(\bX, \bx)K_{V_{k_1}}(\bX, \by) \right| \right] \E_{\bbP} \left[ \left| K_{V_{k_2}}(\bX, \bx)K_{V_{k_2}}(\bX, \by) \right| \right]  d\bx d\by & \lesssim \frac{k_1 \wedge k_2}{n^2} \\
    \frac{1}{n^2} \iint \E_{\bbP}\left[ | K_{V_{k_1}}(\bX,\bx)K_{V_{k_2}}(\bX,\by)|\right]^2  d\bx d\by & \lesssim \frac{k_1 \wedge k_2}{n^2} \\
    \frac{1}{n^2} \iint \E_{\bbP}\left[ |K_{V_{k_1}}(\bX,\bx)K_{V_{k_2}}(\bX,\bx)K_{V_{k_1}}(\bX,\by) |\right] d\bx d\by & \lesssim \frac{k_1 \wedge k_2}{n^2}\\ 
    \frac{1}{n^2} \iint \E_{\bbP}\left[ |K_{V_{k_1}}(\bX,\bx)K_{V_{k_2}}(\bX,\bx)K_{V_{k_2}}(\bX,\by)|\right] d\bx d\by & \lesssim \frac{k_1 \wedge k_2}{n^2} \\
    \frac{1}{n^3} \iint \E_{\bbP}\left[ |K_{V_{k_1}}(\bX,\bx)K_{V_{k_2}}(\bX,\bx)K_{V_{k_1}}(\bX,\by)K_{V_{k_2}}(\bX,\by) |\right] d\bx d\by & \lesssim \frac{(k_1 \wedge k_2)^2}{n^3}
\end{align*}
and by Lemma \ref{lem: exp kernel single},
\begin{equation*}
    \frac{1}{n^3} \iint \E_{\bbP}\left[ |K_{V_{k_1}}(\bX,\bx)K_{V_{k_2}}(\bX,\bx)|\right] \E_{\bbP}\left[|K_{V_{k_1}}(\bX,\by)K_{V_{k_2}}(\bX,\by) | \right] d\bx d\by \lesssim \frac{(k_1 \wedge k_2)^2}{n^3}.
\end{equation*}
Thus,
\begin{equation*}
    \Var_\bbP\left( \int \hat{p}^{(1)}_{k_1}(\bx)\hat{b}^{(1)}_{k_2}(\bx) f(\bx) d\bx  \right) \lesssim \frac{1}{n} +  \frac{k_1 \wedge k_2}{n^2} + \frac{(k_1 \wedge k_2)^2}{n^3}.
\end{equation*}

\subsubsection{Proof of the lower bound}

Let $\bbP \in \cP_{(\alpha,\beta)}$ and $\bX \sim \text{Uniform}([0,1]^d)$. Recall from the proof of the upper bound that
\begin{align*}
    \E_\bbP\left(\hat{\psi}_{k_1, k_2}^{\mathrm{INT}} - \psi(\bbP)\right) & = -\int \Pi(p|V_{k_1}^{\perp})(\bx) \Pi(b|V_{k_2}^{\perp})(\bx) d\bx + \int  b(\bx) \Pi(p|V_{k_1}^{\perp})(\bx)   d\bx \\
    & \qquad + \int p(\bx) \Pi(b|V_{k_2}^{\perp})(\bx)  d\bx  - \int r_{n, k_1, k_2}(\bx) d\bx + O(n^{-1})
\end{align*}
where $r_{n, k_1, k_2}$ is as defined in the proof of the upper bound. That is, the bias of the single sample split estimator is the bias of the double sample split estimator plus $-\int r_{n, k_1, k_2}(\bx) d\bx + O(n^{-1})$.

Suppose without loss of generality that $k_1 \leq k_2$ for $n$ sufficiently large. It follows from the work in the double sample splitting case that
\begin{align*}
    \int \Pi(p|V_{k_1}^{\perp})(\bx) \Pi(b|V_{k_1}^{\perp})(\bx) d\bx & = -\int \Pi(p|V_{k_1}^{\perp})(\bx) \Pi(b|V_{k_2}^{\perp})(\bx) d\bx + \int  b(\bx) \Pi(p|V_{k_1}^{\perp})(\bx)   d\bx  \\
    & \quad + \int p(\bx) \Pi(b|V_{k_2}^{\perp})(\bx) d\bx.
\end{align*}
Therefore,
\begin{equation*}
    \E_\bbP\left(\hat{\psi}_{k_1, k_2}^{\mathrm{INT}} - \psi(\bbP)\right)   =  \int \Pi(p|V_{k_1}^{\perp})(\bx) \Pi(b|V_{k_1}^{\perp})(\bx) d\bx - \int r_{n, k_1, k_2}(\bx) d\bx + O(n^{-1}).
\end{equation*}
Choosing $p$ and $b$ as given in (\ref{eq: p lb}) and (\ref{eq: b lb}), recall that
\begin{equation*}
    \left| \int \Pi(p|V_{k_1}^{\perp})(\bx) \Pi(b|V_{k_1}^{\perp})(\bx) d\bx \right| \asymp k_1^{-\frac{\alpha + \beta}{d}}. 
\end{equation*}
Further, let $\E_{\bbP,1} [AY | \bX] = c \neq 0$. It then follows from Lemma \ref{lem: exp kernel single} that 
\begin{align*}
    \left| \int r_{n, k_1, k_2}(\bx) d\bx \right| & =  \frac{1}{n} \left| \int  \E_{\bbP,1}\left[ AY  K_{V_{k_1}}(\bX, \bx)K_{V_{k_2}}(\bX, \bx) \right] d\bx \right| \\
    & =  \frac{1}{n} \left| \int  \E_{\bbP,1}\left[ \E_{\bbP,1} [AY | \bX]  K_{V_{k_1}}(\bX, \bx)K_{V_{k_2}}(\bX, \bx) \right] d\bx \right| \\
    & \asymp  \frac{1}{n} \left| \int  \E_{\bbP,1}\left[  K_{V_{k_1}}(\bX, \bx)K_{V_{k_2}}(\bX, \bx) \right] d\bx \right| \\
    & \asymp \frac{k_1}{n}.
\end{align*}
This completes the proof of the lower bound provided that $k_1^{-\frac{\alpha + \beta}{d}}$ and $\frac{k_1}{n}$ are not of the same order. When $k_1$ is such that $k_1^{-\frac{\alpha + \beta}{d}} \asymp \frac{k_1}{n}$, we can choose $\epsilon$ (appearing in (\ref{eq: p lb}) and (\ref{eq: b lb}) which define $p$ and $b$) sufficiently small so that
\begin{equation*}
    \left| \int \Pi(p|V_{k_1}^{\perp})(\bx) \Pi(b|V_{k_1}^{\perp})(\bx) d\bx - \int r_{n, k_1, k_2}(\bx) d\bx \right| \gtrsim \frac{k_1}{n}
\end{equation*}
which completes the proof.

\subsection{Monte Carlo-based plug-in estimator}
Since the bias of $\hat{\psi}_{k_1, k_2}^{\mathrm{MC}}$ is the same as that of $\hat{\psi}_{k_1, k_2}^{\mathrm{INT}}$, it suffices to establish the upper bound on the variance of $\hat{\psi}_{k_1, k_2}^{\mathrm{MC}}$.

Similar to the double sample splitting case, we consider the decomposition 
\begin{align*}
    \Var_\bbP(\hat{\psi}_{k_1, k_2}^{\mathrm{MC}}) & = \Var_\bbP\left( \frac{1}{n}\sum_{i \in \D_2} A_i Y_i\right) + \Var_\bbP\left( \frac{1}{n}\sum_{i \in \D_3} \hat{p}^{(1)}_{k_1}(\bX_i)\hat{b}^{(1)}_{k_2}(\bX_i)  \right) \\
    & = \frac{1}{n}\Var_\bbP(AY) +  \Var_{\bbP,1} \left[ \E_{\bbP,3}(\hat{p}^{(1)}_{k_1}(\bX)\hat{b}^{(1)}_{k_2}(\bX))  \right]  \\
    & \quad +  \E_{\bbP,1} \left[ \Var_{\bbP,3}\left( \frac{1}{n}\sum_{i \in \D_3} \hat{p}^{(1)}_{k_1}(\bX_i)\hat{b}^{(1)}_{k_2}(\bX_i)  \right)  \right].
\end{align*}
The second term can be bounded by
\begin{align*}
    \Var_{\bbP,1} \left[ \E_{\bbP,3}(\hat{p}^{(1)}_{k_1}(\bX)\hat{b}^{(1)}_{k_2}(\bX))  \right] & = \Var_{\bbP,1} \left[ \int \hat{p}^{(1)}_{k_1}(\bx) \hat{b}^{(1)}_{k_2}(\bx)f(\bx) d\bx \right]  \\
    & \lesssim \frac{1}{n} + \frac{k_1 \wedge k_2}{n^2} + \frac{(k_1 \wedge k_2)^2}{n^3}
\end{align*}
where the inequality follows the bound of the variance of the single sample-split integral-based plug-in estimator.

To bound the third term, we have that
\begin{align*}
    \E_{\bbP,1} \left[ \Var_{\bbP,3}\left( \frac{1}{n}\sum_{i \in \D_3} \hat{p}^{(1)}_{k_1}(\bX_i)\hat{b}^{(1)}_{k_2}(\bX_i)  \right)  \right] & = \frac{1}{n} \E_{\bbP,1} \left[ \Var_{\bbP,3}\left(  \hat{p}^{(1)}_{k_1}(\bX)\hat{b}^{(1)}_{k_2}(\bX)  \right)  \right] \\
    & \leq \frac{1}{n} \E_{\bbP,1} \left[ \E_{\bbP,3}\left(  \hat{p}^{(1)}_{k_1}(\bX)^2\hat{b}^{(1)}_{k_2}(\bX)^2  \right)  \right] \\
    & = \frac{1}{n} \int \E_{\bbP,1} \left[   \hat{p}^{(1)}_{k_1}(\bx)^2\hat{b}^{(1)}_{k_2}(\bx)^2    \right] f(\bx) d\bx \\
    & \lesssim \frac{1}{n} \int \E_{\bbP,1} \left[   \hat{p}^{(1)}_{k_1}(\bx)^2\hat{b}^{(1)}_{k_2}(\bx)^2    \right] d\bx 
\end{align*}
where the last inequality follows from $f$ being bounded.

By the decomposition of $\E_{\bbP,1} \left[   \hat{p}^{(1)}_{k_1}(\bx)^2\hat{b}^{(1)}_{k_2}(\bx)^2    \right]$ performed in the variance bound of the single sample-split integral-based plug-in estimator,
\begin{align*}
    \int \E_{\bbP,1} \left[   \hat{p}^{(1)}_{k_1}(\bx)^2\hat{b}^{(1)}_{k_2}(\bx)^2    \right] d\bx & \lesssim 1 + \frac{1}{n} \int \E_{\bbP}\left[ K_{V_{k_1}}(\bX,\bx)^2 \right] d\bx \\ 
    & \quad \quad +  \frac{1}{n} \int \E_{\bbP}\left[ K_{V_{k_2}}(\bX,\bx)^2 \right] d\bx \\ 
    & \quad \quad +  \frac{1}{n} \int \E_{\bbP}\left[ |K_{V_{k_1}}(\bX,\bx)K_{V_{k_2}}(\bX,\bx)| \right] d\bx \\ 
    & \quad \quad + \frac{1}{n^2} \int \E_{\bbP}\left[ |K_{V_{k_1}}(\bX,\bx)K_{V_{k_2}}(\bX,\bx)|\right]^2  d\bx \\ 
    & \quad \quad + \frac{1}{n^2} \int \E_{\bbP}\left[ K_{V_{k_1}}(\bX,\bx)^2\right] \E_{\bbP}\left[ K_{V_{k_2}}(\bX,\bx)^2\right]  d\bx \\ 
    & \quad \quad + \frac{1}{n^2} \int \E_{\bbP}\left[ | K_{V_{k_1}}(\bX,\bx)^2K_{V_{k_2}}(\bX,\bx)| \right] d\bx \\
    & \quad \quad + \frac{1}{n^2} \int \E_{\bbP}\left[ | K_{V_{k_1}}(\bX,\bx)K_{V_{k_2}}(\bX,\bx)^2| \right] d\bx \\
    & \quad \quad + \frac{1}{n^3} \int \E_{\bbP}\left[ K_{V_{k_1}}(\bX,\bx)^2 K_{V_{k_2}}(\bX,\bx)^2 \right] d\bx.
\end{align*}
Then, by Lemma \ref{lem: exp kernel single},
\begin{align*}
    \frac{1}{n} \int \E_{\bbP}\left[ K_{V_{k_1}}(\bX,\bx)^2 \right] d\bx & \lesssim \frac{k_1}{n} \\
    \frac{1}{n} \int \E_{\bbP}\left[ K_{V_{k_2}}(\bX,\bx)^2 \right] d\bx & \lesssim \frac{k_2}{n} \\
    \frac{1}{n} \int \E_{\bbP}\left[ | K_{V_{k_1}}(\bX,\bx)K_{V_{k_2}}(\bX,\bx)| \right] d\bx & \lesssim \frac{k_1 \wedge k_2}{n} \\
    \frac{1}{n^2} \int \E_{\bbP}\left[ |K_{V_{k_1}}(\bX,\bx)K_{V_{k_2}}(\bX,\bx)|\right]^2  d\bx & \lesssim \frac{(k_1 \wedge k_2)^2}{n^2} \\
    \frac{1}{n^2} \int \E_{\bbP}\left[ K_{V_{k_1}}(\bX,\bx)^2\right] \E_{\bbP}\left[ K_{V_{k_2}}(\bX,\bx)^2\right]  d\bx & \lesssim \frac{k_1k_2}{n^2} \\
    \frac{1}{n^2} \int \E_{\bbP}\left[ | K_{V_{k_1}}(\bX,\bx)K_{V_{k_2}}(\bX,\bx)^2| \right] d\bx & \lesssim \frac{k_1k_2}{n^2} \\
        \frac{1}{n^2} \int \E_{\bbP}\left[ | K_{V_{k_1}}(\bX,\bx)^2K_{V_{k_2}}(\bX,\bx)| \right] d\bx & \lesssim \frac{k_1k_2}{n^2} \\
    \frac{1}{n^3} \int \E_{\bbP}\left[ K_{V_{k_1}}(\bX,\bx)^2 K_{V_{k_2}}(\bX,\bx)^2 \right] d\bx & \lesssim \frac{(k_1 \wedge k_2)^2 (k_1 \vee k_2)}{n^3}.
\end{align*}
Therefore,
\begin{equation*}
    \int \E_{\bbP,1} \left[   \hat{p}^{(1)}_{k_1}(\bx)^2\hat{b}^{(1)}_{k_2}(\bx)^2    \right] d\bx \lesssim 1 + \frac{k_1 \vee k_2}{n} + \frac{k_1k_2}{n^2} + \frac{(k_1 \wedge k_2)^2 (k_1 \vee k_2)}{n^3}.
\end{equation*}
In summary, the third term is bounded by
\begin{align*}
    \E_{\bbP,1} \left[ \Var_{\bbP,3}\left( \frac{1}{n}\sum_{i \in \D_3} \hat{p}^{(1)}_{k_1}(\bX_i)\hat{b}^{(1)}_{k_2}(\bX_i)  \right)  \right] & \lesssim \frac{1}{n} \int \E_{\bbP,1} \left[   \hat{p}^{(1)}_{k_1}(\bx)^2\hat{b}^{(1)}_{k_2}(\bx)^2    \right] d\bx  \\
    & \lesssim \frac{1}{n} + \frac{k_1 \vee k_2}{n^2} + \frac{k_1k_2}{n^3} + \frac{(k_1 \wedge k_2)^2 (k_1 \vee k_2)}{n^4}.
\end{align*}

\subsection{First-order bias-corrected estimator}

\subsubsection{Proof of the upper bound}

Let $\bbP \in \cP_{(\alpha,\beta)}$ be arbitrary. \\

\noindent \textbf{Bounding the bias:}

Following the same steps in the double sample splitting case, we have that 
\begin{equation*}
    \E_\bbP(\hat{\psi}^{\mathrm{IF}}_{k_1, k_2}) = \psi(\bbP) + \E_{\bbP, 1} \left[ \int (p(\bx) - \hat{p}^{(1)}_{k_1}(\bx))(b(\bx) - \hat{b}^{(1)}_{k_2}(\bx)) f(\bx) d\bx \right].
\end{equation*}
It follows from the proof of Lemma \ref{lem: exp pxpy} that we can express $\E_{\bbP}[\hat{p}^{(1)}_{k_1}(\bx)\hat{b}^{(1)}_{k_2}(\bx)]$ in terms of \newline \noindent $\Pi(p|V_{k_1})(\bx) \Pi(b|V_{k_2})(\bx)$ plus a remainder. Specifically, 
\begin{equation*}
    \E_{\bbP}[\hat{p}^{(1)}_{k_1}(\bx)\hat{b}^{(1)}_{k_2}(\bx)] = \Pi(p|V_{k_1})(\bx) \Pi(b|V_{k_2})(\bx) + O(n^{-1}) + r_{n, k_1, k_2}(\bx)
\end{equation*}
where
\begin{equation*}
    r_{n, k_1, k_2}(\bx) = \frac{1}{n}  \E_{\bbP,1}\left[ \frac{ AY  K_{V_{k_1}}(\bX, \bx)K_{V_{k_2}}(\bX, \bx)}{\left( f(\bX) \right)^2} \right].
\end{equation*}
That is,
\begin{align*}
    \E_\bbP(\hat{\psi}^{\mathrm{IF}}_{k_1, k_2}) - \psi(\bbP) & = \int \left( p(\bx) - \Pi(p | V_{k_1})(\bx)\right)\left( b(\bx) - \Pi(b | V_{k_2})(\bx)\right) f(\bx) d\bx  \\
    & \qquad + \int r_{n, k_1, k_2}(\bx) f(\bx) d\bx + O(n^{-1}).
\end{align*}
Therefore, the bias of the single sample split estimator is the bias of the double sample split estimator plus $\int r_{n, k_1, k_2}(\bx) f(\bx) d\bx + O(n^{-1})$. 

Suppose without loss of generality that $k_1 \geq k_2$ for $n$ sufficiently large. Recall from the proof of the upper bound of the double sample split estimator that 
\begin{equation*}
    \left| \int \left( p(\bx) - \Pi(p | V_{k_1})(\bx)\right)\left( b(\bx) - \Pi(b | V_{k_2})(\bx)\right) f(\bx) d\bx \right| \lesssim k_1^{-(\alpha + \beta)/ d}.
\end{equation*}
Moreover, recall from derivation of the upper bound of the bias of $\hat{\psi}_{k_1, k_2}^{\mathrm{INT}}$ (single sample splitting case) that
\begin{equation*}
     \left| \int r_{n, k_1, k_2}(\bx) d\bx \right| \lesssim \frac{k_2}{n}.
\end{equation*}
Therefore, we can bound the bias by
\begin{equation*}
    \sup_{\bbP \in \cP_{(\alpha,\beta)}} \left|\E_\bbP\left(\hat{\psi}_{k_1, k_2}^{\mathrm{IF}} - \psi(\bbP)\right)\right| \lesssim (k_1 \vee k_2)^{-\frac{\alpha + \beta}{d}} + \frac{k_1 \wedge k_2}{n}. 
\end{equation*}

\textbf{Bounding the variance:}

We will show that
\begin{align}
    \Var_{\bbP,1} \left[ \E_{\bbP,2}\left( \hat{\psi}^{\mathrm{IF}}_{k_1, k_2}\right)  \right] & \lesssim \frac{1}{n} + \frac{k_1 \wedge k_2}{n^2} + \frac{(k_1 \wedge k_2)^2}{n^3} \label{eq: var pt2 cov if single} \\
    \E_{\bbP,1} \left[ \Var_{\bbP,2} \left( \hat{\psi}^{\mathrm{IF}}_{k_1,k_2}\right)  \right] & \lesssim \frac{1}{n} + \frac{k_1 \vee k_2}{n^2} + \frac{k_1k_2}{n^3} + \frac{(k_1 \wedge k_2)^2 (k_1 \vee k_2)}{n^4}. \label{eq: var pt1 cov if single}
\end{align}

We first show (\ref{eq: var pt2 cov if single}). Following the same steps used in the bias derivation,
\begin{equation*}
    \E_{\bbP,2}\left( \hat{\psi}^{\mathrm{IF}}_{k_1, k_2}\right) = \psi(\bbP) + \int (p(\bx) - \hat{p}^{(1)}_{k_1}(\bx))(b(\bx) - \hat{b}^{(1)}_{k_2}(\bx)) f(\bx) d\bx.
\end{equation*}
Expanding the product in the integrand and noting that $\psi(\bbP)$ and $\int p(\bx) b(\bx) f(\bx) d\bx$ are deterministic, we have that
\begin{align*}
    \Var_{\bbP,1} \left[ \E_{\bbP,2}\left( \hat{\psi}^{\mathrm{IF}}_{k_1, k_2}\right)  \right] & \leq 3 \Var_{\bbP,1}\left[ \int \hat{p}^{(1)}_{k_1}(\bx) b(\bx) f(\bx) d\bx \right] + 3 \Var_{\bbP,1}\left[ \int p(\bx)\hat{b}^{(1)}_{k_2}(\bx) f(\bx) d\bx \right]  \\
    & \quad + 3 \Var_{\bbP,1}\left[ \int \hat{p}^{(1)}_{k_1}(\bx)\hat{b}^{(1)}_{k_2}(\bx) f(\bx) d\bx \right].
\end{align*}
We bound the first term, noting that the second term can be bounded in the same manner upon exchanging the roles of $A$ and $Y$. By definition of $\hat{p}^{(1)}_{k_1}$,
\begin{equation*}
    \int \hat{p}^{(1)}_{k_1}(\bx) b(\bx) f(\bx) d\bx = \frac{1}{n} \sum_{i \in \D_1} \frac{A_i \Pi(bf | V_{k_1})(\bX_i)}{f(\bX_i)}.
\end{equation*}
Since $A$ is a bounded random variable, $\Pi(bf | V_{k_1})$ is bounded, and $f$ is bounded away from zero, it follows that
\begin{equation*}
    \Var_{\bbP,1}\left[ \int \hat{p}^{(1)}_{k_1}(\bx) b(\bx) f(\bx) d\bx \right] \lesssim \frac{1}{n}.
\end{equation*}
Last, recall from the derivation of the variance of $\hat{\psi}_{k_1, k_2}^{\mathrm{INT}}$ (single sample splitting case) that
\begin{equation*}
    \Var_{\bbP,1}\left[ \int \hat{p}^{(1)}_{k_1}(\bx)\hat{b}^{(1)}_{k_2}(\bx) f(\bx) d\bx \right] \lesssim \frac{1}{n} + \frac{k_1 \wedge k_2}{n^2} + \frac{(k_1 \wedge k_2)^2}{n^3},
\end{equation*}
which establishes (\ref{eq: var pt2 cov if single}).

It remains to show (\ref{eq: var pt1 cov if single}). Following the strategy in the double sample splitting case (i.e., in establishing (\ref{eq: var pt1 cov if})), we have that
\begin{align*}
    & \E_{\bbP,1} \left[ \Var_{\bbP,2} \left( \hat{\psi}^{\mathrm{IF}}_{k_1,k_2}\right)  \right] \\
    \quad & \leq \frac{1}{n} \E_{\bbP,1} \left[ \E_{\bbP,2} \left( (A - \hat{p}^{(1)}_{k_1}(\bX))^2(Y - \hat{b}^{(1)}_{k_2}(\bX))^2 \right) \right] \\
    \quad & = \frac{1}{n} \E_{\bbP,1} \left[ \E_{\bbP,2} \left( (A^2 - 2A \hat{p}^{(1)}_{k_1}(\bX) + \hat{p}^{(1)}_{k_1}(\bX)^2)(Y^2 - 2Y\hat{b}^{(1)}_{k_2}(\bX) + \hat{b}^{(1)}_{k_2}(\bX)^2) \right) \right] 
\end{align*}
Since $A$, $Y$, $\bX$ are bounded random variables and $\Pi(p | V_{k_1})$ and $\Pi(b | V_{k_2})$ are bounded, 
\begin{align*}
    \E_{\bbP,1} \left[ \Var_{\bbP,2} \left( \hat{\psi}^{\mathrm{IF}}_{k_1, k_2}\right)  \right] & \lesssim \frac{1}{n} + \frac{1}{n} \E_{\bbP,2} \left( \E_{\bbP,1}\left[\hat{p}^{(1)}_{k_1}(\bX)^2\right] \right) +  \frac{1}{n} \E_{\bbP,2} \left( \E_{\bbP,1}\left[\hat{b}^{(1)}_{k_2}(\bX)^2\right] \right) \\
    & \quad + \frac{1}{n} \E_{\bbP,2} \left(  \left| \E_{\bbP,1}\left[\hat{p}^{(1)}_{k_1}(\bX) \hat{b}^{(1)}_{k_2}(\bX)\right] \right| \right)  \\
    & \quad + \frac{1}{n}  \E_{\bbP,2} \left( \left| \E_{\bbP,1}\left[\hat{p}^{(1)}_{k_1}(\bX)^2 \hat{b}^{(1)}_{k_2}(\bX)\right] \right| \right)  \\
    & \quad + \frac{1}{n}  \E_{\bbP,2} \left( \left|  \E_{\bbP,1}\left[\hat{p}^{(1)}_{k_1}(\bX) \hat{b}^{(1)}_{k_2}(\bX)^2\right] \right| \right)  \\
    & \quad + \frac{1}{n} \E_{\bbP,2} \left(  \E_{\bbP,1}\left[\hat{p}^{(1)}_{k_1}(\bX)^2 \hat{b}^{(1)}_{k_2}(\bX)^2\right] \right).
\end{align*}
Recall from Lemma \ref{lem: nuisance function bounds} that
\begin{align*}
    \E_{\bbP,2} \left( \E_{\bbP,1}\left[\hat{p}^{(1)}_{k_1}(\bX)^2\right] \right) \lesssim 1 + \frac{k_1}{n} \\
    \E_{\bbP,2} \left( \E_{\bbP,1}\left[\hat{b}^{(1)}_{k_2}(\bX)^2\right] \right) \lesssim 1 + \frac{k_2}{n}.
\end{align*}
Following the same arguments used in the derivation of the bias of this estimator, we have that
\begin{equation*}
    \E_{\bbP,2} \left( \left| \E_{\bbP,1}\left[\hat{p}^{(1)}_{k_1}(\bX) \hat{b}^{(1)}_{k_2}(\bX)\right] \right| \right) \lesssim 1 + \frac{k_1 \wedge k_2}{n}.
\end{equation*}
Moreover, recall from proof of the upper bound on the variance of the single sample split Monte Carlo-based plug-in estimator that
\begin{equation*}
    \E_{\bbP,2} \left(  \E_{\bbP,1}\left[\hat{p}^{(1)}_{k_1}(\bX)^2 \hat{b}^{(1)}_{k_2}(\bX)^2\right] \right) \lesssim 1 + \frac{k_1 \vee k_2}{n} + \frac{k_1k_2}{n^2} + \frac{(k_1 \wedge k_2)^2 (k_1 \vee k_2)}{n^3}.
\end{equation*}
The two remaining terms, i.e. $\E_{\bbP,2} \left( \left| \E_{\bbP,1}\left[\hat{p}^{(1)}_{k_1}(\bX)^2 \hat{b}^{(1)}_{k_2}(\bX)\right] \right| \right)$ and \newline $\E_{\bbP,2} \left( \left|  \E_{\bbP,1}\left[\hat{p}^{(1)}_{k_1}(\bX) \hat{b}^{(1)}_{k_2}(\bX)^2\right] \right| \right) $, can be seen as special cases of the last term bounded. In particular, we decompose these terms as performed in the derivation of the variance of $\hat{\psi}_{k_1, k_2}^{\mathrm{INT}}$ (single sample splitting case). That is,
\begin{align*}
    \E_{\bbP,2} \left( \left| \E_{\bbP,1}\left[\hat{p}^{(1)}_{k_1}(\bX)^2 \hat{b}^{(1)}_{k_2}(\bX)\right] \right| \right)  & \lesssim 1 + \frac{1}{n} \int \E_{\bbP}\left[ |K_{V_{k_1}}(\bX,\bx)K_{V_{k_2}}(\bX,\bx)| \right] d\bx \\
     & \quad + \frac{1}{n} \int \E_{\bbP}\left[ K_{V_{k_1}}(\bX,\bx)^2 \right] d\bx \\
     & \quad + \frac{1}{n^2} \int \E_{\bbP}\left[ |K_{V_{k_1}}(\bX,\bx)^2K_{V_{k_2}}(\bX,\bx)| \right] d\bx  \\
     & \lesssim 1 + \frac{k_1 \wedge k_2}{n} + \frac{k_1}{n} + \frac{k_1k_2}{n^2}
\end{align*}
and similarly,
\begin{equation*}
    \E_{\bbP,2} \left( \left|  \E_{\bbP,1}\left[\hat{p}^{(1)}_{k_1}(\bX) \hat{b}^{(1)}_{k_2}(\bX)^2\right] \right| \right)  \lesssim 1 + \frac{k_1 \wedge k_2}{n} + \frac{k_2}{n} + \frac{k_1k_2}{n^2}
\end{equation*}
which completes the proof.

\subsubsection{Proof of the lower bound}

Let $\bbP \in \cP_{(\alpha,\beta)}$ and $\bX \sim \text{Uniform}([0,1]^d)$. Recall from the proof of the upper bound that
\begin{align*}
    \E_\bbP\left(\hat{\psi}_{k_1, k_2}^{\mathrm{IF}} - \psi(\bbP)\right) & = \int \Pi(p | V_{k_1}^{\perp})(\bx)\Pi(b | V_{k_2}^{\perp})(\bx) d\bx  + \int r_{n, k_1,k_2}(\bx) d\bx + O(n^{-1})
\end{align*}
where $r_{n, k_1, k_2}$ is as defined in the proof of the upper bound. Suppose without loss of generality that $k_1 \leq k_2$ for $n$ sufficiently large. Choosing $p$ and $b$ as given in (\ref{eq: p lb}) and (\ref{eq: b lb}), it follows from the proof of the lower bound in the double sample splitting case that 
\begin{align*}
    \left| \int \Pi(p | V_{k_1}^{\perp})(\bx)\Pi(b | V_{k_2}^{\perp})(\bx) d\bx  \right| \asymp k_2^{-(\alpha + \beta)/d}. 
\end{align*}
Moreover, letting $\E_{\bbP,1} [AY | \bX] = c \neq 0$, it follows from the proof of the lower bound of the bias of the integral-based plug-in estimator in the single sample splitting case that
\begin{equation*}
    \left| \int r_{n, k_1, k_2}(\bx) d\bx \right| \asymp \frac{k_1}{n}
\end{equation*}
which completes the proof provided that $k_2^{-\frac{\alpha + \beta}{d}}$ and $\frac{k_1}{n}$ are not of the same order. When $k_1$ and $k_2$ are such that $k_2^{-\frac{\alpha + \beta}{d}} \asymp \frac{k_1}{n}$, we can choose $\epsilon$ (appearing in (\ref{eq: p lb}) and (\ref{eq: b lb}) which define $p$ and $b$) sufficiently small so that
\begin{equation*}
    \left| \int \Pi(p|V_{k_1}^{\perp})(\bx) \Pi(b|V_{k_2}^{\perp})(\bx) d\bx + \int r_{n, k_1, k_2}(\bx) d\bx \right| \gtrsim \frac{k_1}{n}
\end{equation*}
which completes the proof.

\section{Proof of Theorem 3} \label{sec: proof known f nr}

As in the proof of Theorem 1, we write out the proof of the upper bound when $f$ is such that 
\begin{align*}
    (i) & \quad  f(\bx) \in [M_1, M_2] \quad \forall \bx \in [0, 1]^d\\
    (ii) & \quad f \in H(\gamma, M)
\end{align*}
where $\gamma \geq \alpha \vee \beta$ and $M, M_1, M_2 \in \mathbb{R}^{+}$ are known constants. Recall that we use the approximate wavelet projection estimators of the nuisance functions given by (\ref{eq: wavelet estimator of p}) and (\ref{eq: wavelet estimator of b}) in this case.

Moreover, the Newey and Robins plug-in estimator that we analyze is given by
\begin{equation*}
   \hat{\psi}_{k}^{\mathrm{NR}} = \frac{1}{n} \sum_{i \in \D_2} A_i(Y_i - \hat{b}_k^{(1)}(\bX_i)).
\end{equation*}

\subsection{Proof of the upper bound}

Let $\bbP \in \cP_{(\alpha,\beta)}$ be arbitrary. \\

\noindent \textbf{Bounding the bias:}

Observe that
\begin{align*}
    \E_\bbP(\hat{\psi}^{\mathrm{NR}}_{k}) & =  \psi(\bbP) +  \E_\bbP\left[ p(\bX) (b(\bX) - \hat{b}^{(1)}_k(\bX)) \right] \\ 
    & =  \psi(\bbP) +  \int  p(\bx) f(\bx) (b(\bx) - \Pi(b|V_k)(\bx)) d\bx.
\end{align*}
Recall from the derivation of the bias of the integral-based plug-in estimator that
\begin{equation*}
    \left| \int  p(\bx) f(\bx) \Pi(b|V_k^{\perp})(\bx) d\bx \right| \lesssim k^{-(\alpha+\beta)/d}
\end{equation*}
which completes the upper bound on the bias of $\hat{\psi}_{k}^{\mathrm{NR}}$. \\

\noindent \textbf{Bounding the variance:} 

We will show that
\begin{align}
    \Var_{\bbP,1} \left[ \E_{\bbP,2}\left( \hat{\psi}^{\mathrm{NR}}_{k}\right)  \right] & \lesssim \frac{1}{n} \label{eq: var known f nr pt1} \\
    \E_{\bbP,1} \left[ \Var_{\bbP,2} \left( \hat{\psi}^{\mathrm{NR}}_{k}\right)  \right] & \lesssim \frac{1}{n} + \frac{k}{n^2} \label{eq: var known f nr pt2}.
\end{align}
We first show (\ref{eq: var known f nr pt1}) as follows:
\begin{align*}
    & \Var_{\bbP,1} \left[ \E_{\bbP,2}\left( \hat{\psi}^{\mathrm{NR}}_{k}\right)  \right] \\
    & \quad = \Var_{\bbP,1} \left[  \E_{\bbP,2} \left(A(Y - \hat{b}^{(1)}_{k}(\bX))\right)  \right]  \\
    & \quad  = \Var_{\bbP,1} \left[\E_{\bbP,2} \left(p(\bX)(b(\bX) - \hat{b}^{(1)}_{k}(\bX)) \right) \right] \\
    & \quad  = \Var_{\bbP,1} \left[\int p(\bx)f(\bx)(b(\bx) - \hat{b}^{(1)}_{k}(\bx)) d\bx \right] \\
    & \quad  = \E_{\bbP,1} \left[ \left( \int p(\bx)f(\bx)(b(\bx) - \hat{b}^{(1)}_{k}(\bx)) d\bx\right)^2  \right] \\
    & \quad \quad  - \left(   \E_{\bbP,1} \left[ \int p(\bx)f(\bx)(b(\bx) - \hat{b}^{(1)}_{k}(\bx)) d\bx  \right] \right)^2 \\
    & \quad  = \iint  p(\bx)p(\by)f(\by)f(\bx) \E_{\bbP,1} \left[(b(\bx) - \hat{b}^{(1)}_{k}(\bx))(b(\by) - \hat{b}^{(1)}_{k}(\by)) \right] d\bx d\by \\
    & \quad  \quad - \iint  p(\bx)p(\by)f(\by)f(\bx) (b(\bx) - \Pi(b|V_k)(\bx))(b(\by) - \Pi(b|V_k)(\by)) d\bx d\by \\
    & \quad  \lesssim \iint  \left| \E_{\bbP,1} \left[(b(\bx) - \hat{b}^{(1)}_{k}(\bx))(b(\by) - \hat{b}^{(1)}_{k}(\by)) \right] - \Pi(b|V_k^{\perp})(\bx) \Pi(b|V_k^{\perp})(\by)  \right| d\bx d\by \\
    & \quad \lesssim \frac{1}{n} + \iint \frac{1}{n} \E_{\bbP} \left[ \left| K_{V_{k}}(\bX, \bx)K_{V_{k}}(\bX, \by) \right| \right] d\bx d\by \quad (\text{by Lemma \ref{lem: exp pxpy}})\\
    & \quad \lesssim \frac{1}{n}  \quad (\text{by Lemma \ref{lem: exp kernel}}).
\end{align*}
Next, we show (\ref{eq: var known f nr pt2}). We have that
\begin{align*}
    \E_{\bbP,1} \left[ \Var_{\bbP,2} \left( \hat{\psi}^{\mathrm{NR}}_{k}\right)  \right] & = \frac{1}{n} \E_{\bbP,1} \left[ \Var_{\bbP,2} \left( A(Y - \hat{b}^{(1)}_{k}(\bX)) \right) \right] \\
    & \leq \frac{1}{n} \E_{\bbP,1} \left[ \E_{\bbP,2} \left( A^2(Y - \hat{b}^{(1)}_{k}(\bX))^2 \right) \right] \\
    & = \frac{1}{n} \E_{\bbP,1} \left[ \E_{\bbP,2} \left( A^2(Y^2 - 2Y\hat{b}^{(1)}_{k}(\bX) + \hat{b}^{(1)}_{k}(\bX)^2) \right) \right] \\
    & = \frac{1}{n} \E_{\bbP,2} \left[ A^2\left(Y^2 - 2Y\Pi(b | V_k)(\bX) + \E_{\bbP,1} \left[\hat{b}^{(1)}_{k}(\bX)^2\right]\right) \right].
\end{align*}
Since $A$, $Y$ are bounded random variables and $\Pi(b | V_{k})$ is bounded, 
\begin{equation*}
    \E_{\bbP,1} \left[ \Var_{\bbP,2} \left( \hat{\psi}^{\mathrm{NR}}_{k}\right)  \right] \lesssim \frac{1}{n} + \frac{1}{n} \E_{\bbP,2} \left( \E_{\bbP,1}\left[\hat{b}^{(1)}_{k}(\bX)^2\right] \right). 
\end{equation*}
It then follows from Lemma \ref{lem: nuisance function bounds} that
\begin{equation*}
    \E_{\bbP,1} \left[ \Var_{\bbP,2} \left( \hat{\psi}^{\mathrm{NR}}_{k}\right)  \right] \lesssim \frac{1}{n} + \frac{k}{n^2}.
\end{equation*}

\subsection{Proof of the lower bound}

Let $\bbP \in \cP_{(\alpha,\beta)}$ and $\bX \sim \text{Uniform}([0,1]^d)$. Recall from the proof of the upper bound that
\begin{equation*}
    \E_\bbP(\hat{\psi}^{\mathrm{NR}}_{k}) - \psi(\bbP)  = \int  p(\bx) \Pi(b|V_k^{\perp})(\bx) d\bx. 
\end{equation*}
By orthogonality of $\Pi(p|V_k)$ and $\Pi(b|V_k^{\perp})$ with respect to the Lebesgue measure,
\begin{align*}
    \E_\bbP(\hat{\psi}^{\mathrm{NR}}_{k}) - \psi(\bbP)  & = \int  \Pi(p|V_k)(\bx) \Pi(b|V_k^{\perp})(\bx) d\bx  + \int  \Pi(p|V_k^{\perp})(\bx) \Pi(b|V_k^{\perp})(\bx) d\bx  \\
    & = \int  \Pi(p|V_k^{\perp})(\bx) \Pi(b|V_k^{\perp})(\bx) d\bx.
\end{align*}
Choosing $p$ and $b$ as given in (\ref{eq: p lb}) and (\ref{eq: b lb}), recall that
\begin{equation*}
    \left| \int  \Pi(p|V_k^{\perp})(\bx) \Pi(b|V_k^{\perp})(\bx) d\bx \right| \gtrsim k^{-(\alpha + \beta) / d}
\end{equation*}
which completes the proof.

\section{Proof of Theorem 4} \label{sec: proof known f no ss}

As in the proof of Theorem 1, we write out the proofs of the upper bounds when $f$ is such that 
\begin{align*}
    (i) & \quad  f(\bx) \in [M_1, M_2] \quad \forall \bx \in [0, 1]^d\\
    (ii) & \quad f \in H(\gamma, M)
\end{align*}
where $\gamma \geq \alpha \vee \beta$ and $M, M_1, M_2 \in \mathbb{R}^{+}$ are known constants. Recall that we use the approximate wavelet projection estimators of the nuisance functions given by (\ref{eq: wavelet estimator of p}) and (\ref{eq: wavelet estimator of b}) in this case.

The estimators of $\psi(\bbP)$ that we analyze are given by
\begin{align*}
    \hat{\psi}_{k_1, k_2}^{\mathrm{INT}} & = \frac{1}{n}\sum_{i \in \D_1} A_i Y_i -  \int \hat{p}_{k_1}^{(1)}(\bx)\hat{b}_{k_2}^{(1)}(\bx)f(\bx) d\bx \\
    \hat{\psi}^{\mathrm{MC}}_{k_1, k_2} & = \frac{1}{n}\sum_{i \in \D_1} A_i Y_i -  \frac{1}{n}\sum_{i \in \D_1} \hat{p}^{(1)}_{k_1}(\bX_i)\hat{b}^{(1)}_{k_2}(\bX_i) \\
    \hat{\psi}^{\mathrm{IF}}_{k_1, k_2} & = \frac{1}{n} \sum_{i \in \D_1} (A_i - \hat{p}^{(1)}_{k_1}(\bX_i))(Y_i - \hat{b}^{(1)}_{k_2}(\bX_i)).
\end{align*}

For notational simplicity, we will remove the superscript of $(1)$ from the nuisance function estimators. Further, let $\hat{p}^{(-j)}_{k_1}$ denote the approximate wavelet projection estimator without the contribution of the $j$th observation in $\D_1$, i.e.
\begin{equation*}
    \hat{p}_k^{(-j)}(\bx) = \frac{1}{n} \sum_{\substack{i \in \D_1 \\ i \neq j}} \frac{A_i K_{V_{k}}(\bX_i,\bx)}{f(\bX_i)}, \qquad \bx \in [0, 1]^d.
\end{equation*}
Let $\hat{b}^{(-j)}_{k_2}$ be defined analogously. Let $\E_{\bbP, -j}$ and $\Var_{\bbP, -j}$ denote the expectation and variance under distribution $\bbP$ over set $\D_1 \setminus \{ \bO_j \}$ conditional on $\{ \bO_j \}$. Let $\E_{\bbP, j}$ and $\Var_{\bbP, j}$ denote the expectation and variance under distribution $\bbP$ over $\{\bO_j\}$ conditional on $\D_1 \setminus \{\bO_j\}$. 

Note that 
\begin{align*}
    \E_{\bbP, -j}\left[ \hat{p}_k^{(-j)} (\bx) \right] & = \frac{n-1}{n} \Pi(p|V_k)(\bx) \\
    \E_{\bbP, -j}\left[ \hat{b}_k^{(-j)} (\bx) \right] & = \frac{n-1}{n} \Pi(b|V_k)(\bx).
\end{align*}

\subsection{Integral-based plug-in estimator} 

Let $\bbP \in \cP_{(\alpha,\beta)}$ be arbitrary. The bias is trivially the same as in the single sample splitting case. To bound the variance, observe that
\begin{align*}
    \Var_{\bbP}(\hat{\psi}_{k_1, k_2}^{\mathrm{INT}}) \leq 2 \Var_{\bbP}\left( \frac{1}{n}\sum_{i \in \D_1} A_i Y_i \right) + 2 \Var_{\bbP}\left( \int \hat{p}_{k_1}(\bx)\hat{b}_{k_2}(\bx)f(\bx) d\bx \right)
\end{align*}
and recall that we showed in the single sample splitting case that
\begin{align*}
    \Var_{\bbP}\left( \frac{1}{n}\sum_{i \in \D_1} A_i Y_i \right) & \lesssim \frac{1}{n} \\
    \Var_{\bbP}\left( \int \hat{p}_{k_1}(\bx)\hat{b}_{k_2}(\bx)f(\bx) d\bx \right) & \lesssim \frac{1}{n} + \frac{k_{\mathrm{min}}}{n^2} + \frac{k_{\mathrm{min}}^2}{n^3}.
\end{align*}

\subsection{Monte Carlo-based plug-in estimator}

Recall that we can write
\begin{align}
    \hat{p}_{k_1}(\bx) & = \hat{p}_{k_1}^{(-i)}(\bx) + \frac{A_i}{nf(\bX_i)} K_{V_{k_1}}(\bX_i, \bx) \label{eq: phat minus i}\\
    \hat{b}_{k_2}(\bx) & = \hat{b}_{k_2}^{(-i)}(\bx) + \frac{Y_i}{nf(\bX_i)} K_{V_{k_2}}(\bX_i, \bx) \label{eq: bhat minus i}.
\end{align}
Using (\ref{eq: phat minus i}) and (\ref{eq: bhat minus i}), we can express $\hat{\psi}^{\mathrm{MC}}_{k_1, k_2}$ as
\begin{align}
    \hat{\psi}^{\mathrm{MC}}_{k_1, k_2} & = \frac{1}{n}\sum_{i \in \D_1} A_i Y_i -  \frac{1}{n}\sum_{i \in \D_1} \hat{p}_{k_1}^{(-i)}(\bX_i) \hat{b}_{k_2}^{(-i)}(\bX_i) \nonumber \\
    & \quad - \frac{1}{n^2}\sum_{i \in \D_1} \frac{A_i}{f(\bX_i)} K_{V_{k_1}}(\bX_i, \bX_i) \hat{b}^{(-i)}_{k_2}(\bX_i) \nonumber \\
    & \quad - \frac{1}{n^2}\sum_{i \in \D_1} \frac{Y_i}{f(\bX_i)} K_{V_{k_2}}(\bX_i, \bX_i) \hat{p}^{(-i)}_{k_1}(\bX_i) \nonumber \\
    & \quad - \frac{1}{n^3}\sum_{i \in \D_1} \frac{A_iY_i}{f(\bX_i)^2} K_{V_{k_1}}(\bX_i, \bX_i)K_{V_{k_2}}(\bX_i, \bX_i) . \label{eq: mc no ss rep}
\end{align}

\subsubsection{Proof of the upper bound}

Let $\bbP \in \cP_{(\alpha,\beta)}$ be arbitrary. \\

\noindent \textbf{Bounding the bias:}

We first analyze the expected value of the second term in (\ref{eq: mc no ss rep}). We have that
\begin{align*}
    \E_{\bbP} \left( \frac{1}{n}\sum_{i \in \D_1} \hat{p}_{k_1}^{(-i)}(\bX_i) \hat{b}_{k_2}^{(-i)}(\bX_i) \right) & = \E_{\bbP} \left( \hat{p}_{k_1}^{(-i)}(\bX_i) \hat{b}_{k_2}^{(-i)}(\bX_i) \right) \\
    & = \int \E_{\bbP} \left( \hat{p}_{k_1}^{(-i)}(\bx) \hat{b}_{k_2}^{(-i)}(\bx) \right) f(\bx) d\bx \\
    & = \left( 1 + O\left(\frac{1}{n}\right) \right) \int \E_{\bbP} \left( \hat{p}_{k_1}(\bx) \hat{b}_{k_2}(\bx) \right) f(\bx) d\bx.
\end{align*}
Recall from the single sample splitting case that
\begin{align*}
    \int \E_{\bbP} \left( \hat{p}_{k_1}(\bx) \hat{b}_{k_2}(\bx) \right) f(\bx) d\bx & = \int \Pi(p|V_{k_1})(\bx) \Pi(b|V_{k_2})(\bx) f(\bx) d\bx  +  O\left( \frac{k_1 \wedge k_2}{n} \right)
\end{align*}
and so the second term in (\ref{eq: mc no ss rep}) can be expressed as
\begin{equation*}
    \E_{\bbP} \left( \frac{1}{n}\sum_{i \in \D_1} \hat{p}_{k_1}^{(-i)}(\bX_i) \hat{b}_{k_2}^{(-i)}(\bX_i) \right) = \int \Pi(p|V_{k_1})(\bx) \Pi(b|V_{k_2})(\bx) f(\bx) d\bx  +  O\left( \frac{k_1 \wedge k_2}{n} \right).
\end{equation*}

Next, we bound the expectation of the third term in (\ref{eq: mc no ss rep}). We have that
\begin{align*}
    & \left| \E_{\bbP} \left( \frac{1}{n^2}\sum_{i \in \D_1} \frac{A_i}{f(\bX_i)} K_{V_{k_1}}(\bX_i, \bX_i) \hat{b}^{(-i)}_{k_2}(\bX_i)\right) \right| \\
    & \quad = \frac{1}{n} \left| \E_{\bbP} \left( \frac{A_i}{f(\bX_i)} K_{V_{k_1}}(\bX_i, \bX_i) \hat{b}^{(-i)}_{k_2}(\bX_i) \right) \right| \\
    & \quad = \frac{1}{n} \left| \int p(\bx) K_{V_{k_1}}(\bx, \bx) \E_{\bbP} \left(\hat{b}^{(-i)}_{k_2}(\bx)\right) d\bx \right| \\
    & \quad \lesssim \frac{1}{n}  \int K_{V_{k_1}}(\bx, \bx)  d\bx   \\
    & \quad \lesssim \frac{k_1}{n} \quad (\text{by Lemma \ref{lem: exp kernel single same}}).
\end{align*}

Similarly, the expectation of the fourth term in (\ref{eq: mc no ss rep}) is bounded by $O(\frac{k_2}{n})$. Last, we bound the expectation of the fifth term in (\ref{eq: mc no ss rep}). We have that
\begin{align*}
    & \left| \E_{\bbP}\left(  \frac{1}{n^3}\sum_{i \in \D_1} \frac{A_iY_i}{f(\bX_i)^2} K_{V_{k_1}}(\bX_i, \bX_i)K_{V_{k_2}}(\bX_i, \bX_i) \right) \right| \\
    & \quad = \frac{1}{n^2} \left| \E_{\bbP}\left( \frac{A_iY_i}{f(\bX_i)^2} K_{V_{k_1}}(\bX_i, \bX_i)K_{V_{k_2}}(\bX_i, \bX_i) \right) \right| \\
    & \quad \lesssim \frac{1}{n^2}  \E_{\bbP}\left(  K_{V_{k_1}}(\bX_i, \bX_i)K_{V_{k_2}}(\bX_i, \bX_i)  \right) \\
    & \quad \lesssim \frac{1}{n^2}  \int   K_{V_{k_1}}(\bx, \bx)K_{V_{k_2}}(\bx, \bx)  d\bx \\
    & \quad \lesssim \frac{k_1 k_2}{n^2} \quad (\text{by Lemma \ref{lem: exp kernel single same}}).
\end{align*}

Therefore, 
\begin{equation*}
    \E_{\bbP}(\hat{\psi}^{\mathrm{MC}}_{k_1, k_2}) = \E_{\bbP}(AY) - \int \Pi(p|V_{k_1})(\bx) \Pi(b|V_{k_2})(\bx) f(\bx) d\bx +  O\left( \frac{k_1 \vee k_2}{n} \right) +  O\left( \frac{k_1 k_2}{n^2} \right). 
\end{equation*}
Observe that the expectation of the Monte Carlo-based plug-in estimator without sample splitting is equal to the expectation of the double sample split estimator plus $O\left( \frac{k_1 \vee k_2}{n} \right) +  O\left( \frac{k_1 k_2}{n^2} \right)$. It then follows from the work in the double sample splitting case that
\begin{equation*}
    \left|\E_\bbP\left(\hat{\psi}_{k_1, k_2}^{\mathrm{MC}} - \psi(\bbP)\right)\right| \lesssim (k_1 \wedge k_2)^{-(\alpha + \beta)/d} + \frac{k_1 \vee k_2}{n} + \frac{k_1k_2}{n^2}.
\end{equation*}
The result then follows when noting that $\frac{k_1k_2}{n^2} \ll \frac{k_1 \vee k_2}{n}$ when $k_1, k_2 \ll n$.\\

\noindent \textbf{Bounding the variance:}

By (\ref{eq: mc no ss rep}), we have that
\begin{align*}
    \Var_{\bbP}(\hat{\psi}^{\mathrm{MC}}_{k_1, k_2}) & \leq 5\Var_{\bbP} \left( \frac{1}{n}\sum_{i \in \D_1} A_i Y_i \right) +   5\Var_{\bbP} \left(\frac{1}{n}\sum_{i \in \D_1} \hat{p}_{k_1}^{(-i)}(\bX_i) \hat{b}_{k_2}^{(-i)}(\bX_i) \right) \\
    & \quad +  5\Var_{\bbP} \left(\frac{1}{n^3}\sum_{i \in \D_1} \frac{A_iY_i}{f(\bX_i)^2} K_{V_{k_1}}(\bX_i, \bX_i)K_{V_{k_2}}(\bX_i, \bX_i) \right) \\
    & \quad +  5\Var_{\bbP} \left(\frac{1}{n^2}\sum_{i \in \D_1} \frac{A_i}{f(\bX_i)} K_{V_{k_1}}(\bX_i, \bX_i) \hat{b}^{(-i)}_{k_2}(\bX_i)\right) \\
    & \quad +  5\Var_{\bbP} \left(\frac{1}{n^2}\sum_{i \in \D_1} \frac{Y_i}{f(\bX_i)} K_{V_{k_2}}(\bX_i, \bX_i) \hat{p}^{(-i)}_{k_1}(\bX_i) \right).
\end{align*}

Recall that $\Var_{\bbP} \left( \frac{1}{n}\sum_{i \in \D_1} A_i Y_i \right) \lesssim \frac{1}{n}$. We next bound \newline $\Var_{\bbP} \left(\frac{1}{n^3}\sum_{i \in \D_1} \frac{A_iY_i}{f(\bX_i)^2} K_{V_{k_1}}(\bX_i, \bX_i)K_{V_{k_2}}(\bX_i, \bX_i) \right)$ as follows
\begin{align*}
    & \Var_{\bbP} \left(\frac{1}{n^3}\sum_{i \in \D_1} \frac{A_iY_i}{f(\bX_i)^2} K_{V_{k_1}}(\bX_i, \bX_i)K_{V_{k_2}}(\bX_i, \bX_i) \right) \\
    & \quad = \frac{1}{n^5} \Var_{\bbP} \left( \frac{A_iY_i}{f(\bX_i)^2} K_{V_{k_1}}(\bX_i, \bX_i)K_{V_{k_2}}(\bX_i, \bX_i) \right) \\
    & \quad \leq \frac{1}{n^5} \E_{\bbP} \left[ \left( \frac{A_iY_i}{f(\bX_i)^2} K_{V_{k_1}}(\bX_i, \bX_i)K_{V_{k_2}}(\bX_i, \bX_i) \right)^2 \right] \\
    & \quad \lesssim \frac{1}{n^5} \int K_{V_{k_1}}(\bx, \bx)^2K_{V_{k_2}}(\bx, \bx)^2 d\bx \\
    & \quad \lesssim \frac{k_1^2k_2^2}{n^5} \quad (\text{by Lemma \ref{lem: exp kernel single same}}).
\end{align*}

We next show that
\begin{equation} \label{eq: var a k bhat}
    \Var_{\bbP} \left(\frac{1}{n^2}\sum_{i \in \D_1} \frac{A_i}{f(\bX_i)} K_{V_{k_1}}(\bX_i, \bX_i) \hat{b}^{(-i)}_{k_2}(\bX_i)\right) \lesssim \frac{k_1^2}{n^3} + \frac{k_1^2k_2}{n^4}.
\end{equation}
To do so, we will show that for all $i \in \D_1$
\begin{equation}
        \Var_{\bbP}\left( \frac{A_i}{f(\bX_i)} K_{V_{k_1}}(\bX_i, \bX_i) \hat{b}^{(-i)}_{k_2}(\bX_i)\right) \lesssim k_1^2 + \frac{k_1^2k_2}{n} \label{eq: no ss mc var pt1a}
\end{equation}
and for all $i, j \in \D_1, i \neq j$
\begin{align}
    & \left| \Cov_{\bbP}\left(\begin{array}{c} \frac{A_i}{f(\bX_i)} K_{V_{k_1}}(\bX_i, \bX_i) \hat{b}^{(-i)}_{k_2}(\bX_i), \\\quad \frac{A_j}{f(\bX_j)} K_{V_{k_1}}(\bX_j, \bX_j) \hat{b}^{(-j)}_{k_2}(\bX_j) \end{array}\right) \right| \lesssim \frac{k_1^2}{n} + \frac{k_1^2k_2}{n^2} \label{eq: no ss mc var pt2a}.
\end{align}
To see that (\ref{eq: no ss mc var pt1a}) holds, observe that for any $i \in \D_1$
\begin{align*}
    \Var_{\bbP} \left(\frac{A_i}{f(\bX_i)} K_{V_{k_1}}(\bX_i, \bX_i) \hat{b}^{(-i)}_{k_2}(\bX_i)\right) & \lesssim \E_{\bbP} \left[ K_{V_{k_1}}(\bX_i, \bX_i)^2 \hat{b}^{(-i)}_{k_2}(\bX_i)^2\right] \\
    & \lesssim \int  K_{V_{k_1}}(\bx, \bx)^2 \E_{\bbP} \left[\hat{b}^{(-i)}_{k_2}(\bx)^2\right] d\bx \\
    & \lesssim k_1^2 \int  \E_{\bbP} \left[\hat{b}^{(-i)}_{k_2}(\bx)^2\right] d\bx \quad (\text{by Lemma \ref{lem: exp kernel single same}})\\
    & \lesssim k_1^2 \int   \E_{\bbP} \left[\hat{b}_{k_2}(\bx)^2\right] d\bx \\
    & \lesssim k_1^2 + \frac{k_1^2k_2}{n} \quad (\text{by Lemma \ref{lem: nuisance function bounds}}).
\end{align*}
We next show (\ref{eq: no ss mc var pt2a}). For any $i,j \in \D_1$ where $i \neq j$,
\begin{align*}
    & \Cov_{\bbP}\left(\frac{A_i}{f(\bX_i)} K_{V_{k_1}}(\bX_i, \bX_i) \hat{b}^{(-i)}_{k_2}(\bX_i), \frac{A_j}{f(\bX_j)} K_{V_{k_1}}(\bX_j, \bX_j) \hat{b}^{(-j)}_{k_2}(\bX_j)\right)\\
    & \quad = \E_{\bbP}\left[ \frac{A_iA_j}{f(\bX_i)f(\bX_j)}K_{V_{k_1}}(\bX_i, \bX_i)K_{V_{k_1}}(\bX_j, \bX_j)\hat{b}_{k_2}^{(-i)}(\bX_i)\hat{b}_{k_2}^{(-j)}(\bX_j) \right] \\
    & \quad \quad - \left(\E_{\bbP}\left[ \frac{A_i}{f(\bX_i)} K_{V_{k_1}}(\bX_i, \bX_i) \hat{b}^{(-i)}_{k_2}(\bX_i) \right]\right)^2.
\end{align*}
Observe that
\begin{align*}
    & \left( \E_{\bbP}\left[ \frac{A_i}{f(\bX_i)} K_{V_{k_1}}(\bX_i, \bX_i) \hat{b}^{(-i)}_{k_2}(\bX_i) \right] \right)^2 \\
    & \quad = \frac{(n-1)^2}{n^2} \left( \E_{\bbP} \left[ \frac{A_iY_{i^{\prime}}}{f(\bX_i)f(\bX_{i^{\prime}})} K_{V_{k_1}}(\bX_i, \bX_i) K_{V_{k_2}}(\bX_{i^{\prime}}, \bX_i) \right] \right)^2
\end{align*}
where $i^\prime \neq i$. To bound the covariance, we decompose the following summation
\begin{align*}
    & \E_{\bbP}\left[ \frac{A_iA_j}{f(\bX_i)f(\bX_j)}K_{V_{k_1}}(\bX_i, \bX_i)K_{V_{k_1}}(\bX_j, \bX_j)\hat{b}_{k_2}^{(-i)}(\bX_i)\hat{b}_{k_2}^{(-j)}(\bX_j) \right] \\
    & \quad = \frac{1}{n^2}\sum_{\substack{i^{\prime}, j^{\prime} \in \D_1\\ i^{\prime} \neq i, j^{\prime} \neq j}} \E_{\bbP}\left[ \frac{A_iA_jY_{i^{\prime}}Y_{j^{\prime}}}{f(\bX_i)f(\bX_j)f(\bX_{i^{\prime}})f(\bX_{j^{\prime}})} \left(\begin{array}{c} K_{V_{k_1}}(\bX_i, \bX_i)K_{V_{k_1}}(\bX_j, \bX_j) \\ \times K_{V_{k_2}}(\bX_{i^{\prime}}, \bX_i) K_{V_{k_2}}(\bX_{j^{\prime}}, \bX_j) \end{array} \right) \right].
\end{align*}
The contribution to the summation when $i^\prime = j^\prime$ is given by
\begin{align*}
    \frac{(n-2)}{n^2}  \E_{\bbP}\left[ \frac{A_iA_jY_{i^{\prime}}^2}{f(\bX_i)f(\bX_j)f(\bX_{i^{\prime}})^2} \left( \begin{array}{c} K_{V_{k_1}}(\bX_i, \bX_i)K_{V_{k_1}}(\bX_j, \bX_j) \\ \times K_{V_{k_2}}(\bX_{i^{\prime}}, \bX_i) K_{V_{k_2}}(\bX_{i^{\prime}}, \bX_j) \end{array}\right) \right]
\end{align*}
where we can bound the expectation by
\begin{align*}
    & \left| \E_{\bbP}\left[ \frac{A_iA_jY_{i^{\prime}}^2}{f(\bX_i)f(\bX_j)f(\bX_{i^{\prime}})^2} \left( \begin{array}{c} K_{V_{k_1}}(\bX_i, \bX_i)K_{V_{k_1}}(\bX_j, \bX_j) \\ \times K_{V_{k_2}}(\bX_{i^{\prime}}, \bX_i) K_{V_{k_2}}(\bX_{i^{\prime}}, \bX_j) \end{array}\right) \right] \right| \\
    & \quad \lesssim \E_{\bbP}\left[ \left| K_{V_{k_1}}(\bX_i, \bX_i)K_{V_{k_1}}(\bX_j, \bX_j) K_{V_{k_2}}(\bX_{i^{\prime}}, \bX_i) K_{V_{k_2}}(\bX_{i^{\prime}}, \bX_j) \right| \right] \\
    & \quad \lesssim k_1^2 \E_{\bbP}\left[ \left| K_{V_{k_2}}(\bX_{i^{\prime}}, \bX_i) K_{V_{k_2}}(\bX_{i^{\prime}}, \bX_j) \right| \right] \quad (\text{by Lemma \ref{lem: exp kernel single same}}) \\
    & \quad \lesssim k_1^2 \iint \E_{\bbP}\left[ \left| K_{V_{k_2}}(\bX_{i^{\prime}}, \bx) K_{V_{k_2}}(\bX_{i^{\prime}}, \by) \right|\right] d\bx d\by  \\
    & \quad \lesssim k_1^2 \quad (\text{by Lemma \ref{lem: exp kernel}})
\end{align*}
We next consider the contributions to the summation when $i^\prime \neq j^\prime$. When $i^\prime \neq j$ and $j^\prime \neq i$, the contribution is given by
\begin{equation*}
    \frac{(n-2)(n-3)}{n^2} \left( \E_{\bbP} \left[ \frac{A_iY_{i^{\prime}}}{f(\bX_i)f(\bX_{i^{\prime}})} K_{V_{k_1}}(\bX_i, \bX_i) K_{V_{k_2}}(\bX_{i^{\prime}}, \bX_i) \right] \right)^2.
\end{equation*}
When $i^\prime \neq j$ and $j^\prime = i$, the contribution is given by
\begin{equation*}
    \frac{(n-2)}{n^2} \E_{\bbP}\left[ \frac{A_iA_jY_{i^{\prime}}Y_{i}}{f(\bX_i)f(\bX_j)f(\bX_{i^{\prime}})f(\bX_{i})} \left(\begin{array}{c} K_{V_{k_1}}(\bX_i, \bX_i)K_{V_{k_1}}(\bX_j, \bX_j) \\ \times K_{V_{k_2}}(\bX_{i^{\prime}}, \bX_i) K_{V_{k_2}}(\bX_i, \bX_j) \end{array} \right) \right]
\end{equation*}
where we can bound the expectation by
\begin{align*}
    & \left| \E_{\bbP}\left[ \frac{A_iA_jY_{i^{\prime}}Y_{i}}{f(\bX_i)f(\bX_j)f(\bX_{i^{\prime}})f(\bX_{i})} \left(\begin{array}{c} K_{V_{k_1}}(\bX_i, \bX_i)K_{V_{k_1}}(\bX_j, \bX_j) \\ \times K_{V_{k_2}}(\bX_{i^{\prime}}, \bX_i) K_{V_{k_2}}(\bX_i, \bX_j) \end{array} \right) \right] \right| \\
    & \quad \lesssim \E_{\bbP}\left[ \left| K_{V_{k_1}}(\bX_i, \bX_i)K_{V_{k_1}}(\bX_j, \bX_j) K_{V_{k_2}}(\bX_{i^{\prime}}, \bX_i) K_{V_{k_2}}(\bX_i, \bX_j) \right| \right] \\
    & \quad \lesssim k_1^2 \E_{\bbP}\left[ \left| K_{V_{k_2}}(\bX_{i^{\prime}}, \bX_i) K_{V_{k_2}}(\bX_i, \bX_j) \right| \right] \quad (\text{by Lemma \ref{lem: exp kernel single same}}) \\
    & \quad \lesssim k_1^2 \iint \E_{\bbP}\left[ \left| K_{V_{k_2}}(\bX_{i}, \bx) K_{V_{k_2}}(\bX_{i}, \by) \right|\right] d\bx d\by  \\
    & \quad \lesssim k_1^2 \quad (\text{by Lemma \ref{lem: exp kernel}}).
\end{align*}
Similarly, when $i^\prime = j$ and $j^\prime \neq i$, there are $O(n)$ terms in the summation and each term is $O(k_1^2)$. Lastly, when $i^\prime = j$ and $j^\prime = i$, the contribution is given by
\begin{equation*}
    \frac{1}{n^2}\E_{\bbP}\left[ \frac{A_iA_jY_{j}Y_{i}}{f(\bX_i)^2f(\bX_j)^2} K_{V_{k_1}}(\bX_i, \bX_i)K_{V_{k_1}}(\bX_j, \bX_j) K_{V_{k_2}}(\bX_{j}, \bX_i)^2 \right]
\end{equation*}
where 
\begin{align*}
    & \left| \E_{\bbP}\left[ \frac{A_iA_jY_{j}Y_{i}}{f(\bX_i)^2f(\bX_j)^2} K_{V_{k_1}}(\bX_i, \bX_i)K_{V_{k_1}}(\bX_j, \bX_j) K_{V_{k_2}}(\bX_{j}, \bX_i)^2 \right] \right|\\
    & \quad \lesssim \E_{\bbP}\left[ \left|K_{V_{k_1}}(\bX_i, \bX_i)K_{V_{k_1}}(\bX_j, \bX_j) K_{V_{k_2}}(\bX_{j}, \bX_i)^2 \right| \right] \\
    & \quad \lesssim k_1^2 \E_{\bbP}\left[ \left|K_{V_{k_2}}(\bX_{j}, \bX_i)^2 \right| \right]  \quad (\text{by Lemma \ref{lem: exp kernel single same}}) \\
    & \quad \lesssim k_1^2 k_2 \quad (\text{by Lemma \ref{lem: exp kernel single}}).
\end{align*}
Upon subtracting $\left( \E_{\bbP}\left[ \frac{A_i}{f(\bX_i)} K_{V_{k_1}}(\bX_i, \bX_i) \hat{b}^{(-i)}_{k_2}(\bX_i) \right] \right)^2 $ from the decomposed summation, we can bound the covariance as follows
\begin{align*}
    & \left| \Cov_{\bbP}\left(\frac{A_i}{f(\bX_i)} K_{V_{k_1}}(\bX_i, \bX_i) \hat{b}^{(-i)}_{k_2}(\bX_i), \frac{A_j}{f(\bX_j)} K_{V_{k_1}}(\bX_j, \bX_j) \hat{b}^{(-j)}_{k_2}(\bX_j)\right) \right|\\
    & \quad \lesssim \frac{k_1^2}{n} + \frac{k_1^2k_2}{n^2} + \frac{1}{n} \left( \E_{\bbP} \left[ \frac{A_iY_{i^{\prime}}}{f(\bX_i)f(\bX_{i^{\prime}})} K_{V_{k_1}}(\bX_i, \bX_i) K_{V_{k_2}}(\bX_{i^{\prime}}, \bX_i) \right] \right)^2.
\end{align*}
We can bound the remaining expectation by
\begin{align*}
    & \left| \E_{\bbP} \left[ \frac{A_iY_{i^{\prime}}}{f(\bX_i)f(\bX_{i^{\prime}})} K_{V_{k_1}}(\bX_i, \bX_i) K_{V_{k_2}}(\bX_{i^{\prime}}, \bX_i) \right] \right| \\
    & \quad \lesssim  \E_{\bbP} \left[\left| K_{V_{k_1}}(\bX_i, \bX_i) K_{V_{k_2}}(\bX_{i^{\prime}}, \bX_i)  \right|\right] \\
    & \quad \lesssim  k_1 \E_{\bbP} \left[\left| K_{V_{k_2}}(\bX_{i^{\prime}}, \bX_i)  \right|\right] \quad (\text{by Lemma \ref{lem: exp kernel single same}})\\
    & \quad \lesssim  k_1 \quad (\text{by Lemma \ref{lem: exp kernel single}})
\end{align*}
which establishes  (\ref{eq: no ss mc var pt2a}).

It follows in the same manner that
\begin{equation*}
    \Var_{\bbP} \left(\frac{1}{n^2}\sum_{i \in \D_1} \frac{Y_i}{f(\bX_i)} K_{V_{k_2}}(\bX_i, \bX_i) \hat{p}^{(-i)}_{k_1}(\bX_i) \right) \lesssim \frac{k_2^2}{n^3} + \frac{k_1k_2^2}{n^4}.
\end{equation*}

Last, we bound $\Var_{\bbP} \left(\frac{1}{n}\sum_{i \in \D_1} \hat{p}_{k_1}^{(-i)}(\bX_i) \hat{b}_{k_2}^{(-i)}(\bX_i) \right)$ using a similar approach used to show (\ref{eq: var a k bhat}). We will show that for all $i \in \D_1$
\begin{equation}
    \Var_{\bbP} \left( \hat{p}_{k_1}^{(-i)}(\bX_i) \hat{b}_{k_2}^{(-i)}(\bX_i) \right) \lesssim 1 + \frac{k_{\mathrm{max}}}{n} + \frac{k_{\mathrm{min}}k_{\mathrm{max}}}{n^2} + \frac{k_{\mathrm{min}}^2 k_{\mathrm{max}}}{n^3}  \label{eq: no ss mc var pt1} 
\end{equation}
and for all $i,j\in \D_1, i \neq j$
\begin{align}
    \Cov_{\bbP} \left( \hat{p}_{k_1}^{(-i)}(\bX_i) \hat{b}_{k_2}^{(-i)}(\bX_i), \hat{p}_{k_1}^{(-j)}(\bX_j) \hat{b}_{k_2}^{(-j)}(\bX_j) \right) & \lesssim \frac{1}{n} + \frac{k_{\mathrm{max}}}{n^2} + \frac{k_{\mathrm{min}}k_{\mathrm{max}}}{n^3} \nonumber \\
    & \quad + \frac{k_{\mathrm{min}}^2 k_{\mathrm{max}}}{n^4}. \label{eq: no ss mc var pt2}
\end{align}
We first establish (\ref{eq: no ss mc var pt1}). For any $i \in \D_1$, we have that
\begin{align*}
    \Var_{\bbP} \left( \hat{p}_{k_1}^{(-i)}(\bX_i) \hat{b}_{k_2}^{(-i)}(\bX_i) \right) & = \E_{\bbP,-i} \left[ \Var_{\bbP,i} \left[ \hat{p}_{k_1}^{(-i)}(\bX_i) \hat{b}_{k_2}^{(-i)}(\bX_i) \right] \right] \\
    & \quad + \Var_{\bbP,-i} \left[ \E_{\bbP,i} \left[ \hat{p}_{k_1}^{(-i)}(\bX_i) \hat{b}_{k_2}^{(-i)}(\bX_i) \right] \right].
\end{align*}
Recall from the proof of the upper bound of the variance of $\hat{\psi}_{k_1, k_2}^{\mathrm{MC}}$ (single sample splitting case) that
\begin{align*}
    \E_{\bbP,-i} \left[ \Var_{\bbP,i} \left[ \hat{p}_{k_1}^{(-i)}(\bX_i) \hat{b}_{k_2}^{(-i)}(\bX_i) \right] \right] & \leq \E_{\bbP,-i} \left[ \E_{\bbP,i} \left[ (\hat{p}_{k_1}^{(-i)}(\bX_i) \hat{b}_{k_2}^{(-i)}(\bX_i))^2 \right] \right] \\
    & = \int \E_{\bbP,-i} \left[ \hat{p}_{k_1}^{(-i)}(\bx)^2 \hat{b}_{k_2}^{(-i)}(\bx)^2 \right] f(\bx) d\bx \\
    & \lesssim 1 + \frac{k_{\mathrm{max}}}{n} + \frac{k_{\mathrm{min}}k_{\mathrm{max}}}{n^2} + \frac{k_{\mathrm{min}}^2 k_{\mathrm{max}}}{n^3}.
\end{align*}
Recall from the proof of the upper bound of the variance of $\hat{\psi}_{k_1, k_2}^{\mathrm{INT}}$ (single sample splitting case) that
\begin{align*}
    \Var_{\bbP,-i} \left[ \E_{\bbP,i} \left[ \hat{p}_{k_1}^{(-i)}(\bX_i) \hat{b}_{k_2}^{(-i)}(\bX_i) \right] \right] & = \Var_{\bbP,-i} \left[ \int \hat{p}_{k_1}^{(-i)}(\bx) \hat{b}_{k_2}^{(-i)}(\bx) f(\bx) d\bx\right] \\
    & \lesssim \frac{1}{n} + \frac{k_{\mathrm{min}}}{n^2} + \frac{k_{\mathrm{min}}^2}{n^3}
\end{align*}
which establishes (\ref{eq: no ss mc var pt1}). 

The remainder of the proof will establish (\ref{eq: no ss mc var pt2}). We follow a similar approach used to show (\ref{eq: no ss mc var pt2a}). Let $i,j \in \D_1$ such that $i \neq j$ be arbitrary. We have that
\begin{align*}
    & \Cov_{\bbP} \left( \hat{p}_{k_1}^{(-i)}(\bX_i) \hat{b}_{k_2}^{(-i)}(\bX_i), \hat{p}_{k_1}^{(-j)}(\bX_j) \hat{b}_{k_2}^{(-j)}(\bX_j) \right) \\
    & = \E_{\bbP} \left( \hat{p}_{k_1}^{(-i)}(\bX_i) \hat{b}_{k_2}^{(-i)}(\bX_i) \hat{p}_{k_1}^{(-j)}(\bX_j) \hat{b}_{k_2}^{(-j)}(\bX_j) \right) - \left[ \E_{\bbP} \left( \hat{p}_{k_1}^{(-i)}(\bX_i) \hat{b}_{k_2}^{(-i)}(\bX_i)\right) \right]^2.
\end{align*}
Recall from the single sample splitting case that
\begin{align*}
    & \E_{\bbP} \left( \hat{p}_{k_1}^{(-i)}(\bX_i) \hat{b}_{k_2}^{(-i)}(\bX_i)\right)  \\
    & \quad =  \int \E_{\bbP} \left( \hat{p}_{k_1}^{(-i)}(\bx) \hat{b}_{k_2}^{(-i)}(\bx)\right) f(\bx) d\bx  \\
    & \quad = \left( 1 + O\left(\frac{1}{n}\right) \right) \int \E_{\bbP} \left( \hat{p}_{k_1}(\bx) \hat{b}_{k_2}(\bx)\right) f(\bx) d\bx  \\
    & \quad = \left( 1 + O\left(\frac{1}{n}\right) \right) \bigg[ \int \Pi(p|V_{k_1})(\bx) \Pi(b|V_{k_2})(\bx)f(\bx)d\bx \\
    & \quad \qquad \qquad \qquad \qquad \quad+ \frac{1}{n} \int \E_{\bbP}\left[ \frac{AY  K_{V_{k_1}}(\bX, \bx)K_{V_{k_2}}(\bX, \bx)}{\left( f(\bX) \right)^2} \right] f(\bx) d\bx  + O\left( \frac{1}{n} \right) \bigg].
\end{align*}
By definition of $\hat{p}_{k_1}^{(-i)}$ and $\hat{b}_{k_2}^{(-i)}$, we can write
\begin{align*}
    & \E_{\bbP} \left( \hat{p}_{k_1}^{(-i)}(\bX_i) \hat{b}_{k_2}^{(-i)}(\bX_i) \hat{p}_{k_1}^{(-j)}(\bX_j) \hat{b}_{k_2}^{(-j)}(\bX_j) \right) \\
    & \quad = \frac{1}{n^4}  \sum_{\substack{i_1, i_2, j_1, j_2 \in \D_1 \\ i_1 \neq i, i_2 \neq j \\ j_1 \neq i, j_2 \neq j}} \E_{\bbP}\left[ \frac{A_{i_1}A_{i_2}Y_{j_1}Y_{j_2}\left( \begin{array}{c} K_{V_{k_1}}(\bX_{i_1},\bX_i)K_{V_{k_1}}(\bX_{i_2},\bX_j) \\ \times K_{V_{k_2}}(\bX_{j_1},\bX_i)K_{V_{k_2}}(\bX_{j_2},\bX_j) \end{array} \right)}{f(\bX_{i_1})f(\bX_{i_2})f(\bX_{j_1})f(\bX_{j_2})}  \right] \\
    & \quad = \frac{1}{n^4}  \sum_{\substack{i_1, i_2, j_1, j_2 \in \D_1 \\ i_1 \neq i, i_2 \neq j \\ j_1 \neq i, j_2 \neq j \\ \{i_1,i_2,j_1,j_2\} \cap \{i, j\} = \emptyset}} \E_{\bbP}\left[ \frac{A_{i_1}A_{i_2}Y_{j_1}Y_{j_2}\left( \begin{array}{c} K_{V_{k_1}}(\bX_{i_1},\bX_i)K_{V_{k_1}}(\bX_{i_2},\bX_j) \\ \times K_{V_{k_2}}(\bX_{j_1},\bX_i)K_{V_{k_2}}(\bX_{j_2},\bX_j) \end{array} \right)}{f(\bX_{i_1})f(\bX_{i_2})f(\bX_{j_1})f(\bX_{j_2})}  \right] \\
    & \quad \quad + \frac{1}{n^4}  \sum_{\substack{i_1, i_2, j_1, j_2 \in \D_1 \\ i_1 \neq i, i_2 \neq j \\ j_1 \neq i, j_2 \neq j \\ \{i_1,i_2,j_1,j_2\} \cap \{i, j\} \neq \emptyset}} \E_{\bbP}\left[ \frac{A_{i_1}A_{i_2}Y_{j_1}Y_{j_2}\left( \begin{array}{c} K_{V_{k_1}}(\bX_{i_1},\bX_i)K_{V_{k_1}}(\bX_{i_2},\bX_j) \\ \times K_{V_{k_2}}(\bX_{j_1},\bX_i)K_{V_{k_2}}(\bX_{j_2},\bX_j) \end{array} \right)}{f(\bX_{i_1})f(\bX_{i_2})f(\bX_{j_1})f(\bX_{j_2})}  \right]  \\
    & \quad = \frac{1}{n^4}  \sum_{\substack{i_1, i_2, j_1, j_2 \in \D_1 \\ i_1 \neq i, i_2 \neq j \\ j_1 \neq i, j_2 \neq j \\ \{i_1,i_2,j_1,j_2\} \cap \{i, j\} = \emptyset}} \E_{\bbP}\left[ \frac{A_{i_1}A_{i_2}Y_{j_1}Y_{j_2}\left(\begin{array}{c} \int K_{V_{k_1}}(\bX_{i_1},\bx) K_{V_{k_2}}(\bX_{j_1},\bx)f(\bx)d\bx \\ \times \int K_{V_{k_1}}(\bX_{i_2},\by)K_{V_{k_2}}(\bX_{j_2},\by) f(\by)  d\by \end{array}\right)}{f(\bX_{i_1})f(\bX_{i_2})f(\bX_{j_1})f(\bX_{j_2})}  \right] \\
    & \quad \quad + \frac{1}{n^4}  \sum_{\substack{i_1, i_2, j_1, j_2 \in \D_1 \\ i_1 \neq i, i_2 \neq j \\ j_1 \neq i, j_2 \neq j \\ \{i_1,i_2,j_1,j_2\} \cap \{i, j\} \neq \emptyset}} \E_{\bbP}\left[ \frac{A_{i_1}A_{i_2}Y_{j_1}Y_{j_2} \left( \begin{array}{c} K_{V_{k_1}}(\bX_{i_1},\bX_i)K_{V_{k_1}}(\bX_{i_2},\bX_j) \\ \times K_{V_{k_2}}(\bX_{j_1},\bX_i)K_{V_{k_2}}(\bX_{j_2},\bX_j) \end{array} \right)}{f(\bX_{i_1})f(\bX_{i_2})f(\bX_{j_1})f(\bX_{j_2})}  \right]  \\
    & = I + II.
\end{align*}
It follows from the analyses in the single sample splitting case that
\begin{align*}
    \left| I - \left[ \E_{\bbP} \left( \hat{p}_{k_1}^{(-i)}(\bX_i) \hat{b}_{k_2}^{(-i)}(\bX_i)\right) \right]^2 \right| \lesssim \frac{1}{n} + \frac{k_1 \wedge k_2}{n^2} + \frac{(k_1 \wedge k_2)^2}{n^3}.
\end{align*}
It remains to bound $II$. We decompose $II$ into four cases based on 
\begin{equation*}
    u = | \{i_1,i_2,j_1,j_2\} \cap \{i,j\} |.
\end{equation*}
\begin{itemize}
    \item $u = 4$. This implies $i_1 = j_1 = j$ and $i_2 = j_2 = i$ in $II$. The contribution in $II$ is then
\begin{align*}
    &  \frac{1}{n^4} \left| \sum_{\substack{i_1, i_2, j_1, j_2 \in \D_1 \\ i_1 = j_1 = j, i_2 = j_2 = i}} \E_{\bbP}\left[ \frac{A_{i_1}A_{i_2}Y_{j_1}Y_{j_2} \left(\begin{array}{c} K_{V_{k_1}}(\bX_{i_1},\bX_i)K_{V_{k_1}}(\bX_{i_2},\bX_j) \\ \times K_{V_{k_2}}(\bX_{j_1},\bX_i)K_{V_{k_2}}(\bX_{j_2},\bX_j) \end{array}\right)}{f(\bX_{i_1})f(\bX_{i_2})f(\bX_{j_1})f(\bX_{j_2})} \right] \right| \\
    & \quad = \frac{1}{n^4} \left|  \E_{\bbP}\left[ \frac{A_{j}A_{i}Y_{j}Y_{i} K_{V_{k_1}}(\bX_{i},\bX_j)^2K_{V_{k_2}}(\bX_{i},\bX_j)^2}{f(\bX_{i})^2f(\bX_{j})^2} \right] \right| \\
    & \quad \lesssim \frac{1}{n^4} \E_{\bbP} \left[ K_{V_{k_1}}(\bX_{i},\bX_j)^2 K_{V_{k_2}}(\bX_{i},\bX_j)^2 \right] \\
    & \quad \lesssim \frac{1}{n^4} \int \E_{\bbP} \left[K_{V_{k_1}}(\bX,\bx)^2 K_{V_{k_2}}(\bX,\bx)^2 \right] d\bx \\
    & \quad \lesssim \frac{(k_1 \wedge k_2)^2 (k_1 \vee k_2)}{n^4} \quad (\text{by Lemma \ref{lem: exp kernel single}}).
\end{align*}

\item $u = 3$. Suppose without loss of generality that $j_2 \not\in \{i,j\}$. This implies $i_1 = j_1 = j$, and $i_2 = i$. The contribution in $II$ is
\begin{align*}
    &  \frac{1}{n^4} \left| \sum_{\substack{i_1, i_2, j_1, j_2 \in \D_1 \\ i_1 = j_1 = j, i_2 = i, j_2 \not \in \{i,j\}}} \E_{\bbP}\left[ \frac{A_{i_1}A_{i_2}Y_{j_1}Y_{j_2} \left(\begin{array}{c} K_{V_{k_1}}(\bX_{i_1},\bX_i)K_{V_{k_1}}(\bX_{i_2},\bX_j) \\ \times K_{V_{k_2}}(\bX_{j_1},\bX_i)K_{V_{k_2}}(\bX_{j_2},\bX_j) \end{array}\right)}{f(\bX_{i_1})f(\bX_{i_2})f(\bX_{j_1})f(\bX_{j_2})} \right] \right| \\
    & \quad = \frac{1}{n^4} \left| \sum_{\substack{i_1, i_2, j_1, j_2 \in \D_1 \\ i_1 = j_1 = j, i_2 = i, j_2 \not \in \{i,j\}}} \E_{\bbP}\left[ \frac{A_{j}A_{i}Y_{j}Y_{j_2} \left( \begin{array}{c} K_{V_{k_1}}(\bX_{i},\bX_j)^2 \\ \times K_{V_{k_2}}(\bX_{j},\bX_i) K_{V_{k_2}}(\bX_{j_2},\bX_j) \end{array} \right)}{f(\bX_{j})^2f(\bX_{i})f(\bX_{j_2})} \right] \right| \\
    & \quad \lesssim \frac{1}{n^3} \E_{\bbP}\left[ | K_{V_{k_1}}(\bX_{i},\bX_j)^2 K_{V_{k_2}}(\bX_{i},\bX_j)K_{V_{k_2}}(\bX_{j_2},\bX_j) | \right] \\
    & \quad \lesssim \frac{1}{n^3} \E_{\bbP}\left[ |K_{V_{k_1}}(\bX_{i},\bX_j)^2 K_{V_{k_2}}(\bX_{i},\bX_j)| \right]   \quad (\text{by Lemma \ref{lem: exp kernel single}}) \\
    & \quad \lesssim \frac{1}{n^3} \int  \E_{\bbP}\left[ | K_{V_{k_1}}(\bX,\bx)^2 K_{V_{k_2}}(\bX,\bx) | \right] d\bx \\
    & \quad \lesssim \frac{k_1k_2}{n^3}  \quad (\text{by Lemma \ref{lem: exp kernel single}}).
\end{align*}

\item $u = 2$. We will consider each of the $\binom{4}{2} = 6$ cases where $u = 2$. First, suppose that $i_1, j_1 \in \{i,j\}$. This implies that $i_1 = j_1 = j$. The contribution in II when $i_2 \neq j_2$ is
\begin{align*}
    & \frac{1}{n^4} \left| \sum_{\substack{i_1, i_2, j_1, j_2 \in \D_1 \\ i_1 = j_1 = j \\ i_2, j_2 \not\in \{i,j\}, i_2 \neq j_2}} \E_{\bbP}\left[ \frac{A_{i_1}A_{i_2}Y_{j_1}Y_{j_2} \left( \begin{array}{c} K_{V_{k_1}}(\bX_{i_1},\bX_i)K_{V_{k_1}}(\bX_{i_2},\bX_j) \\ \times K_{V_{k_2}}(\bX_{j_1},\bX_i)K_{V_{k_2}}(\bX_{j_2},\bX_j) \end{array} \right)}{f(\bX_{i_1})f(\bX_{i_2})f(\bX_{j_1})f(\bX_{j_2})} \right] \right| \\
    & \quad = \frac{1}{n^4}  \sum_{\substack{i_1, i_2, j_1, j_2 \in \D_1 \\ i_1 = j_1 = j \\ i_2, j_2 \not\in \{i,j\}, i_2 \neq j_2}} \left| \E_{\bbP}\left[ \frac{A_{j}A_{i_2}Y_{j}Y_{j_2} \left( \begin{array}{c} K_{V_{k_1}}(\bX_{j},\bX_i)K_{V_{k_1}}(\bX_{i_2},\bX_j) \\ \times K_{V_{k_2}}(\bX_{j},\bX_i)K_{V_{k_2}}(\bX_{j_2},\bX_j) \end{array} \right)}{f(\bX_{j})^2f(\bX_{i_2})f(\bX_{j_2})} \right] \right| \\
    & \quad \lesssim \frac{1}{n^2}  \E_{\bbP}\left[ | K_{V_{k_1}}(\bX_{j},\bX_i)K_{V_{k_1}}(\bX_{i_2},\bX_j)K_{V_{k_2}}(\bX_{j},\bX_i)K_{V_{k_2}}(\bX_{j_2},\bX_j) | \right] \\
    & \quad \lesssim \frac{1}{n^2}  \E_{\bbP}\left[ | K_{V_{k_1}}(\bX_{j},\bX_i)K_{V_{k_2}}(\bX_{j},\bX_i) | \right]   \quad (\text{by Lemma \ref{lem: exp kernel single}}) \\
    & \quad \lesssim \frac{1}{n^2}  \int \E_{\bbP}\left[ | K_{V_{k_1}}(\bX_{j},\bx)K_{V_{k_2}}(\bX_{j},\bx) | \right] d\bx  \\
    & \quad \lesssim \frac{k_1 \wedge k_2}{n^2}  \quad (\text{by Lemma \ref{lem: exp kernel single}})
\end{align*}
and when $i_2 = j_2$ is
\begin{align*}
    & \frac{1}{n^4} \left|  \sum_{\substack{i_1, i_2, j_1, j_2 \in \D_1 \\ i_1 = j_1 = j \\ i_2, j_2 \not\in \{i,j\}, i_2 = j_2}} \E_{\bbP}\left[ \frac{A_{i_1}A_{i_2}Y_{j_1}Y_{j_2} \left(\begin{array}{c} K_{V_{k_1}}(\bX_{i_1},\bX_i)K_{V_{k_1}}(\bX_{i_2},\bX_j) \\ \times K_{V_{k_2}}(\bX_{j_1},\bX_i)K_{V_{k_2}}(\bX_{j_2},\bX_j) \end{array} \right)}{f(\bX_{i_1})f(\bX_{i_2})f(\bX_{j_1})f(\bX_{j_2})} \right] \right| \\
    & \quad = \frac{1}{n^4} \left|  \sum_{\substack{i_1, i_2, j_1, j_2 \in \D_1 \\ i_1 = j_1 = j \\ i_2, j_2 \not\in \{i,j\}, i_2 = j_2}} \E_{\bbP}\left[ \frac{A_{j}A_{i_2}Y_{j}Y_{i_2} \left( \begin{array}{c} K_{V_{k_1}}(\bX_{j},\bX_i)K_{V_{k_1}}(\bX_{i_2},\bX_j) \\ \times K_{V_{k_2}}(\bX_{j},\bX_i)K_{V_{k_2}}(\bX_{i_2},\bX_j) \end{array} \right)}{f(\bX_{j})^2f(\bX_{i_2})^2} \right] \right| \\
    & \quad \lesssim \frac{1}{n^3}  \E_{\bbP}\left[ | K_{V_{k_1}}(\bX_{j},\bX_i)K_{V_{k_1}}(\bX_{i_2},\bX_j)K_{V_{k_2}}(\bX_{j},\bX_i)K_{V_{k_2}}(\bX_{i_2},\bX_j) |\right]  \\
    & \quad \lesssim \frac{1}{n^3}  \int \E_{\bbP}\left[ | K_{V_{k_1}}(\bx,\bX_i)K_{V_{k_1}}(\bX_{i_2},\bx)K_{V_{k_2}}(\bx,\bX_i)K_{V_{k_2}}(\bX_{i_2},\bx) | \right] d\bx \\
    & \quad \lesssim \frac{1}{n^3}  \int \left(\E_{\bbP}\left[ | K_{V_{k_1}}(\bx,\bX_i)K_{V_{k_2}}(\bx,\bX_i) | \right] \right)^2 d\bx  \\
    & \quad \lesssim \frac{(k_1 \wedge k_2)^2}{n^3}.
\end{align*}

By symmetry, the case where $i_2,j_2 \in \{i,j\}$ also contributes $O(\frac{k_1 \wedge k_2}{n^2}) + O(\frac{(k_1 \wedge k_2)^2}{n^3})$ to $II$. 

For the third case, consider that $i_1, i_2 \in \{i,j\}$. This implies that $i_1 = j$ and $i_2 = i$. By similar arguments as before, the contribution in $II$ when $j_1 \neq j_2$ is bounded by
\begin{align*}
    & \frac{1}{n^2} \E_{\bbP}\left[| K_{V_{k_1}}(\bX_{j},\bX_i)K_{V_{k_1}}(\bX_{i},\bX_j)K_{V_{k_2}}(\bX_{j_1},\bX_i)K_{V_{k_2}}(\bX_{j_2},\bX_j) |\right] \\
    & \quad \lesssim \frac{1}{n^2} \E_{\bbP}\left[ K_{V_{k_1}}(\bX_i,\bX_j)^2\right] \quad (\text{by Lemma \ref{lem: exp kernel single}}) \\
    & \quad \lesssim \frac{k_1}{n^2}.
\end{align*}
The contribution in $II$ when $j_1 = j_2$ is bounded by
\begin{align*}
    & \frac{1}{n^3} \E_{\bbP}\left[| K_{V_{k_1}}(\bX_i,\bX_j)^2K_{V_{k_2}}(\bX_{j_1},\bX_i)K_{V_{k_2}}(\bX_{j_1},\bX_j) |\right] \\
    & \quad \lesssim \frac{k_1k_2}{n^3} \quad (\text{by Lemma \ref{lem: exp kernel triple}}).
\end{align*}
By symmetric arguments, the case where $j_1,j_2 \in \{i,j\}$ contributes $O(\frac{k_2}{n^2}) + O(\frac{k_1k_2}{n^3})$ in $II$.

For the fifth case, consider that $j_1,i_2 \in \{i,j\}$. This implies that $j_1 = j$ and $i_2 = i$. The contribution in $II$ when $i_1 \neq j_2$ is bounded by
\begin{align*}
    & \frac{1}{n^2} \E_{\bbP}\left[ |K_{V_{k_1}}(\bX_{i_1},\bX_i)K_{V_{k_1}}(\bX_{i},\bX_j)K_{V_{k_2}}(\bX_{j},\bX_i)K_{V_{k_2}}(\bX_{j_2},\bX_j)| \right] \\
    & \quad \lesssim \frac{1}{n^2} \E_{\bbP}\left[ |K_{V_{k_1}}(\bX_{i},\bX_j)K_{V_{k_2}}(\bX_{j},\bX_i)| \right]  \quad (\text{by Lemma \ref{lem: exp kernel single}}) \\
    & \quad \lesssim \frac{k_1 \wedge k_2}{n^2}  \quad (\text{by Lemma \ref{lem: exp kernel single}}).
\end{align*}
The contribution in $II$ when $i_1 = j_2$ is bounded by
\begin{align*}
    & \frac{1}{n^3} \E_{\bbP}\left[ |K_{V_{k_1}}(\bX_{i_1},\bX_i)K_{V_{k_1}}(\bX_{i},\bX_j)K_{V_{k_2}}(\bX_{j},\bX_i)K_{V_{k_2}}(\bX_{i_1},\bX_j)| \right] \\
    & \quad \lesssim \frac{(k_1 \wedge k_2)^2}{n^3} \quad (\text{by Lemma \ref{lem: exp kernel triple}}).
\end{align*}
By symmetric arguments, the case where $i_1,j_2 \in \{i,j\}$ contributes $O(\frac{k_1 \wedge k_2}{n^2}) + O(\frac{(k_1 \wedge k_2)^2}{n^3})$ in $II$. 

In summary, the $u=2$ cases contribute $O(\frac{k_1 \vee k_2}{n^2}) + O(\frac{k_1k_2}{n^3})$ in $II$.

\item $u = 1$. Suppose without loss of generality that $i_1 \in \{i,j\}$. This implies that $i_1 = j$. We consider 3 broad cases based on how many of the $i_2,j_1,j_2$ are distinct:
\begin{itemize}
    \item $i_2,j_1,j_2$ are distinct: The contribution in $II$ is bounded by
    \begin{align*}
        & \frac{1}{n} \E_{\bbP}\left[ |K_{V_{k_1}}(\bX_i,\bX_j)K_{V_{k_1}}(\bX_{i_2},\bX_j)K_{V_{k_2}}(\bX_{j_1},\bX_i)K_{V_{k_2}}(\bX_{j_2},\bX_j)| \right] \\
        & \quad \lesssim \frac{1}{n} \E_{\bbP}\left[ |K_{V_{k_1}}(\bX_i,\bX_j)| \right]  \quad (\text{by Lemma \ref{lem: exp kernel single}}) \\
        & \quad \lesssim \frac{1}{n}  \quad (\text{by Lemma \ref{lem: exp kernel single}}).
    \end{align*}

    \item One pair of equal $i_2,j_1,j_2$. When $i_2 = j_1$, the contribution in $II$ is bounded by
    \begin{align*}
        & \frac{1}{n^2} \E_{\bbP}\left[ |K_{V_{k_1}}(\bX_i,\bX_j)K_{V_{k_1}}(\bX_{i_2},\bX_j)K_{V_{k_2}}(\bX_{i_2},\bX_i)K_{V_{k_2}}(\bX_{j_2},\bX_j)| \right] \\
        & \quad \lesssim \frac{1}{n^2} \E_{\bbP}\left[ |K_{V_{k_1}}(\bX_i,\bX_j)K_{V_{k_1}}(\bX_{i_2},\bX_j)K_{V_{k_2}}(\bX_{i_2},\bX_i)| \right]  \quad (\text{by Lemma \ref{lem: exp kernel single}}) \\
        & \quad \lesssim \frac{k_1 \wedge k_2}{n^2}  \quad (\text{by Lemma \ref{lem: exp kernel triple}}).
    \end{align*}
    Similarly, when $j_1 = j_2$, the contribution in $II$ is bounded by
    \begin{align*}
        & \frac{1}{n^2} \E_{\bbP}\left[ |K_{V_{k_1}}(\bX_i,\bX_j)K_{V_{k_1}}(\bX_{i_2},\bX_j)K_{V_{k_2}}(\bX_{j_1},\bX_i)K_{V_{k_2}}(\bX_{j_1},\bX_j)| \right] \\
        & \quad \lesssim \frac{1}{n^2} \E_{\bbP}\left[ |K_{V_{k_1}}(\bX_i,\bX_j)K_{V_{k_2}}(\bX_{j_1},\bX_i)K_{V_{k_2}}(\bX_{j_1},\bX_j)| \right]  \quad (\text{by Lemma \ref{lem: exp kernel single}}) \\
        & \quad \lesssim \frac{k_1 \wedge k_2}{n^2}  \quad (\text{by Lemma \ref{lem: exp kernel triple}}).
    \end{align*}
    When $i_2 = j_2$, the contribution in $II$ is bounded by
    \begin{align*}
        & \frac{1}{n^2} \E_{\bbP}\left[ |K_{V_{k_1}}(\bX_i,\bX_j)K_{V_{k_1}}(\bX_{i_2},\bX_j)K_{V_{k_2}}(\bX_{j_1},\bX_i)K_{V_{k_2}}(\bX_{i_2},\bX_j)| \right] \\
        & \quad \lesssim \frac{1}{n^2} \E_{\bbP}\left[ |K_{V_{k_1}}(\bX_{i_2},\bX_j)K_{V_{k_2}}(\bX_{i_2},\bX_j)| \right]  \quad (\text{by Lemma \ref{lem: exp kernel single}}) \\
        & \quad \lesssim \frac{k_1 \wedge k_2}{n^2}  \quad (\text{by Lemma \ref{lem: exp kernel single}}).
    \end{align*}

    \item All equal $i_2, j_1, j_2$. The contribution in $II$ is bounded by
    \begin{align*}
        & \frac{1}{n^3} \E_{\bbP}\left[ |K_{V_{k_1}}(\bX_i,\bX_j)K_{V_{k_1}}(\bX_{i_2},\bX_j)K_{V_{k_2}}(\bX_{i_2},\bX_i)K_{V_{k_2}}(\bX_{i_2},\bX_j)| \right] \\
        & \quad \lesssim \frac{(k_1 \wedge k_2)^2}{n^3} \quad (\text{by Lemma \ref{lem: exp kernel triple}}).
    \end{align*}
\end{itemize}

\end{itemize}
In summary,
\begin{equation*}
    II \lesssim \frac{1}{n} + \frac{k_1 \vee k_2}{n^2} + \frac{k_1k_2}{n^3} + \frac{(k_1 \wedge k_2)^2 (k_1 \vee k_2)}{n^4}
\end{equation*}
which completes the proof of the variance bound.

\subsubsection{Proof of the lower bound}

Let $\bbP \in \cP_{(\alpha,\beta)}$. Recall from the proof of the upper bound that
\begin{align*}
    \E_{\bbP}(\hat{\psi}^{\mathrm{MC}}_{k_1, k_2}) & = \E_{\bbP}(AY) - \left(\frac{n-1}{n}\right)^2 \int \E_{\bbP} \left( \hat{p}_{k_1}(\bx) \hat{b}_{k_2}(\bx) \right) f(\bx) d\bx \\
    & \quad - \frac{1}{n} \E_{\bbP} \left( \frac{A_i}{f(\bX_i)} K_{V_{k_1}}(\bX_i, \bX_i) \hat{b}^{(-i)}_{k_2}(\bX_i) \right) \\
    & \quad - \frac{1}{n} \E_{\bbP} \left( \frac{Y_i}{f(\bX_i)} K_{V_{k_2}}(\bX_i, \bX_i) \hat{p}^{(-i)}_{k_1}(\bX_i) \right) \\
    & \quad + O\left( \frac{k_1k_2}{n^2} \right).
\end{align*}
We perform a more precise analysis of the second term for the lower bound compared to the upper bound. Specifically, recall from the single sample splitting case that
\begin{align*}
    & \left(\frac{n-1}{n}\right)^2 \int \E_{\bbP} \left( \hat{p}_{k_1}(\bx) \hat{b}_{k_2}(\bx) \right) f(\bx) d\bx \\
    & \quad = \int \E_{\bbP} \left( \hat{p}_{k_1}(\bx) \hat{b}_{k_2}(\bx) \right) f(\bx) d\bx + O\left( \frac{1}{n} \right) + O\left( \frac{k_1 \wedge k_2}{n^2} \right) \\
    & \quad = \int \Pi(p|V_{k_1})(\bx) \Pi(b|V_{k_2})(\bx) f(\bx) d\bx + \int r_{n, k_1, k_2}(\bx) f(\bx) d\bx \\
    & \quad \quad + O\left( \frac{1}{n} \right) + O\left( \frac{k_1 \wedge k_2}{n^2} \right).
\end{align*}
Let $\bX \sim \text{Uniform}([0,1]^d)$. Suppose without loss of generality that $k_1 \leq k_2$ for $n$ sufficiently large. Then, we can express the bias of $\hat{\psi}_{k_1, k_2}^{\mathrm{MC}}$ by
\begin{align} 
    \E_\bbP\left(\hat{\psi}_{k_1, k_2}^{\mathrm{MC}} - \psi(\bbP)\right) & =  \int \Pi(p|V_{k_1}^{\perp})(\bx) \Pi(b|V_{k_1}^{\perp})(\bx) d\bx - \int r_{n, k_1, k_2}(\bx) d\bx \nonumber \\
    & \quad - \frac{1}{n} \E_{\bbP} \left( A_i K_{V_{k_1}}(\bX_i, \bX_i) \hat{b}^{(-i)}_{k_2}(\bX_i) \right) \nonumber \\
    & \quad - \frac{1}{n} \E_{\bbP} \left( Y_i K_{V_{k_2}}(\bX_i, \bX_i) \hat{p}^{(-i)}_{k_1}(\bX_i) \right) \nonumber \\
    & \quad + O\left( \frac{1}{n} \right) + O\left( \frac{k_1k_2}{n^2} \right). \label{eq: mc bias lb no ss}
\end{align}
Next, we establish lower bounds on the absolute value of each of the first four terms in (\ref{eq: mc bias lb no ss}). 

\begin{itemize}
    \item Case 1: $k_1^{-(\alpha + \beta)/d} \gg \frac{k_2}{n}$. Choose $p$ and $b$ as given in (\ref{eq: p lb}) and (\ref{eq: b lb}). Then, recall from the double sample splitting case that
\begin{align*}
    \left| \int \Pi(p|V_{k_1}^{\perp})(\bx) \Pi(b|V_{k_1}^{\perp})(\bx) d\bx \right| & \gtrsim k_1^{-(\alpha + \beta)/d}.
\end{align*}
Recall from the proof of the upper bound on the bias that
\begin{align*}
    \left| \int r_{n, k_1, k_2}(\bx) d\bx \right| & \lesssim \frac{k_1}{n}  \ll k_1^{-(\alpha + \beta)/d} \\
    \frac{1}{n} \left| \E_{\bbP} \left( A_i K_{V_{k_1}}(\bX_i, \bX_i) \hat{b}^{(-i)}_{k_2}(\bX_i) \right) \right| & \lesssim \frac{k_1}{n} \ll k_1^{-(\alpha + \beta)/d} \\
    \frac{1}{n} \left| \E_{\bbP} \left( Y_i K_{V_{k_2}}(\bX_i, \bX_i) \hat{p}^{(-i)}_{k_1}(\bX_i) \right) \right| & \lesssim \frac{k_2}{n} \ll k_1^{-(\alpha + \beta)/d}.
\end{align*}
Therefore, we have that
\begin{equation*}
    \left| \E_\bbP\left(\hat{\psi}_{k_1, k_2}^{\mathrm{MC}} - \psi(\bbP)\right) \right| \gtrsim k_1^{-(\alpha + \beta)/d}.
\end{equation*}

\item Case 2: $k_1^{-(\alpha + \beta)/d} \lesssim \frac{k_2}{n}$. Choose $p(\bx) = b(\bx) = \epsilon > 0$ and $\E_{\bbP}(AY | \bx) = 0$ for all $\bx \in [0,1]^d$. Since $p,b \in V_{k_1}$ holds by Property P3 (Appendix \ref{sec: wavelets}), we have that
\begin{equation*}
    \int \Pi(p|V_{k_1}^{\perp})(\bx) \Pi(b|V_{k_1}^{\perp})(\bx) d\bx = 0.
\end{equation*}
Since $\E_{\bbP}(AY|\bx) = 0$, we have that
\begin{align*}
    \int r_{n, k_1, k_2}(\bx) d\bx & = \frac{1}{n}  \int \E_{\bbP,1}\left[ AY  K_{V_{k_1}}(\bX, \bx)K_{V_{k_2}}(\bX, \bx) \right] d\bx \\
    & = 0.
\end{align*}
Recall from the proof of the upper bound on the bias that
\begin{align*}
    \frac{1}{n} \E_{\bbP} \left( Y_i K_{V_{k_2}}(\bX_i, \bX_i) \hat{p}^{(-i)}_{k_1}(\bX_i) \right) & = \frac{\epsilon^2}{n} \int K_{V_{k_2}}(\bx, \bx)  d\bx + O\left(\frac{k_2}{n^2}\right) \\
    \frac{1}{n} \E_{\bbP} \left( A_i K_{V_{k_1}}(\bX_i, \bX_i) \hat{b}^{(-i)}_{k_2}(\bX_i) \right) & = \frac{\epsilon^2}{n} \int K_{V_{k_1}}(\bx, \bx)  d\bx + O\left(\frac{k_1}{n^2}\right).
\end{align*}
Therefore, we can write
\begin{align*}
    \E_\bbP\left(\hat{\psi}_{k_1, k_2}^{\mathrm{MC}} - \psi(\bbP)\right) & = - \frac{\epsilon^2}{n} \int K_{V_{k_2}}(\bx, \bx)  d\bx - \frac{\epsilon^2}{n} \int K_{V_{k_1}}(\bx, \bx)  d\bx \\
    & \quad + O\left(\frac{1}{n}\right) + O\left(\frac{k_1k_2}{n^2}\right).
\end{align*}
Recall from Lemma \ref{lem: exp kernel single same} that
\begin{equation*}
    \int K_{V_{k}}(\bx, \bx)  d\bx \asymp k
\end{equation*}
and $K_{V_{k}}(\bx, \bx) \geq 0$ for all $\bx \in [0, 1]^d$. Therefore, 
\begin{equation*}
    \left| \E_\bbP\left(\hat{\psi}_{k_1, k_2}^{\mathrm{MC}} - \psi(\bbP)\right) \right| \gtrsim \frac{k_2}{n}.
\end{equation*}

\end{itemize}

\subsection{First-order bias-corrected estimator} 

\subsubsection{Proof of the upper bound}

Let $\bbP \in \cP_{(\alpha,\beta)}$ be arbitrary.\\

\noindent \textbf{Bounding the bias:}

We follow a similar approach used for analyzing the Monte Carlo-based plug-in estimator in the no sample splitting case. Using (\ref{eq: phat minus i}) and (\ref{eq: bhat minus i}), we can express $\hat{\psi}^{\mathrm{IF}}_{k_1, k_2}$ as
\begin{align}
    \hat{\psi}^{\mathrm{IF}}_{k_1, k_2} & = \frac{1}{n} \sum_{i \in \D_1} (A_i - \hat{p}^{(-i)}_{k_1}(\bX_i))(Y_i - \hat{b}^{(-i)}_{k_2}(\bX_i)) \nonumber \\
    & \quad - \frac{1}{n^2} \sum_{i \in \D_1} \left( \frac{A_i K_{V_{k_1}}(\bX_i,\bX_i)}{f(\bX_i)} \right)(Y_i - \hat{b}^{(-i)}_{k_2}(\bX_i)) \nonumber \\
    & \quad - \frac{1}{n^2} \sum_{i \in \D_1} (A_i - \hat{p}^{(-i)}_{k_1}(\bX_i))\left( \frac{Y_i K_{V_{k_2}}(\bX_i,\bX_i)}{f(\bX_i)} \right) \nonumber \\
    & \quad + \frac{1}{n^3} \sum_{i \in \D_1} \left( \frac{A_i K_{V_{k_1}}(\bX_i,\bX_i)}{f(\bX_i)} \right)\left( \frac{Y_i K_{V_{k_2}}(\bX_i,\bX_i)}{f(\bX_i)} \right). \label{eq: if no ss rep}
\end{align}

We first analyze the expected value of the first term in (\ref{eq: if no ss rep}). We will show that the expectation of this term is the same as the expectation of the double sample split estimator plus a remainder of $O\left( \frac{k_1 \wedge k_2}{n} \right)$. We have that
\begin{align}
    & \E_{\bbP} \left[ \frac{1}{n} \sum_{i \in \D_1} (A_i - \hat{p}^{(-i)}_{k_1}(\bX_i))(Y_i - \hat{b}^{(-i)}_{k_2}(\bX_i)) \right] \nonumber \\
    & \quad = \E_{\bbP} \left[ (A_i - \hat{p}^{(-i)}_{k_1}(\bX_i))(Y_i - \hat{b}^{(-i)}_{k_2}(\bX_i)) \right] \nonumber \\
    & \quad = \E_{\bbP}(AY) - \E_{\bbP}[\hat{p}^{(-i)}_{k_1}(\bX_i)Y_i] - \E_{\bbP}[A_i\hat{b}^{(-i)}_{k_2}(\bX_i)] + \E_{\bbP}[\hat{p}^{(-i)}_{k_1}(\bX_i)\hat{b}^{(-i)}_{k_2}(\bX_i)] \nonumber \\
    & \quad = \E_{\bbP}(AY) - \int \E_{\bbP}(\hat{p}^{(-i)}_{k_1}(\bx))b(\bx)f(\bx)d\bx - \int p(\bx) \E_{\bbP}[\hat{b}^{(-i)}_{k_2}(\bx)] f(\bx) d\bx \nonumber \\
    & \quad \quad + \int \E_{\bbP}[\hat{p}^{(-i)}_{k_1}(\bx)\hat{b}^{(-i)}_{k_2}(\bx)]f(\bx)d\bx. \label{eq: if bias no ss term1}
\end{align}
Observe that
\begin{align}
    \int \E_{\bbP}(\hat{p}^{(-i)}_{k_1}(\bx))b(\bx)f(\bx)d\bx & = \int \Pi(p|V_{k_1})(\bx)b(\bx)f(\bx)d\bx + O\left(\frac{1}{n}\right) \label{eq: if bias no ss temp1} \\
    \int p(\bx) \E_{\bbP}[\hat{b}^{(-i)}_{k_2}(\bx)] f(\bx) d\bx & = \int p(\bx) \Pi(b|V_{k_2})(\bx) f(\bx) d\bx + O\left(\frac{1}{n}\right). \label{eq: if bias no ss temp2} 
\end{align}
Moreover, recall from the proof of the upper bound of the bias of $\hat{\psi}^{\mathrm{MC}}_{k_1, k_2}$ (no sample splitting case) that
\begin{align}
        & \int \E_{\bbP}[\hat{p}^{(-i)}_{k_1}(\bx)\hat{b}^{(-i)}_{k_2}(\bx)]f(\bx)d\bx \nonumber \\
        & \quad = \int \Pi(p|V_{k_1})(\bx) \Pi(b|V_{k_2})(\bx) f(\bx) d\bx +  \int r_{n,k_1,k_2}(\bx)f(\bx) d\bx \nonumber \\
        & \quad = \int \Pi(p|V_{k_1})(\bx) \Pi(b|V_{k_2})(\bx) f(\bx) d\bx +  O\left( \frac{k_1 \wedge k_2}{n} \right). \label{eq: if bias no ss temp3} 
\end{align}
Plugging in (\ref{eq: if bias no ss temp1}), (\ref{eq: if bias no ss temp2}), and (\ref{eq: if bias no ss temp3}) into (\ref{eq: if bias no ss term1}), we can see that the expected value of the first term in (\ref{eq: if no ss rep}) equals expected value of the double sample split first-order estimator plus a remainder of $O\left( \frac{k_1 \wedge k_2}{n} \right)$. Therefore,
\begin{align*}
     & \E_{\bbP} \left[ \frac{1}{n} \sum_{i \in \D_1} (A_i - \hat{p}^{(-i)}_{k_1}(\bX_i))(Y_i - \hat{b}^{(-i)}_{k_2}(\bX_i)) \right] \\
     & \quad = \psi(\bbP) + O\left( (k_1 \vee k_2)^{-(\alpha + \beta)/d} \right) + O\left( \frac{k_1 \wedge k_2}{n} \right).
\end{align*}

Next, we bound the expectation of the second term in (\ref{eq: if no ss rep}). We have that
\begin{align*}
    & \left| \E_{\bbP} \left( \frac{1}{n^2} \sum_{i \in \D_1} \left( \frac{A_i K_{V_{k_1}}(\bX_i,\bX_i)}{f(\bX_i)} \right)(Y_i - \hat{b}^{(-i)}_{k_2}(\bX_i)) \right) \right| \\
    & \quad = \frac{1}{n} \left| \E_{\bbP} \left[  \left( \frac{A_i K_{V_{k_1}}(\bX_i,\bX_i)}{f(\bX_i)} \right)(Y_i - \hat{b}^{(-i)}_{k_2}(\bX_i)) \right] \right| \\
    & \quad \leq \frac{1}{n} \left| \E_{\bbP} \left[   \frac{A_i Y_i K_{V_{k_1}}(\bX_i,\bX_i)}{f(\bX_i)} \right] \right| + \frac{1}{n} \left| \E_{\bbP} \left[  \frac{A_i K_{V_{k_1}}(\bX_i,\bX_i)\hat{b}^{(-i)}_{k_2}(\bX_i)}{f(\bX_i)}  \right] \right|. 
\end{align*}
Recall from the proof of the upper bound of the bias of $\hat{\psi}^{\mathrm{MC}}_{k_1, k_2}$ (no sample splitting case) that
\begin{align*}
    \frac{1}{n} \left| \E_{\bbP} \left[   \frac{A_i Y_i K_{V_{k_1}}(\bX_i,\bX_i)}{f(\bX_i)} \right] \right| \lesssim \frac{k_1}{n}.
\end{align*}
It follows in the exact same manner that 
\begin{equation*}
    \frac{1}{n} \left| \E_{\bbP} \left[  \frac{A_i K_{V_{k_1}}(\bX_i,\bX_i)\hat{b}^{(-i)}_{k_2}(\bX_i)}{f(\bX_i)}  \right] \right| \lesssim \frac{k_1}{n}.
\end{equation*}

Similarly, the expectation of the third term in (\ref{eq: if no ss rep}) is bounded by $O(\frac{k_2}{n})$. Finally, recall from the proof of the upper bound of the bias of $\hat{\psi}^{\mathrm{MC}}_{k_1, k_2}$ (no sample splitting case) that the expectation of the fourth term in (\ref{eq: if no ss rep}) is bounded by $O(\frac{k_1k_2}{n^2})$. 

Therefore,
\begin{equation*}
    \left|\E_\bbP\left(\hat{\psi}_{k_1, k_2}^{\mathrm{IF}} - \psi(\bbP)\right)\right| \lesssim (k_1 \vee k_2)^{-(\alpha + \beta)/d} + \frac{k_1 \vee k_2}{n} + \frac{k_1k_2}{n^2}.
\end{equation*}
The result then follows when noting that $\frac{k_1k_2}{n^2} \ll \frac{k_1 \vee k_2}{n}$ when $k_1, k_2 \ll n$. \\

\noindent \textbf{Bounding the variance:}

We have that
\begin{align*}
    \Var_{\bbP}(\hat{\psi}^{\mathrm{IF}}_{k_1, k_2}) & \leq 4 \Var_{\bbP} \left( \frac{1}{n}\sum_{i \in \D_1} A_i Y_i \right) + 4 \Var_{\bbP} \left( \frac{1}{n}\sum_{i \in \D_1} A_i \hat{b}_{k_2}(\bX_i) \right) \\
    & \quad + 4 \Var_{\bbP} \left( \frac{1}{n}\sum_{i \in \D_1} \hat{p}_{k_1}(\bX_i) Y_i \right) + 4 \Var_{\bbP} \left( \frac{1}{n}\sum_{i \in \D_1} \hat{p}_{k_1}(\bX_i) \hat{b}_{k_2}(\bX_i) \right).
\end{align*}
Recall that $\Var_{\bbP} \left( \frac{1}{n}\sum_{i \in \D_1} A_i Y_i \right) \lesssim \frac{1}{n}$. Additionally, recall from the proof of the upper bound of the variance of $\hat{\psi}^{\mathrm{MC}}_{k_1, k_2}$ (no sample splitting case) that
\begin{equation*}
    \Var_{\bbP} \left( \frac{1}{n}\sum_{i \in \D_1} \hat{p}_{k_1}(\bX_i) \hat{b}_{k_2}(\bX_i) \right) \lesssim \frac{1}{n} + \frac{k_{\mathrm{max}}}{n^2} + \frac{k_{\mathrm{max}}^2}{n^3} + \frac{k_{\mathrm{min}} k_{\mathrm{max}}^2}{n^4} + \frac{k_{\mathrm{min}}^2 k_{\mathrm{max}}^2}{n^5}.
\end{equation*}
Next, we bound $\Var_{\bbP} \left( \frac{1}{n}\sum_{i \in \D_1} A_i \hat{b}_{k_2}(\bX_i) \right)$. Using (\ref{eq: bhat minus i}), 
\begin{align*}
    \Var_{\bbP} \left( \frac{1}{n}\sum_{i \in \D_1} A_i \hat{b}_{k_2}(\bX_i) \right) & \leq 2 \Var_{\bbP} \left( \frac{1}{n}\sum_{i \in \D_1} A_i \hat{b}^{(-i)}_{k_2}(\bX_i) \right) \\
    & \quad + 2\Var_{\bbP} \left( \frac{1}{n^2}\sum_{i \in \D_1}  \frac{A_iY_i}{f(\bX_i)} K_{V_{k_2}}(\bX_i, \bX_i) \right).
\end{align*}
Then,
\begin{align*}
    \Var_{\bbP} \left( \frac{1}{n^2}\sum_{i \in \D_1}  \frac{A_iY_i}{f(\bX_i)} K_{V_{k_2}}(\bX_i, \bX_i) \right) & = \frac{1}{n^3} \Var_{\bbP} \left( \frac{A_iY_i}{f(\bX_i)} K_{V_{k_2}}(\bX_i, \bX_i) \right) \\
    & \leq \frac{1}{n^3} \E_{\bbP} \left[ \left( \frac{A_iY_i}{f(\bX_i)} K_{V_{k_2}}(\bX_i, \bX_i) \right)^2 \right] \\
    & \lesssim \frac{1}{n^3} \int K_{V_{k_2}}(\bx, \bx)^2 d\bx \\
    & \lesssim \frac{k_2^2}{n^3} \quad (\text{by Lemma \ref{lem: exp kernel single same}}).
\end{align*}
We next show that
\begin{equation} \label{eq: no ss nr var helper}
    \Var_{\bbP} \left( \frac{1}{n}\sum_{i \in \D_1} A_i \hat{b}^{(-i)}_{k_2}(\bX_i) \right) \lesssim \frac{1}{n} + \frac{k_2}{n^2}.
\end{equation}
To see that (\ref{eq: no ss nr var helper}) holds, it suffices to show that
\begin{align*} 
    \Var_{\bbP}(A_i\hat{b}_{k_2}^{(-i)}(\bX_i)) & \lesssim 1 + \frac{k_2}{n}, \qquad \forall i \in \D_1  \\
    \left| \Cov_{\bbP}(A_i\hat{b}_{k_2}^{(-i)}(\bX_i), A_j\hat{b}_{k_2}^{(-j)}(\bX_j)) \right| & \lesssim \frac{1}{n} + \frac{k_2}{n^2}, \qquad \forall i, j \in \D_1, i \neq j .
\end{align*}
which follow from the same steps used to show (\ref{eq: no ss mc var pt1a}) and (\ref{eq: no ss mc var pt2a}) respectively (appearing in the proof of the upper bound of the variance of $\hat{\psi}^{\mathrm{MC}}_{k_1, k_2}$ in the no sample splitting case).

By exchanging the roles of $A$ and $Y$, we also have that
\begin{equation*}
    \Var_{\bbP} \left( \frac{1}{n}\sum_{i \in \D_1} \hat{p}_{k_1}(\bX_i) Y_i \right) \lesssim \frac{1}{n} + \frac{k_1}{n^2} + \frac{k_1^2}{n^3}
\end{equation*}
which completes the proof of the variance bound.

\subsubsection{Proof of the lower bound}

Let $\bbP \in \cP_{(\alpha,\beta)}$. Recall from the proof of the upper bound that
\begin{align*}
    \E_{\bbP}(\hat{\psi}^{\mathrm{IF}}_{k_1, k_2}) & =  \E_{\bbP}\left[ (A_i - \hat{p}^{(-i)}_{k_1}(\bX_i))(Y_i - \hat{b}^{(-i)}_{k_2}(\bX_i)) \right]  \\
    & \quad - \frac{1}{n} \E_{\bbP}\left[ \left( \frac{A_i K_{V_{k_1}}(\bX_i,\bX_i)}{f(\bX_i)} \right)(Y_i - \hat{b}^{(-i)}_{k_2}(\bX_i)) \right]  \\
    & \quad - \frac{1}{n} \E_{\bbP}\left[ (A_i - \hat{p}^{(-i)}_{k_1}(\bX_i))\left( \frac{Y_i K_{V_{k_2}}(\bX_i,\bX_i)}{f(\bX_i)} \right) \right] \\
    & \quad + O\left( \frac{k_1k_2}{n^2} \right) 
\end{align*}
where
\begin{align*}
    \E_{\bbP}\left[ (A_i - \hat{p}^{(-i)}_{k_1}(\bX_i))(Y_i - \hat{b}^{(-i)}_{k_2}(\bX_i)) \right] & = \psi(\bbP) + \int \Pi(p|V_{k_1}^{\perp})(\bx) \Pi(b|V_{k_2}^{\perp})(\bx) d\bx \\
    & \quad + \int r_{n, k_1, k_2}(\bx)f(\bx) d\bx + O\left(\frac{1}{n}\right).
\end{align*}
Let $\bX \sim \text{Uniform}([0,1]^d)$. Suppose without loss of generality that $k_1 \leq k_2$ for $n$ sufficiently large. Then, we can express the bias of $\hat{\psi}^{\mathrm{IF}}_{k_1, k_2}$ by
\begin{align} 
    \E_\bbP\left(\hat{\psi}_{k_1, k_2}^{\mathrm{IF}} - \psi(\bbP)\right) & = \int \Pi(p|V_{k_1}^{\perp})(\bx) \Pi(b|V_{k_2}^{\perp})(\bx) d\bx + \int r_{n, k_1, k_2}(\bx) d\bx \nonumber \\
    & \quad - \frac{1}{n} \E_{\bbP} \left( A_i K_{V_{k_1}}(\bX_i, \bX_i) (Y_i - \hat{b}^{(-i)}_{k_2}(\bX_i)) \right) \nonumber \\
    & \quad - \frac{1}{n} \E_{\bbP} \left( (A_i - \hat{p}^{(-i)}_{k_1}(\bX_i)) Y_i K_{V_{k_2}}(\bX_i, \bX_i)  \right) \nonumber \\
    & \quad + O\left( \frac{1}{n} \right) + O\left( \frac{k_1k_2}{n^2} \right). \label{eq: if bias lb no ss}
\end{align}

Next, we establish lower bounds on absolute value of each of the first four terms in (\ref{eq: if bias lb no ss}). 

\begin{itemize}
    \item Case 1: $k_2^{-(\alpha + \beta)/d} \gg \frac{k_2}{n}$. Choose $p$ and $b$ as given in (\ref{eq: p lb}) and (\ref{eq: b lb}). Then, recall from the double sample splitting case that
\begin{align*}
    \left| \int \Pi(p|V_{k_1}^{\perp})(\bx) \Pi(b|V_{k_2}^{\perp})(\bx) d\bx \right| & \gtrsim k_2^{-(\alpha + \beta)/d}.
\end{align*}
Recall from the proof of the upper bound on the bias that
\begin{align*}
    \left| \int r_{n, k_1, k_2}(\bx) d\bx \right| & \lesssim \frac{k_1}{n}  \ll k_2^{-(\alpha + \beta)/d} \\
    \left| \frac{1}{n} \E_{\bbP} \left( A_i K_{V_{k_1}}(\bX_i, \bX_i) (Y_i - \hat{b}^{(-i)}_{k_2}(\bX_i)) \right) \right| & \lesssim \frac{k_1}{n} \ll k_2^{-(\alpha + \beta)/d} \\
    \left| \frac{1}{n} \E_{\bbP} \left( (A_i - \hat{p}^{(-i)}_{k_1}(\bX_i)) Y_i K_{V_{k_2}}(\bX_i, \bX_i)  \right) \right| & \lesssim \frac{k_2}{n} \ll k_2^{-(\alpha + \beta)/d}.
\end{align*}
Therefore, we have that
\begin{equation*}
    \left| \E_\bbP\left(\hat{\psi}_{k_1, k_2}^{\mathrm{IF}} - \psi(\bbP)\right) \right| \gtrsim k_2^{-(\alpha + \beta)/d}.
\end{equation*}

\item Case 2: $k_2^{-(\alpha + \beta)/d} \lesssim \frac{k_2}{n}$. Choose $p(\bx) = b(\bx) = \epsilon > 0$ and $\E_{\bbP}(AY | \bx) = 0$ for all $\bx \in [0,1]^d$. It follows from the same arguments in proof of the lower bound of $\hat{\psi}_{k_1, k_2}^{\mathrm{MC}}$ (no sample splitting, Case 2) that 
\begin{align*}
    \int \Pi(p|V_{k_1}^{\perp})(\bx) \Pi(b|V_{k_2}^{\perp})(\bx) d\bx & = 0 \\
    \int r_{n, k_1, k_2}(\bx) d\bx & = 0 \\
    \frac{1}{n} \E_{\bbP} \left( A_i K_{V_{k_1}}(\bX_i, \bX_i) (Y_i - \hat{b}^{(-i)}_{k_2}(\bX_i)) \right) & = - \frac{\epsilon^2}{n} \int K_{V_{k_1}}(\bx, \bx)  d\bx + O\left(\frac{k_1}{n^2}\right) \\
    \frac{1}{n} \E_{\bbP} \left( (A_i - \hat{p}^{(-i)}_{k_1}(\bX_i)) Y_i K_{V_{k_2}}(\bX_i, \bX_i)  \right) & = - \frac{\epsilon^2}{n} \int K_{V_{k_2}}(\bx, \bx)  d\bx + O\left(\frac{k_2}{n^2}\right).
\end{align*}
Therefore, we can write
\begin{align*}
    \E_\bbP\left(\hat{\psi}_{k_1, k_2}^{\mathrm{IF}} - \psi(\bbP)\right) & = \frac{\epsilon^2}{n} \int K_{V_{k_1}}(\bx, \bx)  d\bx +  \frac{\epsilon^2}{n} \int K_{V_{k_2}}(\bx, \bx)  d\bx \\
     & \quad + O\left( \frac{1}{n} \right) + O\left( \frac{k_1k_2}{n^2} \right).
\end{align*}
Recall from Lemma \ref{lem: exp kernel single same} that
\begin{equation*}
    \int K_{V_{k}}(\bx, \bx)  d\bx \asymp k
\end{equation*}
and $K_{V_{k}}(\bx, \bx) \geq 0$ for all $\bx \in [0, 1]^d$. Therefore,
\begin{equation*}
    \left| \E_\bbP\left(\hat{\psi}_{k_1, k_2}^{\mathrm{IF}} - \psi(\bbP)\right) \right| \gtrsim \frac{k_2}{n}.
\end{equation*}
\end{itemize}

\section{Proof of Theorem 5} \label{sec: proof known f nr no ss}

As in the proof of Theorem 1, we write out the proofs of the upper bounds when $f$ is such that 
\begin{align*}
    (i) & \quad  f(\bx) \in [M_1, M_2] \quad \forall \bx \in [0, 1]^d\\
    (ii) & \quad f \in H(\gamma, M)
\end{align*}
where $\gamma \geq \alpha \vee \beta$ and $M, M_1, M_2 \in \mathbb{R}^{+}$ are known constants. Recall that we use the approximate wavelet projection estimators of the nuisance functions given by (\ref{eq: wavelet estimator of p}) and (\ref{eq: wavelet estimator of b}) in this case.

The Newey and Robins plug-in estimator that we analyze is given by
\begin{equation*}
   \hat{\psi}_{k}^{\mathrm{NR}} = \frac{1}{n} \sum_{i \in \D_1} A_i(Y_i - \hat{b}_k^{(1)}(\bX_i)).
\end{equation*}

Throughout the proof, we adopt the additional notation introduced in the proof of Theorem 4 (see Appendix \ref{sec: proof known f no ss}).

\subsection{Proof of the upper bound}

Let $\bbP \in \cP_{(\alpha,\beta)}$ be arbitrary. \\

\noindent \textbf{Bounding the bias:}

Using (\ref{eq: bhat minus i}), we can express $\hat{\psi}^{\mathrm{NR}}_k$ as
\begin{equation}
    \hat{\psi}_{k}^{\mathrm{NR}} = \frac{1}{n} \sum_{i \in \D_1} A_i(Y_i - \hat{b}_k^{(-i)}(\bX_i)) - \frac{1}{n} \sum_{i \in \D_1} \frac{A_iY_i}{nf(\bX_i)} K_{V_k}(\bX_i, \bX_i). \label{eq: nr no ss rep}
\end{equation}
We first analyze the expected value of the first term in (\ref{eq: nr no ss rep}). We have that
\begin{align*}
    \E_{\bbP}\left[ \frac{1}{n} \sum_{i \in \D_1} A_i(Y_i - \hat{b}_k^{(-i)}(\bX_i)) \right] & = \E_{\bbP}\left[ A_i(Y_i - \hat{b}_k^{(-i)}(\bX_i)) \right] \\
    & = \psi(\bbP) + \E_{\bbP}\left[ p(\bX_i)(b(\bX_i) - \hat{b}_k^{(-i)}(\bX_i)) \right] \\
    & = \psi(\bbP) + \int p(\bx)\left(b(\bx) - \frac{n-1}{n}\Pi(b | V_{k})(\bx)\right) f(\bx) d\bx \\
    & = \psi(\bbP) + \int p(\bx) (b(\bx) - \Pi(b | V_{k})(\bx)) f(\bx) d\bx \\
    & \quad + O\left(\frac{1}{n}\right)
\end{align*}
which recall is equal to the expected value of the sample split Newey and Robins plug-in estimator plus a remainder of $O(\frac{1}{n})$. Therefore, 
\begin{equation*}
    \E_{\bbP}\left[ \frac{1}{n} \sum_{i \in \D_1} A_i(Y_i - \hat{b}_k^{(-i)}(\bX_i)) \right] = \psi(\bbP) + O\left( k^{-(\alpha + \beta) / d}\right) + O\left( \frac{1}{n} \right). 
\end{equation*}

Next, we bound the expectation of the second term. Recall from the proof of the upper bound of the bias of $\hat{\psi}_{k_1,k_2}^{\mathrm{IF}}$ (no sample splitting case) that
\begin{equation*}
    \E_{\bbP} \left[ \frac{1}{n} \sum_{i \in \D_1} \frac{A_iY_i}{nf(\bX_i)} K_{V_k}(\bX_i, \bX_i) \right] \lesssim \frac{k}{n}.
\end{equation*}

Therefore,
\begin{equation*}
    \left|\E_\bbP\left(\hat{\psi}_{k}^{\mathrm{NR}} - \psi(\bbP)\right)\right| \lesssim k^{-(\alpha + \beta)/d} + \frac{k}{n}.
\end{equation*}

\noindent \textbf{Bounding the variance:}

We have that
\begin{equation*}
    \Var_{\bbP}(\hat{\psi}_{k}^{\mathrm{NR}}) \leq 2 \Var_{\bbP} \left( \frac{1}{n}\sum_{i \in \D_1} A_i Y_i \right) + 2 \Var_{\bbP} \left( \frac{1}{n}\sum_{i \in \D_1} A_i \hat{b}_{k}(\bX_i) \right).
\end{equation*}
Then, recall from the proof of the upper bound of the variance of $\hat{\psi}_{k_1,k_2}^{\mathrm{IF}}$ (no sample splitting case) that
\begin{align*}
    \Var_{\bbP} \left( \frac{1}{n}\sum_{i \in \D_1} A_i Y_i \right) & \lesssim \frac{1}{n} \\
    \Var_{\bbP} \left( \frac{1}{n}\sum_{i \in \D_1} A_i \hat{b}_{k}(\bX_i) \right) & \lesssim \frac{1}{n} + \frac{k}{n^2} + \frac{k^2}{n^3}. 
\end{align*}

\subsection{Proof of the lower bound}

Let $\bbP \in \cP_{(\alpha,\beta)}$. Recall from the proof of the upper bound that
\begin{align*}
    \E_{\bbP}(\hat{\psi}_{k}^{\mathrm{NR}}) - \psi(\bbP) & =  \int p(\bx) (b(\bx) - \Pi(b | V_{k})(\bx)) f(\bx) d\bx \\
    & \quad - \E_{\bbP}\left[\frac{1}{n} \sum_{i \in \D_1} \frac{A_iY_i}{nf(\bX_i)} K_{V_k}(\bX_i, \bX_i)\right] + O\left(\frac{1}{n}\right).
\end{align*}
Let $\bX \sim \text{Uniform}([0,1]^d)$ and choose $p$ and $b$ as given in (\ref{eq: p lb}) and (\ref{eq: b lb}). Then, recall from the single sample splitting case that
\begin{align*}
    \left| \int p(\bx) (b(\bx) - \Pi(b | V_{k})(\bx)) d\bx \right| \asymp k^{-(\alpha + \beta)/d}.
\end{align*}
Further, let $\E_{\bbP}(AY | \bx) = c > 0$. Then, 
\begin{align*}
    \E_{\bbP} \left[ \frac{1}{n} \sum_{i \in \D_1} \frac{A_iY_i}{n} K_{V_k}(\bX_i, \bX_i) \right] & = \frac{1}{n}\E_{\bbP} \left[ A_iY_i K_{V_k}(\bX_i, \bX_i) \right]  \\
    & = \frac{1}{n} \int \E_{\bbP}(A_iY_i|\bx) K_{V_k}(\bx, \bx) d\bx  \\
    & = \frac{c}{n} \int K_{V_k}(\bx, \bx) d\bx \\
    & \asymp \frac{k}{n} \quad (\text{by Lemma \ref{lem: exp kernel single same}}).
\end{align*}
This completes the proof of the lower bound provided that $k^{-(\alpha + \beta)/d}$ and $\frac{k}{n}$ are not of the same order. When $k$ is such that $k^{-(\alpha + \beta)/d} \asymp \frac{k}{n}$, we can choose $\epsilon$ (appearing in (\ref{eq: p lb}) and (\ref{eq: b lb}) which define $p$ and $b$) sufficiently small so that
\begin{equation*}
    \left| \int p(\bx) (b(\bx) - \Pi(b | V_{k})(\bx)) d\bx - \frac{c}{n} \int K_{V_k}(\bx, \bx) d\bx  \right| \gtrsim \frac{k}{n}.
\end{equation*}

\section{Proof of Theorem \ref{theorem: plugin unknown f}} \label{sec: proof unknown f}

Throughout the proof, we will use the following notation. For $\bx \in [0, 1]^d$, let  
\begin{align}
    p_{k_1}(p(\bx)) &:= \E_{\bbP, 1, 2} [\hat{p}^{(1, 2)}_{k_1}(\bx)] =  \E_{\bbP,2} \left[ \Pi\left( \frac{pf}{\hat{f}^{(2)}} | V_{k_1}\right)(\bx) \right] \label{eq: p projection unknown f}\\ 
    p_{k_2}(b(\bx)) &:= \E_{\bbP, 3, 4} [\hat{b}^{(3, 4)}_{k_2}(\bx)]  = \E_{\bbP,4} \left[ \Pi\left( \frac{bf}{\hat{f}^{(4)}} | V_{k_2}\right)(\bx) \right]. \label{eq: b projection unknown f}
\end{align}
Moreover, it will be convenient in our proofs in the unknown $f$ case to define estimators of $p$ and $b$ which assume that $f$ is known. With a slight abuse of notation, to distinguish whether the nuisance function estimators assume that $f$ is known or unknown, we use the tilde symbol in the known $f$ case, i.e.,
\begin{align*}
    \tilde{p}^{(1)}_{k_1}(\bx) = \frac{1}{n} \sum_{i \in \D_1} \frac{A_i K_{V_{k_1}}(\bX_i,\bx)}{f(\bX_i)} \\
    \tilde{b}^{(3)}_{k_2}(\bx) = \frac{1}{n} \sum_{i \in \D_3} \frac{Y_i K_{V_{k_2}}(\bX_i,\bx)}{f(\bX_i)}.
\end{align*}

\subsection{Integral-based plug-in estimator} 

\subsubsection{Proof of the upper bound}

Let $\bbP \in \cP_{(\alpha,\beta,\gamma)}$ be arbitrary. \\

\noindent \textbf{Bounding the bias:}

Observe that
\begin{align*}
    \E_\bbP(\hat{\psi}_{k_1, k_2}^{\mathrm{INT}}) & = \E_\bbP\left( A Y \right) -  \int  p_{k_1}(p(\bx)) p_{k_2}(b(\bx)) \E_{\bbP,6}(\hat{f}^{(6)}(\bx)) d\bx  \\
    & = \E_\bbP\left( A Y \right) -  \int  p_{k_1}(p(\bx)) p_{k_2}(b(\bx)) f(\bx) d\bx \\
    & \quad + \int  p_{k_1}(p(\bx)) p_{k_2}(b(\bx)) [f(\bx) - \E_{\bbP,6}(\hat{f}^{(6)}(\bx)) ]d\bx  \\
    & = \E_\bbP\left( A Y \right) -  \int p(\bx)b(\bx) f(\bx) d\bx \\
    & \quad -  \int \left( p_{k_1}(p(\bx)) -  p(\bx) \right) \left( p_{k_2}(b(\bx)) - b(\bx)  \right) f(\bx)  d\bx   \\ & \quad -  \int p(\bx) f(\bx) \left( p_{k_2}(b(\bx)) - b(\bx)  \right)  d\bx - \int  b(\bx) f(\bx) \left( p_{k_1}(p(\bx)) -  p(\bx) \right)   d\bx \\ & \quad + \int  p_{k_1}(p(\bx)) p_{k_2}(b(\bx)) [f(\bx) - \E_{\bbP,6}(\hat{f}^{(6)}(\bx)) ]d\bx  
\end{align*}
which implies
\begin{align*}
    \left|\E_\bbP\left(\hat{\psi}_{k_1, k_2}^{\mathrm{INT}} - \psi(\bbP)\right)\right|  & \leq  \left| \int \left(p_{k_1}(p(\bx)) -  p(\bx) \right) \left( p_{k_2}(b(\bx)) - b(\bx)  \right)  f(\bx)  d\bx \right |  \\
    & \quad + \left| \int p(\bx) f(\bx) \left( p_{k_2}(b(\bx)) - b(\bx)  \right)  d\bx \right|  \\
    & \quad + \left| \int  b(\bx) f(\bx) \left( p_{k_1}(p(\bx)) -  p(\bx) \right)   d\bx  \right| \\ & \quad +\left|  \int  p_{k_1}(p(\bx)) p_{k_2}(b(\bx)) [f(\bx) - \E_{\bbP,6}(\hat{f}^{(6)}(\bx)) ]d\bx \right|.
\end{align*}

By boundedness of $f$ and Lemma \ref{lem: proj dist},
\begin{align*}
    & \left| \int \left(p_{k_1}(p(\bx)) -  p(\bx) \right) \left( p_{k_2}(b(\bx)) - b(\bx)  \right)  f(\bx)  d\bx \right | \\
    & \quad \lesssim \int \| p_{k_1}(p(\cdot)) -  p \|_{\infty} \|  p_{k_2}(b(\cdot)) - b   \|_{\infty}  d\bx \\ 
    &  \quad \lesssim k_1^{-\alpha/d}k_2^{-\beta/d} + \left( \frac{n}{\log n}\right)^{-\frac{\gamma}{2\gamma + d}}.
\end{align*}

Next, we bound $\left| \int  p(\bx) f(\bx) \left( p_{k_2}(b(\bx)) -  b(\bx) \right) d\bx \right|$. We have that
\begin{align*}
    \left| \int  p(\bx) f(\bx) \left( p_{k_2}(b(\bx)) -  b(\bx) \right) d\bx \right| & \leq \left| \int  p(\bx) f(\bx) (p_{k_2}(b(\bx)) - \Pi(b | V_{k_2})(\bx)) d\bx \right| \\
    & \quad  + \left| \int  p(\bx) f(\bx) (\Pi(b | V_{k_2})(\bx) - b(\bx)) d\bx \right|. 
\end{align*}
Recall from the known $f$ case, 
\begin{equation*}
    \left| \int  p(\bx) f(\bx) (\Pi(b | V_{k_2})(\bx) - b(\bx)) d\bx \right|  \lesssim k_2^{-(\alpha + \beta)/d}.
\end{equation*}
By (\ref{eq: p_k proj}) and (\ref{eq: exp proj}), 
\begin{align*}
    & \left| \int  p(\bx) f(\bx) (p_{k_2}(b(\bx)) - \Pi(b | V_{k_2})(\bx)) d\bx \right| \\
    & \quad = \left| \int  p(\bx) f(\bx) \E_{\bbP,4} \left[ \Pi\left( \frac{ b(f - \hat{f}^{(4)})}{\hat{f}^{(4)}} | V_{k_2}\right)(\bx) \right] d\bx \right| \\ 
    & \quad \lesssim \E_{\bbP,4} \left[ \left\| \Pi\left( \frac{ b(f - \hat{f}^{(4)})}{\hat{f}^{(4)}} | V_{k_2}\right) \right\|_{\infty} \right]  \\
    & \quad \lesssim \left( \frac{n}{\log n}\right)^{-\frac{\gamma}{2\gamma + d}}.
\end{align*}
Thus,
\begin{equation*}
    \left| \int  p(\bx) f(\bx) \left( p_{k_2}(b(\bx)) -  b(\bx) \right) d\bx \right| \lesssim  k_2^{-(\alpha + \beta) / d} +  \left( \frac{n}{\log n}\right)^{-\frac{\gamma}{2\gamma + d}}.
\end{equation*}
Similarly, it can be shown in the same manner that
\begin{equation*}
    \left| \int  b(\bx) f(\bx) \left( p_{k_1}(p(\bx)) -  p(\bx) \right)   d\bx \right| \lesssim  k_1^{-(\alpha + \beta) / d} +  \left( \frac{n}{\log n}\right)^{-\frac{\gamma}{2\gamma + d}}.
\end{equation*}

Last, we bound $\left|  \int  p_{k_1}(p(\bx)) p_{k_2}(b(\bx)) [f(\bx) - \E_{\bbP,6}(\hat{f}^{(6)}(\bx)) ]d\bx \right|$. By boundedness of $p$ and $b$ and Lemma \ref{lem: proj dist}, 
\begin{align*}
    \| p_{k_1}(p(\cdot)) \|_{\infty}  & \leq \| p_{k_1}(p(\cdot)) - p \|_{\infty} + \| p \|_{\infty}  \lesssim 1 \\
    \| p_{k_2}(b(\cdot)) \|_{\infty}  & \leq \| p_{k_2}(b(\cdot)) - b \|_{\infty} + \| b \|_{\infty}  \lesssim 1.
\end{align*}
Then,
\begin{align*}
    & \left|  \int  p_{k_1}(p(\bx)) p_{k_2}(b(\bx)) [f(\bx) - \E_{\bbP,6}(\hat{f}^{(6)}(\bx)) ]d\bx \right| \\
    & \quad \lesssim  \| p_{k_1}(p(\cdot))\|_{\infty} \| p_{k_2}(b(\cdot))\|_{\infty} \E_{\bbP,6} \left[ \| f - \hat{f}^{(6)} \|_{\infty} \right] \\
    & \quad \lesssim \left( \frac{n}{\log n}\right)^{-\frac{\gamma}{2\gamma + d}}.
\end{align*}

Therefore, we can bound the bias by
\begin{equation*}
    \left|\E_\bbP\left(\hat{\psi}_{k_1, k_2}^{\mathrm{INT}} - \psi(\bbP)\right)\right| \lesssim (k_1 \wedge k_2)^{-(\alpha + \beta)/d} + \left( \frac{n}{\log n}\right)^{-\frac{\gamma}{2\gamma + d}}.
\end{equation*}

\noindent \textbf{Bounding the variance:}

We bound the variance in a similar manner as the known $f$ case. We have that
\begin{align*}
    \Var_\bbP(\hat{\psi}_{k_1, k_2}^{\mathrm{INT}}) & = \Var_\bbP\left( \frac{1}{n}\sum_{i \in \D_5} A_i Y_i\right) + \Var_\bbP\left( \int \hat{p}^{(1, 2)}_{k_1}(\bx)\hat{b}^{(3, 4)}_{k_2}(\bx) \hat{f}^{(6)}(\bx) d\bx  \right) \\
    & = \frac{1}{n}\Var_\bbP(AY) +  \Var_{\bbP,1,2} \left[ \E_{\bbP,3,4,6}\left(\int \hat{p}^{(1, 2)}_{k_1}(\bx)\hat{b}^{(3, 4)}_{k_2}(\bx) \hat{f}^{(6)}(\bx) d\bx\right)  \right]  \\ & \quad +  \E_{\bbP,1,2} \left[ \Var_{\bbP,3,4,6}\left( \int \hat{p}^{(1, 2)}_{k_1}(\bx)\hat{b}^{(3, 4)}_{k_2}(\bx) \hat{f}^{(6)}(\bx) d\bx  \right)  \right].
\end{align*}

We will show that the second and third terms in the above expression have the following upper bounds

\begin{align}
    \Var_{\bbP,1,2} \left[ \E_{\bbP,3,4,6}\left( \int \hat{p}^{(1, 2)}_{k_1}(\bx)\hat{b}^{(3, 4)}_{k_2}(\bx) \hat{f}^{(6)}(\bx) d\bx \right)  \right] & \lesssim \frac{1}{n} + \left( \frac{n}{\log n}\right)^{-\frac{2\gamma}{2\gamma + d}} \label{eq: var exp v5} \\
    \E_{\bbP,1,2} \left[ \Var_{\bbP,3,4,6} \left( \int \hat{p}^{(1, 2)}_{k_1}(\bx)\hat{b}^{(3, 4)}_{k_2}(\bx) \hat{f}^{(6)}(\bx) d\bx \right)  \right] & \lesssim \frac{1}{n} + \frac{k_1 \wedge k_2}{n^2} + \left( \frac{n}{\log n}\right)^{-\frac{2\gamma}{2\gamma + d}} \label{eq: exp var v5}. 
\end{align}

We first show (\ref{eq: var exp v5}). We have that
\begin{align*}
    & \Var_{\bbP,1,2} \left[ \E_{\bbP,3,4,6}\left( \int \hat{p}^{(1, 2)}_{k_1}(\bx)\hat{b}^{(3, 4)}_{k_2}(\bx) \hat{f}^{(6)}(\bx) d\bx \right)  \right] \\
    & \quad = \Var_{\bbP,1,2} \left[  \int \hat{p}^{(1,2)}_{k_1}(\bx) p_{k_2}(b(\bx)) \E_{\bbP,6}(\hat{f}^{(6)}(\bx)) d\bx \right] \\
    & \quad \leq 2 \Var_{\bbP,1} \bigg[ \int \tilde{p}^{(1)}_{k_1}(\bx) p_{k_2}(b(\bx)) \E_{\bbP,6}(\hat{f}^{(6)}(\bx)) d\bx \bigg] \\ & \quad \quad  + 2\Var_{\bbP,1,2} \bigg[ \int (\hat{p}^{(1,2)}_{k_1}(\bx) - \tilde{p}^{(1)}_{k_1}(\bx)) p_{k_2}(b(\bx)) \E_{\bbP,6}(\hat{f}^{(6)}(\bx)) d\bx  \bigg].
\end{align*}
We can bound the first term as follows
\begin{align*}
    & \Var_{\bbP,1} \bigg[ \int \tilde{p}^{(1)}_{k_1}(\bx) p_{k_2}(b(\bx)) \E_{\bbP,6}(\hat{f}^{(6)}(\bx)) d\bx \bigg] \\
    & \quad = \E_{\bbP,1} \left[ \left( \int \tilde{p}^{(1)}_{k_1}(\bx) p_{k_2}(b(\bx)) \E_{\bbP,6}(\hat{f}^{(6)}(\bx)) d\bx \right)^2 \right] \\ & \quad \quad - \left[ \E_{\bbP,1}  \left( \int \tilde{p}^{(1)}_{k_1}(\bx) p_{k_2}(b(\bx)) \E_{\bbP,6}(\hat{f}^{(6)}(\bx)) d\bx \right) \right]^2 \\
    & \quad  = \E_{\bbP,1} \left[  \iint \tilde{p}^{(1)}_{k_1}(\bx)\tilde{p}^{(1)}_{k_1}(\by) p_{k_2}(b(\bx))p_{k_2}(b(\by)) \E_{\bbP,6}(\hat{f}^{(6)}(\bx))\E_{\bbP,6}(\hat{f}^{(6)}(\by)) d\bx d\by  \right] \\ & \quad \quad - \left[ \int  \Pi(p|V_{k_1})(\bx) p_{k_2}(b(\bx)) \E_{\bbP,6}(\hat{f}^{(6)}(\bx)) d\bx \right]^2 \\
    & \quad  =  \iint  \E_{\bbP,1} \left[\tilde{p}^{(1)}_{k_1}(\bx)\tilde{p}^{(1)}_{k_1}(\by) \right] p_{k_2}(b(\bx))p_{k_2}(b(\by)) \E_{\bbP,6}(\hat{f}^{(6)}(\bx))\E_{\bbP,6}(\hat{f}^{(6)}(\by)) d\bx d\by   \\
    & \quad \quad -  \iint \Pi(p|V_{k_1})(\bx)\Pi(p|V_{k_1})(\by) p_{k_2}(b(\bx))p_{k_2}(b(\by)) \E_{\bbP,6}(\hat{f}^{(6)}(\bx))\E_{\bbP,6}(\hat{f}^{(6)}(\by)) d\bx d\by \\ 
    & \quad  \lesssim \iint \left|  \E_{\bbP,1} \left[\tilde{p}^{(1)}_{k_1}(\bx)\tilde{p}^{(1)}_{k_1}(\by) \right] - \Pi(p|V_{k_1})(\bx)\Pi(p|V_{k_1})(\by)  \right| d\bx d\by\\
    & \quad  \lesssim \frac{1}{n} + \iint \frac{1}{n} \E_{\bbP} \left[ \left| K_{V_{k_1}}(\bX, \bx)K_{V_{k_1}}(\bX, \by) \right| \right] d\bx d\by \quad (\text{by Lemma \ref{lem: exp pxpy}}) \\
    & \quad  \lesssim \frac{1}{n} \quad (\text{by Lemma \ref{lem: exp kernel}}).
\end{align*}
We can bound the second term as follows
\begin{align*}
    & \Var_{\bbP,1,2} \bigg[ \int (\hat{p}^{(1,2)}_{k_1}(\bx) - \tilde{p}^{(1)}_{k_1}(\bx)) p_{k_2}(b(\bx)) \E_{\bbP,6}(\hat{f}^{(6)}(\bx)) d\bx  \bigg] \\
    & \quad \leq \E_{\bbP,1,2} \left[ \left( \int (\hat{p}^{(1,2)}_{k_1}(\bx) - \tilde{p}^{(1)}_{k_1}(\bx)) p_{k_2}(b(\bx)) \E_{\bbP,6}(\hat{f}^{(6)}(\bx)) d\bx \right)^2 \right] \\
    & \quad = \E_{\bbP,1,2} \bigg[ \begin{array}{c}
     \iint (\hat{p}^{(1,2)}_{k_1}(\bx) - \tilde{p}^{(1)}_{k_1}(\bx))(\hat{p}^{(1,2)}_{k_1}(\by) - \tilde{p}^{(1)}_{k_1}(\by))\\ \times  p_{k_2}(b(\bx))p_{k_2}(b(\by)) \E_{\bbP,6}(\hat{f}^{(6)}(\bx))\E_{\bbP,6}(\hat{f}^{(6)}(\by)) d\bx d\by 
   \end{array} \bigg] \\
    & \quad =   \iint \begin{array}{c}\E_{\bbP,1,2} \bigg[ (\hat{p}^{(1,2)}_{k_1}(\bx) - \tilde{p}^{(1)}_{k_1}(\bx))(\hat{p}^{(1,2)}_{k_1}(\by) - \tilde{p}^{(1)}_{k_1}(\by))\bigg] \\ \times p_{k_2}(b(\bx))p_{k_2}(b(\by)) \E_{\bbP,6}(\hat{f}^{(6)}(\bx))\E_{\bbP,6}(\hat{f}^{(6)}(\by)) \end{array} d\bx d\by   \\
    & \quad \lesssim   \iint \left| \E_{\bbP,1,2} \bigg[ (\hat{p}^{(1,2)}_{k_1}(\bx) - \tilde{p}^{(1)}_{k_1}(\bx))(\hat{p}^{(1,2)}_{k_1}(\by) - \tilde{p}^{(1)}_{k_1}(\by))\bigg]  \right| d\bx d\by  \\
    & \quad \lesssim \left( \frac{n}{\log n}\right)^{-\frac{2\gamma}{2\gamma + d}}  + \iint \frac{1}{n} \E_{\bbP} \left[ \left| K_{V_{k_1}}(\bX, \bx)K_{V_{k_1}}(\bX, \by) \right| \right] d\bx d\by \quad (\text{by Lemma \ref{lem: exp pxpy}}) \\
    & \quad \lesssim \frac{1}{n}  + \left( \frac{n}{\log n}\right)^{-\frac{2\gamma}{2\gamma + d}}   \quad (\text{by Lemma \ref{lem: exp kernel}}).
\end{align*}

Next we show (\ref{eq: exp var v5}). We have that
\begin{align*}
    & \E_{\bbP,1,2} \left[ \Var_{\bbP,3,4,6} \left( \int \hat{p}^{(1, 2)}_{k_1}(\bx)\hat{b}^{(3, 4)}_{k_2}(\bx) \hat{f}^{(6)}(\bx) d\bx \right)  \right] \\
    & \quad = \E_{\bbP,1,2} \left[ \E_{\bbP,3,4}  \left[ \Var_{\bbP,6} \left( \int \hat{p}^{(1, 2)}_{k_1}(\bx)\hat{b}^{(3, 4)}_{k_2}(\bx) \hat{f}^{(6)}(\bx) d\bx \right)  \right] \right] \\ & \quad \quad + \E_{\bbP,1,2} \left[ \Var_{\bbP,3,4} \left[ \E_{\bbP,6} \left( \int \hat{p}^{(1, 2)}_{k_1}(\bx)\hat{b}^{(3, 4)}_{k_2}(\bx) \hat{f}^{(6)}(\bx) d\bx \right)  \right] \right]. 
\end{align*}
The remainder of the proof will establish the following:
\begin{align}
    & \E_{\bbP,1,2} \left[ \E_{\bbP,3,4}  \left[ \Var_{\bbP,6} \left( \int \hat{p}^{(1, 2)}_{k_1}(\bx)\hat{b}^{(3, 4)}_{k_2}(\bx) \hat{f}^{(6)}(\bx) d\bx \right)  \right] \right] \nonumber \\
    & \quad \lesssim \left( \frac{n}{\log n}\right)^{-\frac{2\gamma}{2\gamma + d}} \left[ 1 + \frac{k_1 \wedge k_2}{n^2} \right] \label{eq: mc unknown f exp var helper1} 
\end{align}
and
\begin{align}
    & \E_{\bbP,1,2} \left[ \Var_{\bbP,3,4} \left[ \E_{\bbP,6} \left( \int \hat{p}^{(1, 2)}_{k_1}(\bx)\hat{b}^{(3, 4)}_{k_2}(\bx) \hat{f}^{(6)}(\bx) d\bx \right)  \right] \right] \nonumber \\
    & \quad \lesssim  \frac{1}{n} + \frac{k_1 \wedge k_2}{n^2} + \left( \frac{n}{\log n}\right)^{-\frac{2\gamma}{2\gamma + d}}. \label{eq: mc unknown f exp var helper2}
\end{align}

We can show (\ref{eq: mc unknown f exp var helper1}) as follows
\begin{align*}
    & \E_{\bbP,1,2} \left[ \E_{\bbP,3,4}  \left[ \Var_{\bbP,6} \left( \int \hat{p}^{(1, 2)}_{k_1}(\bx)\hat{b}^{(3, 4)}_{k_2}(\bx) \hat{f}^{(6)}(\bx) d\bx \right)  \right] \right] \\
    & \quad = \E_{\bbP,1,2} \left[ \E_{\bbP,3,4}  \left[ \Var_{\bbP,6} \left( \int \hat{p}^{(1, 2)}_{k_1}(\bx)\hat{b}^{(3, 4)}_{k_2}(\bx) [\hat{f}^{(6)}(\bx) - f(\bx)] d\bx \right)  \right] \right] \\
    & \quad  \leq \E_{\bbP,1,2,3,4,6}  \left[ \iint \begin{array}{c} \hat{p}^{(1,2)}_{k_1}(\bx)\hat{p}^{(1,2)}_{k_1}(\by)\hat{b}^{(3,4)}_{k_2}(\bx) \hat{b}^{(3,4)}_{k_2}(\by) \\ \times [\hat{f}^{(6)}(\bx) - f(\bx)][\hat{f}^{(6)}(\by) - f(\by)] \end{array} d\bx d\by  \right] \\
    & \quad = \iint \begin{array}{c}\E_{\bbP,1,2}  \left[ \hat{p}^{(1,2)}_{k_1}(\bx)\hat{p}^{(1,2)}_{k_1}(\by) \right] \E_{\bbP,3,4}  \left[ \hat{b}^{(3,4)}_{k_2}(\bx)\hat{b}^{(3,4)}_{k_2}(\by) \right] \\ \times \E_{\bbP, 6} \left[ [\hat{f}^{(6)}(\bx) - f(\bx)][\hat{f}^{(6)}(\by) - f(\by)] \right] \end{array} d\bx d\by  \\
    & \quad \leq \iint \begin{array}{c} \left| \E_{\bbP,1,2} \left[ \hat{p}^{(1,2)}_{k_1}(\bx)\hat{p}^{(1,2)}_{k_1}(\by) \right] \right| \left| \E_{\bbP,3,4}  \left[ \hat{b}^{(3,4)}_{k_2}(\bx)\hat{b}^{(3,4)}_{k_2}(\by) \right] \right| \\ \times \E_{\bbP, 6} \left[ \left\| \hat{f}^{(6)} - f \right\|_{\infty}^2 \right] \end{array} d\bx d\by \\
    & \quad \leq \left( \frac{n}{\log n}\right)^{-\frac{2\gamma}{2\gamma + d}} \iint \left| \E_{\bbP,1,2} \left[ \hat{p}^{(1,2)}_{k_1}(\bx)\hat{p}^{(1,2)}_{k_1}(\by) \right] \right| \left| \E_{\bbP,3,4}  \left[ \hat{b}^{(3,4)}_{k_2}(\bx)\hat{b}^{(3,4)}_{k_2}(\by) \right] \right| d\bx d\by \\
    & \quad \lesssim \left( \frac{n}{\log n}\right)^{-\frac{2\gamma}{2\gamma + d}} \iint \begin{array}{c}\left( 1 + \frac{1}{n}\E_{\bbP} \left[ \left| K_{V_{k_1}}(\bX, \bx)K_{V_{k_1}}(\bX, \by) \right| \right] \right) \\ \times \left( 1 + \frac{1}{n}\E_{\bbP} \left[ \left| K_{V_{k_2}}(\bX, \bx)K_{V_{k_2}}(\bX, \by) \right| \right] \right) \end{array} d\bx d\by \quad (\text{by Lemma \ref{lem: exp pxpy}})\\
    & \quad \lesssim \left( \frac{n}{\log n}\right)^{-\frac{2\gamma}{2\gamma + d}} \left[ 1 + \frac{k_1 \wedge k_2}{n^2} \right] \quad (\text{by Lemma \ref{lem: exp kernel}}).
\end{align*}

To show (\ref{eq: mc unknown f exp var helper2}), we note that
\begin{align*}
    & \E_{\bbP,1,2} \left[ \Var_{\bbP,3,4} \left[ \E_{\bbP,6} \left( \int \hat{p}^{(1, 2)}_{k_1}(\bx)\hat{b}^{(3, 4)}_{k_2}(\bx) \hat{f}^{(6)}(\bx) d\bx \right)  \right] \right]  \\
    & \quad = \E_{\bbP,1,2} \left[ \Var_{\bbP,3,4} \left[  \int \hat{p}^{(1, 2)}_{k_1}(\bx)\hat{b}^{(3, 4)}_{k_2}(\bx) \E_{\bbP,6}(\hat{f}^{(6)}(\bx)) d\bx   \right] \right] \\ 
    & \quad \leq 2 \E_{\bbP,1,2} \left[ \Var_{\bbP,3} \left[  \int \hat{p}^{(1, 2)}_{k_1}(\bx)\tilde{b}^{(3)}_{k_2}(\bx) \E_{\bbP,6}(\hat{f}^{(6)}(\bx)) d\bx   \right] \right] \\ 
    & \quad \quad + 2 \E_{\bbP,1,2} \left[ \Var_{\bbP,3,4} \left[  \int \hat{p}^{(1, 2)}_{k_1}(\bx) \left[ \hat{b}^{(3, 4)}_{k_2}(\bx) - \tilde{b}^{(3)}_{k_2}(\bx) \right]\E_{\bbP,6}(\hat{f}^{(6)}(\bx)) d\bx   \right] \right].
\end{align*}
We can bound the first term as follows
\begin{align*}
    & \E_{\bbP,1,2} \left[ \Var_{\bbP,3} \left[  \int \hat{p}^{(1, 2)}_{k_1}(\bx)\tilde{b}^{(3)}_{k_2}(\bx) \E_{\bbP,6}(\hat{f}^{(6)}(\bx)) d\bx   \right] \right] \\
    & \quad  = \E_{\bbP,1,2,3} \left[  \iint \hat{p}^{(1,2)}_{k_1}(\bx)\hat{p}^{(1,2)}_{k_1}(\by)\tilde{b}^{(3)}_{k_2}(\bx)\tilde{b}^{(3)}_{k_2}(\by) \E_{\bbP,6}(\hat{f}^{(6)}(\bx))\E_{\bbP,6}(\hat{f}^{(6)}(\by)) d\bx d\by   \right] \\ & \quad \quad  - \E_{\bbP,1,2} \left[ \left(\int \hat{p}^{(1,2)}_{k_1}(\bx) \Pi(b|V_{k_2})(\bx) \E_{\bbP,6}(\hat{f}^{(6)}(\bx)) d\bx   \right)^2 \right] \\
    & \quad  =  \iint \begin{array}{c} \E_{\bbP,1,2} \left[ \hat{p}^{(1,2)}_{k_1}(\bx)\hat{p}^{(1,2)}_{k_1}(\by) \right] \E_{\bbP,3} \left[ \tilde{b}^{(3)}_{k_2}(\bx)\tilde{b}^{(3)}_{k_2}(\by) \right] \\ \times \E_{\bbP,6}(\hat{f}^{(6)}(\bx))\E_{\bbP,6}(\hat{f}^{(6)}(\by)) \end{array}d\bx d\by   \\ & \quad \quad  -  \iint \begin{array}{c} \E_{\bbP,1,2} \left[\hat{p}^{(1,2)}_{k_1}(\bx) \hat{p}^{(1,2)}_{k_1}(\by)\right] \Pi(b|V_{k_2})(\bx)\Pi(b|V_{k_2})(\by) \\ \times \E_{\bbP,6}(\hat{f}^{(6)}(\bx))\E_{\bbP,6}(\hat{f}^{(6)}(\by)) \end{array} d\bx d\by   \\
    & \quad  \lesssim \iint \left| \begin{array}{c} \E_{\bbP,1,2} \left[ \hat{p}^{(1,2)}_{k_1}(\bx)\hat{p}^{(1,2)}_{k_1}(\by) \right] \\ \times \left\{ \E_{\bbP,3} \left[ \tilde{b}^{(3)}_{k_2}(\bx)\tilde{b}^{(3)}_{k_2}(\by)\right] - \Pi(b|V_{k_2})(\bx)\Pi(b|V_{k_2})(\by) \right\} \end{array} \right| d\bx d\by \\
    & \quad \lesssim \iint \begin{array}{c} \left( 1 + \frac{1}{n}\E_{\bbP} \left[ \left| K_{V_{k_1}}(\bX, \bx)K_{V_{k_1}}(\bX, \by) \right| \right] \right) \\ \times \left( \frac{1}{n} + \frac{1}{n}\E_{\bbP} \left[ \left| K_{V_{k_2}}(\bX, \bx)K_{V_{k_2}}(\bX, \by) \right| \right] \right) \end{array} d\bx d\by \quad (\text{by Lemma \ref{lem: exp pxpy}})\\
    & \quad \lesssim \frac{1}{n} + \frac{k_1 \wedge k_2}{n^2} \quad (\text{by Lemma \ref{lem: exp kernel}}).
\end{align*}

We can bound the second term as follows
\begin{align*}
    & \E_{\bbP,1,2} \left[ \Var_{\bbP,3,4} \left[  \int \hat{p}^{(1, 2)}_{k_1}(\bx) \left[ \hat{b}^{(3, 4)}_{k_2}(\bx) - \tilde{b}^{(3)}_{k_2}(\bx) \right]\E_{\bbP,6}(\hat{f}^{(6)}(\bx)) d\bx   \right] \right] \\
    & \quad  \leq \E_{\bbP,1,2,3,4} \left[  \iint \begin{array}{c}\hat{p}^{(1,2)}_{k_1}(\bx)\hat{p}^{(1,2)}_{k_1}(\by) \E_{\bbP,6}(\hat{f}^{(6)}(\bx))\E_{\bbP,6}(\hat{f}^{(6)}(\by)) \\ \times \left[ \hat{b}^{(3, 4)}_{k_2}(\bx) - \tilde{b}^{(3)}_{k_2}(\bx) \right] \left[ \hat{b}^{(3, 4)}_{k_2}(\by) - \tilde{b}^{(3)}_{k_2}(\by) \right]  \end{array} d\bx d\by   \right] \\ 
    & \quad  =  \iint \begin{array}{c}\E_{\bbP,1,2} \left[ \hat{p}^{(1,2)}_{k_1}(\bx)\hat{p}^{(1,2)}_{k_1}(\by) \right] \E_{\bbP,6}(\hat{f}^{(6)}(\bx))\E_{\bbP,6}(\hat{f}^{(6)}(\by)) \\ \E_{\bbP,3,4} \left[ \left[ \hat{b}^{(3, 4)}_{k_2}(\bx) - \tilde{b}^{(3)}_{k_2}(\bx) \right] \left[ \hat{b}^{(3, 4)}_{k_2}(\by) - \tilde{b}^{(3)}_{k_2}(\by) \right] \right]  \end{array} d\bx d\by   \\
    & \quad  \lesssim \iint \begin{array}{c} \left| \E_{\bbP,1,2} \left[ \hat{p}^{(1,2)}_{k_1}(\bx)\hat{p}^{(1,2)}_{k_1}(\by) \right] \right| \\ \times \left| \E_{\bbP,3,4} \left[ \left[ \hat{b}^{(3, 4)}_{k_2}(\bx) - \tilde{b}^{(3)}_{k_2}(\bx) \right]\left[ \hat{b}^{(3, 4)}_{k_2}(\by) - \tilde{b}^{(3)}_{k_2}(\by) \right] \right] \right| \end{array} d\bx d\by \\
    & \quad \lesssim \iint  \begin{array}{c}\left( 1 + \frac{1}{n} \E_{\bbP} \left[ \left| K_{V_{k_1}}(\bX, \bx)K_{V_{k_1}}(\bX, \by) \right| \right] \right) \\ \times \left( \left( \frac{n}{\log n}\right)^{-\frac{2\gamma}{2\gamma + d}} + \frac{1}{n} \E_{\bbP} \left[ \left| K_{V_{k_2}}(\bX, \bx)K_{V_{k_2}}(\bX, \by) \right| \right] \right)\end{array} d\bx d\by \quad (\text{by Lemma \ref{lem: exp pxpy}})\\
    & \quad \lesssim \left( \frac{n}{\log n}\right)^{-\frac{2\gamma}{2\gamma + d}} + \frac{1}{n} + \frac{k_1 \wedge k_2}{n^2} \quad (\text{by Lemma \ref{lem: exp kernel}}).
\end{align*}

\subsubsection{Proof of the lower bound}
Let $\bbP \in \cP_{(\alpha,\beta,\gamma)}$, where $\bX \sim \text{Uniform}([0,1]^d)$. Recall from the proof of the upper bound on the bias that
\begin{align*}
    \E_\bbP(\hat{\psi}^{\mathrm{INT}}_{k_1, k_2}) - \psi(\bbP) & =  -  \int \left( p_{k_1}(p(\bx)) -  p(\bx) \right) \left( p_{k_2}(b(\bx)) - b(\bx)  \right)  d\bx   \\ & \quad -  \int p(\bx)  \left( p_{k_2}(b(\bx)) - b(\bx)  \right)   d\bx - \int  b(\bx)  \left( p_{k_1}(p(\bx)) -  p(\bx) \right)    d\bx \\
    & \quad + \int  p_{k_1}(p(\bx)) p_{k_2}(b(\bx)) [1 - \E_{\bbP,6}(\hat{f}^{(6)}(\bx)) ]d\bx  . 
\end{align*}
Recall from the upper bound that
\begin{equation*}
    \left| \int  p_{k_1}(p(\bx)) p_{k_2}(b(\bx)) [1 - \E_{\bbP,6}(\hat{f}^{(6)}(\bx)) ]d\bx \right| \lesssim \left( \frac{n}{\log n}\right)^{-\frac{\gamma}{2\gamma + d}}.
\end{equation*}
Using the decomposition in (\ref{eq: p_k proj}) (in Lemma \ref{lem: proj dist}), 
\begin{align*}
    \int \left( p_{k_1}(p(\bx)) -  p(\bx) \right) \left( p_{k_2}(b(\bx)) - b(\bx)  \right) d\bx & = \int \Pi(p | V_{k_1}^{\perp})(\bx) \Pi(b | V_{k_2}^{\perp})(\bx)  d\bx + R_{1, n, k_1, k_2} \\
    \int p(\bx)  \left( p_{k_2}(b(\bx)) - b(\bx)  \right)   d\bx & = - \int p(\bx)  \Pi(b | V_{k_2}^{\perp})(\bx)   d\bx + R_{2, n, k_1, k_2} \\
    \int  b(\bx)  \left( p_{k_1}(p(\bx)) -  p(\bx) \right)    d\bx & = - \int  b(\bx)  \Pi(p | V_{k_1}^{\perp})(\bx)    d\bx + R_{3, n, k_1, k_2}
\end{align*}
where 
\begin{equation*}
    |R_{j, n, k_1, k_2}| \lesssim \left( \frac{n}{\log n}\right)^{-\frac{\gamma}{2\gamma + d}}, \qquad j = 1, 2, 3.
\end{equation*}
Suppose without loss of generality that $k_1 \leq k_2$ for $n$ sufficiently large. It follows from the proof of the lower bound on the bias in the known $f$ case that
\begin{equation*}
     \E_\bbP(\hat{\psi}^{\mathrm{INT}}_{k_1, k_2}) - \psi(\bbP)   =    \int \Pi(p | V_{k_1}^{\perp})(\bx) \Pi(b | V_{k_1}^{\perp})(\bx)   d\bx  + O\left( \left( \frac{n}{\log n}\right)^{-\frac{\gamma}{2\gamma + d}} \right).
\end{equation*}
Choosing $p$ and $b$ as in the known $f$ case (i.e., as in (\ref{eq: p lb}) and (\ref{eq: b lb})),
\begin{equation*}
     \int \Pi(p | V_{k_1}^{\perp})(\bx) \Pi(b | V_{k_1}^{\perp})(\bx)   d\bx \gtrsim k_1^{-(\alpha + \beta)/d}. 
\end{equation*}
The lower bound on the bias is then established by noting that
\begin{equation*}
    k_1^{-(\alpha + \beta)/d} \gg \left( \frac{n}{\log n}\right)^{-\frac{\gamma}{2\gamma + d}}
\end{equation*}
for $k_1 \ll n^{\frac{d/4}{(\alpha + \beta) / 2}}$ and $\gamma$ sufficiently large.

\subsection{Monte Carlo-based plug-in estimator}

\subsubsection{Proof of the upper bound}

Let $\bbP \in \cP_{(\alpha,\beta,\gamma)}$ be arbitrary.\\

\noindent \textbf{Bounding the bias:}

Following the same approach used in the known $f$ case, we have that
\begin{align*}
    \E_\bbP(\hat{\psi}^{\mathrm{MC}}_{k_1, k_2}) & =  \E_\bbP\left( A Y \right) - \int p(\bx)b(\bx)f(\bx)  d\bx \\
    & \quad -  \int \left( p_{k_1}(p(\bx)) -  p(\bx) \right) \left( p_{k_2}(b(\bx)) - b(\bx)  \right) f(\bx) d\bx   \\ & \quad -  \int p(\bx) f(\bx) \left( p_{k_2}(b(\bx)) - b(\bx)  \right)   d\bx - \int  b(\bx) f(\bx) \left( p_{k_1}(p(\bx)) -  p(\bx) \right)    d\bx 
\end{align*}
Observe that the expectation of $\hat{\psi}^{\mathrm{MC}}_{k_1, k_2}$ is the same as that of $\hat{\psi}^{\mathrm{INT}}_{k_1, k_2}$ (in the unknown $f$ case) without the additional term of $\int  p_{k_1}(p(\bx)) p_{k_2}(b(\bx)) [f(\bx) - \E_{\bbP,6}(\hat{f}^{(6)}(\bx)) ]d\bx$. Moreover, recall that
\begin{equation*}
    \left| \int  p_{k_1}(p(\bx)) p_{k_2}(b(\bx)) [f(\bx) - \E_{\bbP,6}(\hat{f}^{(6)}(\bx)) ]d\bx \right| \lesssim \left( \frac{n}{\log n}\right)^{-\frac{\gamma}{2\gamma + d}}.
\end{equation*}
Therefore, 
\begin{equation*}
    \left|\E_\bbP\left(\hat{\psi}_{k_1, k_2}^{\mathrm{MC}} - \psi(\bbP)\right)\right| \lesssim (k_1 \wedge k_2)^{-(\alpha + \beta)/d} + \left( \frac{n}{\log n}\right)^{-\frac{\gamma}{2\gamma + d}}.
\end{equation*}

\noindent \textbf{Bounding the variance:}

To bound the variance, we consider the decomposition
\begin{align*}
    \Var_\bbP(\hat{\psi}_{k_1, k_2}^{\mathrm{MC}}) & = \Var_\bbP\left( \frac{1}{n}\sum_{i \in \D_5} A_i Y_i\right) + \Var_\bbP\left( \frac{1}{n}\sum_{i \in \D_6} \hat{p}^{(1, 2)}_{k_1}(\bX_i)\hat{b}^{(3, 4)}_{k_2}(\bX_i)  \right) \\
    & = \frac{1}{n}\Var_\bbP(AY) +  \Var_{\bbP,1, 2} \left[ \E_{\bbP,3,4,6}(\hat{p}^{(1,2)}_{k_1}(\bX)\hat{b}^{(3,4)}_{k_2}(\bX))  \right]  \\
    & \quad +  \E_{\bbP,1,2} \left[ \Var_{\bbP,3,4,6}\left( \frac{1}{n}\sum_{i \in \D_6} \hat{p}^{(1,2)}_{k_1}(\bX_i)\hat{b}^{(3,4)}_{k_2}(\bX_i)  \right)  \right].
\end{align*}

We can bound the second term by 
\begin{align*}
    \Var_{\bbP,1, 2} \left[ \E_{\bbP,3,4,6}(\hat{p}^{(1,2)}_{k_1}(\bX)\hat{b}^{(3,4)}_{k_2}(\bX))  \right] & = \Var_{\bbP,1,2} \left[  \int \hat{p}^{(1,2)}_{k_1}(\bx) p_{k_2}(b(\bx)) f(\bx) d\bx \right] \\
    & \lesssim \frac{1}{n} + \left( \frac{n}{\log n}\right)^{-\frac{2\gamma}{2\gamma + d}},
\end{align*}
where the last line follows from the exact same steps used to show (\ref{eq: var exp v5}) (in the derivation of the variance of $\hat{\psi}_{k_1, k_2}^{\mathrm{INT}}$ in the unknown $f$ case) when replacing $\E_{\bbP,6}(\hat{f}^{(6)}(\cdot))$ with $f(\cdot)$.  

It remains to bound the third term. We will show that
\begin{align} \label{eq: var unknown f mc pt3}
    & \E_{\bbP,1,2} \left[ \Var_{\bbP,3,4,6}\left( \frac{1}{n}\sum_{i \in \D_6} \hat{p}^{(1,2)}_{k_1}(\bX_i)\hat{b}^{(3,4)}_{k_2}(\bX_i)  \right)  \right] \nonumber \\
    & \quad \lesssim \frac{1}{n} + \frac{k_1 \vee k_2}{n^2} +  \frac{k_1k_2}{n^3} + \left( \frac{n}{\log n}\right)^{-\frac{2\gamma}{2\gamma + d}}.
\end{align}
We have that
\begin{align*}
    & \E_{\bbP,1,2} \left[ \Var_{\bbP,3,4,6}\left( \frac{1}{n}\sum_{i \in \D_6} \hat{p}^{(1,2)}_{k_1}(\bX_i)\hat{b}^{(3,4)}_{k_2}(\bX_i)  \right)  \right] \\
    & \quad = \E_{\bbP,1,2} \left[ \E_{\bbP,3,4}  \left[ \Var_{\bbP,6}\left( \frac{1}{n}\sum_{i \in \D_6} \hat{p}^{(1, 2)}_{k_1}(\bX_i)\hat{b}^{(3,4)}_{k_2}(\bX_i)  \right)  \right] \right] \\ & \quad \quad + \E_{\bbP,1,2} \left[ \Var_{\bbP,3,4} \left[ \E_{\bbP,6}\left( \frac{1}{n}\sum_{i \in \D_6} \hat{p}^{(1,2)}_{k_1}(\bX_i)\hat{b}^{(3,4)}_{k_2}(\bX_i)  \right)  \right] \right] \\
    & \quad = \frac{1}{n} \E_{\bbP,1,2,3,4}  \left[ \Var_{\bbP,6}\left( \hat{p}^{(1,2)}_{k_1}(\bX)\hat{b}^{(3,4)}_{k_2}(\bX)  \right)  \right] \\
    & \quad \quad + \E_{\bbP,1,2} \left[ \Var_{\bbP,3,4} \left[ \E_{\bbP,6}\left( \hat{p}^{(1,2)}_{k_1}(\bX)\hat{b}^{(3,4)}_{k_2}(\bX)  \right)  \right] \right]. 
\end{align*}

To bound $\E_{\bbP,1,2} \left[ \Var_{\bbP,3,4} \left[ \E_{\bbP,6}\left( \hat{p}^{(1,2)}_{k_1}(\bX)\hat{b}^{(3,4)}_{k_2}(\bX)  \right)  \right] \right]$, we note that 
\begin{align*}
    & \E_{\bbP,1,2} \left[ \Var_{\bbP,3,4} \left[ \E_{\bbP,6}\left( \hat{p}^{(1,2)}_{k_1}(\bX)\hat{b}^{(3,4)}_{k_2}(\bX)  \right)  \right] \right] \\
    & \quad = \E_{\bbP,1,2} \left[ \Var_{\bbP,3,4} \left[  \int \hat{p}^{(1, 2)}_{k_1}(\bx)\hat{b}^{(3, 4)}_{k_2}(\bx) f(\bx) d\bx   \right] \right] \\
    & \quad \lesssim  \frac{1}{n} + \frac{k_1 \wedge k_2}{n^2} + \left( \frac{n}{\log n}\right)^{-\frac{2\gamma}{2\gamma + d}}
\end{align*}
where the last line follows from the exact same steps used to show (\ref{eq: mc unknown f exp var helper2}) (in the derivation of the variance of $\hat{\psi}_{k_1, k_2}^{\mathrm{INT}}$ in the unknown $f$ case) when replacing $\E_{\bbP,6}(\hat{f}^{(6)}(\cdot))$ with $f(\cdot)$. 

Next, we bound $\frac{1}{n} \E_{\bbP,1,2,3,4}  \left[ \Var_{\bbP,6}\left( \hat{p}^{(1,2)}_{k_1}(\bX)\hat{b}^{(3,4)}_{k_2}(\bX)  \right)  \right]$. We have that
\begin{align*}
    & \frac{1}{n} \E_{\bbP,1,2,3,4}  \left[ \Var_{\bbP,6}\left( \hat{p}^{(1,2)}_{k_1}(\bX)\hat{b}^{(3,4)}_{k_2}(\bX)  \right)  \right] \\
    & \quad \leq \frac{1}{n} \E_{\bbP,1,2,3,4}  \left[ \E_{\bbP,6}\left[ \left( \hat{p}^{(1,2)}_{k_1}(\bX)\hat{b}^{(3,4)}_{k_2}(\bX)  \right)^2 \right]  \right] \\
    & \quad = \frac{1}{n} \int   \E_{\bbP,1,2} \left[ \hat{p}^{(1,2)}_{k_1}(\bx)^2 \right] \E_{\bbP,3,4} \left[ \hat{b}^{(3,4)}_{k_2}(\bx)^2 \right]  f(\bx) d\bx \\
    & \quad \lesssim \frac{1}{n} \int   \E_{\bbP,1,2} \left[ \hat{p}^{(1,2)}_{k_1}(\bx)^2 \right] \E_{\bbP,3,4} \left[ \hat{b}^{(3,4)}_{k_2}(\bx)^2 \right]  d\bx \\
    & \quad \lesssim \frac{1}{n} \int \left( 1 + \frac{1}{n} \E_{\bbP} \left[ K_{V_{k_1}}(\bX, \bx)^2  \right] \right) \left( 1 + \frac{1}{n} \E_{\bbP} \left[ K_{V_{k_2}}(\bX, \bx)^2 \right] \right) d\bx \quad (\text{by Lemma \ref{lem: exp pxpy}})\\
    & \quad \lesssim \frac{1}{n} + \frac{k_1}{n^2} + \frac{k_2}{n^2} + \frac{k_1k_2}{n^3}  \quad (\text{by Lemma \ref{lem: exp kernel single}}).
\end{align*}
which establishes (\ref{eq: var unknown f mc pt3}) and thus the variance bound.

\subsubsection{Proof of the lower bound}

Let $\bbP \in \cP_{(\alpha,\beta,\gamma)}$ be arbitrary.\\ 

\noindent \textbf{Bounding the bias:}

As noted in the proof of the upper bound of the bias, the expectation of $\hat{\psi}^{\mathrm{MC}}_{k_1, k_2}$ is the same as that of $\hat{\psi}^{\mathrm{INT}}_{k_1, k_2}$ (in the unknown $f$ case) without the additional term of $\int  p_{k_1}(p(\bx)) p_{k_2}(b(\bx)) [f(\bx) - \E_{\bbP,6}(\hat{f}^{(6)}(\bx)) ]d\bx$. Moreover, recall that
\begin{equation*}
    \left| \int  p_{k_1}(p(\bx)) p_{k_2}(b(\bx)) [f(\bx) - \E_{\bbP,6}(\hat{f}^{(6)}(\bx)) ]d\bx \right| \lesssim \left( \frac{n}{\log n}\right)^{-\frac{\gamma}{2\gamma + d}}.
\end{equation*}
Now, suppose that $\bbP \in \cP_{(\alpha,\beta,\gamma)}$ where $\bX \sim \text{Uniform}([0,1]^d)$ and $p$ and $b$ are given by (\ref{eq: p lb}) and (\ref{eq: b lb}). The lower bound on the bias of $\hat{\psi}^{\mathrm{MC}}_{k_1, k_2}$ then follows from the proof of the lower bound of the bias of $\hat{\psi}^{\mathrm{INT}}_{k_1, k_2}$. 

\noindent \textbf{Bounding the variance:}

Suppose that $k_1, k_2 \gg n$. Following the approach used in the known $f$ case, it suffices to show that 
\begin{equation}
    \E_{\bbP, 1, 2, 3, 4} \left[\Var_{\bbP, 6} \left( \hat{p}^{(1,2)}_{k_1}(\bX) \hat{b}^{(3,4)}_{k_2}(\bX)  \right) \right] \gtrsim \frac{k_1k_2}{n^2}.  \label{eq: mc unknown f exp var lb}
\end{equation}

We have that
\begin{align*}
    \E_{\bbP, 1, 2, 3, 4} \left[\Var_{\bbP, 6} \left( \hat{p}^{(1, 2)}_{k_1}(\bX) \hat{b}^{(3, 4)}_{k_2}(\bX)  \right) \right] & =   \E_{\bbP, 1, 2, 3, 4} \left[\E_{\bbP, 6} \left[ \left( \hat{p}^{(1, 2)}_{k_1}(\bX) \hat{b}^{(3, 4)}_{k_2}(\bX) \right)^2  \right] \right] \\
    & \quad -   \E_{\bbP, 1, 2, 3, 4} \left[\E_{\bbP, 6} \left[ \hat{p}^{(1, 2)}_{k_1}(\bX) \hat{b}^{(3, 4)}_{k_2}(\bX)   \right]^2 \right].
\end{align*}
By Lemma \ref{lem: nuisance function bounds}
\begin{equation*}
    \E_{\bbP, 1, 2, 3, 4} \left[\E_{\bbP, 6} \left[ \left( \hat{p}^{(1, 2)}_{k_1}(\bX) \hat{b}^{(3, 4)}_{k_2}(\bX) \right)^2  \right] \right] \gtrsim \frac{k_1k_2}{n^2}.
\end{equation*}
Moreover,
\begin{align*}
    & \E_{\bbP, 1, 2, 3, 4} \left[\E_{\bbP, 6} \left[ \hat{p}^{(1, 2)}_{k_1}(\bX) \hat{b}^{(3, 4)}_{k_2}(\bX)   \right]^2 \right] \\
    & \quad = \E_{\bbP, 1, 2, 3, 4} \left[ \iint \hat{p}^{(1, 2)}_{k_1}(\bx)\hat{p}^{(1, 2)}_{k_1}(\by) \hat{b}^{(3, 4)}_{k_2}(\bx)\hat{b}^{(3, 4)}_{k_2}(\by) d\bx d\by \right] \\
    & \quad =  \iint \E_{\bbP, 1, 2} \left[\hat{p}^{(1, 2)}_{k_1}(\bx)\hat{p}^{(1, 2)}_{k_1}(\by) \right] \E_{\bbP, 3, 4} \left[ \hat{b}^{(3, 4)}_{k_2}(\bx)\hat{b}^{(3, 4)}_{k_2}(\by) \right] d\bx d\by \\
    & \quad \lesssim \iint \begin{array}{c} \left( 1 + \frac{1}{n} \E_{\bbP} \left[ \left| K_{V_{k_1}}(\bX, \bx)K_{V_{k_1}}(\bX, \by) \right| \right] \right) \\ \times \left( 1 + \frac{1}{n} \E_{\bbP} \left[ \left| K_{V_{k_2}}(\bX, \bx)K_{V_{k_2}}(\bX, \by) \right| \right] \right) \end{array} d\bx d\by \quad (\text{by Lemma \ref{lem: exp pxpy}}) \\
    & \quad \lesssim 1 + \frac{k_1 \wedge k_2}{n^2} \quad (\text{by Lemma \ref{lem: exp kernel single}})
\end{align*}
which establishes (\ref{eq: mc unknown f exp var lb}).

\subsection{Newey and Robins plug-in estimator}

\subsubsection{Proof of the upper bound}
Let $\bbP \in \cP_{(\alpha,\beta,\gamma)}$ be arbitrary.\\

\noindent \textbf{Bounding the bias:}

 First, we find an upper bound on the bias of $\hat{\psi}_{k}^{\mathrm{NR}}$. Observe that
\begin{align*}
    \E_\bbP(\hat{\psi}^{\mathrm{NR}}_{k}) & =  \psi(\bbP) +  \E_\bbP\left[ p(\bX) (b(\bX) - \hat{b}^{(1,2)}_k(\bX)) \right] \\ 
    & =  \psi(\bbP) +  \int  p(\bx) f(\bx) (b(\bx) - p_k(b(\bx))) d\bx. 
\end{align*}
Recall from the derivation of the bias of $\hat{\psi}_{k_1, k_2}^{\mathrm{INT}}$ (in the unknown $f$ case) that 
\begin{equation*}
    \left| \int  p(\bx) f(\bx) (b(\bx) - p_k(b(\bx))) d\bx  \right| \lesssim k^{-(\alpha + \beta) / d} +  \left( \frac{n}{\log n}\right)^{-\frac{\gamma}{2\gamma + d}}
\end{equation*}
which completes the upper bound on the bias of $\hat{\psi}^{\mathrm{NR}}_{k}$.

\noindent \textbf{Bounding the variance:}

To bound the variance of $\hat{\psi}^{\mathrm{NR}}_{k}$, we first bound $\Var_{\bbP,1,2} \left[ \E_{\bbP,3}\left( \hat{\psi}^{\mathrm{NR}}_{k}\right)  \right]$ following the same approach used in the proof of (\ref{eq: var exp v5}) (in the derivation of the variance of $\hat{\psi}_{k_1, k_2}^{\mathrm{INT}}$ in the unknown $f$ case). That is, we may write,
\begin{align*}
    \Var_{\bbP,1,2} \left[ \E_{\bbP,3}\left( \hat{\psi}^{\mathrm{NR}}_{k}\right)  \right] & = \Var_{\bbP,1,2} \left[ \E_{\bbP,3}\left( p(\bX)(b(\bX) - \hat{b}^{(1,2)}_{k}(\bX))\right)  \right] \\
    & = \Var_{\bbP,1,2} \left[ \int  p(\bx)f(\bx)(b(\bx) - \hat{b}^{(1,2)}_{k}(\bx)) d\bx  \right] \\
    & \leq 2\Var_{\bbP,1,2} \left[ \int  p(\bx)f(\bx)\tilde{b}^{(1)}_{k}(\bx) d\bx  \right]  \\
    & \quad + 2\Var_{\bbP,1,2} \left[ \int  p(\bx)f(\bx)(\tilde{b}^{(1)}_{k}(\bx) - \hat{b}^{(1,2)}_{k}(\bx)) d\bx  \right]. 
\end{align*}
It then follows from the same steps used in the proof of (\ref{eq: var exp v5}) that
\begin{align*}
    \Var_{\bbP,1,2} \left[ \int  p(\bx)f(\bx)\tilde{b}^{(1)}_{k}(\bx) d\bx  \right] & \lesssim \frac{1}{n} \\
    \Var_{\bbP,1,2} \left[ \int  p(\bx)f(\bx)(\tilde{b}^{(1)}_{k}(\bx) - \hat{b}^{(1,2)}_{k}(\bx)) d\bx  \right]  & \lesssim \frac{1}{n} + \left( \frac{n}{\log n}\right)^{-\frac{2\gamma}{2\gamma + d}}.
\end{align*}
Therefore,
\begin{equation*}
    \Var_{\bbP,1,2} \left[ \E_{\bbP,3}\left( \hat{\psi}^{\mathrm{NR}}_{k}\right)  \right] \lesssim \frac{1}{n} + \left( \frac{n}{\log n}\right)^{-\frac{2\gamma}{2\gamma + d}}.
\end{equation*}
Second, we bound $\E_{\bbP,1,2} \left[ \Var_{\bbP,3}\left( \hat{\psi}^{\mathrm{NR}}_{k}\right)  \right]$ in the same manner as in the known $f$ case. We have that
\begin{align*}
    \E_{\bbP,1,2} \left[ \Var_{\bbP,3} \left( \hat{\psi}^{\mathrm{NR}}_{k}\right)  \right] & = \frac{1}{n} \E_{\bbP,1,2} \left[ \Var_{\bbP,3} \left( A(Y - \hat{b}^{(1,2)}_{k}(\bX)) \right) \right] \\
    & \leq \frac{1}{n} \E_{\bbP,1,2} \left[ \E_{\bbP,3} \left( A^2(Y - \hat{b}^{(1,2)}_{k}(\bX))^2 \right) \right] \\
    & = \frac{1}{n} \E_{\bbP,1,2} \left[ \E_{\bbP,3} \left( A^2(Y^2 - 2Y\hat{b}^{(1,2)}_{k}(\bX) + \hat{b}^{(1,2)}_{k}(\bX)^2) \right) \right] \\
    & = \frac{1}{n} \E_{\bbP,3} \left[ A^2\left(Y^2 - 2Yp_k(b(\bX)) + \E_{\bbP,1,2} \left[\hat{b}^{(1,2)}_{k}(\bX)^2\right] \right) \right].
\end{align*}
Since $A$, $Y$ are bounded random variables and $p_k(b(\cdot))$ is bounded, 
\begin{equation*}
    \E_{\bbP,1,2} \left[ \Var_{\bbP,3} \left( \hat{\psi}^{\mathrm{NR}}_{k}\right)  \right] \lesssim \frac{1}{n} + \frac{1}{n} \E_{\bbP,3} \left( \E_{\bbP,1,2}\left[\hat{b}^{(1,2)}_{k}(\bX)^2\right] \right). 
\end{equation*}
It then follows from Lemma \ref{lem: nuisance function bounds} that
\begin{equation*}
    \E_{\bbP,1,2} \left[ \Var_{\bbP,3} \left( \hat{\psi}^{\mathrm{NR}}_{k}\right)  \right] \lesssim \frac{1}{n} + \frac{k}{n^2}.
\end{equation*}

\subsubsection{Proof of the lower bound}

Let $\bbP \in \cP_{(\alpha,\beta,\gamma)}$, where $\bX \sim \text{Uniform}([0,1]^d)$. Recall from the proof of the upper bound on the bias that
\begin{align*}
    \E_\bbP(\hat{\psi}^{\mathrm{NR}}_{k}) - \psi(\bbP) =  \int  p(\bx) (b(\bx) - p_k(b(\bx))) d\bx.
\end{align*}
Moreover, recall from the lower bound on the bias of the integral-based plug-in estimator that
\begin{equation*}
    \int  p(\bx) (b(\bx) - p_k(b(\bx))) d\bx =  \int p(\bx)  \Pi(b | V_{k}^{\perp})(\bx)   d\bx - R_{2, n, k}
\end{equation*}
where
\begin{equation*}
    |R_{2, n, k}| \lesssim \left( \frac{n}{\log n}\right)^{-\frac{\gamma}{2\gamma + d}}.
\end{equation*}
Choosing $p$ and $b$ as given in (\ref{eq: p lb}) and (\ref{eq: b lb}), recall from the known $f$ case that
\begin{equation*}
    \left| \int p(\bx)  \Pi(b | V_{k}^{\perp})(\bx)   d\bx \right| \gtrsim k^{-(\alpha + \beta)/d}.
\end{equation*}
The lower bound on the bias is then established by noting that
\begin{equation*}
    k^{-(\alpha + \beta)/d} \gg \left( \frac{n}{\log n}\right)^{-\frac{\gamma}{2\gamma + d}}
\end{equation*}
for $k \ll n^{\frac{d/4}{(\alpha + \beta) / 2}}$ and $\gamma$ sufficiently large.

\subsection{First-order bias-corrected estimator} 

\subsubsection{Proof of the upper bound}

Let $\bbP \in \cP_{(\alpha,\beta,\gamma)}$ be arbitrary.\\

\noindent \textbf{Bounding the bias:}

Following the same approach used in the known $f$ case,
\begin{equation*}
     \E_\bbP(\hat{\psi}^{\mathrm{IF}}_{k_1, k_2}) - \psi(\bbP)  =  \int (p_{k_1}(p(\bx)) - p(\bx))(p_{k_2}(b(\bx)) - b(\bx)) f(\bx) d\bx. 
\end{equation*}
Using the decomposition in (\ref{eq: p_k proj}) (in Lemma \ref{lem: proj dist}), 
\begin{align*} 
    \E_\bbP(\hat{\psi}^{\mathrm{IF}}_{k_1, k_2}) - \psi(\bbP)  & = \int \begin{array}{c} \left(\E_{\bbP,2} \left[ \Pi\left( \frac{ p(f - \hat{f}^{(2)})}{\hat{f}^{(2)}} | V_{k_1}\right)(\bx) \right] -\Pi(p | V_{k_1}^{\perp})(\bx) \right) \\ \times \left(\E_{\bbP,4} \left[ \Pi\left( \frac{ b(f - \hat{f}^{(4)})}{\hat{f}^{(4)}} | V_{k_2}\right)(\bx) \right] - \Pi(b | V_{k_2}^{\perp})(\bx) \right)  \end{array} f(\bx) d\bx  \\
    & = \int \Pi(p | V_{k_1}^{\perp})(\bx)\Pi(b | V_{k_2}^{\perp})(\bx) f(\bx) d\bx \\
    & \quad + \int \begin{array}{c} \E_{\bbP,2} \left[ \Pi\left( \frac{ p(f - \hat{f}^{(2)})}{\hat{f}^{(2)}} | V_{k_1}\right)(\bx) \right] \\ \times \E_{\bbP,4} \left[ \Pi\left( \frac{ b(f - \hat{f}^{(4)})}{\hat{f}^{(4)}} | V_{k_2}\right)(\bx) \right] \end{array} f(\bx) d\bx \\
    & \quad - \int \Pi(b | V_{k_2}^{\perp})(\bx) \E_{\bbP,2} \left[ \Pi\left( \frac{ p(f - \hat{f}^{(2)})}{\hat{f}^{(2)}} | V_{k_1}\right)(\bx) \right] f(\bx) d\bx \\
    & \quad - \int \Pi(p | V_{k_1}^{\perp})(\bx) \E_{\bbP,4} \left[ \Pi\left( \frac{ b(f - \hat{f}^{(4)})}{\hat{f}^{(4)}} | V_{k_2}\right)(\bx) \right] f(\bx) d\bx.
\end{align*}
Recall from the known $f$ case that
\begin{equation*}
    \int \Pi(p | V_{k_1}^{\perp})(\bx)\Pi(b | V_{k_2}^{\perp})(\bx) f(\bx) d\bx \lesssim (k_1 \vee k_2)^{-\frac{\alpha + \beta}{d}}.
\end{equation*}
Moreover, recall from the proof of Lemma \ref{lem: proj dist} that
\begin{align*}
    \E_{\bbP,2} \left[ \left\| \Pi\left( \frac{ p(f - \hat{f}^{(2)})}{\hat{f}^{(2)}} | V_{k_1}\right) \right\|_{\infty} \right]  & \lesssim \left( \frac{n}{\log n}\right)^{-\frac{\gamma}{2\gamma + d}} \\
    \E_{\bbP,4} \left[ \left\| \Pi\left( \frac{ b(f - \hat{f}^{(4)})}{\hat{f}^{(4)}} | V_{k_2}\right) \right\|_{\infty} \right] & \lesssim \left( \frac{n}{\log n}\right)^{-\frac{\gamma}{2\gamma + d}}.
\end{align*}
Therefore,
\begin{align*}
    \left| \int \begin{array}{c} \E_{\bbP,2} \left[ \Pi\left( \frac{ p(f - \hat{f}^{(2)})}{\hat{f}^{(2)}} | V_{k_1}\right)(\bx) \right] \\ \times \E_{\bbP,4} \left[ \Pi\left( \frac{ b(f - \hat{f}^{(4)})}{\hat{f}^{(4)}} | V_{k_2}\right)(\bx) \right] \end{array} f(\bx) d\bx \right| & \lesssim \left( \frac{n}{\log n}\right)^{-\frac{2\gamma}{2\gamma + d}} \\
    \left| \int \Pi(b | V_{k_2}^{\perp})(\bx) \E_{\bbP,2} \left[ \Pi\left( \frac{ p(f - \hat{f}^{(2)})}{\hat{f}^{(2)}} | V_{k_1}\right)(\bx) \right] f(\bx) d\bx \right| & \lesssim k_2^{-\beta/d} \left( \frac{n}{\log n}\right)^{-\frac{\gamma}{2\gamma + d}} \\
    \left| \int \Pi(p | V_{k_1}^{\perp})(\bx) \E_{\bbP,4} \left[ \Pi\left( \frac{ b(f - \hat{f}^{(4)})}{\hat{f}^{(4)}} | V_{k_2}\right)(\bx) \right] f(\bx) d\bx \right| & \lesssim  k_1^{-\alpha/d} \left( \frac{n}{\log n}\right)^{-\frac{\gamma}{2\gamma + d}} 
\end{align*}

Therefore, we can bound the bias by
\begin{equation*}
    \left|\E_\bbP\left(\hat{\psi}^{\mathrm{IF}}_{k_1, k_2} - \psi(\bbP)\right)\right| \lesssim (k_1 \vee k_2)^{-\frac{\alpha + \beta}{d}} + \left( \frac{n}{\log n}\right)^{-\frac{\gamma}{2\gamma + d}}.
\end{equation*}

\noindent \textbf{Bounding the variance:}

We will show that
\begin{align}
    \Var_{\bbP,1,2,3,4} \left[ \E_{\bbP,5}\left( \hat{\psi}^{\mathrm{IF}}_{k_1, k_2}\right)  \right] & \lesssim \frac{1}{n} + \frac{k_1 \wedge k_2}{n^2} + \left( \frac{n}{\log n}\right)^{-\frac{2\gamma}{2\gamma + d}} \label{eq: if var exp unknown f} \\
    \E_{\bbP,1,2,3,4} \left[ \Var_{\bbP,5} \left( \hat{\psi}^{\mathrm{IF}}_{k_1, k_2}\right)  \right] & \lesssim \frac{1}{n} + \frac{k_1 \vee k_2}{n^2} + \frac{k_1k_2}{n^3}. \label{eq: if exp var ub unknown f}
\end{align}

We first show (\ref{eq: if var exp unknown f}). It follows from the work in the known $f$ case that we may write
\begin{align*}
    & \Var_{\bbP,1,2,3,4} \left[ \E_{\bbP,5}\left( \hat{\psi}^{\mathrm{IF}}_{k_1, k_2}\right)  \right] \\
    & \quad = \underbrace{\Var_{\bbP,1,2} \left[ \E_{\bbP,3,4} \left(  \int  (p(\bx) - \hat{p}^{(1,2)}_{k_1}(\bx))(b(\bx) - \hat{b}^{(3,4)}_{k_2}(\bx)) f(\bx) d\bx \right) \right]}_{:=T_1} \\ 
    & \quad \quad + \underbrace{\E_{\bbP,1,2} \left[ \Var_{\bbP,3,4} \left(  \int  (p(\bx) - \hat{p}^{(1,2)}_{k_1}(\bx))(b(\bx) - \hat{b}^{(3,4)}_{k_2}(\bx)) f(\bx) d\bx \right) \right]}_{:=T_2}.
\end{align*}
We first bound $T_1$. We have that
\begin{align*}
    T_1 & = \Var_{\bbP,1,2} \left[  \int  (p(\bx) - \hat{p}^{(1,2)}_{k_1}(\bx))(b(\bx) - p_{k_2}(b(\bx))) f(\bx) d\bx \right] \\
    & = \Var_{\bbP,1,2} \left[  \int  \hat{p}^{(1,2)}_{k_1}(\bx)(b(\bx) - p_{k_2}(b(\bx))) f(\bx) d\bx \right] \\
    & \lesssim \frac{1}{n}  + \left( \frac{n}{\log n}\right)^{-\frac{2\gamma}{2\gamma + d}} 
\end{align*}
where the last line follows from the same manner as in the proof of (\ref{eq: var exp v5}) (in the derivation of the variance of $\hat{\psi}_{k_1, k_2}^{\mathrm{INT}}$ in the unknown $f$ case). 

Next, we bound $T_2$. We have that
\begin{align*}
    T_2 & \leq 2 \E_{\bbP,1,2} \left[ \Var_{\bbP,3} \left(  \int  (p(\bx) - \hat{p}^{(1,2)}_{k_1}(\bx))\tilde{b}^{(3)}_{k_2}(\bx) f(\bx) d\bx \right) \right] \\
    & \quad + 2 \E_{\bbP,1,2} \left[ \Var_{\bbP,3,4} \left(  \int  (p(\bx) - \hat{p}^{(1,2)}_{k_1}(\bx))(\hat{b}^{(3,4)}_{k_2}(\bx) - \tilde{b}^{(3)}_{k_2}(\bx)) f(\bx) d\bx \right) \right].
\end{align*}
By following the same argument as used to prove (\ref{eq: mc unknown f exp var helper2}) (in the derivation of the variance of $\hat{\psi}_{k_1, k_2}^{\mathrm{INT}}$),
\begin{align*}
    & \E_{\bbP,1,2} \left[ \Var_{\bbP,3} \left(  \int  (p(\bx) - \hat{p}^{(1,2)}_{k_1}(\bx))\tilde{b}^{(3)}_{k_2}(\bx) f(\bx) d\bx \right) \right] \\
    & \quad \lesssim \iint \left| \begin{array}{c}\E_{\bbP,1,2} \left[ (p(\bx) - \hat{p}^{(1,2)}_{k_1}(\bx))(p(\by) - \hat{p}^{(1,2)}_{k_1}(\by)) \right] \\  \times \left\{ \E_{\bbP,3} \left[ \tilde{b}^{(3)}_{k_2}(\bx)\tilde{b}^{(3)}_{k_2}(\by)\right] - \Pi(b|V_{k_2})(\bx)\Pi(b|V_{k_2})(\by) \right\}\end{array} \right| d\bx d\by \\
    & \quad \lesssim \iint \begin{array}{c}\left| \E_{\bbP,1,2} \left[ (p(\bx) - \hat{p}^{(1,2)}_{k_1}(\bx))(p(\by) - \hat{p}^{(1,2)}_{k_1}(\by)) \right] \right| \\ \times \left( \frac{1}{n} + \frac{1}{n}\E_{\bbP} \left[ \left| K_{V_{k_2}}(\bX, \bx)K_{V_{k_2}}(\bX, \by) \right| \right] \right) \end{array} d\bx d\by. 
\end{align*}
By Lemma \ref{lem: exp pxpy}, there exists a $C \in \mathbb{R}^+$ such that for all $\bx, \by \in [0, 1]^d$
\begin{align*}
    & \left| \E_{\bbP,1,2} \left[ (p(\bx) - \hat{p}^{(1,2)}_{k_1}(\bx))(p(\by) - \hat{p}^{(1,2)}_{k_1}(\by)) \right] \right| \\ & \quad \leq |p(\bx)p(\by)| + |p(\bx) p_{k_1}(p(\by))| + |p(\by) p_{k_1}(p(\bx))| +  \left| \E_{\bbP,1,2} \left[ \hat{p}^{(1,2)}_{k_1}(\bx)\hat{p}^{(1,2)}_{k_1}(\by) \right] \right| \\
    & \quad \leq C \left( 1 + \frac{1}{n} \E_{\bbP} \left[ \left| K_{V_{k_1}}(\bX, \bx)K_{V_{k_1}}(\bX, \by) \right| \right] \right).
\end{align*}
Therefore,
\begin{align*}
    & \E_{\bbP,1,2} \left[ \Var_{\bbP,3} \left(  \int  (p(\bx) - \hat{p}^{(1,2)}_{k_1}(\bx))\tilde{b}^{(3)}_{k_2}(\bx) f(\bx) d\bx \right) \right] \\
    & \quad \lesssim \frac{1}{n} + \frac{1}{n} \iint  \E_{\bbP} \left[ \left| K_{V_{k_2}}(\bX, \bx)K_{V_{k_2}}(\bX, \by) \right| \right] d\bx d\by \\
    & \quad \quad + \frac{1}{n^2} \iint  \E_{\bbP} \left[ \left| K_{V_{k_1}}(\bX, \bx)K_{V_{k_1}}(\bX, \by) \right| \right] \E_{\bbP} \left[ \left| K_{V_{k_2}}(\bX, \bx)K_{V_{k_2}}(\bX, \by) \right| \right] d\bx d\by \\
    & \quad \lesssim \frac{1}{n} + \frac{k_1 \wedge k_2}{n^2} \quad (\text{by Lemma \ref{lem: exp kernel}}).
\end{align*}

Finally, we bound $\E_{\bbP,1,2} \left[ \Var_{\bbP,3,4} \left(  \int  (p(\bx) - \hat{p}^{(1,2)}_{k_1}(\bx))(\hat{b}^{(3,4)}_{k_2}(\bx) - \tilde{b}^{(3)}_{k_2}(\bx)) f(\bx) d\bx \right) \right]$ as follows:
\begin{align*}
    & \E_{\bbP,1,2} \left[ \Var_{\bbP,3,4} \left(  \int  (p(\bx) - \hat{p}^{(1,2)}_{k_1}(\bx))(\hat{b}^{(3,4)}_{k_2}(\bx) - \tilde{b}^{(3)}_{k_2}(\bx)) f(\bx) d\bx \right) \right] \\
    & \quad \lesssim \iint \left| \begin{array}{c} \E_{\bbP,1,2} \left[ (p(\bx) - \hat{p}^{(1,2)}_{k_1}(\bx))(p(\by) - \hat{p}^{(1,2)}_{k_1}(\by)) \right] \\  \times \E_{\bbP,3,4} \left[(\hat{b}^{(3,4)}_{k_2}(\bx) - \tilde{b}^{(3)}_{k_2}(\bx))(\hat{b}^{(3,4)}_{k_2}(\by) - \tilde{b}^{(3)}_{k_2}(\by))\right] \end{array}  \right| d\bx d\by \\
    & \quad \lesssim \iint \begin{array}{c}\left( 1 + \frac{1}{n} \E_{\bbP} \left[ \left| K_{V_{k_1}}(\bX, \bx)K_{V_{k_1}}(\bX, \by) \right| \right] \right) 
 \\ \times \left( \left( \frac{n}{\log n}\right)^{-\frac{2\gamma}{2\gamma + d}} + \frac{1}{n} \E_{\bbP} \left[ \left| K_{V_{k_2}}(\bX, \bx)K_{V_{k_2}}(\bX, \by) \right| \right] \right) \end{array} d\bx d\by \quad (\text{by Lemma \ref{lem: exp pxpy}}) \\
    & \quad \lesssim \left(  \frac{n}{\log n}\right)^{-\frac{2\gamma}{2\gamma + d}} + \frac{1}{n} + \frac{k_1 \wedge k_2}{n^2} \quad (\text{by Lemma \ref{lem: exp kernel}})
\end{align*}
which completes the proof of (\ref{eq: if var exp unknown f}).

It remains to show (\ref{eq: if exp var ub unknown f}). We follow the same approach used in the known $f$ case. We have that
\begin{align*}
    \E_{\bbP,1,2,3,4} \left[ \Var_{\bbP,5} \left( \hat{\psi}^{\mathrm{IF}}_{k_1, k_2}\right)  \right] & = \frac{1}{n} \E_{\bbP,1,2,3,4} \left[ \Var_{\bbP,5} \left( (A - \hat{p}^{(1,2)}_{k_1}(\bX))(Y - \hat{b}^{(3,4)}_{k_2}(\bX)) \right) \right] \\
    & \leq \frac{1}{n} \E_{\bbP,1,2,3,4} \left[ \E_{\bbP,5} \left( (A - \hat{p}^{(1,2)}_{k_1}(\bX))^2(Y - \hat{b}^{(3,4)}_{k_2}(\bX))^2 \right) \right] \\
    & = \frac{1}{n} \E_{\bbP,1,2,3,4} \left[ \E_{\bbP,5} \left( \begin{array}{c}(A^2 - 2A \hat{p}^{(1,2)}_{k_1}(\bX) + \hat{p}^{(1,2)}_{k_1}(\bX)^2) \\ \times (Y^2 - 2Y\hat{b}^{(3,4)}_{k_2}(\bX) + \hat{b}^{(3,4)}_{k_2}(\bX)^2) \end{array} \right) \right] \\
    & = \frac{1}{n} \E_{\bbP,5} \left( \begin{array}{c}\left(A^2 - 2A p_{k_1}(p(\bX)) + \E_{\bbP,1,2} \left[\hat{p}^{(1,2)}_{k_1}(\bX)^2 \right]\right) \\ \times \left(Y^2 - 2Yp_{k_2}(b(\bX)) + \E_{\bbP,3,4}\left[\hat{b}^{(3,4)}_{k_2}(\bX)^2\right] \right) \end{array} \right).
\end{align*}
Since $A$, $Y$, $\bX$ are bounded random variables and $p_{k_1}(p(\cdot))$ and $p_{k_2}(b(\cdot))$ are bounded, 
\begin{align*}
    \E_{\bbP,1,2,3,4} \left[ \Var_{\bbP,5} \left( \hat{\psi}^{\mathrm{IF}}_{k_1, k_2}\right)  \right] & \lesssim \frac{1}{n} + \frac{1}{n} \E_{\bbP,5} \left( \E_{\bbP,1,2}\left[\hat{p}^{(1,2)}_{k_1}(\bX)^2\right] \right)  \\
    & \quad + \frac{1}{n} \E_{\bbP,5} \left( \E_{\bbP,3,4}\left[\hat{b}^{(3,4)}_{k_2}(\bX)^2\right] \right) \\ & \quad + \frac{1}{n} \E_{\bbP,5} \left(  \E_{\bbP,1,2}\left[\hat{p}^{(1,2)}_{k_1}(\bX)^2\right] \E_{\bbP,3,4}\left[\hat{b}^{(3,4)}_{k_2}(\bX)^2\right] \right).
\end{align*}
Recall that it was shown in the proof of (\ref{eq: var unknown f mc pt3}) (in the derivation of the variance of $\hat{\psi}^{\mathrm{MC}}_{k_1, k_2}$) that
\begin{equation*}
    \frac{1}{n} \E_{\bbP,5} \left(  \E_{\bbP,1,2}\left[\hat{p}^{(1,2)}_{k_1}(\bX)^2\right] \E_{\bbP,3,4}\left[\hat{b}^{(3,4)}_{k_2}(\bX)^2\right] \right) \lesssim \frac{1}{n} + \frac{k_1}{n^2} + \frac{k_2}{n^2} + \frac{k_1k_2}{n^3}. 
\end{equation*}
Moreover,
\begin{align*}
    \frac{1}{n} \E_{\bbP,5} \left( \E_{\bbP,1,2}\left[\hat{p}^{(1,2)}_{k_1}(\bX)^2\right] \right) & = \frac{1}{n} \int \E_{\bbP,1,2}\left[\hat{p}^{(1,2)}_{k_1}(\bx)^2\right] f(\bx) d\bx \\
    & \lesssim \frac{1}{n} \int \left( 1 + \frac{1}{n} \E_{\bbP} \left[ K_{V_{k_1}}(\bX, \bx)^2 \right] \right) d\bx \quad (\text{by Lemma \ref{lem: exp pxpy}})\\
    & \lesssim \frac{1}{n} + \frac{k_1}{n^2} \quad (\text{by Lemma \ref{lem: exp kernel single}}).
\end{align*}
Similarly, 
\begin{equation*}
    \frac{1}{n} \E_{\bbP,5} \left( \E_{\bbP,3,4}\left[\hat{b}^{(3,4)}_{k_2}(\bX)^2\right] \right) \lesssim \frac{1}{n} + \frac{k_2}{n^2}.
\end{equation*}
Therefore,
\begin{equation*}
    \E_{\bbP,1,2,3,4} \left[ \Var_{\bbP,5} \left( \hat{\psi}^{\mathrm{IF}}_{k_1, k_2}\right)  \right] \lesssim \frac{1}{n} + \frac{k_1 \vee k_2}{n^2} + \frac{k_1k_2}{n^3}
\end{equation*}
which establishes (\ref{eq: if exp var ub unknown f}).

\subsubsection{Proof of the lower bound}

Let $\bbP \in \cP_{(\alpha,\beta,\gamma)}$, where $\bX \sim \text{Uniform}([0,1]^d)$. Recall from the proof of the upper bound that 
\begin{equation*}
     \E_\bbP(\hat{\psi}^{\mathrm{IF}}_{k_1,k_2}) - \psi(\bbP)  =  \int \Pi(p | V_{k_1}^{\perp})(\bx) \Pi(b | V_{k_2}^{\perp})(\bx) d\bx + R^{\prime}_{1,n,k_1,k_2}
\end{equation*}
where 
\begin{equation*}
    | R^{\prime}_{1,n,k_1,k_2} | \lesssim \left( \frac{n}{\log n}\right)^{-\frac{2\gamma}{2\gamma + d}} + (k_1^{-\alpha / d} + k_2^{-\beta / d}) \left( \frac{n}{\log n}\right)^{-\frac{\gamma}{2\gamma + d}}.
\end{equation*}
Suppose without loss of generality that $k_1 \leq k_2$ for $n$ sufficiently large. Recall from the known $f$ case that when $p$ and $b$ is as given in (\ref{eq: p lb}) and (\ref{eq: b lb}), 
\begin{equation*}
    \left| \int \Pi(p | V_{k_1}^{\perp})(\bx) \Pi(b | V_{k_2}^{\perp})(\bx) d\bx \right| \gtrsim k_2^{-(\alpha+\beta)/d}.
\end{equation*}
Suppose that $k_2 \ll n^{\frac{d/4}{(\alpha + \beta) / 2}}$. Then for sufficiently large $\gamma$, $k_2^{-(\alpha+\beta)/d} \gg | R^{\prime}_{1,n,k_1,k_2} |$ which completes the bound of the bias.

Next, we bound the variance using the same approach as in the known $f$ case. It suffices to show that for $k_1,k_2 \gg n$
\begin{align}
    \E_{\bbP, 1, 2, 3, 4} \left[\Var_{\bbP, 5} \left( \hat{p}^{(1,2)}_{k_1}(\bX) \hat{b}^{(3,4)}_{k_2}(\bX)  \right) \right] & \asymp \frac{k_1k_2}{n^2} \label{eq: if var lb helper1} \\
    \E_{\bbP, 1, 2, 3, 4} \left[ \Var_{\bbP, 5} \left( AY - Y\hat{p}^{(1,2)}_{k_1}(\bX) - A\hat{b}^{(3,4)}_{k_2}(\bX)    \right) \right] & \lesssim \frac{k_1 \vee k_2}{n} \label{eq: if var lb helper2}.
\end{align}

We first bound (\ref{eq: if var lb helper1}). We have that
\begin{align*}
    \E_{\bbP, 1, 2, 3, 4} \left[\Var_{\bbP, 5} \left( \hat{p}^{(1, 2)}_{k_1}(\bX) \hat{b}^{(3, 4)}_{k_2}(\bX)  \right) \right] & =   \E_{\bbP, 1, 2, 3, 4} \left[\E_{\bbP, 5} \left[ \left( \hat{p}^{(1, 2)}_{k_1}(\bX) \hat{b}^{(3, 4)}_{k_2}(\bX) \right)^2  \right] \right] \\
    & \quad -   \E_{\bbP, 1, 2, 3, 4} \left[\E_{\bbP, 5} \left[ \hat{p}^{(1, 2)}_{k_1}(\bX) \hat{b}^{(3, 4)}_{k_2}(\bX)   \right]^2 \right].
\end{align*}
By Lemma \ref{lem: nuisance function bounds},
\begin{equation*}
    \E_{\bbP, 1, 2, 3, 4} \left[\E_{\bbP, 5} \left[ \left( \hat{p}^{(1, 2)}_{k_1}(\bX) \hat{b}^{(3, 4)}_{k_2}(\bX) \right)^2  \right] \right] \asymp \frac{k_1k_2}{n^2}.
\end{equation*}
Moreover,
\begin{align*}
     & \E_{\bbP, 1, 2, 3, 4} \left[\E_{\bbP, 5} \left[ \hat{p}^{(1, 2)}_{k_1}(\bX) \hat{b}^{(3, 4)}_{k_2}(\bX)   \right]^2 \right] \\
     & \quad = \E_{\bbP, 1, 2, 3, 4} \left[ \iint \hat{p}^{(1, 2)}_{k_1}(\bx)\hat{p}^{(1, 2)}_{k_1}(\by) \hat{b}^{(3, 4)}_{k_2}(\bx)\hat{b}^{(3, 4)}_{k_2}(\by) d\bx d\by \right] \\
    & \quad \lesssim  \iint \left| \E_{\bbP, 1, 2} \left[\hat{p}^{(1, 2)}_{k_1}(\bx)\hat{p}^{(1, 2)}_{k_1}(\by) \right] \E_{\bbP, 3, 4} \left[ \hat{b}^{(3, 4)}_{k_2}(\bx)\hat{b}^{(3, 4)}_{k_2}(\by) \right] \right| d\bx d\by \\
    & \quad  \lesssim \iint \begin{array}{c}\left( 1 + \frac{1}{n} \E_{\bbP} \left[ \left| K_{V_{k_1}}(\bX, \bx)K_{V_{k_1}}(\bX, \by) \right| \right] \right) \\ \times \left( 1 + \frac{1}{n} \E_{\bbP} \left[ \left| K_{V_{k_2}}(\bX, \bx)K_{V_{k_2}}(\bX, \by) \right| \right] \right) \end{array} d\bx d\by \quad (\text{by Lemma \ref{lem: exp pxpy}})\\
    & \quad  \lesssim 1 + \frac{k_1 \wedge k_2}{n^2} \quad (\text{by Lemma \ref{lem: exp kernel single}}).
\end{align*}

The bound in (\ref{eq: if var lb helper2}) can be seen to hold from the following bounds, which follows from Lemma \ref{lem: nuisance function bounds} and the observation that $A, Y, \bX$ are bounded random variables and $p_{k_1}(p(\cdot)), p_{k_2}(b(\cdot))$ are bounded and 
\begin{align*}
    \E_{\bbP, 1, 2} [\E_{\bX} ( \hat{p}^{(1, 2)}_{k_1}(\bX)^2 )] & \lesssim \frac{k_1}{n} \\
    \E_{\bbP, 3, 4} [\E_{\bX} ( \hat{b}^{(3, 4)}_{k_2}(\bX)^2 )]  & \lesssim \frac{k_2}{n}. 
\end{align*}

\section{Proof of Theorem \ref{theorem: pred opt unknown f}} \label{sec: proof pred opt unknown f}

Without loss of generality, we let $j_1 = 1, j_1^{\prime} = 2$ for notational simplicity. Moreover, throughout the proof, we use the notation $p_{k_1}(p(\cdot))$ and $p_{k_2}(b(\cdot))$ as defined in (\ref{eq: p projection unknown f}) and (\ref{eq: b projection unknown f}) (appearing at the beginning of the proof of Theorem \ref{theorem: plugin unknown f}).

\subsection{Proof of the upper bound}

Let $\bbP \in \cP_{(\alpha,\beta,\gamma)}$ be arbitrary.\\

\noindent \textbf{Bounding the bias:}

By Lemma \ref{lem: proj dist}, we can bound the bias by
\begin{equation*}
    \| p_{k_1}(p(\cdot)) - p \|_{\infty} \lesssim k_1^{-\alpha/d} + \left( \frac{n}{\log n}\right)^{-\frac{\gamma}{2\gamma + d}}.
\end{equation*}
For $k_1 \ll n^{\frac{d}{2\alpha}}$ and $\gamma$ sufficiently large, we have that $k_1^{-\alpha/d} \gg \left( \frac{n}{\log n}\right)^{-\frac{\gamma}{2\gamma + d}}$. Therefore, we can express the bias by
\begin{equation*}
    \| p_{k_1}(p(\cdot)) - p \|_{\infty} \lesssim k_1^{-\alpha/d}.
\end{equation*}
Thus,
\begin{equation*}
    \int \left( \E_\bbP\left(\hat{p}^{(1, 2)}_{k_1}(\bx) - p(\bx)\right)\right)^2 d\bx \lesssim \| p_{k_1}(p(\cdot)) - p \|_{\infty}^2 \lesssim k_1^{-2\alpha/d}.
\end{equation*}

\noindent \textbf{Bounding the variance:}

Let $\bx \in [0, 1]^d$ be arbitrary. We have that
\begin{align*}
    \Var_{\bbP}(\hat{p}^{(1, 2)}_{k_1}(\bx)) & = \E_{\bbP}(\hat{p}^{(1, 2)}_{k_1}(\bx)^2) - p_{k_1}(p(\bx))^2.
\end{align*}
Following the same approach used in the proof of Lemma \ref{lem: exp pxpy}, we can express $\E_{\bbP}(\hat{p}^{(1, 2)}_{k_1}(\bx)^2)$ as
\begin{align*}
    \E_{\bbP}(\hat{p}^{(1, 2)}_{k_1}(\bx)^2) & = \left( 1 - \frac{1}{n}\right)\E_{\bbP}\left[ \Pi\left( \frac{pf}{\hat{f}^{(2)}} \Big| V_{k_1}\right)(\bx)^2 \right] + \frac{1}{n}  \E_{\bbP}\left[ \frac{A^2 K_{V_{k_1}}(\bX, \bx)^2}{\left( \hat{f}^{(2)}(\bX) \right)^2} \right]. 
\end{align*}
Recalling that $p_{k_1}(p(\bx)) = \E_{\bbP}[\Pi( \frac{pf}{\hat{f}^{(2)}} | V_{k_1})(\bx)]$, we can write
\begin{align*}
    \Var_{\bbP}(\hat{p}^{(1, 2)}_{k_1}(\bx)) & = \left( 1 - \frac{1}{n}\right) \Var_{\bbP}\left( \Pi\left( \frac{pf}{\hat{f}^{(2)}} \Big| V_{k_1}\right)(\bx) \right) - \frac{1}{n} p_{k_1}(p(\bx))^2 \\
    & \quad + \frac{1}{n}  \E_{\bbP}\left[ \frac{A^2 K_{V_{k_1}}(\bX, \bx)^2}{\left( \hat{f}^{(2)}(\bX) \right)^2} \right].
\end{align*}
We bound each of the three terms after integrating over $\bx$. The first term can be bounded by
\begin{align*}
    \int \Var_{\bbP}\left( \Pi\left( \frac{pf}{\hat{f}^{(2)}} \Big| V_{k_1}\right)(\bx) \right) d\bx & \leq \int \E_{\bbP}\left[ \left( \Pi\left( \frac{pf}{\hat{f}^{(2)}} \Big| V_{k_1}\right)(\bx) - \Pi\left( p | V_{k_1}\right)(\bx) \right)^2 \right] d\bx \\
    & \leq \E_{\bbP}\left[ \left\| \Pi\left( \frac{p(f - \hat{f}^{(2)})}{\hat{f}^{(2)}} \Big| V_{k_1}\right) \right\|_{\infty}^2 \right] \\
    & \lesssim \E_{\bbP}\left[ \| f - \hat{f}^{(2)} \|_{\infty}^2 \right] \quad (\text{by boundedness of $p, \hat{f}$})\\
    & \lesssim \left( \frac{n}{\log n}\right)^{-\frac{2\gamma}{2\gamma + d}}  \quad (\text{by assumption (ii) on $\hat{f}$}).
\end{align*}
Since $p_{k_1}$ is uniformly bounded, we can bound the second term by $\frac{1}{n} \int p_{k_1}(p(\bx))^2 d\bx \lesssim \frac{1}{n}$. Last, since $\hat{f}$ is uniformly bounded and $A$ is a bounded random variable, it follows from Lemma \ref{lem: exp kernel single} that
\begin{equation*}
    \frac{1}{n} \int \E_{\bbP}\left[ \frac{A^2 K_{V_{k_1}}(\bX, \bx)^2}{\left( \hat{f}^{(2)}(\bX) \right)^2} \right] d\bx \lesssim \frac{1}{n} \int \E_{\bbP}\left[ K_{V_{k_1}}(\bX, \bx)^2  \right] d\bx \lesssim \frac{k_1}{n}.
\end{equation*}
Therefore,
\begin{equation*}
    \int \Var_{\bbP}(\hat{p}^{(1, 2)}_{k_1}(\bx)) d\bx \lesssim \frac{k_1}{n} + \left( \frac{n}{\log n}\right)^{-\frac{2\gamma}{2\gamma + d}} \lesssim \frac{k_1}{n}.
\end{equation*}
where the last inequality follows by taking $\gamma$ sufficiently large and $k \gtrsim n^\epsilon$.

\subsection{Proof of the lower bound}

\noindent \textbf{Bounding the bias:}

Let $\bx \in [0, 1]^d$ and $\bbP \in \cP_{(\alpha,\beta,\gamma)}$ be arbitrary. Recall from the proof of Lemma \ref{lem: proj dist} that we can express the expectation of $\hat{p}^{(1, 2)}_{k_1}(\bx)$ as
\begin{equation*}
    \E_\bbP[\hat{p}^{(1, 2)}_{k_1}(\bx)] = \Pi(p | V_{k_1})(\bx) + R_{n, k_1}(\bx)
\end{equation*}
where $\|R_{n, k_1}\|_{\infty} \lesssim \left( \frac{n}{\log n}\right)^{-\frac{\gamma}{2\gamma + d}}$. Then,
\begin{align*}
    & \int \left( \E_\bbP\left(\hat{p}^{(1, 2)}_{k_1}(\bx) - p(\bx)\right)\right)^2 d\bx \\
    & \quad = \int (\Pi(p | V_{k_1})(\bx) - p(\bx))^2 d\bx  -2 \int (\Pi(p | V_{k_1})(\bx) - p(\bx)) R_{n, k_1}(\bx) d\bx \\
    & \quad + \int R_{n, k_1}(\bx)^2 d\bx  \\
    & = \int \Pi(p | V_{k_1}^{\perp})(\bx)^2 d\bx + 2 \int \Pi(p | V_{k_1}^{\perp})(\bx) R_{n, k_1}(\bx) d\bx  + \int R_{n, k_1}(\bx)^2 d\bx.   
\end{align*}
By the triangle inequality,
\begin{align*}
    \int \left( \E_\bbP\left(\hat{p}^{(1, 2)}_{k_1}(\bx) - p(\bx)\right)\right)^2 d\bx & \geq \left| \int \Pi(p | V_{k_1}^{\perp})(\bx)^2 d\bx \right| - 2 \left| \int \Pi(p | V_{k_1}^{\perp})(\bx) R_{n, k_1}(\bx) d\bx  \right| \\ & \quad -  \left| \int R_{n, k_1}(\bx)^2 d\bx \right|.
\end{align*}
We next derive upper bounds for the second and third terms and then derive a lower bound for the first term. For $k_1 \ll n^{\frac{d}{2\alpha}}$ and $\gamma$ sufficiently large, we have that $k_1^{-\alpha/d} \gg \left( \frac{n}{\log n}\right)^{-\frac{\gamma}{2\gamma + d}}$. Since $p$ is $\alpha$-Hölder, we have by Property P1 (Appendix \ref{sec: wavelets}) that $\| \Pi(p | V_{k_1}^{\perp}) \|_{\infty} \lesssim k_1^{-\alpha / d}$. Therefore,
\begin{align*}
    \left| \int \Pi(p | V_{k_1}^{\perp})(\bx) R_{n, k_1}(\bx) d\bx \right| & \lesssim \| \Pi(p | V_{k_1}^{\perp}) \|_{\infty} \|R_{n, k_1}\|_{\infty} \\
    & \lesssim k_1^{-\alpha / d} \left( \frac{n}{\log n}\right)^{-\frac{\gamma}{2\gamma + d}} \\
    & \ll k_1^{-2\alpha / d}.
\end{align*}
Moreover,
\begin{align*}
    \left| \int R_{n, k_1}(\bx)^2 d\bx   \right| & \lesssim \|R_{n, k_1}\|_{\infty}^2 \\
    & \lesssim \left( \frac{n}{\log n}\right)^{-\frac{2\gamma}{2\gamma + d}} \\
    & \ll k_1^{-2\alpha / d}.
\end{align*}
Now, suppose that we choose $p$ as given in (\ref{eq: p lb}). Then, recall from the derivation of the bias of the integral-based plug-in estimator that
\begin{equation*}
    \left| \int \Pi(p | V_{k_1}^{\perp})(\bx)^2 d\bx \right| \gtrsim k_1^{-2\alpha / d}
\end{equation*}
which completes the lower bound on the bias.

\noindent \textbf{Bounding the variance:}

Let $\bx \in [0, 1]^d$ and $\bbP \in \cP_{(\alpha,\beta,\gamma)}$ be arbitrary. Recall from the derivation of the upper bound of the variance that
\begin{align*}
    \Var_{\bbP}(\hat{p}^{(1, 2)}_{k_1}(\bx)) & = \left( 1 - \frac{1}{n}\right) \Var_{\bbP}\left( \Pi\left( \frac{pf}{\hat{f}^{(2)}} \Big| V_{k_1}\right)(\bx) \right) - \frac{1}{n} p_{k_1}(p(\bx))^2 \\
    & \quad + \frac{1}{n}  \E_{\bbP}\left[ \frac{A^2 K_{V_{k_1}}(\bX, \bx)^2}{\left( \hat{f}^{(2)}(\bX) \right)^2} \right].
\end{align*}
Since the first term is nonnegative and $p_{k_1}$ is uniformly bounded,
\begin{equation*}
    \int \Var_{\bbP}(\hat{p}^{(1, 2)}_{k_1}(\bx)) d\bx \geq -\frac{C}{n} + \frac{1}{n} \int \E_{\bbP}\left[ \frac{A^2 K_{V_{k_1}}(\bX, \bx)^2}{\left( \hat{f}^{(2)}(\bX) \right)^2} \right] d\bx 
\end{equation*}
for some $C > 0$. Suppose that $|A| \geq c$ with probability 1 for some $c>0$ and $\bX \sim \text{Uniform}([0,1]^d)$. Then,
\begin{align*}
    \int \E_{\bbP}\left[ \frac{A^2 K_{V_{k_1}}(\bX, \bx)^2}{\left( \hat{f}^{(2)}(\bX) \right)^2} \right] d\bx & \gtrsim \int \E_{\bbP}\left[ K_{V_{k_1}}(\bX, \bx)^2 \right] d\bx \quad (\text{by boundedness of $\hat{f}$})\\
    & \gtrsim k_1 \quad (\text{by Lemma \ref{lem: exp kernel single}})
\end{align*}
which completes the lower bound on the variance when taking $k_1 \gtrsim n^\epsilon$.

\section{Technical lemmas} \label{sec: technical lemmas}

\begin{lemma} \label{lem: exp pxpy}
For any $\bx, \by \in [0,1]^d$ and $\bbP \in \cP_{(\alpha, \beta, \gamma)}$, the following statements hold:
\begin{itemize}
    \item[(i)] 
    \begin{align*}
    \E_{\bbP,1} \left[\tilde{p}^{(1)}_{k_1}(\bx)\tilde{p}^{(1)}_{k_1}(\by) \right]  & = \Pi( p | V_{k_1})(\bx) \Pi( p | V_{k_1})(\by)  + r_{n,k_1}(\bx,\by)\\
        \E_{\bbP,3} \left[\tilde{b}^{(3)}_{k_2}(\bx)\tilde{b}^{(3)}_{k_2}(\by) \right]  & = \Pi( b | V_{k_2})(\bx) \Pi( b | V_{k_2})(\by)  + r^{\prime}_{n,k_2}(\bx,\by)  
        \end{align*}
    where for $R_{n,k}(\bx,\by)$ equal to $r_{n,k}(\bx,\by)$ or $r^{\prime}_{n,k}(\bx,\by)$
\begin{equation*}
    | R_{n,k}(\bx,\by) | \leq c\left( \frac{1}{n} + \frac{1}{n} \E_{\bbP} \left[ \left| K_{V_k}(\bX, \bx)K_{V_k}(\bX, \by) \right| \right] \right)
\end{equation*}
for some $c \in \mathbb{R}^+$ that does not depend on $\bx, \by$.

    \item[(ii)] For $c_1, c_2 \in \mathbb{R}^+$ that do not depend on $\bx, \by$,
    \begin{align*}
    \left| \E_{\bbP,1,2} \left[\hat{p}^{(1,2)}_{k_1}(\bx)\hat{p}^{(1,2)}_{k_1}(\by) \right] \right| & \leq c_1\left(1 + \frac{1}{n} \E_{\bbP} \left[ \left| K_{V_{k_1}}(\bX, \bx)K_{V_{k_1}}(\bX, \by) \right| \right]\right)  \\
    \left| \E_{\bbP,3,4} \left[\hat{b}^{(3,4)}_{k_2}(\bx)\hat{b}^{(3,4)}_{k_2}(\by) \right] \right| & \leq c_2 \left(1 + \frac{1}{n} \E_{\bbP} \left[ \left| K_{V_{k_2}}(\bX, \bx)K_{V_{k_2}}(\bX, \by) \right| \right]\right).  
    \end{align*}
    
    \item[(iii)] For $c_3, c_4 \in \mathbb{R}^+$ that do not depend on $\bx, \by$,
    \begin{align*}
    & \left| \E_{\bbP,1,2} \bigg[ (\hat{p}^{(1,2)}_{k_1}(\bx) - \tilde{p}^{(1)}_{k_1}(\bx))(\hat{p}^{(1,2)}_{k_1}(\by) - \tilde{p}^{(1)}_{k_1}(\by))\bigg] \right|\\
    & \quad \leq c_3 \left( \left( \frac{n}{\log n}\right)^{-\frac{2\gamma}{2\gamma + d}} + \frac{1}{n} \E_{\bbP} \left[ \left| K_{V_{k_1}}(\bX, \bx)K_{V_{k_1}}(\bX, \by) \right| \right]\right)  \\
    & \left| \E_{\bbP,3,4} \bigg[ (\hat{b}^{(3,4)}_{k_2}(\bx) - \tilde{b}^{(3)}_{k_2}(\bx))(\hat{b}^{(3,4)}_{k_2}(\by) - \tilde{b}^{(3)}_{k_2}(\by))\bigg] \right| \\
    & \quad \leq c_4\left(\left( \frac{n}{\log n}\right)^{-\frac{2\gamma}{2\gamma + d}} + \frac{1}{n} \E_{\bbP} \left[ \left| K_{V_{k_2}}(\bX, \bx)K_{V_{k_2}}(\bX, \by) \right| \right] \right).
\end{align*}
\end{itemize}

\end{lemma}
\begin{proof}
We will prove the statements in (i)-(iii) involving estimators of $p$, noting that the statements involving estimators of $b$ can be shown in the same manner. For notational convenience, we will let $\lesssim$ denote a uniform bound over $\bx, \by \in [0, 1]^d$ (i.e., for positive sequences $a_n(\bx, \by)$ and $b_n(\bx, \by)$, $a_n(\bx, \by) \lesssim b_n(\bx, \by)$ imply that there exists a positive constant $C$ such that $a_n(\bx, \by) \leq C b_n(\bx, \by) $ for all $n \geq n_0$ for all $\bx, \by \in [0, 1]^d$) throughout the proof. 

\begin{itemize}
    \item[(i)] We have that
\begin{align*}
    \E_{\bbP,1} \left[\tilde{p}^{(1)}_{k_1}(\bx)\tilde{p}^{(1)}_{k_1}(\by) \right] & = \E_{\bbP,1}\left[ \frac{1}{n^2} \sum_{\substack{i,j \in \D_1 \\ i \neq j}} \left( \frac{A_i K_{V_{k_1}}(\bX_i,\bx)}{f(\bX_i)} \right) \left(  \frac{A_j K_{V_{k_1}}(\bX_j,\by)}{f(\bX_j)} \right) \right] \\ & \quad  +  \E_{\bbP,1}\left[ \frac{1}{n^2} \sum_{\substack{i,j \in \D_1 \\ i = j}} \left( \frac{A_i K_{V_{k_1}}(\bX_i,\bx)}{f(\bX_i)} \right) \left(  \frac{A_j K_{V_{k_1}}(\bX_j,\by)}{f(\bX_j)} \right) \right]
\end{align*}
where
\begin{align*}
    & \E_{\bbP,1}\left[ \frac{1}{n^2} \sum_{\substack{i,j \in \D_1 \\ i \neq j}} \left( \frac{A_i K_{V_{k_1}}(\bX_i,\bx)}{f(\bX_i)} \right) \left(  \frac{A_j K_{V_{k_1}}(\bX_j,\by)}{f(\bX_j)} \right) \right] \\
    & \quad  = \frac{n(n-1)}{n^2}  \E_{\bbP,1}\left( \frac{A K_{V_{k_1}}(\bX, \bx)}{f(\bX)} \right)\E_{\bbP,1} \left(\frac{A K_{V_{k_1}}(\bX, \by)}{f(\bX)} \right)  \\
    & \quad  =  \left( 1 - \frac{1}{n}\right)\Pi( p | V_{k_1})(\bx) \Pi( p | V_{k_1})(\by) 
\end{align*}
and
\begin{align*}
    & \left| \E_{\bbP,1}\left[ \frac{1}{n^2} \sum_{\substack{i,j \in \D_1 \\ i = j}} \left( \frac{A_i K_{V_{k_1}}(\bX_i,\bx)}{f(\bX_i)} \right) \left(  \frac{A_j K_{V_{k_1}}(\bX_j,\by)}{f(\bX_j)} \right) \right] \right| \\
    & \quad \leq  \frac{1}{n}  \E_{\bbP,1}\left[ \frac{A^2 \left| K_{V_{k_1}}(\bX, \bx)K_{V_{k_1}}(\bX, \by) \right|}{\left( f(\bX) \right)^2} \right] \\
    & \quad \lesssim \frac{1}{n} \E_{\bbP, 1} \left[ \left| K_{V_{k_1}}(\bX, \bx)K_{V_{k_1}}(\bX, \by) \right| \right].  
\end{align*}
The result then follows from $\Pi( p | V_{k_1})$ being bounded.

\item[(ii)] We have that
\begin{align*}
    & \E_{\bbP,1,2} \left[\hat{p}^{(1,2)}_{k_1}(\bx)\hat{p}^{(1,2)}_{k_1}(\by) \right] \\
    & \quad = \E_{\bbP,1,2}\left[ \frac{1}{n^2} \sum_{\substack{i,j \in \D_1 \\ i \neq j}} \left( \frac{A_i K_{V_{k_1}}(\bX_i,\bx)}{\hat{f}^{(2)}(\bX_i)} \right) \left(  \frac{A_j K_{V_{k_1}}(\bX_j,\by)}{\hat{f}^{(2)}(\bX_j)} \right) \right] \\ & \quad \quad  +  \E_{\bbP,1,2}\left[ \frac{1}{n^2} \sum_{\substack{i,j \in \D_1 \\ i = j}} \left( \frac{A_i K_{V_{k_1}}(\bX_i,\bx)}{\hat{f}^{(2)}(\bX_i)} \right) \left(  \frac{A_j K_{V_{k_1}}(\bX_j,\by)}{\hat{f}^{(2)}(\bX_j)} \right) \right].
\end{align*}
The first term can be bounded by
\begin{align*}
    & \left| \E_{\bbP,1,2}\left[ \frac{1}{n^2} \sum_{\substack{i,j \in \D_1 \\ i \neq j}} \left( \frac{A_i K_{V_{k_1}}(\bX_i,\bx)}{\hat{f}^{(2)}(\bX_i)} \right) \left(  \frac{A_j K_{V_{k_1}}(\bX_j,\by)}{\hat{f}^{(2)}(\bX_j)} \right) \right] \right| \\
    & \quad = \frac{n(n-1)}{n^2} \left| \E_{\bbP, 2} \left[ \Pi\left( \frac{pf}{\hat{f}^{(2)}} | V_{k_1}\right)(\bx) \Pi\left( \frac{pf}{\hat{f}^{(2)}} | V_{k_1}\right)(\by) \right] \right|  \\
    & \quad \leq \frac{n(n-1)}{n^2}  \E_{\bbP, 2} \left[ \left\| \Pi\left( \frac{pf}{\hat{f}^{(2)}} | V_{k_1}\right) \right\|_{\infty}^2 \right]  \\
    & \quad \lesssim \frac{n(n-1)}{n^2}  \E_{\bbP, 2} \left[ \left\| \frac{pf}{\hat{f}^{(2)}} \right\|_{\infty}^2 \right] \\
    & \quad \lesssim 1
\end{align*}
where the last inequality follows from boundedness of $p$, $f$, and $\hat{f}^{(2)}$. The second term can be bounded by 
\begin{align*}
    & \left| \E_{\bbP,1,2}\left[ \frac{1}{n^2} \sum_{\substack{i,j \in \D_1 \\ i = j}} \left( \frac{A_i K_{V_{k_1}}(\bX_i,\bx)}{\hat{f}^{(2)}(\bX_i)} \right) \left(  \frac{A_j K_{V_{k_1}}(\bX_j,\by)}{\hat{f}^{(2)}(\bX_j)} \right) \right] \right| \\
    & \quad \leq  \frac{1}{n}  \E_{\bbP,1,2}\left[ \frac{A^2 \left| K_{V_{k_1}}(\bX, \bx)K_{V_{k_1}}(\bX, \by) \right|}{\left( \hat{f}^{(2)}(\bX) \right)^2} \right] \\
    & \quad \lesssim \frac{1}{n} \E_{\bbP, 1} \left[ \left| K_{V_{k_1}}(\bX, \bx)K_{V_{k_1}}(\bX, \by) \right| \right].  
\end{align*}

\item[(iii)] We have that
\begin{align*}
    & \E_{\bbP,1,2} \bigg[ (\hat{p}^{(1,2)}_{k_1}(\bx) - \tilde{p}^{(1)}_{k_1}(\bx))(\hat{p}^{(1,2)}_{k_1}(\by) - \tilde{p}^{(1)}_{k_1}(\by))\bigg] \\
    & \quad = \E_{\bbP,1,2} \left[ \frac{1}{n^2} \sum_{\substack{i,j \in \D_1 \\ i \neq j}}  \begin{array}{c}\left( \frac{A_i K_{V_{k_1}}(\bX_i,\bx)(f(\bX_i) - \hat{f}^{(2)}(\bX_i))}{f(\bX_i)\hat{f}^{(2)}(\bX_i)} \right) \\ \times   \left( \frac{A_j K_{V_{k_1}}(\bX_j,\by)(f(\bX_j) - \hat{f}^{(2)}(\bX_j))}{f(\bX_j)\hat{f}^{(2)}(\bX_j)} \right) \end{array} \right] \\ & \quad \quad + \E_{\bbP,1,2} \left[ \frac{1}{n^2} \sum_{\substack{i,j \in \D_1 \\ i = j}}  \begin{array}{c}\left( \frac{A_i K_{V_{k_1}}(\bX_i,\bx)(f(\bX_i) - \hat{f}^{(2)}(\bX_i))}{f(\bX_i)\hat{f}^{(2)}(\bX_i)} \right)  \\ \times \left( \frac{A_j K_{V_{k_1}}(\bX_j,\by)(f(\bX_j) - \hat{f}^{(2)}(\bX_j))}{f(\bX_j)\hat{f}^{(2)}(\bX_j)} \right) \end{array} \right] 
\end{align*}
where
\begin{align*}
    & \left| \E_{\bbP,1,2} \left[ \frac{1}{n^2} \sum_{\substack{i,j \in \D_1 \\ i \neq j}}  \begin{array}{c}\left( \frac{A_i K_{V_{k_1}}(\bX_i,\bx)(f(\bX_i) - \hat{f}^{(2)}(\bX_i))}{f(\bX_i)\hat{f}^{(2)}(\bX_i)} \right) \\ \times   \left( \frac{A_j K_{V_{k_1}}(\bX_j,\by)(f(\bX_j) - \hat{f}^{(2)}(\bX_j))}{f(\bX_j)\hat{f}^{(2)}(\bX_j)} \right) \end{array} \right] \right|\\
    & \quad = \frac{n(n-1)}{n^2} \left| \E_{\bbP,2} \left[ \Pi\left( \frac{p(f - \hat{f}^{(2)})}{\hat{f}^{(2)}} | V_{k_1}\right)(\bx) \Pi\left( \frac{p(f - \hat{f}^{(2)})}{\hat{f}^{(2)}} | V_{k_1}\right)(\by) \right] \right| \\
    & \quad \leq \E_{\bbP,2} \left[ \left\| \Pi\left( \frac{p(f - \hat{f}^{(2)})}{\hat{f}^{(2)}} | V_{k_1}\right) \right\|_{\infty}^2 \right] \\
    & \quad \lesssim \E_{\bbP,2} \left[ \left\| f - \hat{f}^{(2)} \right\|_{\infty}^2 \right] \\
    & \quad \lesssim \left( \frac{n}{\log n}\right)^{-\frac{2\gamma}{2\gamma + d}}
\end{align*}
and
\begin{align*}
    & \left| \E_{\bbP,1,2} \left[ \frac{1}{n^2} \sum_{\substack{i,j \in \D_1 \\ i = j}}  \begin{array}{c}\left( \frac{A_i K_{V_{k_1}}(\bX_i,\bx)(f(\bX_i) - \hat{f}^{(2)}(\bX_i))}{f(\bX_i)\hat{f}^{(2)}(\bX_i)} \right)  \\ \times \left( \frac{A_j K_{V_{k_1}}(\bX_j,\by)(f(\bX_j) - \hat{f}^{(2)}(\bX_j))}{f(\bX_j)\hat{f}^{(2)}(\bX_j)} \right) \end{array} \right] \right| \\
    & \quad \leq  \frac{1}{n}  \E_{\bbP,1,2} \left[ \frac{A^2 \left|K_{V_{k_1}}(\bX,\bx)K_{V_{k_1}}(\bX,\by)\right|(f(\bX) - \hat{f}^{(2)}(\bX))^2}{\left[f(\bX)\hat{f}^{(2)}(\bX)\right]^2}  \right] \\
    & \quad \lesssim  \frac{1}{n}  \E_{\bbP,1} \left[  \left|K_{V_{k_1}}(\bX,\bx)K_{V_{k_1}}(\bX,\by)\right|  \right]. 
\end{align*}
\end{itemize}

\end{proof}

\begin{lemma} \label{lem: exp kernel} For any $\bbP \in \cP_{(\alpha, \beta)} \cup \cP_{(\alpha, \beta, \gamma)}$,
\begin{align}
    \iint \E_{\bbP} \left[ \left| K_{V_{k_1}}(\bX, \bx)K_{V_{k_2}}(\bX, \by) \right| \right] d\bx d\by & \lesssim 1 \label{eq: exp kernel}\\
    \iint \begin{array}{c} \E_{\bbP} \left[ \left| K_{V_{k_1}}(\bX, \bx)K_{V_{k_1}}(\bX, \by) \right| \right] \\ \times \E_{\bbP} \left[ \left| K_{V_{k_2}}(\bX, \bx)K_{V_{k_2}}(\bX, \by) \right| \right] \end{array} d\bx d\by & \lesssim k_1 \wedge k_2  \label{eq: exp kernel squared} \\
    \iint \E_{\bbP} \left[ \left| K_{V_{k_1}}(\bX, \bx)K_{V_{k_2}}(\bX, \by) \right| \right]^2  d\bx d\by & \lesssim k_1 \wedge k_2  \label{eq: exp kernel squared v2} \\
    \iint \E_{\bbP} \left[ |K_{V_{k_1}}(\bX,\bx)K_{V_{k_2}}(\bX,\bx)K_{V_{k_1}}(\bX,\by) |\right] d\bx d\by & \lesssim k_1 \wedge k_2 \label{eq: exp kernel cubed}\\ 
    \iint \E_{\bbP} \left[ |K_{V_{k_1}}(\bX,\bx)K_{V_{k_2}}(\bX,\bx)K_{V_{k_2}}(\bX,\by)|\right] d\bx d\by & \lesssim k_1 \wedge k_2 \label{eq: exp kernel cubed v2}\\
    \iint \E_{\bbP} \left[ |K_{V_{k_1}}(\bX,\bx)K_{V_{k_2}}(\bX,\bx)K_{V_{k_1}}(\bX,\by)K_{V_{k_2}}(\bX,\by)| \right] d\bx d\by & \lesssim (k_1 \wedge k_2)^2.  \label{eq: exp kernel fourth}
\end{align}
\end{lemma}

\begin{proof}
As in the proof of Lemma \ref{lem: exp pxpy}, we will let $\lesssim$ denote a uniform bound over $\bx \in [0, 1]^d$ for notational convenience. While we write the proofs of the first two inequalities (i.e., (\ref{eq: exp kernel}) and (\ref{eq: exp kernel squared})) in full detail, we write the proofs of the remaining inequalities more concisely since they follow the exact same approach.

\textbf{Showing (\ref{eq: exp kernel})}: Recall from Appendix \ref{sec: wavelets} that for $\bx, \by \in [0, 1]^d$,
\begin{equation*}
    K_{V_k}(\bx, \by) = \sum_{\boldm \in \mathcal{Z}_k} \Phi_{k,\boldm}(\bx)\Phi_{k,\boldm}(\by)
\end{equation*}
where $|\mathcal{Z}_k| \asymp k$. Then,
\begin{equation*}
    K_{V_{k_1}}(\bX, \bx)K_{V_{k_2}}(\bX, \by) = \sum_{(\boldm, \boldm^{\prime}) \in \mathcal{Z}_{k_1} \times \mathcal{Z}_{k_2}} \Phi_{k_1, \boldm}(\bX) \Phi_{k_2, \boldm^{\prime}}(\bX) \Phi_{k_1, \boldm}(\bx) \Phi_{k_2, \boldm^{\prime}}(\by)
\end{equation*}
and so
\begin{align*}
    & \E_{\bbP} \left[ \left| K_{V_{k_1}}(\bX, \bx)K_{V_{k_2}}(\bX, \by) \right| \right] \\
    & \quad \leq \sum_{(\boldm, \boldm^{\prime}) \in \mathcal{Z}_{k_1} \times \mathcal{Z}_{k_2}}  |\Phi_{k_1, \boldm}(\bx) \Phi_{k_2, \boldm^{\prime}}(\by)| \E_{\bbP}[|\Phi_{k_1, \boldm}(\bX) \Phi_{k_2, \boldm^{\prime}}(\bX)|].
\end{align*}
Since $f$ is bounded,
\begin{equation*}
    \E_{\bbP}[|\Phi_{k_1, \boldm}(\bX) \Phi_{k_2, \boldm^{\prime}}(\bX)|] \lesssim \int |\Phi_{k_1, \boldm}(\bx) \Phi_{k_2, \boldm^{\prime}}(\bx)| d\bx.
\end{equation*}
Assume without loss of generality that $k_1 \geq k_2$. For each $\boldm \in \mathcal{Z}_{k_1}$, there are $O(1)$ values of $\boldm^{\prime} \in \mathcal{Z}_{k_2}$ such that the supports of $\Phi_{k_1, \boldm}$ and $\Phi_{k_2, \boldm^{\prime}}$ overlap. Letting $\mathcal{S}_{k_1, k_2}$ denote the set of such $(\boldm, \boldm^{\prime})$, this implies $|\mathcal{S}_{k_1, k_2}| \asymp k_1$. For any $(\boldm, \boldm^{\prime}) \in \mathcal{S}_{k_1, k_2}$
\begin{equation} \label{eq: exp kernel bound temp}
    \int |\Phi_{k_1, \boldm}(\bx) \Phi_{k_2, \boldm^{\prime}}(\bx)| d\bx \lesssim \sqrt{\frac{k_2}{k_1}}
\end{equation}
since (i) $\Phi_{k_1, \boldm}$ has a support which is contained in a set of Lebesgue measure $O(\frac{1}{k_1})$ and (ii) $|\Phi_{k_1, \boldm}|$ and $|\Phi_{k_2, \boldm^{\prime}}|$ are uniformly bounded above by $O(\sqrt{k_1})$ and $O(\sqrt{k_2})$, respectively. Therefore, we can write
\begin{equation}  \label{eq: exp kernel xy}
    \E_{\bbP} \left[ \left| K_{V_{k_1}}(\bX, \bx)K_{V_{k_2}}(\bX, \by) \right| \right] \lesssim \sqrt{\frac{k_2}{k_1}} \sum_{(\boldm, \boldm^{\prime}) \in \mathcal{S}_{k_1, k_2}}  |\Phi_{k_1, \boldm}(\bx) \Phi_{k_2, \boldm^{\prime}}(\by)|.
\end{equation}
Finally, we have that
\begin{align*}
    & \iint \E_{\bbP} \left[ \left| K_{V_{k_1}}(\bX, \bx)K_{V_{k_2}}(\bX, \by) \right| \right] d\bx d\by \\
    & \quad \lesssim \sqrt{\frac{k_2}{k_1}} \sum_{(\boldm, \boldm^{\prime}) \in \mathcal{S}_{k_1, k_2}} \underbrace{\int |\Phi_{k_1, \boldm}(\bx)| d\bx}_{\lesssim \frac{1}{\sqrt{k_1}}} \underbrace{\int |\Phi_{k_2, \boldm^{\prime}}(\by)| d\by}_{\lesssim \frac{1}{\sqrt{k_2}}} \lesssim 1.
\end{align*}

\textbf{Showing (\ref{eq: exp kernel squared})}: By (\ref{eq: exp kernel xy}),
\begin{align}
    & \iint \E_{\bbP} \left[ \left| K_{V_{k_1}}(\bX, \bx)K_{V_{k_1}}(\bX, \by) \right| \right] \E_{\bbP} \left[ \left| K_{V_{k_2}}(\bX, \bx)K_{V_{k_2}}(\bX, \by) \right| \right] d\bx d\by \nonumber \\
    & \quad  \lesssim  \sum_{\substack{(\boldm, \boldm^{\prime}) \in \mathcal{S}_{k_1, k_1} \\ (\boldr, \boldr^{\prime}) \in \mathcal{S}_{k_2, k_2}}} \int  \left|\Phi_{k_1, \boldm}(\bx)\Phi_{k_2, \boldr}(\bx)\right| d\bx   \int \left|\Phi_{k_1, \boldm^{\prime}}(\by)\Phi_{k_2, \boldr^{\prime}}(\by)\right| d\by.   \label{eq: summation for exp kernel squared}
\end{align}
Assume without loss of generality that $k_1 \geq k_2$. For each $\boldm$, there are $O(1)$ distinct $\boldr$ such that the supports of $\Phi_{k_1, \boldm}$ and $\Phi_{k_2, \boldr}$ overlap. For such $\boldm$ and $\boldr$, 
\begin{equation*}
    \int  \left|\Phi_{k_1, \boldm}(\bx)\Phi_{k_2, \boldr}(\bx)\right| d\bx \lesssim \sqrt{\frac{k_2}{k_1}}
\end{equation*}
since (i) $\Phi_{k_1, \boldm}$ has a support which is contained in a set of Lebesgue measure $O(\frac{1}{k_1})$ and (ii) $|\Phi_{k_1, \boldm}|$ and $|\Phi_{k_2, \boldr}|$ are uniformly bounded above by $O(\sqrt{k_1})$ and $O(\sqrt{k_2})$, respectively. Similarly, for each $\boldm^{\prime}$, there are $O(1)$ distinct $\boldr^{\prime}$ such that the supports of $\Phi_{k_1, \boldm^{\prime}}$ and $\Phi_{k_2, \boldr^{\prime}}$ overlap. For such $\boldm^{\prime}$ and $\boldr^{\prime}$,
\begin{equation*}
    \int  \left|\Phi_{k_1, \boldm^{\prime}}(\by)\Phi_{k_2, \boldr^{\prime}}(\by)\right| d\by \lesssim \sqrt{\frac{k_2}{k_1}}.
\end{equation*}
In summary, there are $O(k_1)$ non-zero terms in the summation in (\ref{eq: summation for exp kernel squared}), and each of the non-zero terms in the summation are bounded above by $O(\frac{k_2}{k_1})$. This implies (\ref{eq: exp kernel squared}).

\textbf{Showing (\ref{eq: exp kernel squared v2})}: Assume without loss of generality that $k_1 \geq k_2$. By (\ref{eq: exp kernel xy}),
\begin{align*}
    & \iint \E_{\bbP} \left[ \left| K_{V_{k_1}}(\bX, \bx)K_{V_{k_2}}(\bX, \by) \right| \right]^2 d\bx d\by \nonumber \\
    & \quad \lesssim \frac{k_2}{k_1} \sum_{\substack{(\boldm, \boldm^{\prime}) \in \mathcal{S}_{k_1, k_2} \\ (\boldr, \boldr^{\prime}) \in \mathcal{S}_{k_1, k_2}}} \int  |\Phi_{k_1, \boldm}(\bx) \Phi_{k_1, \boldr}(\bx) | d\bx \int |\Phi_{k_2, \boldm^{\prime}}(\by)  \Phi_{k_2, \boldr^{\prime}}(\by)| d\by. 
\end{align*}
By similar arguments made earlier in the proof, there are $O(k_1)$ non-zero terms in the summation in and each of the non-zero terms in the summation are bounded above by $O(1)$. This implies (\ref{eq: exp kernel squared v2}).

\textbf{Showing (\ref{eq: exp kernel cubed})}: We have that
\begin{align*}
    & K_{V_{k_1}}(\bX, \bx)K_{V_{k_2}}(\bX, \bx)K_{V_{k_1}}(\bX, \by) \\
    & \quad = \sum_{\substack{ \boldm, \boldr \in \mathcal{Z}_{k_1} \\ \boldm^{\prime} \in \mathcal{Z}_{k_2}}} \Phi_{k_1, \boldm}(\bX) \Phi_{k_2, \boldm^{\prime}}(\bX)\Phi_{k_1, \boldr}(\bX) \Phi_{k_1, \boldm}(\bx) \Phi_{k_2, \boldm^{\prime}}(\bx) \Phi_{k_1, \boldr}(\by). 
\end{align*}
and so
\begin{align*}
    & \E_{\bbP} \left[ |K_{V_{k_1}}(\bX, \bx)K_{V_{k_2}}(\bX, \bx)K_{V_{k_1}}(\bX, \by)| \right] \\
    & \quad \leq \sum_{\substack{ \boldm, \boldr \in \mathcal{Z}_{k_1} \\ \boldm^{\prime} \in \mathcal{Z}_{k_2}}} \E_{\bbP} \left[ |\Phi_{k_1, \boldm}(\bX) \Phi_{k_2, \boldm^{\prime}}(\bX)\Phi_{k_1, \boldr}(\bX) | \right] | \Phi_{k_1, \boldm}(\bx) \Phi_{k_2, \boldm^{\prime}}(\bx) \Phi_{k_1, \boldr}(\by) |.
\end{align*}
Then,
\begin{equation*}
    \E_{\bbP} \left[ | \Phi_{k_1, \boldm}(\bX) \Phi_{k_2, \boldm^{\prime}}(\bX)\Phi_{k_1, \boldr}(\bX)| \right] \lesssim \int |\Phi_{k_1, \boldm}(\bx) \Phi_{k_2, \boldm^{\prime}}(\bx)\Phi_{k_1, \boldr}(\bx)| d\bx. 
\end{equation*}
First, assume that $k_1 \geq k_2$. By similar arguments made earlier in the proof, 
\begin{equation*}
    \int |\Phi_{k_1, \boldm}(\bx) \Phi_{k_2, \boldm^{\prime}}(\bx)\Phi_{k_1, \boldr}(\bx)| d\bx = \begin{cases} O(\sqrt{k_2}) \quad & \text{if $(\boldm, \boldm^{\prime}, \boldr) \in \tilde{\mathcal{S}}_{k_1, k_2}$} \\ 0 & \text{otherwise}\end{cases}
\end{equation*}
where $|\tilde{\mathcal{S}}_{k_1, k_2}| \asymp k_1$. Therefore, we can write,
\begin{align*}
    & \iint \E_{\bbP} \left[ |K_{V_{k_1}}(\bX, \bx)K_{V_{k_2}}(\bX, \bx)K_{V_{k_1}}(\bX, \by)| \right] d\bx d\by \\
    & \quad \lesssim \sqrt{k_2} \sum_{(\boldm, \boldm^{\prime}, \boldr) \in \tilde{\mathcal{S}}_{k_1, k_2}} \underbrace{\int |\Phi_{k_1, \boldm}(\bx) \Phi_{k_2, \boldm^{\prime}}(\bx)| d\bx}_{\lesssim \sqrt{\frac{k_2}{k_1}}} \underbrace{\int |\Phi_{k_1, \boldr}(\by)| d\by}_{\lesssim \frac{1}{\sqrt{k_1}}} \lesssim k_2.
\end{align*}
Next, assume that $k_1 \leq k_2$. By similar arguments made earlier in the proof, 
\begin{equation*}
    \int |\Phi_{k_1, \boldm}(\bx) \Phi_{k_2, \boldm^{\prime}}(\bx)\Phi_{k_1, \boldr}(\bx)| d\bx = \begin{cases} O\left(\frac{k_1}{\sqrt{k_2}}\right) \quad & \text{if $(\boldm, \boldm^{\prime}, \boldr) \in \tilde{\mathcal{S}}^{\prime}_{k_1, k_2}$} \\ 0 & \text{otherwise}\end{cases}
\end{equation*}
where $|\tilde{\mathcal{S}}^{\prime}_{k_1, k_2}| \asymp k_2$. Therefore, 
\begin{align*}
    & \iint \E_{\bbP} \left[ |K_{V_{k_1}}(\bX, \bx)K_{V_{k_2}}(\bX, \bx)K_{V_{k_1}}(\bX, \by)| \right] d\bx d\by \\
    & \quad \lesssim \frac{k_1}{\sqrt{k_2}} \sum_{(\boldm, \boldm^{\prime}, \boldr) \in \tilde{\mathcal{S}}^{\prime}_{k_1, k_2}} \underbrace{\int |\Phi_{k_1, \boldm}(\bx) \Phi_{k_2, \boldm^{\prime}}(\bx)| d\bx}_{\lesssim \sqrt{\frac{k_1}{k_2}}} \underbrace{\int |\Phi_{k_1, \boldr}(\by)| d\by}_{\lesssim \frac{1}{\sqrt{k_1}}} \lesssim k_1.
\end{align*}

\textbf{Showing (\ref{eq: exp kernel cubed v2})}: We have that 
\begin{align*}
    & K_{V_{k_1}}(\bX,\bx)K_{V_{k_2}}(\bX,\bx)K_{V_{k_2}}(\bX,\by) \\
    & \quad = \sum_{\substack{ \boldm \in \mathcal{Z}_{k_1} \\ \boldm^{\prime}, \boldr \in \mathcal{Z}_{k_2}}} \Phi_{k_1, \boldm}(\bX) \Phi_{k_2, \boldm^{\prime}}(\bX)\Phi_{k_2, \boldr}(\bX) \Phi_{k_1, \boldm}(\bx) \Phi_{k_2, \boldm^{\prime}}(\bx) \Phi_{k_2, \boldr}(\by)
\end{align*}
and so
\begin{align*}
    & \E_{\bbP} \left[ |K_{V_{k_1}}(\bX, \bx)K_{V_{k_2}}(\bX, \bx)K_{V_{k_2}}(\bX, \by)| \right] \\
    & \quad \leq \sum_{\substack{ \boldm \in \mathcal{Z}_{k_1} \\ \boldm^{\prime}, \boldr \in \mathcal{Z}_{k_2}}} \E_{\bbP} \left[ |\Phi_{k_1, \boldm}(\bX) \Phi_{k_2, \boldm^{\prime}}(\bX)\Phi_{k_2, \boldr}(\bX) | \right] | \Phi_{k_1, \boldm}(\bx) \Phi_{k_2, \boldm^{\prime}}(\bx) \Phi_{k_2, \boldr}(\by) |.
\end{align*}
Then,
\begin{equation*}
     \E_{\bbP} \left[ | \Phi_{k_1, \boldm}(\bX) \Phi_{k_2, \boldm^{\prime}}(\bX)\Phi_{k_2, \boldr}(\bX)| \right] \lesssim \int |\Phi_{k_1, \boldm}(\bx) \Phi_{k_2, \boldm^{\prime}}(\bx)\Phi_{k_2, \boldr}(\bx)| d\bx. 
\end{equation*}
First, assume that $k_1 \geq k_2$. By similar arguments made earlier in the proof
\begin{equation*}
    \int |\Phi_{k_1, \boldm}(\bx) \Phi_{k_2, \boldm^{\prime}}(\bx)\Phi_{k_2, \boldr}(\bx)| d\bx  = \begin{cases} O\left(\frac{k_2}{\sqrt{k_1}}\right) \quad & \text{if $(\boldm, \boldm^{\prime}, \boldr) \in \check{\mathcal{S}}_{k_1, k_2}$} \\ 0 & \text{otherwise}\end{cases}
\end{equation*}
where $|\check{\mathcal{S}}_{k_1, k_2}| \asymp k_1$. Therefore, we can write 
\begin{align*}
    & \iint \E_{\bbP} \left[ |K_{V_{k_1}}(\bX, \bx)K_{V_{k_2}}(\bX, \bx)K_{V_{k_2}}(\bX, \by)| \right] d\bx d\by \\
    & \quad \lesssim \frac{k_2}{\sqrt{k_1}} \sum_{(\boldm, \boldm^{\prime}, \boldr) \in \check{\mathcal{S}}_{k_1, k_2}} \underbrace{\int |\Phi_{k_1, \boldm}(\bx) \Phi_{k_2, \boldm^{\prime}}(\bx)| d\bx}_{\lesssim \sqrt{\frac{k_2}{k_1}}} \underbrace{\int |\Phi_{k_2, \boldr}(\by)| d\by}_{\lesssim \frac{1}{\sqrt{k_2}}} \lesssim k_2.
\end{align*}
Next, assume that $k_1 \leq k_2$. By similar arguments made earlier in the proof
\begin{equation*}
    \int |\Phi_{k_1, \boldm}(\bx) \Phi_{k_2, \boldm^{\prime}}(\bx)\Phi_{k_2, \boldr}(\bx)| d\bx  = \begin{cases} O(\sqrt{k_1}) \quad & \text{if $(\boldm, \boldm^{\prime}, \boldr) \in \check{\mathcal{S}}^{\prime}_{k_1, k_2}$} \\ 0 & \text{otherwise}\end{cases}
\end{equation*}
where $|\check{\mathcal{S}}^{\prime}_{k_1, k_2}| \asymp k_2$. Therefore, 
\begin{align*}
    & \iint \E_{\bbP} \left[ |K_{V_{k_1}}(\bX, \bx)K_{V_{k_2}}(\bX, \bx)K_{V_{k_2}}(\bX, \by)| \right] d\bx d\by \\
    & \quad \lesssim \sqrt{k_1} \sum_{(\boldm, \boldm^{\prime}, \boldr) \in \check{\mathcal{S}}^{\prime}_{k_1, k_2}} \underbrace{\int |\Phi_{k_1, \boldm}(\bx) \Phi_{k_2, \boldm^{\prime}}(\bx)| d\bx}_{\lesssim \sqrt{\frac{k_1}{k_2}}} \underbrace{\int |\Phi_{k_2, \boldr}(\by)| d\by}_{\lesssim \frac{1}{\sqrt{k_2}}} \lesssim k_1.
\end{align*}

\textbf{Showing (\ref{eq: exp kernel fourth})}: We have that
\begin{align*}
    & K_{V_{k_1}}(\bX, \bx)K_{V_{k_2}}(\bX, \bx)K_{V_{k_1}}(\bX, \by)K_{V_{k_2}}(\bX, \by) \\
    & \quad = \sum_{\substack{\boldm, \boldr  \in \mathcal{Z}_{k_1} \\ \boldm^{\prime}, \boldr^{\prime} \in \mathcal{Z}_{k_2}}} \begin{array}{c} \Phi_{k_1, \boldm}(\bX) \Phi_{k_2, \boldm^{\prime}}(\bX)\Phi_{k_1, \boldr}(\bX) \Phi_{k_2, \boldr^{\prime}}(\bX) \\ \times \Phi_{k_1, \boldm}(\bx) \Phi_{k_2, \boldm^{\prime}}(\bx) \Phi_{k_1, \boldr}(\by) \Phi_{k_2, \boldr^{\prime}}(\by) \end{array}
\end{align*}
and so
\begin{align*}
    & \E_{\bbP} \left[ |K_{V_{k_1}}(\bX, \bx)K_{V_{k_2}}(\bX, \bx)K_{V_{k_1}}(\bX, \by)K_{V_{k_2}}(\bX, \by)| \right] \\
    & \quad \leq \sum_{\substack{\boldm, \boldr  \in \mathcal{Z}_{k_1} \\ \boldm^{\prime}, \boldr^{\prime} \in \mathcal{Z}_{k_2}}} \begin{array}{c} \E_{\bbP} \left[ |\Phi_{k_1, \boldm}(\bX) \Phi_{k_2, \boldm^{\prime}}(\bX)\Phi_{k_1, \boldr}(\bX) \Phi_{k_2, \boldr^{\prime}}(\bX)| \right] \\ \times |\Phi_{k_1, \boldm}(\bx) \Phi_{k_2, \boldm^{\prime}}(\bx) \Phi_{k_1, \boldr}(\by) \Phi_{k_2, \boldr^{\prime}}(\by) |\end{array}.
\end{align*}
Then,
\begin{align*}
    & \E_{\bbP} \left[ |\Phi_{k_1, \boldm}(\bX) \Phi_{k_2, \boldm^{\prime}}(\bX)\Phi_{k_1, \boldr}(\bX) \Phi_{k_2, \boldr^{\prime}}(\bX)| \right]\\
    & \quad \lesssim \int | \Phi_{k_1, \boldm}(\bx) \Phi_{k_2, \boldm^{\prime}}(\bx)\Phi_{k_1, \boldr}(\bx) \Phi_{k_2, \boldr^{\prime}}(\bx)| d\bx.
\end{align*}
Assume without loss of generality that $k_1 \geq k_2$. By similar arguments made earlier in the proof,
\begin{equation*}
    \int | \Phi_{k_1, \boldm}(\bx) \Phi_{k_2, \boldm^{\prime}}(\bx)\Phi_{k_1, \boldr}(\bx) \Phi_{k_2, \boldr^{\prime}}(\bx)| d\bx = \begin{cases} O(k_2) \quad & \text{if $(\boldm, \boldm^{\prime}, \boldr, \boldr^{\prime}) \in \bar{\mathcal{S}}_{k_1, k_2}$} \\ 0 & \text{otherwise}\end{cases}
\end{equation*}
where $|\bar{\mathcal{S}}_{k_1, k_2}| \asymp k_1$. Therefore, we can write,
\begin{align*}
    & \iint \E_{\bbP} \left[ |K_{V_{k_1}}(\bX, \bx)K_{V_{k_2}}(\bX, \bx)K_{V_{k_1}}(\bX, \by)K_{V_{k_2}}(\bX, \by)| \right] d\bx d\by \\
    & \quad \lesssim k_2 \sum_{(\boldm, \boldm^{\prime}, \boldr, \boldr^{\prime}) \in \bar{\mathcal{S}}_{k_1, k_2}} \underbrace{\int |\Phi_{k_1, \boldm}(\bx) \Phi_{k_2, \boldm^{\prime}}(\bx)| d\bx}_{\lesssim \sqrt{\frac{k_2}{k_1}}} \underbrace{\int |\Phi_{k_1, \boldr}(\by)\Phi_{k_2, \boldr^{\prime}}(\by)| d\by}_{\lesssim \sqrt{\frac{k_2}{k_1}}} \lesssim k_2^2.
\end{align*}

\end{proof}

\begin{lemma} \label{lem: exp kernel single}
The following statements hold:
\begin{itemize}
    \item[(i)] For any $\bbP \in \cP_{(\alpha, \beta)} \cup \cP_{(\alpha, \beta, \gamma)}$,
\begin{align}
    \| \E_{\bbP} \left[ \left| K_{V_k}(\bX, \cdot) \right| \right] \|_{\infty}  & \lesssim 1 \label{eq: exp kernel single one}\\
    \int \E_{\bbP} \left[ |K_{V_{k_1}}(\bX, \bx) K_{V_{k_2}}(\bX, \bx)| \right] d\bx  & \lesssim k_1 \wedge k_2 \label{eq: exp kernel single second} \\
    \int \E_{\bbP}\left[ | K_{V_{k_1}}(\bX,\bx)^2K_{V_{k_2}}(\bX,\bx)| \right] d\bx & \lesssim k_1k_2 \label{eq: exp kernel single third} \\
    \int \E_{\bbP}\left[ K_{V_{k_1}}(\bX,\bx)^2 K_{V_{k_2}}(\bX,\bx)^2 \right] d\bx & \lesssim (k_1 \wedge k_2)^2 (k_1 \vee k_2) \label{eq: exp kernel single fourth} \\
    \int \E_{\bbP} \left[ K_{V_{k_1}}(\bX, \bx)^2\right] \E_{\bbP}\left[ K_{V_{k_2}}(\bX, \bx)^2 \right] d\bx  & \lesssim k_1 k_2 \label{eq: exp kernel single pt3} \\
    \int \E_{\bbP} \left[ |K_{V_{k_1}}(\bX, \bx) K_{V_{k_2}}(\bX, \bx)| \right]^2 d\bx  & \lesssim (k_1 \wedge k_2)^2 \label{eq: exp kernel single second squared}.
\end{align}

\item[(ii)] For any $\bbP \in \cP_{(\alpha, \beta)} \cup \cP_{(\alpha, \beta, \gamma)}$ with $\bX \sim \text{Uniform}([0,1]^d)$, 
\begin{align}
    \int \E_{\bbP} \left[ K_{V_{k_1}}(\bX, \bx)^2\right] \E_{\bbP}\left[ K_{V_{k_2}}(\bX, \bx)^2 \right] d\bx  & \gtrsim k_1 k_2 \label{eq: exp kernel single lb pt1}\\
    \int \E_{\bbP} \left[ K_{V_{k_1}}(\bX, \bx) K_{V_{k_2}}(\bX, \bx) \right] d\bx  & \gtrsim k_1 \wedge k_2. \label{eq: exp kernel single lb pt2}
\end{align}
\end{itemize}

\end{lemma}

\begin{proof}

\begin{itemize}
    \item[(i)] We follow an approach similar to the proof of Lemma \ref{lem: exp kernel}. As in the proof of Lemma \ref{lem: exp pxpy}, we will let $\lesssim$ denote a uniform bound over $\bx \in [0, 1]^d$ for notational convenience.

\textbf{Showing (\ref{eq: exp kernel single one})}: Recall from Appendix \ref{sec: wavelets} that for $\bx, \by \in [0, 1]^d$,
\begin{equation*}
    K_{V_k}(\bx, \by) = \sum_{\boldm \in \mathcal{Z}_k} \Phi_{k,\boldm}(\bx)\Phi_{k,\boldm}(\by)
\end{equation*}
where $|\mathcal{Z}_k| \asymp k$. Then,
\begin{equation*}
    \E_{\bbP} \left[ \left| K_{V_k}(\bX, \bx) \right| \right] \lesssim \sum_{\boldm \in \mathcal{Z}_k} | \Phi_{k,\boldm}(\bx) | \int | \Phi_{k,\boldm}(\bx^{\prime}) | d\bx^{\prime}. 
\end{equation*}
Since $|\Phi_{k,\boldm}|$ is uniformly bounded above by $O(\sqrt{k})$ and has support contained in a set of Lebesgue measure $O(\frac{1}{k})$, $\int | \Phi_{k,\boldm}(\bx^{\prime}) | d\bx^{\prime}  \lesssim \frac{1}{\sqrt{k}}$. Therefore,
\begin{align*}
    \E_{\bbP} \left[ \left| K_{V_k}(\bX, \bx) \right| \right] &  \lesssim \frac{1}{\sqrt{k}} \sum_{\boldm \in \mathcal{Z}_k} | \Phi_{k,\boldm}(\bx) | \lesssim 1
\end{align*}
where the last inequality follows from the observation that (i) there are $O(1)$ distinct $\boldm$ such that $\Phi_{k,\boldm}(\bx)$ is non-zero, and (ii) $|\Phi_{k,\boldm}|$ is uniformly bounded above by $O(\sqrt{k})$. 

\textbf{Showing (\ref{eq: exp kernel single second})}: Recall that we can write
\begin{equation*}
    K_{V_{k_1}}(\bX, \bx)K_{V_{k_2}}(\bX, \bx) = \sum_{\boldm, \boldm^{\prime} \in \mathcal{Z}_{k_1} \times \mathcal{Z}_{k_2}} \Phi_{k_1, \boldm}(\bx) \Phi_{k_2, \boldm^{\prime}}(\bx) \Phi_{k_1, \boldm}(\bX) \Phi_{k_2, \boldm^{\prime}}(\bX).
\end{equation*}
Then,
\begin{align*}
    & \E_{\bbP}\left[ \left| K_{V_{k_1}}(\bX, \bx)K_{V_{k_2}}(\bX, \bx) \right| \right] \\
    & \quad \lesssim \sum_{\boldm, \boldm^{\prime} \in \mathcal{Z}_{k_1} \times \mathcal{Z}_{k_2}} |\Phi_{k_1, \boldm}(\bx) \Phi_{k_2, \boldm^{\prime}}(\bx)| \int | \Phi_{k_1, \boldm}(\bx^{\prime}) \Phi_{k_2, \boldm^{\prime}}(\bx^{\prime}) | d\bx^{\prime}
\end{align*}
and so
\begin{equation*}
    \int \E_{\bbP}\left[ \left| K_{V_{k_1}}(\bX, \bx)K_{V_{k_2}}(\bX, \bx) \right| \right] d \bx \lesssim \sum_{\boldm, \boldm^{\prime} \in \mathcal{Z}_{k_1} \times \mathcal{Z}_{k_2}} \left(\int | \Phi_{k_1, \boldm}(\bx) \Phi_{k_2, \boldm^{\prime}}(\bx) | d\bx \right)^2.
\end{equation*}
Suppose without loss of generality that $k_1 \geq k_2$. Recall from the proof of (\ref{eq: exp kernel}) in Lemma \ref{lem: exp kernel} (specifically, (\ref{eq: exp kernel bound temp})), 
\begin{equation*}
    \int | \Phi_{k_1, \boldm}(\bx) \Phi_{k_2, \boldm^{\prime}}(\bx) | d\bx  = \begin{cases} O\left(\sqrt{\frac{k_2}{k_1}}\right) \quad & \text{if $(\boldm, \boldm^{\prime}) \in \mathcal{S}_{k_1, k_2}$} \\ 0 & \text{otherwise}\end{cases}
\end{equation*}
where $|\mathcal{S}_{k_1, k_2}| \asymp k_1$. Therefore, 
\begin{equation*}
    \int \E_{\bbP}\left[ \left| K_{V_{k_1}}(\bX, \bx)K_{V_{k_2}}(\bX, \bx) \right| \right] d\bx \lesssim  \sum_{(\boldm, \boldm^{\prime}) \in \mathcal{S}_{k_1, k_2}} \left(\sqrt{\frac{k_2}{k_1}}\right)^2 \lesssim k_2.
\end{equation*}

\textbf{Showing (\ref{eq: exp kernel single third})}: We follow a nearly identical approach used to show (\ref{eq: exp kernel cubed}). We have that
\begin{align*}
    & K_{V_{k_1}}(\bX, \bx)^2K_{V_{k_2}}(\bX, \bx) \\
    & \quad = \sum_{\substack{ \boldm, \boldr \in \mathcal{Z}_{k_1} \\ \boldm^{\prime} \in \mathcal{Z}_{k_2}}} \Phi_{k_1, \boldm}(\bX) \Phi_{k_2, \boldm^{\prime}}(\bX)\Phi_{k_1, \boldr}(\bX) \Phi_{k_1, \boldm}(\bx) \Phi_{k_2, \boldm^{\prime}}(\bx) \Phi_{k_1, \boldr}(\bx). 
\end{align*}
Then,
\begin{align*}
    & \E_{\bbP} \left[ |K_{V_{k_1}}(\bX, \bx)^2K_{V_{k_2}}(\bX, \bx)| \right] \\
    & \quad \lesssim \sum_{\substack{ \boldm, \boldr \in \mathcal{Z}_{k_1} \\ \boldm^{\prime} \in \mathcal{Z}_{k_2}}}  | \Phi_{k_1, \boldm}(\bx) \Phi_{k_2, \boldm^{\prime}}(\bx) \Phi_{k_1, \boldr}(\bx) | \int |\Phi_{k_1, \boldm}(\bx^{\prime}) \Phi_{k_2, \boldm^{\prime}}(\bx^{\prime})\Phi_{k_1, \boldr}(\bx^{\prime}) | d\bx^{\prime}
\end{align*}
and so
\begin{align*}
    & \int \E_{\bbP} \left[ |K_{V_{k_1}}(\bX, \bx)^2K_{V_{k_2}}(\bX, \bx)| \right] d\bx \\
    & \quad \lesssim \sum_{\substack{ \boldm, \boldr \in \mathcal{Z}_{k_1} \\ \boldm^{\prime} \in \mathcal{Z}_{k_2}}} \left( \int |\Phi_{k_1, \boldm}(\bx) \Phi_{k_2, \boldm^{\prime}}(\bx)\Phi_{k_1, \boldr}(\bx) | d\bx \right)^2.
\end{align*}
Suppose that $k_1 \geq k_2$. Recall from the proof of (\ref{eq: exp kernel cubed}) in Lemma \ref{lem: exp kernel} that 
\begin{equation*}
    \int |\Phi_{k_1, \boldm}(\bx) \Phi_{k_2, \boldm^{\prime}}(\bx)\Phi_{k_1, \boldr}(\bx) | d\bx = \begin{cases} O(\sqrt{k_2}) \quad & \text{if $(\boldm, \boldm^{\prime}, \boldr) \in \tilde{\mathcal{S}}_{k_1, k_2}$} \\ 0 & \text{otherwise}\end{cases}
\end{equation*}
where $|\tilde{\mathcal{S}}_{k_1, k_2}| \asymp k_1$. Therefore, 
\begin{equation*}
    \int \E_{\bbP} \left[ |K_{V_{k_1}}(\bX, \bx)^2K_{V_{k_2}}(\bX, \bx)| \right] d\bx \lesssim \sum_{(\boldm, \boldm^{\prime}, \boldr) \in \tilde{\mathcal{S}}_{k_1, k_2}} \left(\sqrt{k_2}\right)^2  \lesssim k_1 k_2.
\end{equation*}
Next, suppose that $k_1 \leq k_2$. Recall from the proof of (\ref{eq: exp kernel cubed}) in Lemma \ref{lem: exp kernel} that 
\begin{equation*}
    \int |\Phi_{k_1, \boldm}(\bx) \Phi_{k_2, \boldm^{\prime}}(\bx)\Phi_{k_1, \boldr}(\bx) | d\bx = \begin{cases} O\left(\frac{k_1}{\sqrt{k_2}}\right) \quad & \text{if $(\boldm, \boldm^{\prime}, \boldr) \in \tilde{\mathcal{S}}^{\prime}_{k_1, k_2}$} \\ 0 & \text{otherwise}\end{cases}
\end{equation*}
where $|\tilde{\mathcal{S}}^{\prime}_{k_1, k_2}| \asymp k_2$. Therefore, 
\begin{equation*}
    \int \E_{\bbP} \left[ |K_{V_{k_1}}(\bX, \bx)^2K_{V_{k_2}}(\bX, \bx)| \right] d\bx  \lesssim  \sum_{(\boldm, \boldm^{\prime}, \boldr) \in \tilde{\mathcal{S}}^{\prime}_{k_1, k_2}} \left( \frac{k_1}{\sqrt{k_2}} \right)^2  \lesssim k_1^2 \lesssim k_1k_2.
\end{equation*}

\textbf{Showing (\ref{eq: exp kernel single fourth})}: 
We follow a nearly identical approach used to show (\ref{eq: exp kernel fourth}) in Lemma \ref{lem: exp kernel}. We have that
\begin{align*}
    & K_{V_{k_1}}(\bX,\bx)^2K_{V_{k_2}}(\bX,\bx)^2 \\
    & \quad = \sum_{\substack{\boldm, \boldr  \in \mathcal{Z}_{k_1} \\ \boldm^{\prime}, \boldr^{\prime} \in \mathcal{Z}_{k_2}}} \begin{array}{c} \Phi_{k_1, \boldm}(\bX) \Phi_{k_2, \boldm^{\prime}}(\bX)\Phi_{k_1, \boldr}(\bX) \Phi_{k_2, \boldr^{\prime}}(\bX) \\ \times \Phi_{k_1, \boldm}(\bx) \Phi_{k_2, \boldm^{\prime}}(\bx) \Phi_{k_1, \boldr}(\bx) \Phi_{k_2, \boldr^{\prime}}(\bx) \end{array}.
\end{align*}
Then,
\begin{align*}
    & \E_{\bbP}\left[ K_{V_{k_1}}(\bX,\bx)^2 K_{V_{k_2}}(\bX,\bx)^2 \right] \\
    & \quad \lesssim \sum_{\substack{\boldm, \boldr  \in \mathcal{Z}_{k_1} \\ \boldm^{\prime}, \boldr^{\prime} \in \mathcal{Z}_{k_2}}} \begin{array}{c} |\Phi_{k_1, \boldm}(\bx) \Phi_{k_2, \boldm^{\prime}}(\bx) \Phi_{k_1, \boldr}(\bx) \Phi_{k_2, \boldr^{\prime}}(\bx)| \\ \times \int |\Phi_{k_1, \boldm}(\bx^{\prime}) \Phi_{k_2, \boldm^{\prime}}(\bx^{\prime})\Phi_{k_1, \boldr}(\bx^{\prime}) \Phi_{k_2, \boldr^{\prime}}(\bx^{\prime})| d\bx^{\prime} \end{array} 
\end{align*}
and so
\begin{align*}
    & \int \E_{\bbP}\left[ K_{V_{k_1}}(\bX,\bx)^2 K_{V_{k_2}}(\bX,\bx)^2 \right] d\bx \\
    & \quad \lesssim \sum_{\substack{\boldm, \boldr  \in \mathcal{Z}_{k_1} \\ \boldm^{\prime}, \boldr^{\prime} \in \mathcal{Z}_{k_2}}}\left( \int |\Phi_{k_1, \boldm}(\bx) \Phi_{k_2, \boldm^{\prime}}(\bx)\Phi_{k_1, \boldr}(\bx) \Phi_{k_2, \boldr^{\prime}}(\bx)|  d\bx \right)^2.
\end{align*}
Assume without loss of generality that $k_1 \geq k_2$. Recall from the proof of (\ref{eq: exp kernel fourth}) in Lemma \ref{lem: exp kernel} that
\begin{equation*}
    \int | \Phi_{k_1, \boldm}(\bx) \Phi_{k_2, \boldm^{\prime}}(\bx)\Phi_{k_1, \boldr}(\bx) \Phi_{k_2, \boldr^{\prime}}(\bx)| d\bx = \begin{cases} O(k_2) \quad & \text{if $(\boldm, \boldm^{\prime}, \boldr, \boldr^{\prime}) \in \bar{\mathcal{S}}_{k_1, k_2}$} \\ 0 & \text{otherwise}\end{cases}
\end{equation*}
where $|\bar{\mathcal{S}}_{k_1, k_2}| \asymp k_1$. Therefore, 
\begin{equation*}
    \int \E_{\bbP}\left[ K_{V_{k_1}}(\bX,\bx)^2 K_{V_{k_2}}(\bX,\bx)^2 \right] d\bx \lesssim \sum_{(\boldm, \boldm^{\prime}, \boldr, \boldr^{\prime}) \in \bar{\mathcal{S}}_{k_1, k_2}} \left( k_2 \right)^2 \lesssim k_1k_2^2.
\end{equation*}

\textbf{Showing (\ref{eq: exp kernel single pt3})}: Recall from the proof of (\ref{eq: exp kernel single second}) that, 
\begin{align*}
    \E_{\bbP} \left[ K_{V_{k_1}}(\bX, \bx)^2\right] & \lesssim \sum_{(\boldm, \boldm^{\prime}) \in \mathcal{S}_{k_1, k_1}} | \Phi_{k_1, \boldm}(\bx) \Phi_{k_1, \boldm^{\prime}}(\bx) | \\
    \E_{\bbP} \left[ K_{V_{k_2}}(\bX, \bx)^2\right] & \lesssim \sum_{(\boldr, \boldr^{\prime}) \in \mathcal{S}_{k_2, k_2}} | \Phi_{k_2, \boldr}(\bx) \Phi_{k_2, \boldr^{\prime}}(\bx) |
\end{align*}
where $|\mathcal{S}_{k_1, k_1}| \asymp k_1$ and $|\mathcal{S}_{k_2, k_2}| \asymp k_2$. Then, 
\begin{align*}
    & \int \E_{\bbP} \left[ K_{V_{k_1}}(\bX, \bx)^2\right] \E_{\bbP}\left[ K_{V_{k_2}}(\bX, \bx)^2 \right] d\bx \\
    &  \quad \lesssim \sum_{\substack{(\boldm, \boldm^{\prime}) \in \mathcal{S}_{k_1, k_1} \\ (\boldr, \boldr^{\prime}) \in \mathcal{S}_{k_2, k_2}}} \int  \left|  \Phi_{k_1, \boldm}(\bx) \Phi_{k_1, \boldm^{\prime}}(\bx)\Phi_{k_2, \boldr}(\bx) \Phi_{k_2, \boldr^{\prime}}(\bx)   \right| d\bx. 
\end{align*}
Suppose without loss of generality that $k_1 \geq k_2$. Similar to the strategy used throughout the proof of Lemma \ref{lem: exp pxpy}, we observe the following: The integral in the above expression is non-zero for $O(1)$ distinct $\boldr$ for each $\boldm$ and for $O(1)$ distinct $\boldr^{\prime}$ for each $\boldm^{\prime}$. Since $|\mathcal{S}_{k_1, k_1}| \asymp k_1$, there are $O(k_1)$ nonzero terms in the summation. Moreover, the integrand is uniformly bounded above by $O(k_1k_2)$ and has support contained in a set of Lebesgue measure $O(\frac{1}{k_1})$. Therefore, (\ref{eq: exp kernel single pt3}) holds.

\textbf{Showing (\ref{eq: exp kernel single second squared})}: 
Assume without loss of generality that $k_1 \geq k_2$. Recall from the proof of (\ref{eq: exp kernel single second}) that, 
\begin{equation*}
    \E_{\bbP} \left[ |K_{V_{k_1}}(\bX, \bx) K_{V_{k_2}}(\bX, \bx)|\right] \lesssim \sqrt{\frac{k_2}{k_1}}\sum_{(\boldm, \boldm^{\prime}) \in \mathcal{S}_{k_1, k_2}} | \Phi_{k_1, \boldm}(\bx) \Phi_{k_2, \boldm^{\prime}}(\bx) |
\end{equation*}
where $|\mathcal{S}_{k_1, k_2}| \asymp k_1$. Then,
\begin{align*}
    & \int \E_{\bbP} \left[ |K_{V_{k_1}}(\bX, \bx) K_{V_{k_2}}(\bX, \bx)| \right]^2 d\bx \\ 
    & \quad \lesssim \frac{k_2}{k_1} \sum_{\substack{(\boldm, \boldm^{\prime}) \in \mathcal{S}_{k_1, k_2} \\ (\boldr, \boldr^{\prime}) \in \mathcal{S}_{k_1, k_2}}} \int | \Phi_{k_1, \boldm}(\bx) \Phi_{k_2, \boldm^{\prime}}(\bx) \Phi_{k_1, \boldr}(\bx) \Phi_{k_2, \boldr^{\prime}}(\bx) | d\bx.
\end{align*}
As argued in the proof of (\ref{eq: exp kernel single pt3}), there are $O(k_1)$ nonzero terms in the summation and the nonzero terms are bounded above by $O(k_2)$. Therefore, (\ref{eq: exp kernel single second squared}) holds.

\item[(ii)]

\textbf{Showing (\ref{eq: exp kernel single lb pt1})}: We will first focus on obtaining a suitable expression for \newline $\E_{\bbP} \left[ K_{V_{k}}(\bX, \bx)^2\right]$. Let $\bx \in [0,1]^d$. Recall that
\begin{equation*}
    K_{V_k}(\bX, \bx)^2 = \sum_{\boldm, \boldm^{\prime} \in \mathcal{Z}_k} \Phi_{k, \boldm}(\bx) \Phi_{k, \boldm^{\prime}}(\bx) \Phi_{k, \boldm}(\bX) \Phi_{k, \boldm^{\prime}}(\bX)
\end{equation*}
where $|\mathcal{Z}_k| \asymp k$.
Therefore,
\begin{align*}
    \E_{\bbP} \left[ K_{V_{k}}(\bX, \bx)^2\right] & = \sum_{\boldm, \boldm^{\prime} \in \mathcal{Z}_k} \Phi_{k, \boldm}(\bx) \Phi_{k, \boldm^{\prime}}(\bx) \int \Phi_{k, \boldm}(\bx^{\prime}) \Phi_{k, \boldm^{\prime}}(\bx^{\prime}) d\bx^{\prime} \\
    & = \sum_{\boldm \in \mathcal{Z}_k} \Phi_{k, \boldm}(\bx)^2
\end{align*}
where the second line follows from orthonormality of the translates $\Phi_{k, \boldm}$. Moreover,
\begin{align*}
    1
    &=
    \left|
    \sum_{\boldm\in\mathcal{Z}_k}
    \langle 1,\Phi_{k,\boldm}\rangle
    \Phi_{k,\boldm}(\bx)
    \right|^2 \\
    &\leq
    \left(
    \sum_{\boldm\in\mathcal{Z}_k}
    \Phi_{k,\boldm}(\bx)^2
    \right)
    \left(
    \sum_{\substack{\boldm\in\mathcal{Z}_k:\\
    \Phi_{k,\boldm}(\bx)\neq0}}
    \langle 1,\Phi_{k,\boldm}\rangle^2
    \right) \\
    &\lesssim
    k^{-1}
    \sum_{\boldm\in\mathcal{Z}_k}
    \Phi_{k,\boldm}(\bx)^2
\end{align*}
where the first line follows from $1\in V_k$ by Property P3 and the second line follows from the Cauchy--Schwarz inequality. The last line follows from Property P2 and orthonormality, since only $O(1)$ terms are non-zero and
\begin{equation*}
    \left|\langle 1,\Phi_{k,\boldm}\rangle\right|^2
    \leq
    \left|\operatorname{supp}(\Phi_{k,\boldm})\right|
    \int\Phi_{k,\boldm}(\bx^\prime)^2d\bx^\prime
    \lesssim k^{-1}.
\end{equation*}

This implies that $\E_{\bbP} \left[ K_{V_{k}}(\bX, \bx)^2\right] \geq c^{\prime}k$ uniformly in $\bx$ ($c^{\prime} \in \mathbb{R}^+$), which in turn establishes (\ref{eq: exp kernel single lb pt1}).

\textbf{Showing (\ref{eq: exp kernel single lb pt2})}: Let $\bx \in [0,1]^d$. Recall that we can write
\begin{equation*}
    K_{V_{k_1}}(\bX, \bx)K_{V_{k_2}}(\bX, \bx) = \sum_{\boldm, \boldm^{\prime} \in \mathcal{Z}_{k_1} \times \mathcal{Z}_{k_2}} \Phi_{k_1, \boldm}(\bx) \Phi_{k_2, \boldm^{\prime}}(\bx) \Phi_{k_1, \boldm}(\bX) \Phi_{k_2, \boldm^{\prime}}(\bX)
\end{equation*}
where $|\mathcal{Z}_{k_1}| \asymp k_1$ and $|\mathcal{Z}_{k_2}| \asymp k_2$.
Then,
\begin{align}
    & \int \E_{\bbP}[K_{V_{k_1}}(\bX, \bx)K_{V_{k_2}}(\bX, \bx)] d \bx \nonumber \\
    & \quad = \sum_{\boldm, \boldm^{\prime} \in \mathcal{Z}_{k_1} \times \mathcal{Z}_{k_2}} \E_{\bbP} \left[ \Phi_{k_1, \boldm}(\bX) \Phi_{k_2, \boldm^{\prime}}(\bX)\right] \int \Phi_{k_1, \boldm} (\bx) \Phi_{k_2, \boldm^{\prime}}(\bx) d\bx \nonumber \\
    & \quad = \sum_{\boldm, \boldm^{\prime} \in \mathcal{Z}_{k_1} \times \mathcal{Z}_{k_2}} \left( \int \Phi_{k_1, \boldm} (\bx) \Phi_{k_2, \boldm^{\prime}}(\bx) d\bx \right)^2. \label{eq: exp kernel single temp}
\end{align}
Suppose without loss of generality that $k_1 \geq k_2$. Our strategy is to express the integrand into a quantity that does not depend on $k_2$.  To do so, recall that the translates $\Phi_{k, \boldm}$ form an orthonormal basis of the linear space $V_k$. Moreover, recall that the linear spaces are nested, i.e., $V_k \subset V_{k^\prime}$ for $k \leq k^\prime$. Then, since $V_{k_2} \subset V_{k_1}$, we can write 
\begin{equation} \label{eq: change basis}
    \Phi_{k_2, \boldm^{\prime}}(\bx) = \sum_{\tilde{\boldm} \in \mathbb{Z}^d} \alpha_{\boldm^{\prime}, \tilde{\boldm}, k_1, k_2} \Phi_{k_1, \tilde{\boldm}}(\bx).
\end{equation}
We next show that $\sum_{\tilde{\boldm}  \in \mathcal{Z}_{k_1}} \alpha_{\boldm^{\prime}, \tilde{\boldm}, k_1, k_2}^2 = 1$. By orthonormality of the translates $\Phi_{k \boldm}$ and using (\ref{eq: change basis}), 
\begin{align}
    1 & = \int \Phi_{k_2, \boldm^{\prime}}(\bx)^2 d\bx \nonumber\\
    & = \int \left(  \sum_{\tilde{\boldm}  \in \mathbb{Z}^d} \alpha_{\boldm^{\prime}, \tilde{\boldm}, k_1, k_2} \Phi_{k_1, \tilde{\boldm}}(\bx) \right)^2 d \bx \nonumber\\
    & = \sum_{\tilde{\boldm}^{(1)} \in \mathbb{Z}^d} \sum_{\tilde{\boldm}^{(2)} \in \mathbb{Z}^d} \alpha_{\boldm^{\prime}, \tilde{\boldm}^{(1)}, k_1, k_2} \alpha_{\boldm^{\prime}, \tilde{\boldm}^{(2)}, k_1, k_2} \int   \Phi_{k_1, \tilde{\boldm}^{(1)}}(\bx)\Phi_{k_1, \tilde{\boldm}^{(2)}}(\bx)  d \bx \nonumber \\
    & = \sum_{\tilde{\boldm}^{(1)} \in \mathcal{Z}_{k_1}}   \alpha_{\boldm^{\prime}, \tilde{\boldm}^{(1)}, k_1, k_2}^2 \label{eq: coef bound}.
\end{align}

Now, plugging in (\ref{eq: change basis}) in (\ref{eq: exp kernel single temp}), we have that
\begin{align*}
    & \int \E_{\bbP}[K_{V_{k_1}}(\bX, \bx)K_{V_{k_2}}(\bX, \bx)] d \bx  \\
    & \quad \geq \sum_{\boldm^{\prime} \in \mathcal{Z}_{k_2}} \sum_{\boldm \in \mathcal{Z}_{k_1}} \left(  \sum_{\tilde{\boldm}  \in \mathbb{Z}^d} \alpha_{\boldm^{\prime}, \tilde{\boldm}, k_1, k_2} \int \Phi_{k_1, \boldm}(\bx)\Phi_{k_1, \tilde{\boldm}}(\bx) d\bx \right)^2.
\end{align*}
By orthonormality of the translates $\Phi_{k, \boldm}$ with respect to the Lebesgue measure, we then have that
\begin{align*}
    \int \E_{\bbP}[K_{V_{k_1}}(\bX, \bx)K_{V_{k_2}}(\bX, \bx)] d \bx & \geq \sum_{\boldm^{\prime} \in \mathcal{Z}_{k_2}} \sum_{\boldm \in \mathcal{Z}_{k_1}} \alpha_{\boldm^{\prime}, \boldm, k_1, k_2}^2 \\
    & = \sum_{\boldm^{\prime} \in \mathcal{Z}_{k_2}} 1 \quad (\text{by (\ref{eq: coef bound})})\\
    & \gtrsim k_2.
\end{align*}

\end{itemize}

\end{proof}

\begin{lemma} \label{lem: exp kernel single same}
The following statements hold
\begin{align}
    \sup_{\bx \in [0, 1]^d} K_{V_k}(\bx, \bx) & \lesssim k  \label{eq: exp kernel single same ub one} \\
    \int K_{V_k}(\bx, \bx)  d\bx & \gtrsim k. \label{eq: exp kernel single same lb one}
\end{align}
\end{lemma}

\begin{proof}
Let $\bx \in [0, 1]^d$ be arbitrary. Recall from Appendix \ref{sec: wavelets} that 
\begin{equation*}
    K_{V_k}(\bx, \bx) = \sum_{\boldm \in \mathcal{Z}_k} \Phi_{k,\boldm}^2(\bx)
\end{equation*}
where $|\mathcal{Z}_k| \asymp k$. 

To show  (\ref{eq: exp kernel single same ub one}), we note that there are $O(1)$ distinct $\boldm$ such that $\Phi_{k,\boldm}(\bx)$ is non-zero, and $|\Phi_{k,\boldm}|$ is uniformly bounded above by $O(\sqrt{k})$.  

To show (\ref{eq: exp kernel single same lb one}), observe that
\begin{align*}
    \int K_{V_k}(\bx, \bx)  d\bx = \sum_{\boldm \in \mathcal{Z}_k} \int \Phi_{k,\boldm}^2(\bx) d\bx = \sum_{\boldm \in \mathcal{Z}_k} 1 \gtrsim k
\end{align*}
where the second equality follows from the orthonormality of the translates $\Phi_{k, \boldm}$.
\end{proof}

\begin{lemma} \label{lem: exp kernel triple}
For any $\bbP \in \cP_{(\alpha, \beta)} \cup \cP_{(\alpha, \beta, \gamma)}$,
\begin{align}
    \E_{\bbP}\left[ |K_{V_{k_1}}(\bX_{i_1},\bX_{i_2})K_{V_{k_1}}(\bX_{i_2},\bX_{i_3})K_{V_{k_2}}(\bX_{i_1},\bX_{i_3})| \right] & \lesssim k_1 \wedge k_2 \label{eq: exp kernel triple zero}  \\
    \E_{\bbP}\left[| K_{V_{k_1}}(\bX_{i_1},\bX_{i_2})^2K_{V_{k_2}}(\bX_{i_1},\bX_{i_3})K_{V_{k_2}}(\bX_{i_2},\bX_{i_3}) |\right] & \lesssim  k_1 (k_1 \wedge k_2) \lesssim k_1k_2 \label{eq: exp kernel triple second} \\
    \E_{\bbP}\left[ \left| \begin{array}{c} K_{V_{k_1}}(\bX_{i_1},\bX_{i_2})K_{V_{k_1}}(\bX_{i_2},\bX_{i_3}) \\ \times K_{V_{k_2}}(\bX_{i_2},\bX_{i_3})K_{V_{k_2}}(\bX_{i_1},\bX_{i_3}) \end{array} \right| \right] & \lesssim (k_1 \wedge k_2)^2 \label{eq: exp kernel triple first}  
\end{align}
where $i_1,i_2,i_3$ are distinct.
\end{lemma}

\begin{proof}
Since we adopt similar arguments used in the proofs of Lemmas \ref{lem: exp kernel}-\ref{lem: exp kernel single same}, we write this proof more concisely.

\textbf{Showing (\ref{eq: exp kernel triple zero})}: We have that
\begin{align*}
    & K_{V_{k_1}}(\bX_{i_1},\bX_{i_2})K_{V_{k_1}}(\bX_{i_2},\bX_{i_3})K_{V_{k_2}}(\bX_{i_1},\bX_{i_3}) \\
    & \quad = \sum_{\substack{\boldm, \boldr  \in \mathcal{Z}_{k_1} \\ \boldm^{\prime} \in \mathcal{Z}_{k_2}}} \Phi_{k_1, \boldm}(\bX_{i_1}) \Phi_{k_1, \boldm}(\bX_{i_2}) \Phi_{k_1, \boldr}(\bX_{i_2}) \Phi_{k_1, \boldr}(\bX_{i_3})  \Phi_{k_2, \boldm^{\prime}}(\bX_{i_1})    \Phi_{k_2, \boldm^{\prime}}(\bX_{i_3})
\end{align*}
where $|\mathcal{Z}_{k_1}| \asymp k_1$ and $|\mathcal{Z}_{k_2}| \asymp k_2$. By boundedness of $f$,
\begin{align*}
    & \E_{\bbP}\left[ |K_{V_{k_1}}(\bX_{i_1},\bX_{i_2})K_{V_{k_1}}(\bX_{i_2},\bX_{i_3})K_{V_{k_2}}(\bX_{i_1},\bX_{i_3})| \right] \\
    & \lesssim \iiint |K_{V_{k_1}}(\bx,\by)K_{V_{k_1}}(\by,\bz)K_{V_{k_2}}(\bx,\bz)| d\bx d\by d\bz \\
    & = \sum_{\substack{\boldm, \boldr  \in \mathcal{Z}_{k_1} \\ \boldm^{\prime} \in \mathcal{Z}_{k_2}}}  \begin{array}{c} \int | \Phi_{k_1, \boldm}(\bx) \Phi_{k_2, \boldm^{\prime}}(\bx) | d\bx \int | \Phi_{k_1, \boldm}(\by) \Phi_{k_1, \boldr}(\by)|d\by \\ \times \int |\Phi_{k_1, \boldr}(\bz) \Phi_{k_2, \boldm^{\prime}}(\bz) |d\bz \end{array}. 
\end{align*}
Suppose that $k_1 \geq k_2$. For each $\boldm \in \mathcal{Z}_{k_1}$, there are $O(1)$ distinct $\boldr,\boldm^{\prime}$ such that all three integrals are non-zero. Moreover, the three integrands are uniformly bounded above by $O(\sqrt{k_1k_2})$, $O(k_1)$, and $O(\sqrt{k_1k_2})$, respectively, and each has support contained in a set of Lebesgue measure $O(\frac{1}{k_1})$. Therefore, (\ref{eq: exp kernel triple zero}) holds.

Similarly, suppose that $k_1 \leq k_2$. For each $\boldm^{\prime} \in \mathcal{Z}_{k_2}$, there are $O(1)$ distinct $\boldm,\boldr$ such that all three integrals are non-zero. Moreover, the three integrands are uniformly bounded above by $O(\sqrt{k_1k_2})$, $O(k_1)$, and $O(\sqrt{k_1k_2})$, respectively, and have support contained in a set of Lebesgue measure $O(\frac{1}{k_2})$, $O(\frac{1}{k_1})$, and $O(\frac{1}{k_2})$, respectively. Therefore, (\ref{eq: exp kernel triple zero}) holds.

\textbf{Showing (\ref{eq: exp kernel triple second})}: We have that
\begin{align*}
    & K_{V_{k_1}}(\bX_{i_1},\bX_{i_2})^2K_{V_{k_2}}(\bX_{i_1},\bX_{i_3})K_{V_{k_2}}(\bX_{i_2},\bX_{i_3}) \\
    & \quad = \sum_{\substack{\boldm, \boldr  \in \mathcal{Z}_{k_1} \\ \boldm^{\prime}, \boldr^{\prime} \in \mathcal{Z}_{k_2}}} \begin{array}{c} \Phi_{k_1, \boldm}(\bX_{i_1}) \Phi_{k_1, \boldm}(\bX_{i_2}) \Phi_{k_1, \boldr}(\bX_{i_1}) \Phi_{k_1, \boldr}(\bX_{i_2}) \\ \times \Phi_{k_2, \boldm^{\prime}}(\bX_{i_1}) \Phi_{k_2, \boldm^{\prime}}(\bX_{i_3}) \Phi_{k_2, \boldr^{\prime}}(\bX_{i_2})    \Phi_{k_2, \boldr^{\prime}}(\bX_{i_3}). \end{array}
\end{align*}
where $|\mathcal{Z}_{k_1}| \asymp k_1$ and $|\mathcal{Z}_{k_2}| \asymp k_2$. By boundedness of $f$,
\begin{align*}
    & \E_{\bbP}\left[ |K_{V_{k_1}}(\bX_{i_1},\bX_{i_2})^2K_{V_{k_2}}(\bX_{i_1},\bX_{i_3})K_{V_{k_2}}(\bX_{i_2},\bX_{i_3})| \right] \\
    & \lesssim \iiint |K_{V_{k_1}}(\bx,\by)^2K_{V_{k_2}}(\bx,\bz)K_{V_{k_2}}(\by,\bz)| d\bx d\by d\bz \\
    & = \sum_{\substack{\boldm, \boldr  \in \mathcal{Z}_{k_1} \\ \boldm^{\prime}, \boldr^{\prime} \in \mathcal{Z}_{k_2}}}  \begin{array}{c} \int | \Phi_{k_1, \boldm}(\bx) \Phi_{k_1, \boldr}(\bx) \Phi_{k_2, \boldm^{\prime}}(\bx) | d\bx \int | \Phi_{k_1, \boldm}(\by) \Phi_{k_1, \boldr}(\by) \Phi_{k_2, \boldr^{\prime}}(\by)|d\by \\ \times  \int |\Phi_{k_2, \boldm^{\prime}}(\bz) \Phi_{k_2, \boldr^{\prime}}(\bz)   |d\bz \end{array}. 
\end{align*}
Suppose that $k_1 \geq k_2$. For each $\boldm \in \mathcal{Z}_{k_1}$, there are $O(1)$ distinct $\boldr,\boldm^{\prime}, \boldr^{\prime}$ such that all three integrals are non-zero. Moreover, the three integrands are uniformly bounded above by $O(k_1\sqrt{k_2})$, $O(k_1\sqrt{k_2})$, and $O(k_2)$, respectively, and have support contained in a set of Lebesgue measure $O(\frac{1}{k_1})$, $O(\frac{1}{k_1})$, and $O(\frac{1}{k_2})$, respectively. Therefore, (\ref{eq: exp kernel triple second}) holds.

Similarly, suppose that $k_1 \leq k_2$. For each $\boldm^{\prime} \in \mathcal{Z}_{k_2}$, there are $O(1)$ distinct $\boldm, \boldr,\boldr^{\prime}$ such that all three integrals are non-zero. Moreover, the three integrands are uniformly bounded above by $O(k_1\sqrt{k_2})$, $O(k_1\sqrt{k_2})$, and $O(k_2)$, respectively, and each has support contained in a set of Lebesgue measure $O(\frac{1}{k_2})$. Therefore, (\ref{eq: exp kernel triple second}) holds.

\textbf{Showing (\ref{eq: exp kernel triple first})}: We have that
\begin{align*}
    & K_{V_{k_1}}(\bX_{i_1},\bX_{i_2})K_{V_{k_1}}(\bX_{i_2},\bX_{i_3})K_{V_{k_2}}(\bX_{i_2},\bX_{i_3})K_{V_{k_2}}(\bX_{i_1},\bX_{i_3}) \\
    & \quad = \sum_{\substack{\boldm, \boldr  \in \mathcal{Z}_{k_1} \\ \boldm^{\prime}, \boldr^{\prime} \in \mathcal{Z}_{k_2}}} \begin{array}{c} \Phi_{k_1, \boldm}(\bX_{i_1}) \Phi_{k_1, \boldm}(\bX_{i_2}) \Phi_{k_1, \boldr}(\bX_{i_2}) \Phi_{k_1, \boldr}(\bX_{i_3}) \\ \times \Phi_{k_2, \boldm^{\prime}}(\bX_{i_2}) \Phi_{k_2, \boldm^{\prime}}(\bX_{i_3}) \Phi_{k_2, \boldr^{\prime}}(\bX_{i_1})    \Phi_{k_2, \boldr^{\prime}}(\bX_{i_3}) \end{array}
\end{align*}
where $|\mathcal{Z}_{k_1}| \asymp k_1$ and $|\mathcal{Z}_{k_2}| \asymp k_2$. By boundedness of $f$,
\begin{align*}
    & \E_{\bbP}\left[ |K_{V_{k_1}}(\bX_{i_1},\bX_{i_2})K_{V_{k_1}}(\bX_{i_2},\bX_{i_3})K_{V_{k_2}}(\bX_{i_2},\bX_{i_3})K_{V_{k_2}}(\bX_{i_1},\bX_{i_3})| \right] \\
    & \lesssim \iiint |K_{V_{k_1}}(\bx,\by)K_{V_{k_1}}(\by,\bz)K_{V_{k_2}}(\by,\bz)K_{V_{k_2}}(\bx,\bz)| d\bx d\by d\bz \\
    & = \sum_{\substack{\boldm, \boldr  \in \mathcal{Z}_{k_1} \\ \boldm^{\prime}, \boldr^{\prime} \in \mathcal{Z}_{k_2}}}  \begin{array}{c} \int | \Phi_{k_1, \boldm}(\bx) \Phi_{k_2, \boldr^{\prime}}(\bx) | d\bx \int | \Phi_{k_1, \boldm}(\by) \Phi_{k_1, \boldr}(\by) \Phi_{k_2, \boldm^{\prime}}(\by)|d\by \\ \times \int |\Phi_{k_1, \boldr}(\bz) \Phi_{k_2, \boldm^{\prime}}(\bz)    \Phi_{k_2, \boldr^{\prime}}(\bz) |d\bz \end{array}. 
\end{align*}
Suppose without loss of generality that $k_1 \geq k_2$. For each $\boldm \in \mathcal{Z}_{k_1}$, there are $O(1)$ distinct $\boldr,\boldm^{\prime}, \boldr^{\prime}$ such that all three integrals are non-zero. Moreover, the three integrands are uniformly bounded above by $O(\sqrt{k_1k_2})$, $O(k_1\sqrt{k_2})$, and $O(k_2\sqrt{k_1})$, respectively, and each has support contained in a set of Lebesgue measure $O(\frac{1}{k_1})$. Therefore, (\ref{eq: exp kernel triple first}) holds.

\end{proof}

\begin{lemma} \label{lem: nuisance function bounds}
For any $\bbP \in \cP_{(\alpha, \beta, \gamma)}$, the following statements hold
\begin{itemize}
    \item[(i)]
    \begin{align}
    \E_{\bbP, 1} [\E_{\bX} ( \hat{p}^{(1)}_{k_1}(\bX)^2 )] & \lesssim 1 +  \frac{k_1}{n}  \label{eq: nuisance function bound pt1}\\
    \E_{\bbP, 1, 2} [\E_{\bX} ( \hat{p}^{(1, 2)}_{k_1}(\bX)^2 )] & \lesssim 1 +  \frac{k_1}{n}  \label{eq: nuisance function bound pt1 unknown f}
    \end{align}
    and similarly
    \begin{align}
    \E_{\bbP, 2} [\E_{\bX} ( \hat{b}^{(2)}_{k_2}(\bX)^2 )] & \lesssim 1 +  \frac{k_2}{n}  \label{eq: nuisance function bound pt2}  \\
    \E_{\bbP, 3, 4} [\E_{\bX} ( \hat{b}^{(3, 4)}_{k_2}(\bX)^2 )] & \lesssim 1 +  \frac{k_2}{n}.  \label{eq: nuisance function bound pt2 unknown f} 
    \end{align}
    
    \item[(ii)]
\begin{align}
    \E_{\bbP, 1, 2} \left[\E_{\bbP, 3} \left[ \left( \hat{p}^{(1)}_{k_1}(\bX) \hat{b}^{(2)}_{k_2}(\bX) \right)^2  \right] \right] & \lesssim 1 + \frac{k_1}{n} + \frac{k_2}{n} +  \frac{k_1k_2}{n^2} \label{eq: nuisance function bound pt3} \\
    \E_{\bbP, 1, 2, 3, 4} \left[\E_{\bbP, 5} \left[ \left( \hat{p}^{(1, 2)}_{k_1}(\bX) \hat{b}^{(3, 4)}_{k_2}(\bX) \right)^2  \right] \right] & \lesssim 1 + \frac{k_1}{n} + \frac{k_2}{n} +  \frac{k_1k_2}{n^2}. \label{eq: nuisance function bound pt3 unknown f}
\end{align}

\item[(iii)] When $k_1, k_2 \gg n$, $\bX \sim \text{Uniform}([0,1]^d)$, and $|A|, |Y| \geq c$ with probability 1 for some $c>0$,
\begin{align}
    \E_{\bbP, 1, 2} \left[\E_{\bbP, 3} \left[ \left( \hat{p}^{(1)}_{k_1}(\bX) \hat{b}^{(2)}_{k_2}(\bX) \right)^2  \right] \right] & \gtrsim \frac{k_1k_2}{n^2} \label{eq: nuisance function bound pt4} \\
    \E_{\bbP, 1, 2, 3, 4} \left[\E_{\bbP, 5} \left[ \left( \hat{p}^{(1,2)}_{k_1}(\bX) \hat{b}^{(3,4)}_{k_2}(\bX) \right)^2  \right] \right] & \gtrsim \frac{k_1k_2}{n^2} \label{eq: nuisance function bound pt4 unknown f}.
\end{align}

\end{itemize}

\end{lemma}
\begin{proof}

\begin{itemize}
    \item[(i)] First, we show (\ref{eq: nuisance function bound pt1}). We have that
\begin{align*}
    \E_{\bbP, 1} [\E_{\bX} ( \hat{p}^{(1)}_{k_1}(\bX)^2 )] & = \int \E_{\bbP, 1} (\hat{p}^{(1)}_{k_1}(\bx)^2) f(\bx) d\bx \\
    & \lesssim \int \left( 1 + \frac{1}{n} \E_{\bbP} \left[ K_{V_{k_1}}(\bX, \bx)^2 \right] \right) d\bx \quad (\text{by Lemma \ref{lem: exp pxpy}}) \\
    & \lesssim 1 + \frac{k_1}{n} \quad (\text{by Lemma \ref{lem: exp kernel single}}).
\end{align*}
The bounds in (\ref{eq: nuisance function bound pt1 unknown f}), (\ref{eq: nuisance function bound pt2}), and (\ref{eq: nuisance function bound pt2 unknown f}) can be shown in the same manner.

\item[(ii)] 
Next, we show (\ref{eq: nuisance function bound pt3}). We have that
\begin{equation*}
    \E_{\bbP, 1, 2} \left[\E_{\bbP, 3} \left[ \left( \hat{p}^{(1)}_{k_1}(\bX) \hat{b}^{(2)}_{k_2}(\bX) \right)^2  \right] \right] =   \frac{1}{n^4} \sum_{\substack{i_1, i_2 \in \D_1 \\ j_1, j_2 \in \D_2}} S_{i_1, i_2, j_1, j_2}
\end{equation*}
where $S_{i_1, i_2, j_1, j_2}$ is given by
\begin{equation*}
     \E_{\bbP, 1, 2} \left[ \frac{A_{i_1}A_{i_2}Y_{j_1}Y_{j_2}}{f(\bX_{i_1})f(\bX_{i_2})f(\bX_{j_1})f(\bX_{j_2})} \int \begin{array}{c} K_{V_{k_1}}(\bX_{i_1}, \bx)K_{V_{k_1}}(\bX_{i_2}, \bx) \\ \times K_{V_{k_2}}(\bX_{j_1}, \bx)K_{V_{k_2}}(\bX_{j_2}, \bx)\end{array}f(\bx)d\bx\right].
\end{equation*}
The remainder of the proof will bound $|S_{i_1, i_2, j_1, j_2}|$ for each $i_1, i_2, j_1, j_2$. When $i_1 = i_2$ and $j_1 = j_2$,
\begin{align*}
    S_{i, i, j, j} & = \E_{\bbP, 1, 2} \left[ \frac{A_i^2Y_j^2}{f(\bX_{i})^2f(\bX_{j})^2} \int K_{V_{k_1}}(\bX_{i}, \bx)^2K_{V_{k_2}}(\bX_{j}, \bx)^2 f(\bx)d\bx\right] \\
    & \lesssim \E_{\bbP, 1, 2} \left[ \int K_{V_{k_1}}(\bX_{i}, \bx)^2K_{V_{k_2}}(\bX_{j}, \bx)^2d\bx\right] \\
    & =  \int \E_{\bbP, 1} [K_{V_{k_1}}(\bX_{i}, \bx)^2] \E_{\bbP, 2} [K_{V_{k_2}}(\bX_{j}, \bx)^2]d\bx \\
    & \lesssim k_1k_2 \quad (\text{by Lemma \ref{lem: exp kernel single}}).
\end{align*}

When $i_1 = i_2$ and $j_1 \neq j_2$,
\begin{align*}
    | S_{i, i, j_1, j_2} | & \lesssim \E_{\bbP, 1, 2} \left[  \int | K_{V_{k_1}}(\bX_{i}, \bx)^2K_{V_{k_2}}(\bX_{j_1}, \bx)K_{V_{k_2}}(\bX_{j_2}, \bx)|d\bx\right] \\
    & =   \int \E_{\bbP, 1} [K_{V_{k_1}}(\bX_{i}, \bx)^2] \E_{\bbP, 2} [|K_{V_{k_2}}(\bX_{j_1}, \bx)|] \E_{\bbP, 2}[|K_{V_{k_2}}(\bX_{j_2}, \bx)|] d\bx \\
    & \lesssim k_1 \quad (\text{by Lemma \ref{lem: exp kernel single}}).
\end{align*}
It follows in the same manner that $| S_{i_1, i_2, j, j} | \lesssim k_2$. 

When $i_1 \neq i_2$ and $j_1 \neq j_2$,
\begin{align*}
    | S_{i_1, i_2, j_1, j_2} |  & \lesssim \E_{\bbP, 1, 2} \left[  \int | K_{V_{k_1}}(\bX_{i_1}, \bx)K_{V_{k_1}}(\bX_{i_2}, \bx)K_{V_{k_2}}(\bX_{j_1}, \bx)K_{V_{k_2}}(\bX_{j_2}, \bx)|d\bx\right] \\
    & =   \int \begin{array}{c} \E_{\bbP, 1} [|K_{V_{k_1}}(\bX_{i_1}, \bx)|]\E_{\bbP, 1} [|K_{V_{k_1}}(\bX_{i_2}, \bx)|] \\ \times \E_{\bbP, 2} [|K_{V_{k_2}}(\bX_{j_1}, \bx)|] \E_{\bbP, 2}[|K_{V_{k_2}}(\bX_{j_2}, \bx)|] \end{array} d\bx \\
    & \lesssim  1 \quad (\text{by Lemma \ref{lem: exp kernel single}}).
\end{align*}
This establishes (\ref{eq: nuisance function bound pt3}). Since $\hat{f}$ is bounded, (\ref{eq: nuisance function bound pt3 unknown f}) follows in the same manner.

\item[(iii)]
Next, we show (\ref{eq: nuisance function bound pt4}). It follows from the work in establishing the upper bound that it suffices to show that $S_{i_1, i_2, j_1, j_2} \gtrsim k_1k_2$ when $i_1 = i_2$ and $j_1 = j_2$. We have that
\begin{align*}
    S_{i, i, j, j} & = \E_{\bbP, 1, 2} \left[ A_i^2Y_j^2 \int K_{V_{k_1}}(\bX_{i}, \bx)^2K_{V_{k_2}}(\bX_{j}, \bx)^2 d\bx\right] \\
    & \gtrsim \E_{\bbP, 1, 2} \left[ \int K_{V_{k_1}}(\bX_{i}, \bx)^2K_{V_{k_2}}(\bX_{j}, \bx)^2d\bx\right] \\
    & =  \int \E_{\bbP, 1} [K_{V_{k_1}}(\bX_{i}, \bx)^2] \E_{\bbP, 2} [K_{V_{k_2}}(\bX_{j}, \bx)^2]d\bx \\
    & \gtrsim k_1k_2 \quad (\text{by Lemma \ref{lem: exp kernel single}}).
\end{align*}
Since $\hat{f}$ is bounded, (\ref{eq: nuisance function bound pt4 unknown f}) follows in the same manner
\end{itemize}

\end{proof}

\begin{lemma} \label{lem: proj dist}
\begin{align}
    \sup_{\bbP \in \cP_{(\alpha, \beta, \gamma)}} \left\| \E_{\bbP,1,2} (\hat{p}^{(1, 2)}_{k_1}(\cdot)) - p \right\|_{\infty} & \lesssim k_1^{-\alpha/d} + \left( \frac{n}{\log n}\right)^{-\frac{\gamma}{2\gamma + d}} \label{proj dist p}\\
    \sup_{\bbP \in \cP_{(\alpha, \beta, \gamma)}} \left\| \E_{\bbP,3,4} (\hat{b}^{(3, 4)}_{k_2}(\cdot)) - b \right\|_{\infty} & \lesssim k_2^{-\beta/d} + \left( \frac{n}{\log n}\right)^{-\frac{\gamma}{2\gamma + d}}.  \label{proj dist b} 
\end{align}
\end{lemma}
\begin{proof}
We will show (\ref{proj dist p}) and note that (\ref{proj dist b}) can be shown in the same manner. 

Let $\bbP \in \cP_{(\alpha, \beta, \gamma)}$ be arbitrary. Observe that
\begin{align*}
    \E_{\bbP,1} (\hat{p}^{(1, 2)}_{k_1}(\bx)) & = \E_{\bbP,1} (\tilde{p}^{(1)}_{k_1}(\bx)) + \E_{\bbP, 1}(\hat{p}^{(1, 2)}_{k_1}(\bx) - \tilde{p}^{(1)}_{k_1}(\bx)) \\
    & = \E_{\bbP,1} \left[ \frac{A K_{V_{k_1}}(\bX,\bx)}{f(\bX)} \right] + \E_{\bbP,1} \left[ \frac{A K_{V_{k_1}}(\bX,\bx)}{f(\bX)} \left( \frac{f(\bX) - \hat{f}^{(2)}(\bX)}{\hat{f}^{(2)}(\bX)} \right) \right] \\
    & = \Pi(p | V_{k_1})(\bx) + \Pi\left( \frac{ p(f - \hat{f}^{(2)})}{\hat{f}^{(2)}} | V_{k_1}\right)(\bx).
\end{align*}
Then,
\begin{equation} \label{eq: p_k proj}
    \E_{\bbP,1,2} (\hat{p}^{(1, 2)}_{k_1}(\bx)) = \Pi(p | V_{k_1})(\bx) + \E_{\bbP,2} \left[ \Pi\left( \frac{ p(f - \hat{f}^{(2)})}{\hat{f}^{(2)}} | V_{k_1}\right)(\bx) \right]
\end{equation}
and so
\begin{align*}
    \left\| \E_{\bbP,1,2} (\hat{p}^{(1, 2)}_{k_1}(\cdot)) - p \right\|_{\infty} & \leq \| \Pi(p | V_{k_1}) - p  \|_{\infty} + \E_{\bbP,2} \left[ \left\| \Pi\left( \frac{ p(f - \hat{f}^{(2)})}{\hat{f}^{(2)}} | V_{k_1}\right) \right\|_{\infty} \right].
\end{align*}
By Property P1 (Appendix \ref{sec: wavelets}), 
\begin{equation*}
    \| \Pi(p | V_{k_1}) - p  \|_{\infty} \lesssim k_1^{-\alpha/d}.
\end{equation*}
Moreover,
\begin{align}
    \E_{\bbP,2} \left[ \left\| \Pi\left( \frac{ p(f - \hat{f}^{(2)})}{\hat{f}^{(2)}} | V_{k_1}\right) \right\|_{\infty} \right] & \lesssim \E_{\bbP,2} \left[ \left\|  \frac{ p(f - \hat{f}^{(2)})}{\hat{f}^{(2)}}  \right\|_{\infty} \right]\nonumber \\
    & \lesssim \E_{\bbP,2} \left[ \| f - \hat{f}^{(2)} \|_{\infty} \right]  \nonumber \\
    & \lesssim \left( \frac{n}{\log n}\right)^{-\frac{\gamma}{2\gamma + d}} \label{eq: exp proj}
\end{align}
where the second inequality follows by boundedness of $p, \hat{f}$ and the last inequality follows by assumption (ii) on $\hat{f}$ and Jensen's inequality. This completes the proof of (\ref{proj dist p}).
\end{proof}

\section{Simulations}

\subsection{Simulation design}

\subsubsection{Setting Hölder smooth functions}

We set the $\mu$ function with Hölder smoothness class $0<s < 1$ as follows. Let
\begin{equation*}
    \lambda(x) = (1 - |2x-1|) I(x \in [0, 1])
\end{equation*}
and define its translations and scalings by
\begin{equation*}
    \lambda_{lm}(x) = \lambda(2^lx - m), \qquad l \in \mathbb{N}_{\geq0}, \, \, m \in \{0, \dots, 2^l-1\}.
\end{equation*}
Note that the functions $\lambda_{lm}$ are the hat functions in the Faber--Schauder system. We first define the function $\tilde{\mu}$ by
\begin{equation}\label{eq:faber}
    \tilde{\mu}(x) = \sum_{l = 0}^{J-1} \sum_{m = 0}^{2^l-1} d_{lm}\lambda_{lm}(x)
\end{equation}
where $d_{lm}=2^{-ls} \epsilon_{lm}$. We set $\epsilon_{lm}$ by drawing i.i.d.\ Rademacher random variables and set the truncation level to $J = 20$. We then set $\mu$ by centering and rescaling $\tilde{\mu}$, i.e., by setting $\mu(x) = (\tilde{\mu}(x) - \bar{\mu}) / c$ where $\bar{\mu} = \int_0^1 \tilde{\mu}(x)dx$ and $c = \left\{ \int_0^1 (\tilde{\mu}(x) - \bar{\mu})^2 dx \right\}^{1/2}$. 

We can see that $\mu$ has Hölder smoothness class $s$ by noting the following. For each level $l \in \{0, \dots, J-1\}$ and each $x \in [0, 1]$, at most one of the $\lambda_{lm}$ functions is nonzero. Moreover, each $\lambda_{lm}$ function has Lipschitz constant $2^{l+1}$, i.e., $|\lambda_{lm}(x) - \lambda_{lm}(y)| \leq 2^{l+1}|x-y|$. Therefore,
\begin{equation*}
    |\tilde{\mu}(x) - \tilde{\mu}(y)| \leq \sum_{l = 0}^{J-1}2^{-ls} \min\{2, 2^{l+1} |x - y|\}. 
\end{equation*}
Let $\delta = |x-y|>0$ and choose $L$ so that $2^L \delta \leq 1 \leq 2^{L+1} \delta$. Then,
\begin{align*}
    \sum_{l = 0}^{J-1}2^{-ls} \min\{2, 2^{l+1} |x - y|\} & \leq \sum_{l = 0}^L 2^{-ls}2^{l+1}\delta + \sum_{l = L+1}^\infty 2^{-ls}2 \\
    & \lesssim 2^{L}2^{-Ls}\delta + 2^{-(L+1)s} \\ 
    & \lesssim \delta^s
\end{align*}
which completes the Hölder bound. Furthermore, this bound is satisfied for $\mu$ since centering does not change the bound and rescaling changes only its constant. Thus, $\mu$ lies in a Hölder ball 
with smoothness $s$. Moreover, since $|d_{lm}| = 2^{-ls}$, the level-$l$ terms in (\ref{eq:faber}) have the largest order compatible with Hölder-$s$ smoothness.

The $\mu$ function in each of the three regularity settings (i.e., $s \in \{0.05, 0.25, 0.75\}$) is illustrated in Figure 5 in the main text.

\subsubsection{Estimators}

In the double sample splitting case, we split the sample into disjoint subsamples $\{\D_1, \D_2, \D_3\}$, each of size $n_s = n / 3$, and use the following estimators
\begin{align*}
    \hat{\psi}_{k_1, k_2}^{\mathrm{INT}} & = \frac{1}{n_s}\sum_{i \in \D_3} A_i Y_i -  \int \hat{p}_{k_1}^{(1)}(x)\hat{b}_{k_2}^{(2)}(x) dx \\
    \hat{\psi}^{\mathrm{MC}}_{k_1, k_2} & = \frac{1}{n_s}\sum_{i \in \D_3} A_i Y_i -  \frac{1}{n_s}\sum_{i \in \D_3} \hat{p}^{(1)}_{k_1}(X_i)\hat{b}^{(2)}_{k_2}(X_i) \\
    \hat{\psi}^{\mathrm{IF}}_{k_1, k_2} & = \frac{1}{n_s} \sum_{i \in \D_3} (A_i - \hat{p}^{(1)}_{k_1}(X_i))(Y_i - \hat{b}^{(2)}_{k_2}(X_i)).
\end{align*}
In the single sample splitting case, we split the sample into disjoint subsamples $\{\D_1, \D_2\}$, each of size $n_s = n / 2$, and use the following estimators
\begin{align*}
    \hat{\psi}_{k_1, k_2}^{\mathrm{INT}} & = \frac{1}{n_s}\sum_{i \in \D_2} A_i Y_i -  \int \hat{p}_{k_1}^{(1)}(x)\hat{b}_{k_2}^{(1)}(x) dx \\
    \hat{\psi}^{\mathrm{MC}}_{k_1, k_2} & = \frac{1}{n_s}\sum_{i \in \D_2} A_i Y_i -  \frac{1}{n_s}\sum_{i \in \D_2} \hat{p}^{(1)}_{k_1}(X_i)\hat{b}^{(1)}_{k_2}(X_i) \\
    \hat{\psi}^{\mathrm{IF}}_{k_1, k_2} & = \frac{1}{n_s} \sum_{i \in \D_2} (A_i - \hat{p}^{(1)}_{k_1}(X_i))(Y_i - \hat{b}^{(1)}_{k_2}(X_i)) \\
    \hat{\psi}_{k}^{\mathrm{NR}} & = \frac{1}{n_s} \sum_{i \in \D_2} A_i(Y_i - \hat{b}_k^{(1)}(X_i)).
\end{align*}
In the no sample splitting case, we let $\D_1$ denote the full sample (of size $n_s = n$) and use the following estimators
\begin{align*}
    \hat{\psi}_{k_1, k_2}^{\mathrm{INT}} & = \frac{1}{n_s}\sum_{i \in \D_1} A_i Y_i -  \int \hat{p}_{k_1}^{(1)}(x)\hat{b}_{k_2}^{(1)}(x) dx \\
    \hat{\psi}^{\mathrm{MC}}_{k_1, k_2} & = \frac{1}{n_s}\sum_{i \in \D_1} A_i Y_i -  \frac{1}{n_s}\sum_{i \in \D_1} \hat{p}^{(1)}_{k_1}(X_i)\hat{b}^{(1)}_{k_2}(X_i) \\
    \hat{\psi}^{\mathrm{IF}}_{k_1, k_2} & = \frac{1}{n_s} \sum_{i \in \D_1} (A_i - \hat{p}^{(1)}_{k_1}(X_i))(Y_i - \hat{b}^{(1)}_{k_2}(X_i)) \\
    \hat{\psi}_{k}^{\mathrm{NR}} & = \frac{1}{n_s} \sum_{i \in \D_1} A_i(Y_i - \hat{b}_k^{(1)}(X_i)).
\end{align*}

We additionally apply the estimators with cross-fitting in the double sample splitting and single sample splitting cases. In the double sample splitting case, this approach applies the estimators $3! = 6$ times based on exchanging the roles of $\{\D_1, \D_2, \D_3\}$ and averaging the estimates. In the single sample splitting case, this approach applies the estimators exchanging the roles of $\{\D_1, \D_2\}$ and averaging the two estimates. 

As discussed in the main text, we adopt approximate wavelet projection estimators of the nuisance functions using the Haar basis. Specifically, we use the nuisance function estimators given by
\begin{align*}
    \hat{p}_k^{(j)}(x) = \frac{1}{n_s} \sum_{i \in \mathcal{D}_j} A_i K_{V_k}(X_i,x) \\
    \hat{b}_k^{(j)}(x) = \frac{1}{n_s} \sum_{i \in \mathcal{D}_j} Y_i K_{V_k}(X_i,x)
\end{align*}
where the projection kernel is given by
\begin{equation*}
    K_{V_{k}}(X,x) := k \sum_{\ell = 1}^k I(X \in \mathcal{I}_\ell) I(x \in \mathcal{I}_\ell), \quad x \in [0, 1]
\end{equation*}
for $\mathcal{I}_\ell = [\frac{\ell-1}{k}, \frac{\ell}{k}]$, $\ell \in \{1, \dots, k\}$.

\subsubsection{Prediction-optimal and optimal resolutions}

To select the prediction-optimal and optimal resolutions, we consider a grid of values for the $k_j$ given by $\mathcal{K}_n=\{2^l: l = 2, \dots, \lfloor\log_2 10n\rfloor\}$. The candidate values of $(k_1, k_2)$ are formed by the grid $\mathcal{K}_n \times \mathcal{K}_n$. 

We selected the prediction-optimal resolutions in each scenario (i.e., regularity regime and sample size). In each of $M$ Monte Carlo replications, we generate a sample of size $n_s$, fit $\hat{p}_{k}^{(j)}$ and $\hat{b}_{k}^{(j)}$ for each $k \in \mathcal{K}_n$ (where $\D_j$ is the full sample of size $n_s$), and estimate the mean integrated squared errors (MISEs) $\int_0^1 (p(x) - \hat{p}_k^{(j)}(x))^2 \, dx$ and $\int_0^1 (b(x) - \hat{b}_k^{(j)}(x))^2 \, dx$ by Monte Carlo integration with $10^4$ samples. That is, the MISE for   $\hat{p}^{(j)}_k$ is estimated by $\frac{1}{B}
\sum_{i=1}^{B} (p(x_i) - \hat{p}_k^{(j)}(x_i))^2$ where $x_1, \ldots, x_{B}$ are i.i.d. samples from a $\mathrm{Uniform}(0,1)$ distribution  and $B = 10^4$. The prediction-optimal $k_1$ is the $k \in \mathcal{K}_n$ that minimizes the average of the estimated MISE for $p$ over the $M$ replications; the prediction-optimal $k_2$ is defined analogously for $b$. 

We selected the optimal resolutions for each estimator in each scenario as follows. In each of $M$ Monte Carlo replications, we generate a sample of size $n$, compute the estimate of $\psi(\bbP)$ for each $(k_1, k_2) \in \mathcal{K}_n \times \mathcal{K}_n$ (or every $k \in \mathcal{K}_n$ for the Newey--Robins estimator), and compute the squared error $(\hat{\psi}_{k_1, k_2} - \psi(\bbP))^2$. The optimal $(k_1, k_2)$ is the pair that minimizes the average of these squared errors over the $M$ replications; for the Newey-Robins plug-in estimator, the optimal $k$ is the minimizer over $\mathcal{K}_n$. 

We use a value of $M = 30,000$ in the scenarios with $n = 300$ and use $M = 1,000$ in the scenarios with $n = 30,000$.

\subsection{Simulation results}

\subsubsection{Scenario: $n=300$}

Tables  \ref{tab:summary_double n=300 cv} and \ref{tab:summary_single n=300 cv} summarize the simulation results for the estimators with cross-fitting in the double sample splitting and single sample splitting cases, respectively.

\begin{table}[!h]
\centering
\caption{Simulation results for double sample split estimators with cross-fitting in settings with $n = 300$.} 
\label{tab:summary_double n=300 cv}
\begin{tabular}{llllllllllll}
  \hline
& & \multicolumn{5}{c}{Prediction-Optimal Resolutions} & \multicolumn{5}{c}{Optimal Resolutions} \\
\cmidrule(lr){3-7} \cmidrule(lr){8-12}
Regularity & Estimator & MSE & Bias\textsuperscript{2} & Var & $k_1$  & $k_2$ & MSE & Bias\textsuperscript{2} & Var  & $k_1$  & $k_2$ \\ 
  \hline
Low & Integral & 0.748 & 0.734 & 0.014 & 4 & 4 & 0.217 & 0.155 & 0.062 & 512 & 512 \\ 
   & Monte Carlo & 0.748 & 0.734 & 0.014 & 4 & 4 & 0.349 & 0.262 & 0.088 & 128 & 128 \\ 
   & First-Order & 0.748 & 0.734 & 0.014 & 4 & 4 & 0.206 & 0.155 & 0.051 & 512 & 16 \\ 
  Medium & Integral & 0.216 & 0.207 & 0.010 & 8 & 8 & 0.041 & 0.017 & 0.024 & 128 & 128 \\ 
   & Monte Carlo & 0.222 & 0.206 & 0.015 & 8 & 8 & 0.108 & 0.040 & 0.069 & 64 & 64 \\ 
   & First-Order & 0.215 & 0.207 & 0.008 & 8 & 8 & 0.031 & 0.017 & 0.014 & 128 & 8 \\ 
  High & Integral & 0.017 & 0.009 & 0.007 & 8 & 8 & 0.010 & 0.000 & 0.010 & 32 & 32 \\ 
   & Monte Carlo & 0.027 & 0.009 & 0.018 & 8 & 8 & 0.026 & 0.003 & 0.024 & 16 & 16 \\ 
   & First-Order & 0.014 & 0.009 & 0.005 & 8 & 8 & 0.005 & 0.000 & 0.005 & 32 & 8 \\ 
   \hline
\end{tabular}
\end{table}

\begin{table}[!h]
\centering
\caption{Simulation results for single sample split estimators with cross-fitting in settings with $n = 300$.} 
\label{tab:summary_single n=300 cv}
\begin{tabular}{llllllllllll}
  \hline
& & \multicolumn{5}{c}{Prediction-Optimal Resolutions} & \multicolumn{5}{c}{Optimal Resolutions} \\
\cmidrule(lr){3-7} \cmidrule(lr){8-12}
Regularity & Estimator & MSE & Bias\textsuperscript{2} & Var & $k_1$  & $k_2$ & MSE & Bias\textsuperscript{2} & Var  & $k_1$  & $k_2$ \\ 
  \hline
Low & Integral & 0.564 & 0.552 & 0.012 & 8 & 8 & 0.037 & 0.023 & 0.014 & 64 & 64 \\ 
   & Monte Carlo & 0.566 & 0.552 & 0.014 & 8 & 8 & 0.062 & 0.023 & 0.039 & 64 & 64 \\ 
   & First-Order & 0.735 & 0.718 & 0.018 & 8 & 8 & 0.244 & 0.178 & 0.067 & 512 & 4 \\ 
   & Newey-Robins & 0.646 & 0.632 & 0.014 &  & 8 & 0.236 & 0.156 & 0.079 &  & 512 \\ 
  Medium & Integral & 0.173 & 0.164 & 0.009 & 8 & 8 & 0.017 & 0.005 & 0.012 & 32 & 32 \\ 
   & Monte Carlo & 0.178 & 0.164 & 0.014 & 8 & 8 & 0.039 & 0.005 & 0.035 & 32 & 32 \\ 
   & First-Order & 0.265 & 0.254 & 0.011 & 8 & 8 & 0.042 & 0.023 & 0.019 & 128 & 4 \\ 
   & Newey-Robins & 0.216 & 0.206 & 0.010 &  & 8 & 0.045 & 0.016 & 0.029 &  & 128 \\ 
  High & Integral & 0.010 & 0.002 & 0.007 & 8 & 8 & 0.010 & 0.002 & 0.007 & 8 & 8 \\ 
   & Monte Carlo & 0.020 & 0.002 & 0.018 & 8 & 8 & 0.020 & 0.002 & 0.018 & 8 & 8 \\ 
   & First-Order & 0.027 & 0.020 & 0.007 & 8 & 8 & 0.008 & 0.001 & 0.007 & 4 & 32 \\ 
   & Newey-Robins & 0.017 & 0.009 & 0.008 &  & 8 & 0.011 & 0.000 & 0.011 &  & 32 \\ 
   \hline
\end{tabular}
\end{table}

\subsubsection{Scenario: $n=30,000$}

\paragraph{Double sample splitting}

Tables \ref{tab:summary_double n=30000} and \ref{tab:summary_double n=30000 cv} summarize the simulation results for the estimators without and with cross-fitting, respectively.

\begin{table}[!h]
\centering
\caption{Simulation results for double sample split estimators in settings with $n = 30,000$. The MSE, squared bias, and variance results are multiplied by a factor of 100.} 
\label{tab:summary_double n=30000}
\begin{tabular}{llllllllllll}
  \hline
& & \multicolumn{5}{c}{Prediction-Optimal Resolutions} & \multicolumn{5}{c}{Optimal Resolutions} \\
\cmidrule(lr){3-7} \cmidrule(lr){8-12}
Regularity & Estimator & MSE & Bias\textsuperscript{2} & Var & $k_1$  & $k_2$ & MSE & Bias\textsuperscript{2} & Var  & $k_1$  & $k_2$ \\ 
  \hline
Low & Integral & 15.659 & 15.580 & 0.079 & 512 & 512 & 1.212 & 0.388 & 0.824 & 131072 & 131072 \\ 
   & Monte Carlo & 15.654 & 15.579 & 0.075 & 512 & 512 & 2.604 & 1.576 & 1.028 & 32768 & 32768 \\ 
   & First-Order & 15.578 & 15.552 & 0.026 & 512 & 512 & 0.846 & 0.398 & 0.448 & 512 & 131072 \\ 
  Medium & Integral & 0.895 & 0.802 & 0.094 & 256 & 256 & 0.168 & 0.023 & 0.145 & 8192 & 8192 \\ 
   & Monte Carlo & 0.885 & 0.804 & 0.082 & 256 & 256 & 0.194 & 0.050 & 0.144 & 4096 & 4096 \\ 
   & First-Order & 0.819 & 0.804 & 0.014 & 256 & 256 & 0.056 & 0.022 & 0.034 & 8192 & 256 \\ 
  High & Integral & 0.079 & 0.004 & 0.075 & 64 & 64 & 0.076 & 0.000 & 0.076 & 256 & 256 \\ 
   & Monte Carlo & 0.071 & 0.004 & 0.067 & 64 & 64 & 0.068 & 0.001 & 0.068 & 128 & 128 \\ 
   & First-Order & 0.014 & 0.004 & 0.011 & 64 & 64 & 0.011 & 0.000 & 0.011 & 256 & 64 \\ 
   \hline
\end{tabular}
\end{table}

\begin{table}[!h]
\centering
\caption{Simulation results for double sample split estimators with cross-fitting in settings with $n = 30,000$. The MSE, squared bias, and variance results are multiplied by a factor of 100.} 
\label{tab:summary_double n=30000 cv}
\begin{tabular}{llllllllllll}
  \hline
& & \multicolumn{5}{c}{Prediction-Optimal Resolutions} & \multicolumn{5}{c}{Optimal Resolutions} \\
\cmidrule(lr){3-7} \cmidrule(lr){8-12}
Regularity & Estimator & MSE & Bias\textsuperscript{2} & Var & $k_1$  & $k_2$ & MSE & Bias\textsuperscript{2} & Var  & $k_1$  & $k_2$ \\ 
  \hline
Low & Integral & 15.560 & 15.551 & 0.009 & 512 & 512 & 0.492 & 0.130 & 0.362 & 262144 & 262144 \\ 
   & Monte Carlo & 15.570 & 15.552 & 0.017 & 512 & 512 & 2.180 & 1.575 & 0.604 & 32768 & 32768 \\ 
   & First-Order & 15.557 & 15.549 & 0.007 & 512 & 512 & 0.307 & 0.131 & 0.176 & 512 & 262144 \\ 
  Medium & Integral & 0.813 & 0.804 & 0.009 & 256 & 256 & 0.042 & 0.011 & 0.031 & 16384 & 16384 \\ 
   & Monte Carlo & 0.827 & 0.804 & 0.023 & 256 & 256 & 0.117 & 0.050 & 0.067 & 4096 & 4096 \\ 
   & First-Order & 0.808 & 0.804 & 0.004 & 256 & 256 & 0.021 & 0.010 & 0.011 & 16384 & 256 \\ 
  High & Integral & 0.010 & 0.004 & 0.006 & 64 & 64 & 0.006 & 0.000 & 0.006 & 256 & 256 \\ 
   & Monte Carlo & 0.018 & 0.004 & 0.015 & 64 & 64 & 0.016 & 0.000 & 0.015 & 128 & 128 \\ 
   & First-Order & 0.007 & 0.004 & 0.003 & 64 & 64 & 0.003 & 0.000 & 0.003 & 64 & 256 \\ 
   \hline
\end{tabular}
\end{table}

\paragraph{Single sample splitting}

Tables \ref{tab:summary_single n=30000} and \ref{tab:summary_single n=30000 cv} summarize the simulation results for the estimators without and with cross-fitting, respectively.

\begin{table}[!h]
\centering
\caption{Simulation results for single sample split estimators in settings with $n = 30,000$. The MSE, squared bias, and variance results are multiplied by a factor of 100.} 
\label{tab:summary_single n=30000}
\begin{tabular}{llllllllllll}
  \hline
& & \multicolumn{5}{c}{Prediction-Optimal Resolutions} & \multicolumn{5}{c}{Optimal Resolutions} \\
\cmidrule(lr){3-7} \cmidrule(lr){8-12}
Regularity & Estimator & MSE & Bias\textsuperscript{2} & Var & $k_1$  & $k_2$ & MSE & Bias\textsuperscript{2} & Var  & $k_1$  & $k_2$ \\ 
  \hline
Low & Integral & 13.078 & 13.013 & 0.064 & 512 & 512 & 0.175 & 0.072 & 0.103 & 4096 & 4096 \\ 
   & Monte Carlo & 13.075 & 13.013 & 0.062 & 512 & 512 & 0.197 & 0.073 & 0.124 & 4096 & 4096 \\ 
   & First-Order & 18.365 & 18.348 & 0.016 & 512 & 512 & 0.634 & 0.168 & 0.466 & 64 & 262144 \\ 
   & Newey-Robins & 15.598 & 15.568 & 0.030 &  & 512 & 0.760 & 0.398 & 0.362 &  & 131072 \\ 
  Medium & Integral & 0.606 & 0.524 & 0.082 & 256 & 256 & 0.144 & 0.051 & 0.092 & 1024 & 1024 \\ 
   & Monte Carlo & 0.597 & 0.522 & 0.075 & 256 & 256 & 0.140 & 0.052 & 0.089 & 1024 & 1024 \\ 
   & First-Order & 1.143 & 1.133 & 0.009 & 256 & 256 & 0.047 & 0.014 & 0.032 & 32 & 16384 \\ 
   & Newey-Robins & 0.830 & 0.799 & 0.031 &  & 256 & 0.072 & 0.021 & 0.051 &  & 8192 \\ 
  High & Integral & 0.061 & 0.001 & 0.061 & 64 & 64 & 0.061 & 0.001 & 0.061 & 64 & 64 \\ 
   & Monte Carlo & 0.056 & 0.000 & 0.056 & 64 & 64 & 0.056 & 0.000 & 0.056 & 64 & 64 \\ 
   & First-Order & 0.018 & 0.011 & 0.007 & 64 & 64 & 0.008 & 0.000 & 0.008 & 256 & 16 \\ 
   & Newey-Robins & 0.030 & 0.004 & 0.026 &  & 64 & 0.026 & 0.001 & 0.026 &  & 128 \\ 
   \hline
\end{tabular}
\end{table}

\begin{table}[!h]
\centering
\caption{Simulation results for single sample split estimators with cross-fitting in settings with $n = 30,000$. The MSE, squared bias, and variance results are multiplied by a factor of 100.} 
\label{tab:summary_single n=30000 cv}
\begin{tabular}{llllllllllll}
  \hline
& & \multicolumn{5}{c}{Prediction-Optimal Resolutions} & \multicolumn{5}{c}{Optimal Resolutions} \\
\cmidrule(lr){3-7} \cmidrule(lr){8-12}
Regularity & Estimator & MSE & Bias\textsuperscript{2} & Var & $k_1$  & $k_2$ & MSE & Bias\textsuperscript{2} & Var  & $k_1$  & $k_2$ \\ 
  \hline
Low & Integral & 12.988 & 12.979 & 0.009 & 512 & 512 & 0.091 & 0.076 & 0.015 & 4096 & 4096 \\ 
   & Monte Carlo & 13.002 & 12.985 & 0.017 & 512 & 512 & 0.127 & 0.076 & 0.051 & 4096 & 4096 \\ 
   & First-Order & 18.357 & 18.348 & 0.009 & 512 & 512 & 0.461 & 0.172 & 0.289 & 262144 & 64 \\ 
   & Newey-Robins & 15.560 & 15.551 & 0.009 &  & 512 & 0.644 & 0.397 & 0.247 &  & 131072 \\ 
  Medium & Integral & 0.532 & 0.523 & 0.009 & 256 & 256 & 0.062 & 0.052 & 0.011 & 1024 & 1024 \\ 
   & Monte Carlo & 0.546 & 0.522 & 0.024 & 256 & 256 & 0.085 & 0.052 & 0.033 & 1024 & 1024 \\ 
   & First-Order & 1.138 & 1.133 & 0.005 & 256 & 256 & 0.030 & 0.012 & 0.018 & 16 & 16384 \\ 
   & Newey-Robins & 0.807 & 0.798 & 0.009 &  & 256 & 0.044 & 0.021 & 0.023 &  & 8192 \\ 
  High & Integral & 0.007 & 0.000 & 0.006 & 64 & 64 & 0.007 & 0.000 & 0.006 & 64 & 64 \\ 
   & Monte Carlo & 0.015 & 0.000 & 0.015 & 64 & 64 & 0.015 & 0.000 & 0.015 & 64 & 64 \\ 
   & First-Order & 0.014 & 0.011 & 0.004 & 64 & 64 & 0.004 & 0.000 & 0.004 & 256 & 8 \\ 
   & Newey-Robins & 0.010 & 0.004 & 0.006 &  & 64 & 0.007 & 0.000 & 0.007 &  & 256 \\ 
   \hline
\end{tabular}
\end{table}

\paragraph{No sample splitting}

Table \ref{tab:summary_no n=30000} summarizes the simulation results for the estimators in the no sample splitting case.

\begin{table}[!h]
\centering
\caption{Simulation results for the estimators without sample splitting in settings with $n = 30,000$. The MSE, squared bias, and variance results are multiplied by a factor of 100.} 
\label{tab:summary_no n=30000}
\begin{tabular}{llllllllllll}
  \hline
& & \multicolumn{5}{c}{Prediction-Optimal Resolutions} & \multicolumn{5}{c}{Optimal Resolutions} \\
\cmidrule(lr){3-7} \cmidrule(lr){8-12}
Regularity & Estimator & MSE & Bias\textsuperscript{2} & Var & $k_1$  & $k_2$ & MSE & Bias\textsuperscript{2} & Var  & $k_1$  & $k_2$ \\ 
  \hline
Low & Integral & 9.567 & 9.558 & 0.009 & 1024 & 1024 & 0.511 & 0.495 & 0.016 & 8192 & 8192 \\ 
   & Monte Carlo & 6.954 & 6.932 & 0.022 & 1024 & 1024 & 0.094 & 0.054 & 0.040 & 512 & 16384 \\ 
   & First-Order & 12.612 & 12.605 & 0.006 & 1024 & 1024 & 0.020 & 0.010 & 0.010 & 8192 & 16 \\ 
   & Newey-Robins & 9.567 & 9.558 & 0.009 &  & 1024 & 0.511 & 0.495 & 0.016 &  & 8192 \\ 
  Medium & Integral & 0.237 & 0.229 & 0.009 & 512 & 512 & 0.023 & 0.014 & 0.009 & 1024 & 1024 \\ 
   & Monte Carlo & 0.052 & 0.025 & 0.027 & 512 & 512 & 0.028 & 0.000 & 0.028 & 512 & 1024 \\ 
   & First-Order & 0.641 & 0.637 & 0.004 & 512 & 512 & 0.005 & 0.000 & 0.005 & 8192 & 256 \\ 
   & Newey-Robins & 0.237 & 0.229 & 0.009 &  & 512 & 0.023 & 0.014 & 0.009 &  & 1024 \\ 
  High & Integral & 0.008 & 0.002 & 0.006 & 64 & 64 & 0.007 & 0.000 & 0.006 & 128 & 128 \\ 
   & Monte Carlo & 0.015 & 0.000 & 0.015 & 64 & 64 & 0.015 & 0.000 & 0.015 & 64 & 64 \\ 
   & First-Order & 0.010 & 0.007 & 0.003 & 64 & 64 & 0.004 & 0.000 & 0.004 & 32 & 2048 \\ 
   & Newey-Robins & 0.008 & 0.002 & 0.006 &  & 64 & 0.007 & 0.000 & 0.006 &  & 128 \\ 
   \hline
\end{tabular}
\end{table}

\end{appendix}


\bibliographystyle{imsart-number} 
\bibliography{bibliography}       
